\documentclass[10pt]{amsart}
\usepackage{amsmath}
\usepackage{amscd,amsthm,amssymb,amsfonts}
\usepackage{mathrsfs}
\usepackage{dsfont}
\usepackage{stmaryrd}
\usepackage{euscript}
\usepackage{expdlist}
\usepackage{enumerate}

\input xy
\xyoption{all}
\usepackage[OT2,T1]{fontenc}

\theoremstyle{plain}
\newtheorem{thm}{Theorem}[section]
\newtheorem{thmA}{Theorem}

\newtheorem*{thm*}{Theorem}

\newtheorem{lm}[thm]{Lemma}
\newtheorem{cor}[thm]{Corollary}
\newtheorem*{cor*}{Corollary}
\newtheorem{prop}[thm]{Proposition}
\newtheorem*{conj*}{Conjecture}

\theoremstyle{remark}

\newtheorem*{thank}{Acknowledgments}

\theoremstyle{definition}
\newtheorem*{defn*}{Definition}
\newtheorem{Remark}[thm]{Remark}
\newtheorem{I_Remark*}{Remark}
\newtheorem{defn}[thm]{Definition}

\newtheorem*{hypothesis*}{Hypothesis}
\newtheorem{hypothesis}[thm]{Hypothesis}

\newcommand{\nc}{\newcommand}

\newcommand{\beq}{\begin{equation}}
\newcommand{\eeq}{\end{equation}}
\newcommand{\bpmx}{\begin{pmatrix}}
\newcommand{\epmx}{\end{pmatrix}}
\newcommand{\bbmx}{\begin{bmatrix}}
\newcommand{\ebmx}{\end{bmatrix}}
\newcommand{\wh}{\widehat}
\newcommand{\wtd}{\widetilde}

\newcommand{\beqcd}[1]{\begin{equation*}\label{#1}\tag{#1}}
\newcommand{\eeqcd}{\end{equation*}}

\numberwithin{equation}{section}
\newenvironment{mylist}{
  \begin{enumerate}{}{%
      \setlength{\itemsep}{5pt} \setlength{\parsep}{0in}
      \setlength{\parskip}{0in} \setlength{\topsep}{0in}
      \setlength{\partopsep}{0in}
      \setlength{\leftmargin}{0.17in}}}{\end{enumerate}}

\def\parref#1{\ref{#1}}
\def\thmref#1{Theorem~\parref{#1}}

\def\propref#1{Proposition~\parref{#1}}
\def\corref#1{Corollary~\parref{#1}}     \def\remref#1{Remark~\parref{#1}}
\def\secref#1{\S\parref{#1}}

\def\lmref#1{Lemma~\parref{#1}}
\def\subsecref#1{\S\parref{#1}}
\def\defref#1{Definition~\parref{#1}}

\def\hypref#1{Hypothesis~\parref{#1}}

\def\makeop#1{\expandafter\def\csname#1\endcsname
  {\mathop{\rm #1}\nolimits}\ignorespaces}
\makeop{Hom}   \makeop{End}   \makeop{Aut}   %\makeop{Isom}
\makeop{Pic} \makeop{Gal}       \makeop{Div} \makeop{Lie}
\makeop{PGL}   \makeop{Corr} \makeop{PSL} \makeop{sgn} \makeop{Spf}
 \makeop{Tr} \makeop{Nr} \makeop{Fr} \makeop{disc}
\makeop{Proj} \makeop{supp} \makeop{ker}   \makeop{Im} \makeop{dom}
\makeop{coker} \makeop{Stab} \makeop{SO} \makeop{SL} \makeop{SL}
\makeop{Cl}    \makeop{cond} \makeop{Br} \makeop{inv} \makeop{rank}
\makeop{id}    \makeop{Fil} \makeop{Frac}  \makeop{GL} \makeop{SU}
\makeop{Trd}   \makeop{Sp} \makeop{Tr}    \makeop{Trd} \makeop{Res}
\makeop{ind} \makeop{depth} \makeop{Tr} \makeop{st} \makeop{Ad}
\makeop{Int} \makeop{tr}    \makeop{Sym} \makeop{can} \makeop{SO}
\makeop{torsion} \makeop{GSp} \makeop{Tor}\makeop{Ker} \makeop{rec}
\makeop{Ind} \makeop{Coker}
 \makeop{vol} \makeop{Ext} \makeop{gr} \makeop{ad}
 \makeop{Gr}\makeop{corank} \makeop{Ann}
\makeop{Hol} %Holomorphic
\makeop{Fitt} \makeop{Mp} \makeop{CAP}

\def\ord{{ord}}

\def\Ord{{\mathrm{ord}}}
\def\Spec{\mathrm{Spec}\,}

\DeclareMathAlphabet{\mathpzc}{OT1}{pzc}{m}{it}
\DeclareSymbolFont{cyrletters}{OT2}{wncyr}{m}{n}
\DeclareMathSymbol{\SHA}{\mathalpha}{cyrletters}{"58}

\def\makebb#1{\expandafter\def
  \csname bb#1\endcsname{{\mathbb{#1}}}\ignorespaces}
\def\makebf#1{\expandafter\def\csname bf#1\endcsname{{\bf
      #1}}\ignorespaces}
\def\makegr#1{\expandafter\def
  \csname gr#1\endcsname{{\mathfrak{#1}}}\ignorespaces}
\def\makescr#1{\expandafter\def
  \csname scr#1\endcsname{{\EuScript{#1}}}\ignorespaces}
\def\makecal#1{\expandafter\def\csname cal#1\endcsname{{\mathcal
      #1}}\ignorespaces}
\def\doLetters#1{#1A #1B #1C #1D #1E #1F #1G #1H #1I #1J #1K #1L #1M
                 #1N #1O #1P #1Q #1R #1S #1T #1U #1V #1W #1X #1Y #1Z}
\def\doletters#1{#1a #1b #1c #1d #1e #1f #1g #1h #1i #1j #1k #1l #1m
                 #1n #1o #1p #1q #1r #1s #1t #1u #1v #1w #1x #1y #1z}
\doLetters\makebb   \doLetters\makecal  \doLetters\makebf
\doLetters\makescr
\doletters\makebf   \doLetters\makegr   \doletters\makegr

\normalsize

\makeop{Ram} \makeop{Rep} \makeop{mass}

\makeop{Bl}
\def\abs#1{\left|#1\right|}
\def\norm#1{\lVert#1\rVert}

\def\Fpbar{\bar{\mathbb F}_p}

\def\Qbarp{\C_p}
\def\Qp{\Q_p}
\def\Qbar{\ol{\Q}}
\def\Zbar{\ol{\Z}}

\def\Zp{\Z_p}

\def\rmT{{\mathrm T}}
\def\rmN{{\mathrm N}}

\def\cA{{\mathcal A}}  %automorphic forms
\def\cB{\mathcal B}
\def\cD{\mathcal D}
\def\cE{{\mathcal E}}
\def\cF{{\mathcal F}}  %Hida family
\def\cG{{\mathcal G}}
\def\cL{{\mathcal L}}
\def\cH{{\mathcal H}}
\def\cI{\mathcal I}
\def\cJ{\mathcal J}
\def\cK{{\mathcal K}}  %imaginary quadratic field
\def\cM{\mathcal M}
\def\cR{{\mathcal R}}
\def\cO{\mathcal O}
\def\cS{{\mathcal S}}
\def\cf{{\mathcal f}}
\def\cW{{\mathcal W}}

\def\cX{\mathcal X}
\def\cV{{\mathcal V}}
\def\cP{{\mathcal P}}

\def\cJ{\mathcal J}
\def\cN{\mathcal N}

\def\cQ{\mathcal Q}
\def\cU{\mathcal U}

\def\bfc{\mathbf c}

\def\bfH{\mathbf H}

\def\bfK{\mathbf K}

\def\bfU{\mathbf U}
\def\bfT{{\mathbf T}}

\def\bda{\mathbf a}
\def\bff{\mathbf f}

\def\bfu{\mathbf u}

\def\bdS{\mathbf S}

\def\bfw{\mathbf w}

\def\bftheta{\boldsymbol{\theta}}

\def\sF{\mathscr F}
\def\sL{\mathscr L}
\def\sI{\mathscr I}
\def\sJ{\mathscr J}

\def\sA{\mathscr A}
\def\sV{\mathscr V}

\def\sB{\mathscr B}

\def\sS{\mathscr S}

\def\sW{{\mathscr W}}

\def\bbI{\mathbb I}

\newcommand{\Z}{\mathbf Z}
\newcommand{\Q}{\mathbf Q}
\newcommand{\R}{\mathbf R}
\newcommand{\C}{\mathbf C}
\newcommand{\A}{\mathbf A}    % for adele

\def\cond{{c}}

\def\fraka{{\mathfrak a}}

\def\frakc{{\mathfrak c}}

\def\frakf{\mathfrak f}

\def\frakp{{\mathfrak p}}

\def\frakq{\mathfrak q}

\def\frakg{\mathfrak g}
\def\frakm{\mathfrak m}

\def\frakH{{\mathfrak H}}

\def\frakX{\mathfrak X}

\def\frakN{\mathfrak N}

\def\bfone{{\mathbf 1}}

\def\one{\mathbf 1} %trivial character

\def\ulN{\ul{N}}

\def\et{{\acute{e}t}}

\def\etale{{\'{e}tale }}

\def\pont{Pontryagin } % Pontryagin dual
\def\padic{\text{$p$-adic }}

\def\BS{Bruhat-Schwartz }

\def\Teich{Teichm\"{u}ller }

\def\Frob{\mathrm{Frob}}
\newcommand{\<}{\langle}   %\< is not defined yet.
\renewcommand{\>}{\rangle} %\> is already defined.

\def\ot{\otimes}

\def\hookto{\hookrightarrow}
\def\longto{\longrightarrow}
\def\ol{\overline}  \nc{\opp}{\mathrm{opp}} \nc{\ul}{\underline}

\newcommand{\pair}[2]{\< #1, #2\>}
\newcommand{\bbpair}[2]{\<\!\<#1,#2\>\!\>}

\newcommand{\pairing}{\pair{\,}{\,}}

\def\XYmatrix{\xymatrix@M=8pt} % make \xymatrix not too cluttered
\def\ncmd{\newcommand}
\ncmd{\xysubset}[1][r]{\ar@<-2.5pt>@{^(-}[#1]\ar@<2.5pt>@{_(-}[#1]}
\ncmd{\XYmatrixc}[1]{\vcenter{\XYmatrix{#1}}}
\ncmd{\xyto}[1][r]{\ar@{->}[#1]}
\ncmd{\xyinj}[1][r]{\ar@{^(->}[#1]}
\ncmd{\xysurj}[1][r]{\ar@{->>}[#1]}
\ncmd{\xyline}[1][r]{\ar@{-}[#1]}
\ncmd{\xydotsto}[1][r]{\ar@{.>}[#1]}
\ncmd{\xydots}[1][r]{\ar@{.}[#1]}
\ncmd{\xyleadsto}[1][r]{\ar@{~>}[#1]}
\ncmd{\xyeq}[1][r]{\ar@{=}[#1]} \ncmd{\xyequal}[1][r]{\ar@{=}[#1]}
\ncmd{\xyequals}[1][r]{\ar@{=}[#1]}
\ncmd{\xymapsto}[1][r]{l\ar@{|->}[#1]}\ncmd{\xyimplies}[1][r]{\ar@{=>}[#1]}
\ncmd{\xyiso}{\ar[r]_-{\sim}}
\def\injxy{\ar@{^(->}}

\newcommand{\pMX}[4]{\begin{pmatrix}
{#1}& {#2}\\
{#3}&{#4}\end{pmatrix} }

 \newcommand{\pDII}[2]{\begin{pmatrix}{#1}&0
 \\0&{#2}\end{pmatrix}}

\newcommand{\seesaw}[4]{{#1}\ar@{-}[rd]\ar@{-}[d]&{#2}\ar@{-}[d]\\
{#3}\ar@{-}[ru]&{#4}}

\def\ie{i.e. }

\def\cf{\mbox{{\it cf.} }}
\def\loccit{\mbox{{\it loc.cit.} }}

\def\uf{\varpi} %uniformizer
\def\Abs{{|\!\cdot\!|}} %adelic absolute value
\def\ndivides{\nmid}

\def\x{{\times}}

\def\onehalf{{\frac{1}{2}}}
\def\al{\alpha}

\def\Lam{\Lambda}

\def\om{\omega}
\def\dirlim{\varinjlim}
\def\prolim{\varprojlim}
\def\iso{\simeq}
\def\con{\equiv}
\def\bksl{\backslash}
\newcommand\stt[1]{\left\{#1\right\}}
\def\ep{\epsilon}
\def\varep{\varepsilon}

\def\pd{\partial}
\def\lam{\lambda}
\def\sg{\sigma}

\def\disjoint{\sqcup}
\def\bigot{\bigotimes}

\def\dx{{\rm d}^\x}

\newcommand{\powerseries}[1]{\llbracket{#1}\rrbracket}

\renewcommand\pmod[1]{\,(\mbox{mod }{#1})}

\renewcommand\Re{\text{Re}\,}
\newcommand\Dmd[1]{\left<{#1}\right>} %Diamond operator

\def\Cp{\C_p}

\def\Qq{\Q_q} 
\def\pmq{q}
\def\cond{{\rm cond}}

\DeclareMathOperator{\lcm}{lcm}
\usepackage{hyperref}
\usepackage{color}
\def\Bd{\boldsymbol d}
\def\rmd{{\rm d}}

\newcommand\aone[1]{\pDII{#1}{1}}
\renewcommand\cond[1]{c(#1)}

\def\ulQ{{\ul{Q}}}

\def\ulk{\ul{k}}
\def\Qtn{D^\x}
\def\Om{\boldsymbol \omega}
\def\Qx{{Q_1}}
\def\Qy{{Q_2}}
\def\Qz{{Q_3}}
\def\unb{{\rm unb}}
\def\bal{{\rm bal}}

\def\Tau{\boldsymbol\tau}

\def\addchar{\boldsymbol\psi}
\def\newW{W}

\def\bfa{\mathbf a}
\def\bdsH{{\boldsymbol H}}
\def\bdsf{{\boldsymbol f}}
\def\bdsg{{\boldsymbol g}}
\def\bdsh{{\boldsymbol h}}
\def\cyc{\boldsymbol\varepsilon_{\rm cyc}}
\newcommand\Prin[2]{#1\boxplus #2}
\def\pmq{\ell}
\def\LR{V}
\def\sW{\cW}
\def\ad{{\rm Ad}}
\def\val{v}
\def\ord{{\rm ord}}
\def\itPi{\mathit{\Pi}}
\newcommand\Contra[1]{\wtd{#1}}
\def\brchf{\psi_1}
\def\brchg{\psi_2}
\def\brchh{\psi_3}
\def\condf{N_1}
\def\condg{N_2}
\def\condh{N_3}
\def\fQx{f}
\def\gQy{g}
\def\hQz{h}
\def\bdsF{\boldsymbol F}

\def\divides{\mid}
\def\itPhi{{\mathit \Phi}}
\def\Bkappa{\Bbbk}

\def\bftr{t_n}
\def\bftn{t_n}
\def\fgh{F}
\def\rmH{{\rm H}}
\def\eord{e}
\def\sS{\cS}
\def\Qbarp{\ol{\Q}_p}
\def\cls{{\rm cls}}
\def\ari{+}
\def\opcpt{U}
\title[Hida families and $p$-adic triple product $L$-functions]{Hida families and  $p$-adic triple product $L$-functions}
\author{Ming-Lun Hsieh}
\date{\today}
\address{ 
Institute of Mathematics, Academia Sinica~\\ Taipei 10617, Taiwan\and National Center for Theoretic Sciences~
}
\email{mlhsieh@math.sinica.edu.tw}

\thanks{This work was partially supported by a MOST grant 103-2115-M-002-012-MY5.}

\begin{document}
\maketitle
\begin{abstract}
We construct the three-variable $p$-adic triple product $L$-functions attached to Hida families of elliptic newforms and prove the explicit interpolation formulae at all critical specializations by establishing explicit Ichino's formulae for the trilinear period integrals of automorphic forms. Our formulae perfectly fit the conjectural shape of $p$-adic $L$-functions predicted by Coates and Perrin-Riou. As an application, we prove the factorization of certain unbalanced \padic triple product $L$-functions into a product of anticyclotomic \padic $L$-functions for modular forms. By this factorization, we obtain a construction of the square root of the anticyclotomic \padic $L$-functions for elliptic curves in the definite case via the diagonal cycle Euler system \`a la Darmon and Rotger and obtain a Greenberg-Stevens style proof of anticyclotomic exceptional zero conjecture for elliptic curves due to Bertolini and Darmon.  
\end{abstract}
\tableofcontents
%!TEX root = TRIPLE3.tex
\def\WD{{\rm WD}}
\def\Sh{\prime}
\section{Introduction}
The aim of this paper is to construct the three-variable $p$-adic triple product $L$-functions attached to Hida families of ellptic newforms in the unbalanced and balanced case with explicit interpolation formulae at all critical specializations. Let $p$ be an \emph{odd} prime. Let $\cO$ be a valuation ring finite flat over $\Zp$. Let $\bfI$ be a normal domain finite flat over the Iwasawa algebra $\Lam=\cO\powerseries{\Gamma}$ of the topological group $\Gamma=1+p\Zp$. Let
\[\bdsF=(\bdsf,\bdsg,\bdsh)\]
be the triplet of primitive Hida families of tame conductor $(\condf,\condg,\condh)$ and nebentypus $(\brchf,\brchg,\brchh)$ with coefficients in $\bfI$. Roughly speaking, we construct a three-variable Iwasawa function over the weight space of $\bdsF$ interpolating the square root of the algebraic part of central values of the triple product $L$-function attached to $\bdsF_Q$ and prove explicit interpolation formulae at all critical specializations. We would like to emphasize that our formulae completely comply with the conjectural form described in \cite{CP89}, \cite{Coates89Bourbaki} and \cite{Coates89II} and is compatible with other known $p$-adic $L$-functions. For example, when $\bdsg$ and $\bdsh$ are primitive Hida families of CM forms by some imaginary quadratic field, we show that the unbalanced $p$-adic $L$-function is the product of theta elements \`a la Bertolini-Darmon. In order to state our result precisely, we need to introduce some notation from Hida theory for elliptic modular forms and technical items such as the modified Euler factors at $p$ and the canonical periods of Hida families in the theory of \padic $L$-functions. 
\subsection{Galois representations attached to Hida families}
If $\cF=\sum_{n\geq 1}\bfa(n,\cF)q^n\in\bfI\powerseries{q}$ is a primitive cuspidal Hida family of tame conductor $N_\cF$, let $\rho_\cF:G_\Q=\Gal(\Qbar/\Q)\to\GL_2(\Frac\bfI)$ be the associated big Galois representation such that $\Tr\rho_\cF(\Frob_\ell)=\bfa(\ell,\cF)$ for primes $\ell\ndivides N_\cF$, where $\Frob_\ell$ is the geometric Frobenius at $\ell$ and let $V_\cF$ denote the natural realization of $\rho_\cF$ inside the \etale cohomology groups of modular curves. Thus, $V_\cF$ is a lattice in $(\Frac\bfI)^2$ with the continuous Galois action via $\rho_\cF$, and the $\Gal(\Qbarp/\Qp)$-invariant subspace $\Fil^0V_\cF:= V_\cF^{I_p}$ fixed by the inertia group $I_p$ at $p$ is free of rank one over $\bfI$ (\cite[Corollary, page 558]{Ohta00})). We recall the specialization of $V_\cF$ at arithmetic points. A point $Q\in\Spec\bfI(\Qbarp)$ is called an arithmetic point of weight $k_Q$ and finite part $\ep_Q$ if $Q|_{\Gamma}\colon \Gamma\to \Lam^\x{\stackrel Q\longto}\Qbarp^\x$  is given by $Q(x)=x^{k_Q}\ep_Q(x)$ for some integer $k_Q\geq 2$ and a finite order character $\ep_Q:\Gamma\to\Qbarp^\x$. %For an arithmetic point $Q$, denote by $k_Q$ the weight of $Q$ and $\ep_Q$ the finite part of $Q$.  
Let $\frakX_\bfI^+$ be the set of arithmetic points of $\bfI$. For each arithmetic point $Q\in\frakX_\bfI^+$, the specialization $V_{\cF_Q}:=V_{\cF}\ot_{\bfI,Q}\Qbarp$ is the geometric \padic Galois representation associated with the eigenform $\cF_Q$ of constructed by Shimura and Deligne.

\subsection{Triple product $L$-functions}\label{S:tri.1} 
Let $\bfV=V_\bdsf\wh\ot_\cO V_\bdsg\wh\ot_\cO V_\bdsh$ be the triple product Galois representation of rank eight over $\cR$ a finite extension of the three-variable Iwasawa algebra given by
 \[\cR=\bfI\wh\ot_\cO\bfI\wh\ot_\cO\bfI.\]
 % and let \[\bfV=V_\bdsf\ot_{\cO} V_\bdsg\ot_{\cO} V_\bdsh\] be the rank eight $\cR$-adic representation of $G_\Q$. Let $\cyc:G_\Q\to\Zp^\x$ be the \padic cyclotomic character. Let $\cX:G_\Q\to\cR^\x$ be the unique square root of the $\cR$-adic character $(\det \bfV)\cyc^{-1}$ such that $\cX(c)=(-1)^a$ ($c$ is the complex conjugagtion).
 Let $\frakX^+_\cR\subset\Spec\cR(\Qbarp)$ be the weight space of arithmetic points of $\cR$ given by 
\[\frakX_\cR^+:=\stt{\ulQ=(\Qx,\Qy,\Qz)\in(\frakX_\bfI^+)^3\mid k_{\Qx}+k_{\Qy}+k_{\Qz}\con 0\pmod{2}}.\]
For each arithmetic point $\ulQ=(\Qx,\Qy,\Qz)\in\frakX^+_\cR$, 
%let $V_{\bdsf_\Qx}$, $V_{\bdsg_\Qy}$ and $V_{\bdsh_\Qz}$ denote the two dimensional $p$-adic Galois representations associated with $\bdsf_\Qx$, $\bdsg_\Qy$ and $\bdsh_\Qz$ of pure weight $1-k_\Qx$, $1-k_\Qy$ and $1-k_\Qz$. 
the specialization $\bfV_\ulQ=V_{\bdsf_\Qx}\ot V_{\bdsg_\Qy}\ot V_{\bdsh_\Qz}$ is a $p$-adic geometric Galois representation of pure weight $w_\ulQ:=k_\Qx+k_\Qy+k_\Qz-3$. Let $\Om:(\Z/p\Z)^\x\to\mu_{p-1}$ be the \Teich character. We assume that 
\beqcd{ev}\brchf\brchg\brchh=\Om^{2a}\text{ for some }a\in\Z.\eeqcd Then \eqref{ev} implies that the determinant $\det \bfV=\cX^2\cyc$, where $\cyc$ is the $p$-adic cyclotomic character and $\cX$ is a $\cR$-adic $p$-ramified Galois character with $\cX(\bfc)=(-1)^{a}$ ($\bfc$ is the complex conjugation). Note that the specialization of $\cX$ at $\ulQ$ can be written as the product $\cX_{\ulQ}=\chi_{\ulQ}\cyc^{-\frac{w_\ulQ+1}{2}}$ with a finite order character $\chi_{\ulQ}$. We consider the critical twist \[\bfV^\dagger=\bfV\ot\cX^{-1}.\]  
Then $\bfV^\dagger$ is self-dual in the sense that $(\bfV^\dagger)^\vee(1)=\bfV^\dagger$. Next we briefly recall the complex $L$-function associated with the specialization $\bfV^\dagger_\ulQ$. For each place $\ell$, denote by $W_{\Q_\ell}$ the Weil-Deligne group of $\Q_\ell$. To the geometric $p$-adic Galois representation $\bfV_\ulQ^\dagger$, we can associate the Weil-Deligne representation $\WD_\ell(\bfV_\ulQ^\dagger)$ of $W_{\Q_\ell}$ over $\Qbarp$ (See \cite[(4.2.1)]{Tate79Corvallis} for $\ell\not =p$ and \cite[(4.2.3)]{F94} for $\ell=p$). Fixing an isomorphism $\iota_p:\Qbarp\iso\C$ once and for all, we define the complex $L$-function of $\bfV_\ulQ^\dagger$ by the Euler product \[L(\bfV_\ulQ^\dagger,s)=\prod_{\ell<\infty}L_\ell(\bfV_\ulQ^\dagger,s)\] of the local $L$-factors $L_\ell(\bfV_\ulQ^\dagger,s)$ attached to 
$\WD_\ell(\bfV_\ulQ^\dagger)\ot_{\Qbarp,\iota_p}\C$ (\cite[(1.2.2)]{Deligne79}, \cite[page 85]{Taylor04ICM}). On the other hand, we denote by $\pi_{\bdsf_\Qx}=\ot_v\pi_{\bdsf_\Qx,v}$ (resp. $\pi_{\bdsg_\Qx},\pi_{\bdsh_\Qz}$) the irreducible unitary cuspidal automorphic representation of $\GL_2(\A)$ associated with $\bdsf_\Qx$ (resp. $\bdsg_\Qy,\bdsh_\Qz$) and let
\[\itPi_\ulQ=\pi_{\bdsf_\Qx}\times\pi_{\bdsg_\Qy}\times\pi_{\bdsh_\Qz}\ot\chi_\ulQ^{-1}\]
be the irreducible unitary automorphic representation of $\GL_2(\A)\times\GL_2(\A)\times\GL_2(\A)$.
Denote by $L(s,\itPi_\ulQ)$ the automorphic $L$-function defined by Garrett, Piateski-Shapiro and Rallis attached to the triple product $\itPi_\ulQ$. The analytic theory of $L(s,\itPi_\ulQ)$ such as functional equations and analytic continuation has been explored extensively in the literature (\cf\cite{PSR87}), and thanks to \cite[Theorem 4.4.1]{Dinakar00}, we have \[L(s+\frac{1}{2},\itPi_\ulQ)={\mathit \Lambda}(\bfV^\dagger_\ulQ,s):=\Gamma_{\bfV^\dagger_\ulQ}(s)\cdot L(\bfV^\dagger_\ulQ,s).\]
Here $\Gamma_{\bfV^\dagger_\ulQ}(s)$ is the archimedean $L$-factor of $\bfV^\dagger_\ulQ$ and is a finite product of four classical $\Gamma$-functions (see \eqref{E:Gamma.1}). Moreover, there is a positive integer $N(\bfV_{\ulQ}^\dagger)$ and the root number $\varepsilon(\bfV_\ulQ^\dagger)\in\stt{\pm 1}$ such that the complete $L$-function $\mathit \Lambda(\bfV_\ulQ,s)$ satisfies the functional equation
\[{\mathit \Lambda}(\bfV_\ulQ^\dagger,s)=\varepsilon(\bfV_{\ulQ}^\dagger)\cdot N(\bfV_{\ulQ}^\dagger)^{-s}\cdot {\mathit \Lambda}(\bfV_\ulQ^\dagger,-s).\]
We thus have a good understanding of the complex analytic behavior of $L(\bfV^\dagger_\ulQ,s)$. On the arithmetic side, Deligne's conjecture for the critical central value $L(\bfV_\ulQ^\dagger,0)$ has been proved in \cite{HK91Triple}. In this article, we shall investigate the $p$-adic analytic behavior of the algebraic part of $L(\bfV_\ulQ^\dagger,0)$ viewed as a function on the weight space $\frakX_\cR^+$. It is natural to first consider the behavior of the root number $\varepsilon(\bfV_\ulQ^\dagger)$ of $\bfV_\ulQ^\dagger$ (or $\itPi_\ulQ$) over the weight space. The global root number 
\[\varepsilon(\bfV_\ulQ^\dagger)=\prod_{\ell\leq \infty}\varepsilon(\WD_\ell(\bfV_\ulQ^\dagger))\]
is defined as the product of local constants, where $\varepsilon(?)$ is the local epsilon factor attached to a Weil-Deligne representation (\cf\cite[page 21]{Tate79Corvallis}) with respect to the standard choice of a non-trivial additive character of $\Qp$ and measures on $\Qp$ in  \cite[5.3]{Deligne79}.  For each arithmetic point $\ulQ\in\frakX^+_\cR$, we put \[\Sigma^-(\ulQ):=\stt{\pmq\colon\text{ prime factor of }\condf\condg\condh\mid \varepsilon(\WD_\ell(\bfV_\ulQ^\dagger))=-1}. \]
It is known that there is a subset $\Sigma^-$ of prime factors of $\condf\condg\condh$ such that $\Sigma^-=\Sigma^-(\ulQ)$ for all $\ulQ\in\frakX^+_\cR$.
For the archimedean root number, we partition the weight space $\frakX_\cR^+$ into $\frakX_\cR^\bdsf\disjoint\frakX_\cR^\bdsg\disjoint\frakX_\cR^\bdsh\disjoint\frakX_\cR^\bal$, where $\frakX_\cR^\bdsf$ is the unbalanced range dominated by $\bdsf$ given by 
\begin{align*}\frakX_\cR^\bdsf=&\stt{(\Qx,\Qy,\Qz)\in \frakX_\cR^+\mid k_\Qx+ k_\Qy+k_\Qz\leq 2k_\Qx}
\intertext{($\frakX_\cR^\bdsg$ and $\frakX_\cR^\bdsh$ are defined likewise), and $\frakX_\cR^\bal$ is the balanced range}
\frakX_\cR^\bal=&\stt{(\Qx,\Qy,\Qz)\in\frakX_\cR^+\mid k_\Qx+k_\Qy+k_\Qz>2k_{Q_i}\text{ for all }i=1,2,3}.\end{align*} The union $\frakX_\cR^{\rm unb}:=\frakX_\cR^\bdsf\disjoint\frakX_\cR^\bdsg\disjoint\frakX_\cR^\bdsh$ is called the unbalanced range. Then we know that
\begin{align*}\varepsilon(\WD_\infty(\bfV_\ulQ^\dagger))&=+1\text{ if }\ulQ\in\frakX_\cR^{\rm unb};\\
\varepsilon(\WD_\infty(\bfV_\ulQ^\dagger))&=-1\text{ if }\ulQ\in\frakX_\cR^{\rm bal}.
\end{align*}

\subsection{The modified Euler factors at $p$ and $\infty$}\label{S:mod.1}
Let $G_{\Qp}$ be the decomposition group at $p$. We consider the following rank four $G_{\Qp}$-invariant subspaces of $\bfV_\ulQ$:
\beq\label{E:filtration}\begin{aligned}
\Fil_\bdsf\bfV:=&\Fil^0V_{\bdsf}\ot V_{\bdsg}\ot V_{\bdsh};\\
\Fil_\bal\bfV:=&\Fil^0V_{\bdsf}\ot \Fil^0V_{\bdsg}\ot V_{\bdsh}+V_{\bdsf}\ot \Fil^0V_{\bdsg}\ot \Fil^0V_{\bdsh}+\Fil^0V_{\bdsf}\ot V_{\bdsg}\ot \Fil^0V_{\bdsh}.
\end{aligned}\eeq
Let $\bullet\in\stt{\bdsf,\bal}$. Define the filtrations $\Fil^+_\bullet\bfV^\dagger:=\Fil_\bullet\bfV\ot\cX^{-1}\subset \bfV^\dagger$.
%\[\Fil^+\bfV^\dagger:=\begin{cases}
%\Fil_\bdsf\bfV\ot\cX^{-1}\\
%\Fil_\bal\bfV_\ulQ\ot\cX^{-1}
%\end{cases}\]
The pair $(\Fil_\bullet^+\bfV^\dagger,\frakX^\bullet_\cR)$ 
satisfies the \emph{Panchishkin condition} in \cite[page 217]{Greenberg94}) in the sense that for each arithmetic point $\ulQ\in\frakX_\cR^\bullet$, the Hodge-Tate numbers of $\Fil^+_\bullet\bfV_\ulQ^\dagger$ are all positive, while the Hodge-Tate numbers of $\bfV^\dagger/\Fil^+_\bullet\bfV_\ulQ^\dagger$ are all non-positive.\footnote{The Hodge-Tate number of $\Qp(1)$ is one in our convention.} %Indeed, the Hodge-Tate numbers of $\Fil_\bullet^+\bfV_\ulQ^\dagger$ are given by $(\frac{w_\ulQ+1}{2},1-k_\Qx^*,k_\Qy^*,k_\Qz^*)$  if $\bullet=\bdsf$ and in the balanced case, 
%$(\frac{w_\ulQ+1}{2},k_\Qx^*,k_\Qy^*,k_\Qz^*)$.
Now we can define the modified $p$-Euler factor by 
\beq\label{E:modified.intro}\cE_p(\Fil^+_\bullet\bfV_\ulQ^\dagger):= \frac{L_p(\Fil^+_\bullet\bfV_\ulQ^\dagger,0)}{\varepsilon(\WD_p(\Fil^+_\bullet\bfV_\ulQ^\dagger))\cdot L_p(\bfV_\ulQ^\dagger/\Fil^+_\bullet\bfV_\ulQ^\dagger,0)}\cdot\frac{1}{L_p(\bfV^\dagger_\ulQ,0)}.
\eeq
We note that this modified $p$-Euler factor is precisely the ratio between the factor $\cL_p^{(\sqrt{-1})}(\bfV^\dagger_\ulQ)$ in \cite[page 109, (18)]{Coates89II} and the local $L$-factor $L_p(\bfV^\dagger_\ulQ,0)$.
 
 In the theory of $p$-adic $L$-functions, we also need the modified Euler factor $\cE_\infty(\bfV^\dagger_\ulQ)$ at the archimedean place observed by Deligne.  It is defined to be the ratio between the factor $\cL_\infty^{(\sqrt{-1})}(\bfV^\dagger_\ulQ)$ in \cite[page 103 (4)]{Coates89II} and the Gamma factor $\Gamma_{\bfV^\dagger_\ulQ}(0)$ and is explicitly given by \[\cE_\infty(\bfV^\dagger_\ulQ)=(\sqrt{-1})^{-2k_\Qx}\text{ if }\ulQ\in\frakX_\cR^\bdsf;\quad \cE_\infty(\bfV_\ulQ^\dagger)=(\sqrt{-1})^{1-k_\Qx-k_\Qy-k_\Qz}\text{ if }\ulQ\in\frakX_\cR^\bal.\]

 \subsection{Hida's canonical periods}\label{S:period.1}To make our interpolation formula meaningful, we must give the precise definition of periods for the motive $\bfV_\ulQ^\dagger$. We begin by recalling Hida's canonical period of a $\bfI$-adic primitive cuspidal Hida family $\cF$ of tame conductor $N_\cF$.  %To begin with, let $f(q)=\sum_{n\geq 1}\bfa(n,f)q^n\in\cO_f\powerseries{q}$ is an elliptic newform of conductor $N_f$ and nebentype $\chi$. WriteLet $c_p(f)$ be the exponent of the $p$-primary part of $N_f$. Let $w^*_p(f)$ be the root number defined by 
%\[w^*_p(f)=\bfa(p,f)^{-c_p(f)}\cdot \begin{cases}1&\text{ if }c_p(f)=0\\
%-1&\text{ if }c_p(f)=1,\,\chi_{(p)}=1\\
%\frakg(\chi_{(p)})\chi_{(p)}(-1)&\text{ if }\chi_{(p)}\not =1.
%\end{cases}\] 
%Here $\frakg(\chi)$ is the usual Gauss sum of $\chi$.
 %Let $w(f)\in\C^\x$ be the root number of $f$ with modulus one in the functional equation $f|_k\pMX{0}{1}{-N_f}{0}(z)=w(f)\cdot \ol{f(-\ol{z})}$. Define the complex number $\pair{f}{f}_{\rm B}$ by  
%\[\pair{f}{f}_{\rm B}:=(2\sqrt{-1})^{k-2}w(f)N_f^{\frac{k}{2}-1}\cdot \norm{f}_{\Gamma_0(N_f)}\in\C^\x\]
%where $\norm{f}_{\Gamma_0(N_f)}$ is the usual Petersson norm of $f$. The value $\pair{f}{f}_{\rm B}$ is indeed obtained by computing certain Hecke equivariant Poincar\'e pairing of the deRham cohomology classes of $f$ (See \cite[(6.22), page 226]{Hida16Pune} for the precise meaning). On the other hand, Hida in \cite[page 488]{Hida94Duke} introduced the plus/minus canonical periods $\Omega(\pm\,;f)=\Omega(\pm,f;\cO_f)\in\C^\x$ unique up to the multiplication by a unit in $\cO_f$. Under some standard hypothesis, Hida proved that the ratio \[\frac{\pair{f}{f}_{\rm B}}{\Omega(+\,;f)\Omega(-\,;f)}\]
%is $p$-integral and is a generator of the congruence ideal of $f$ in $S_k(\Gamma_1(N))$ (\cf\cite[Corollary 6.24, Theorem 6.28]{Hida16Pune}). 
Let $\frakm_\bfI$ be the maximal ideal of $\bfI$. For a  subset $\Sigma$ of the support of $N_\cF$, we consider the following\begin{hypothesis*}[CR,$\,\Sigma$]
The residual Galois representation $\bar\rho_{\cF}:=\rho_\cF\pmod{\frakm_\bfI}:G_\Q\to\GL_2(\Fpbar)$ is absolutely irreducible and $p$-distinguished. Moreover, $\bar\rho_\cF$ is ramified at every $\ell\in\Sigma$ with $\pmq\con 1\pmod{p}$.
\end{hypothesis*}
When $\Sigma=\emptyset$ is the empty set, we shall simply write (CR) for (CR,$\,\emptyset$).
Recall that $\rho_\cF$ is $p$-distinguished if the semi-simplication of the restriction of the residual Galois representation $\rho_\cF\pmod{\frakm_{\bfI}}$ to the decomposition at $p$ is a sum of two characters $\chi_\cF^+\oplus\chi_\cF^-$ with $\chi_\cF^+\not\con \chi_\cF^-\pmod{\frakm_{\bfI}}$. 
Suppose that $\cF$ satisfies (CR). The local component of the universal cuspidal ordinary Hecke algebra corresponding to $\cF$ is known to be Gorenstein by \cite[Prop.2, \S 9]{MW86} and \cite[Corollary 2, page 482]{Wiles95}, and with this Gorenstein property, Hida proved in \cite[Theorem 0.1]{Hida88AJM} that the congruence module for $\cF$ is isomorphic to $\bfI/(\eta_\cF)$ for some non-zero element $\eta_\cF\in\bfI$. Moreover, for any arithmetic point $Q\in\frakX_\bfI^+$, the specialization $\eta_{\cF_Q}=Q(\eta_\cF)$ generates the congruence ideal of $\cF_Q$. We denote by $\cF_Q^\circ$ the normalized newform of weight $k_Q$, conductor $N_Q=N_\cF p^{n_Q}$ with nebentypus $\chi_Q$ corresponding to $\cF_Q$. There is a unique decomposition $\chi_Q=\chi'_Q\chi_{Q,(p)}$, where  $\chi_Q'$ and $\chi_{Q,(p)}$ are Dirichlet characters modulo $N_\cF$ and $p^{n_Q}$ respectively. Let $\al_Q=\bfa(p,\cF_Q)$. Define the modified Euler factor $\cE_p(\cF_Q,\Ad)$ for adjoint motive of $\cF_Q$ by
\[\cE_p(\cF_Q,\Ad)=\al_Q^{-2n_Q}\begin{cases}
(1-\al_Q^{-2}\chi_Q(p)p^{k_Q-1})(1-\al_Q^{-2}\chi_Q(p)p^{k_Q-2})&\text{ if }n_Q=0,\\
-1&\text{ if }n_Q=1,\chi_{Q,(p)}=1\,(\text{so }k_Q=2),\\
p^{(k_Q-2)n_Q}\frakg(\chi_{Q,(p)})\chi_{Q,(p)}(-1)&\text{ if }n_Q>0,\,\chi_{Q,(p)}\not=1.
\end{cases}\]
Here $\frakg(\chi_{Q,(p)})$ is the usual Gauss sum. Fixing a choice of the generator $\eta_\cF$ and letting $\norm{\cF_Q^\circ}^2_{\Gamma_0(N_Q)}$ be the usual Petersson norm of $\cF_Q^\circ$, we define the \emph{canonical period} $\Omega_{\cF_Q}$ of $\cF$ at $Q$ by
 \beq\label{E:period.1}\Omega_{\cF_Q}:=(-2\sqrt{-1})^{k_Q+1}\cdot \norm{\cF_Q^\circ}^2_{\Gamma_0(N_Q)}\cdot \frac{\cE_p(\cF_Q,\Ad)}{\iota_p(\eta_{\cF_Q})}\in\C^\x.\eeq
By \cite[Corollary 6.24, Theorem 6.28]{Hida16Pune}, one can show that for each arithmetic point $Q$, up to a $p$-adic unit, the period $\Omega_{\cF_Q}$ is equal to the product of the plus/minus canonical period $\Omega(+\,;\cF_Q^\circ)\Omega(-\,;\cF_Q^\circ)$ introduced in \cite[page 488]{Hida94Duke}. 

 \subsection{Definitions of $\Gamma$-factors and an exceptional finite set $\Sigma_{\rm exc}$}
 We recall the definition of $\Gamma$-factors of $\bfV_\ulQ^\dagger$ following the recipe in \cite{Deligne79}:
 \beq\label{E:Gamma.1}\Gamma_{\bfV_\ulQ^\dagger}(s):=\begin{cases}\Gamma_\C(s+\frac{w_\ulQ+1}{2})\Gamma_\C(s+1-k_\Qx^*)\Gamma_\C(s+k_\Qy^*)\Gamma_\C(s+k_\Qz^*)&\text{ if }\ulQ\in\frakX_\cR^\bdsf;\\[1em]
\Gamma_\C(s+\frac{w_\ulQ+1}{2})\Gamma_\C(s+k_\Qx^*)\Gamma_\C(s+k_\Qy^*)\Gamma_\C(s+k_\Qz^*)&\text{ if }\ulQ\in\frakX_\cR^\bal.\end{cases}\eeq
Here $\Gamma_\C(s)=2(2\pi)^{-s}\Gamma(s)$ and \[k_{Q_i}^*=\frac{k_\Qx+k_\Qy+k_\Qz}{2}-k_{Q_i},\,i=1,2,3.\]
For each prime $\ell$, let $\tau_{\Q_{\ell^2}}$ be the unique unramified quadratic character of $\Q_\ell^\x$. Let $(f,g,h)=(\bdsf_\Qx,\bdsg_\Qy,\bdsh_\Qz)$ be the specialization of $\bdsF$ at $\ulQ$ and put
\begin{align*}\Sigma_{\bdsf\bdsg}^{sc}&=\stt{\ell\text{: finite prime}\mid \pi_{f,\ell}\text{ and }\pi_{g,\ell}\text{ are supercuspidal};\,\pi_{h,\ell}\text{ is spherical}};\\
\Sigma_{\bdsf\bdsg}&=\stt{\ell\in\Sigma_{\bdsf\bdsg}^{sc}\mid\pi_{f,\ell}\iso\pi_{f,\ell}\ot\tau_{\Q_{\ell^2}}\iso\pi_{g,\ell}^\vee\ot\sigma\text{ for some }\sigma\text{ unramified character}}.\end{align*}
 Define $\Sigma_{\bdsf\bdsh}$ and $\Sigma_{\bdsf\bdsg}$ likewise. We introduce the finite set
 \beq\label{E:exp.1}\Sigma_{\rm exc}=\Sigma_{\bdsg\bdsh}\disjoint \Sigma_{\bdsf\bdsh} \disjoint \Sigma_{\bdsf\bdsg}.\eeq
It is known that this set $\Sigma_{\rm exc}$ does not depend on any particular choice of the specializations of $(\bdsf,\bdsg,\bdsh)$.
 \subsection{Statement of the main results}
  We impose the following technical assumption:
\beqcd{sf}\text {$\gcd(\condf,\condg,\condh)$ is square-free}.\eeqcd
  Our first result is the construction of the unbalanced $p$-adic triple product $L$-functions: \begin{thmA}In addition to \eqref{ev} and \eqref{sf}, we further suppose that \begin{mylist}\item $\Sigma^-=\emptyset$, \item $\bdsf$ satisfies (CR).\end{mylist} Fix a generator $\eta_\bdsf$ of the congruence ideal of $\bdsf$. There exists a unique element $\cL^\bdsf_{\bdsF}\in\cR$ such that for every $\ulQ=(\Qx,\Qy,\Qz)\in\frakX_\cR^\bdsf$ in the unbalanced range dominated by $\bdsf$, we have
\begin{align*}(\cL^\bdsf_{\bdsF}(\ulQ))^2=&\Gamma_{\bfV_\ulQ^\dagger}(0)\cdot\frac{L(\bfV_\ulQ^\dagger,0)}{(\sqrt{-1})^{2k_\Qx}\Omega_{\bdsf_\Qx}^2}\cdot\cE_p(\Fil^+_\bdsf\bfV_\ulQ^\dagger)\cdot\prod_{\ell\in\Sigma_{\rm exc}}(1+\ell^{-1})^2.\end{align*}
 \end{thmA}
 This $p$-adic $L$-function $\cL^\bdsf_{\bdsF}$ is unique up to a choice of generators of the congruence ideal of $\bdsf$, \ie it is unique up to a unit in $\bfI$, but the ratio $\cL^\bdsf_{\bdsF}/\eta_\bdsf$ is a genuine $p$-adic $L$-function. By symmetry, we actually obtain from Theorem A two more \padic $L$-functions $\cL^\bdsg_{\bdsF}$ and $\cL^\bdsh_{\bdsF}$ which interpolate central $L$-values at $\frakX^\bdsg_\cR$ and $\frakX^\bdsh_\cR$ respectively. These $p$-adic $L$-functions $\cL_{\bdsF}^\bdsf,\cL_{\bdsF}^\bdsg$ and $\cL_{\bdsF}^\bdsh$ are called \emph{unbalanced} $p$-adic triple product $L$-functions as they interpolate a square root of the critical central $L$-values of the triple product $L$-function $L(\bfV^\dagger_\ulQ,s)$ for $\ulQ\in\frakX_\cR^{\rm unb}$ at the unbalanced range; from the interpolation formula, these \padic $L$-functions are distinguished by the choices of the modified Euler factor at $p$ and the complex periods. In the literature, the one-variable unbalanced $p$-adic triple product $L$-functions were first constructed by Harris and Tilouine in \cite{HT01Triple} (when $N_1=N_2=N_3=1$). Darmon and Rotger in \cite{DR14ASEN} extended the method in \cite{HT01Triple} to construct a three-variable power series interpolating the global trilinear period of a triplet of Hida families and proved the interpolation formulae at the \emph{balanced range}, which is in connection with the $p$-adic Abel-Jacobi image of diagonal cycles in a triple product of modular curves. This is a $p$-adic analogue of the classical Gross-Zagier formula and has obtained very significant arithmetic application to certain equivariant BSD conjectures in \cite{DR17JAMS}.  On the other hand, it is well known that the relation of the interpolation at the unbalanced range to central $L$-values is suggested by the main identity of Harris and Kudla \cite{HK91Triple}, or in general, Ichino's formula \cite{Ichino08Duke}, but the interpolation formulae at the unbalanced range in the literature are not precise enough for more refined arithmetic applications such as the formulation of corresponding Iwasawa-Greenberg main conjecture. Therefore, Theorem A complements the literature by providing a precise relation of the values of \padic triple product $L$-functions at all arithmetic points in the unbalanced range to central $L$-values of the complex triple product $L$-functions.
 
Our main motivation is to use Theorem A to prove the factorization of \padic triple product $L$-functions into a product of anticyclotomic $p$-adic $L$-functions. For example, if $\bdsg$ and $\bdsh$ are primitive Hida families of CM forms associated with some imaginary quadratic field, then $\cL^\bdsf_{\bdsF}$ is a product of two square roots of anticyclotomic $p$-adic $L$-functions for modular forms constructed in \cite{BD96} and \cite{CH17Crelle}; in contrast, if $\bdsf$ and $\bdsg$ are primitive Hida families of CM forms, then $\cL^\bdsf_{\bdsF}$ is a product of  two anticyclotomic \padic $L$-functions in \cite{BDP13} divided by some Katz \padic $L$-function. The latter gives a strengthening  of \cite[Theorem 3.9]{DLR15} and \cite{Collins16}. With this factorization, we can easily show that the square root of the anticyclotomic \padic $L$-functions in the definite case can be recovered by the Euler system of generalized Kato classes \cite{DR17JAMS} (See \remref{R:Kato}) and provide a new proof of the anticyclotomic exceptional zero conjecture for elliptic curves. These factorizations of $p$-adic triple product $L$-functions are obtained via the direct comparison of the explicit interpolation formulae of $p$-adic $L$-functions at critical points. These examples are much simpler than the factorization formulae of Katz $p$-adic $L$-functions for imaginary quadratic fields and $p$-adic $L$-functions for the symmetric square of elliptic newforms, proved by Gross and Dasgupta respectively, where no critical interpolation is available. In a joint work with F. Castella \cite{CsH18}, we explore this Euler system construction of the square root of the anticyclotomic $p$-adic $L$-functions for elliptic curves and show the non-vanishing of the generalized Kato classes in the rank two case for elliptic curves of rank two.

Next we state our second result about the balanced $p$-adic triple product $L$-functions. 
\begin{thmA}\label{T:interpolation} Let $N=\lcm(\condf,\condg,\condh)$ and $ N^-$ be the square-free product of primes in $\Sigma^-$. In addition to \eqref{ev} and \eqref{sf}, we further suppose that $p>3$ and \begin{mylist}\item $\#(\Sigma^-)$ is odd, \item  $\bdsf,\bdsg$ and $\bdsh$ satisfy (CR, $\Sigma^-$),
\item $N=N^+N^-$ with $\gcd(N^+,N^-)=1$.
\end{mylist}
Then there exists a unique element $\cL_{\bdsF}^\bal\in\cR$ satisfies the following interpolation property: for any arithmetic point $\ulQ\in\frakX_\cR^\bal$, we have
\begin{align*}\left(\cL_{\bdsF}^\bal(\ulQ)\right)^2=&\Gamma_{\bfV_\ulQ^+}(0)\cdot
\frac{L(\bfV_\ulQ^\dagger,0)}{(\sqrt{-1})^{k_\Qx+k_\Qy+k_\Qz-1}\Omega_{\bdsf_\Qx}\Omega_{\bdsg_\Qy}\Omega_{\bdsh_\Qz}}\cdot\cE_p(\Fil_\bal^+\bfV_\ulQ^\dagger)\cdot\prod_{\ell\in\Sigma_{\rm exc}}(1+\ell^{-1})^2.\end{align*}
\end{thmA} 
We must mention that the \padic interpolation of global trilinear period integrals attached to a triplet of $p$-adic families of modular forms in the balanced range was first investigated by Greenberg and Seveso in a pioneering work \cite{GS16}. Our construction is ostensibly different from theirs for their method heavily relies on the theory of Ash-Stevens while our approach is built on classical Hida theory developed in \cite{Hida88Annals}. Indeed, their method treats more general setting, namely they do not restrict to the ordinary case, while our approach is  more well-suited for the future investigation on the arithmetic of the balanced \padic $L$-functions such as the $\mu$-invariants and the Iwasawa-Greenberg main conjecture. The situation is more or less similar to the two different constructions of two-variable \padic $L$-functions for Hida families given by Greenberg-Stevens and Mazur-Kitagawa. In any case, it is definitely very interesting to compare these two different approaches in the ordinary case.
\begin{Remark} We discuss briefly the exceptional zero phenomenon for the balanced \padic $L$-functions. By the Ramanujan conjecture, the modified $p$-Euler factor $\cE_p(\Fil_\bal^+\bfV_\ulQ^\dagger)$ never vanishes unless either of $\bdsf_\Qx,\bdsg_\Qy,\bdsh_\Qz$ is special at $p$. For example, suppose that $\bdsF=(\bdsf,\bdsg,\bdsh)$ is the triplet of primitive Hida families passing through the $p$-stabilized newforms $(f_1,f_2,f_3)$ attached to elliptic curves $(E_1,E_2,E_3)$ over $\Q$ at the weight two specialization $\ulQ$. Let $\al_i=\bfa(p,f_i)$ be the $p$-th Fourier coefficient of $f_i$ for $i=1,2,3$. Assume $E_1$ is semi-stable at $p$ (\ie $\al_1=\pm 1$). Then the formula of the modified $p$-Euler factor reads 
\[\cE_p(\Fil_\bal^+\bfV_\ulQ^\dagger)=\begin{cases}
-p\al_1\al_2\al_3(1-\al_1\al_2\al_3)^3&\text{ if $E_2$ and $E_3$ are semi-stable at $p$},\\
p\al_3^{-2}(1-\frac{\al_3}{\al_1\al_2})^2(1-\frac{\al_1}{\al_2\al_3})^2&\text{ otherwise.}\\
\end{cases}\]
We thus conclude that $\cL_{\bdsF}^\bal$ posseses an exceptional zero at $\ulQ$ when either (i) $E_2$ and $E_3$ are semi-stable at $p$ and $\al_1\al_2\al_3=1$ or (ii) $E_2$ and $E_3$ has good ordinary reduction at $p$ and $\al_2=\al_3\al_1$. In the case (i), we even have the vanishing of the central value $L(\bfV_\ulQ^\dagger,0)=L(E_1\times E_2\times E_3,2)=0$ as the global root number \[\varepsilon(\bfV_{\ulQ}^\dagger)=\varepsilon(\WD_p(\bfV_\ulQ^\dagger))=-\al_1\al_2\al_3=-1,\] so one might speculate about a $p$-adic Gross-Zaiger formula relating certain ``second partial derivatives'' of $\cL_{\bdsF}^\bal$ at $\ulQ$ to the \padic Abel-Jacobi image of diagonal cycle in the Shimura curve $X_{N^+,pN^-}$ attached to the quaternion algebra ramified precisely at $pN^-$ as  \cite[Theorem 1]{BD07}. We hope to come back to this question in the near future.
\end{Remark}

\subsection{An outline of the proof}
The construction of the unbalanced \padic $L$-function is based on Hida's $p$-adic Rankin-Selberg convolution (\cf\cite{Hida93Blue}). %combined with the new ingredients 
%\begin{itemize}\item  a sophisticated choice $\bdsH^{\rm aux}$ of test $\cR$-adic modular forms of level $N$;
%\item the explicit calculation of certain local zeta integrals for the local component $\itPi_{\ulQ,v}$ at a place $v$ of $\Q$, which we shall call \emph{Ichino local integerals} in the sequel since they arise from Ichino's formula of global trilinear period integrals, \ie the local integral $I_v(\phi_v\ot\phi_v')$ in \cite[page 282]{Ichino08Duke}.\end{itemize} 
 Denote by $\eord\bfS(N,\chi,\bfI)\subset \bfI\powerseries{q}$ the space of ordinary $\bfI$-adic cusp forms with tame nebentypus $\chi$ and by $\bfT(N,\chi,\bfI)$ the universal ordinary cuspidal Hecke algebra. Decompose the tame nebentypus $\psi_1$ of $\bdsf$ into a product of Dirichlet characters $\psi_{1,(p)}$ and $\psi_1^{(p)}$ modulo $p$ and $N_1$ respectively and let $\chi:=\psi_{1,(p)}\ol{\psi^{(p)}_1}$.  Let $\breve\bdsf\in\eord\bfS(\condf,\chi,\bfI)$ be the primitive Hida family of $\bdsf$ twisted by $\ol{\psi_1^{(p)}}$ and let $1_{\breve\bdsf}\in \bfT(\condf,\chi,\bfI)\ot_{\bfI}\Frac\bfI$ be the idempotent corresponding to $\breve\bdsf$.  By the definition of congruence ideals, one can verify that $\eta_\bdsf\cdot 1_{\breve\bdsf}$ indeed belongs to $\bfT(N,\chi,\bfI)$. In \subsecref{SS:36} \eqref{E:bdsH}, we construct an auxiliary $\cR$-adic modular form $\eord\bdsH^{\rm aux}\in \eord\bfS(N,\chi,\bfI)\ot_{\bfI,i_1}\cR\subset \cR\powerseries{q}$, where $i_1:\bfI\to \cR$ is the homomorphism $a\mapsto a\ot 1\ot 1$, and then the unbalanced $p$-adic $L$-function is defined to be  
\[\sL_{\bdsF}^\bdsf:=\text{ the first Fourier coefficient of }\eta_\bdsf\cdot 1_{\breve\bdsf} \Tr_{N/N_1}(\eord\bdsH^{\rm aux})\in\cR,\]
where $\Tr_{N/N_1}\colon \eord\bfS(N,\chi,\bfI)\to \eord\bfS(\condf,\chi,\bfI)$ is the usual trace map. 

In the balanced case, Hida theory for definite quaternion algebras plays an important role. Let $D$ be the definite quaternion algebra over $\Q$ of the absolute discriminant $N^-$, and for each positive integer $m$, let $\wtd X_m$ be the definite Shimura curve of level $\Gamma_1(p^nN)$ associated with $D$ as described in \cite[\S 2.1]{LongoVigni11}. These are curves of genus zero equipped with a natural finite covering map $\wtd\al_{m}:\wtd X_{m}\to \wtd X_{m-1}$. We let $J_m=\Pic \wtd X_m\ot_\Z\Zp$ and let $J_\infty:=\prolim_{n\to\infty} J_m$ be the inverse limit induced by $\wtd\al_{m}$. Then $J_\infty$ is a $\Lam$-module with Hecke action, and its ordinary part $J_\infty^\ord$ is equipped with the action of the $\Sigma^-$-new quotient of the universal ordinay cuspidal Hecke algebra of level $\Gamma_1(Np^\infty)$. The $\bfI$-module $\eord\bfS^D(N,\bfI):=\Hom_\Lam(J^\ord_\infty,\bfI)$ is called the space of Hida families of definite quaternionic forms. Due to the lack of $q$-expansions, we do not have the notion of primitive Hida families on definite quaternion algebras. Nonetheless, using the idea of Pollack and Weston \cite{PW11CoM} and Hida theory, for a primitive Hida family $\cF$ satisfying (CR, $\Sigma^-$), it can be shown that there exists Hecke eigenform $\cF^D\in\eord\bfS^D(N,\bfI)$, unique up to a unit in $\bfI$, characterized by the following properties (i) $\cF^D$ shares the same Hecke eigenvalues with $\cF$; (ii) $\cF^D$ is non-zero modulo $\frakm_\bfI$ (\thmref{T:freeness}). We shall call $\cF^D$ the primitive Jacquet-Langlands lift of $\cF$. Let $\bfJ^\ord_m:=J^\ord_m\wh\ot_\cO J^\ord_m\wh\ot_\cO J^\ord_m$ and $\bfJ^\ord_\infty=\prolim_{m\to\infty} \bfJ^\ord_m$. With the assumption (2) in Theorem B, we thus obtain the primitive Jacquet-Langlands lift $\bdsF^D=\bdsf^D\boxtimes\bdsg^D\boxtimes\bdsh^D\in\Hom(\bfJ^\ord_\infty,\cR)$. On the other hand, in \defref{D:Reg}, we construct a collection of regularized diagonal cycles $\Delta_m^\dagger$ in $\bfJ^\ord_m$ which are compatible with respect to $\wtd\al_m$ and thus get the \emph{big diagonal cycle} $\Delta_\infty^\dagger:=\prolim_{m\to\infty}\Delta_m^\dagger\in\bfJ^\ord_\infty$. In order to achieve the optimal integrality of $p$-adic $L$-functions, we actually take a modification $\bdsF^{D\star}\in\Hom(\bfJ^\ord_\infty,\cR)$ of $\bdsF^D$ in \defref{D:testbal}, and then define the balanced $p$-adic $L$-function 
\[\Theta_{\bdsF^D}:=\bdsF^{D\star}(\Delta^\dagger_\infty)\in\cR\] to be the value of the modified $\bdsF^{D\star}$ at $\Delta_\infty^\dagger$. This $p$-adic $L$-function $\Theta_{\bdsF^D}$ is an analogue of theta elements \`a la Bertolini and Darmon (\cite{BD96}) in the triple product setting.

%To prove the interpolation formulae in Theorem B, we use the method of congruences to show that the evaluation of the $p$-adic $L$-function $\cL_{\bdsf,\Sigma^-}(\ulQ)$ at an arithmetic point $\ulQ\in\frakX_\cR^\bal$ equals the global trilinear period integral of quaternionic forms on $D^\x_\A$ (\propref{P:formulaTheta}), and then again apply Ichino's formula. %Future aspects: has potential of generalizing Vatsal's idea in \cite{Vatsal02} to stuudy the non-vanishing of the \padic triple product $L$-functions via Ratner's theory.

To obtain the interpolation formula in Theorem A and B, we first prove that the interpolation $\sL_{\bdsF}^\bdsf(\ulQ)$ at $\ulQ\in\frakX_\cR^\bdsf$ (resp. $\sL_{\bdsf,\Sigma^-}(\ulQ)$ at $\ulQ\in\frakX_\cR^\bal$)  is given by the global trilinear period integral of certain automorphic forms in the cuspidal automorphic representation $\itPi_\ulQ$ of $\GL_2(\A_E)$ (resp. the automorphic representation $\itPi_\ulQ^D$ of $(D\ot \A_E)^\x$ via the Jacquet-Langlands transfer), where $E=\Q\oplus\Q\oplus\Q$ is the split \etale cubic $\Q$-algebra (See \propref{P:inter1} and \ref{P:formulaTheta}). Thanks to Ichino's formula in \cite{Ichino08Duke}, we can show that the square of this global trilinear period integral is a product of the central $L$-value $L(1/2,\itPi_\ulQ)$ and certain local zeta integrals $I_v(\phi^\star_v\ot\phi^\star_v)$ (See \subsecref{SS:Ichino.unb} for definitions), which we shall call \emph{local Ichino integrals} in the introduction. The proof of the interpolation formulae therefore boils down to the determination of the values of these local Ichino integrals. In the literature, local Ichino integrals were only computed for some special cases \cite{II10GAFA}, \cite{NPS14} and \cite{HuYueKe17}. Local Ichino integrals at the real place are completely determined in a recent work \cite{ChenYao16}, but the explicit calculation of local Ichino integrals at non-archimedean places in the generality we need is a highly laborious task and occupies a substantial part of this paper. The key ingredient in our computation is \propref{P:IRSintegral}, a generalization of \cite[Lemma 3.4.2]{MV10} by removing several restrictive conditions therein, which reduces the calculation of local Ichino integrals to that of certain local Rankin-Selberg integrals in \cite[(1.1.3)]{GJ78}. With local theory of $L$-functions on $\GL(2)\times \GL(2)$ developed by Jacquet in \cite{Jacquet72Part2}, we are able to work out the calculation of local Rankin-Selberg integrals under \eqref{sf} and certain \emph{minimal hypothesis} (See \hypref{H:ram}). It turns out that the \padic Ichino integral gives the modified $p$-Euler factor $\cE_p(\Fil_\bullet^+\bfV_\ulQ^\dagger)$, while local Ichino integrals at ramified places $\ell$ only contributes $p$-adic units if $\ell\not\in\Sigma_{\rm exc}$ or $(1+\ell^{-1})^2$ if $\ell\in\Sigma_{\rm exc}$. This minimal hypothesis, roughly speaking, requires $\bdsF$ to be minimal in the sense that $\bdsF$ has the minimal conductor among Dirichlet twists. By taking a suitable Dirichlet twist $\bdsF'=(\bdsf\ot\chi_1, \bdsg\ot\chi_2,\bdsh\ot\chi_3)$ with $\chi_1\chi_2\chi_3=1$ which satisfies the minimal hypothesis, we obtain the desired $p$-adic $L$-functions \[\cL_{\bdsF}^\bdsf:=\sL^{\bdsf\ot\chi_1}_{\bdsF'};\quad \cL_{\bdsF}^\bal:=\Theta_{\bdsF^{\prime D}}.\] The interpolation formulae is a direct consequence of the explicit evaluation of local Ichino integrals and the comparison between the canonical periods of $\bdsF$ and its Dirichlet twist $\bdsF'$ established in \subsecref{SS:periods}. We conclude this paragraph by mentioning that the method of this paper has been extended by Isao Ishikawa in \cite{Ishikawa} to construct \padic twisted triple product $L$-functions attached a Hida family of Hilbert modular form over a real quadratic field and a Hida family of elliptic modular forms. 

This paper is organized as follows. In \secref{S:modularforms}, we recall basic definitions and facts about classical elliptic modular forms and automorphic forms on $\GL_2(\A)$. In \secref{S:unb}, we give the construction of the unbalanced $p$-adic triple product $L$-functions $\sL_{\bdsF}^\bdsf$. The key items used in the construction of $\bdsH^{\rm aux}$, the test $\Lam$-adic forms $\bdsg^\star$ and $\bdsh^\star$,  are introduced in \defref{D:testunb}. The main formula is derived in \corref{C:Ichino.imb}, where we show the interpolation of the square of $\sL_{\bdsF}^\bdsf$ at the unbalanced range is the product of the central $L$-value of the triple product $L$-function and local Ichino integrals at the prime $p$ and ramified primes. In \subsecref{S:bal}, we consider the balanced case. We review Hida's theory for definite quaterninoic forms in \subsecref{SS:Hida1} and \subsecref{SS:Hida2}. In particular, we present a slightly explicit version of the control theorem in \thmref{T:HidaQ} and explain the notion of primitive Jacquet-Langlands lifts in \thmref{T:freeness}. The construction of the big diagonal cycle $\Delta_\infty^\dagger$ and the balanced $p$-adic $L$-functions are given in \subsecref{SS:theta1} and \subsecref{SS:theta2}. The relation between the interpolation of the square of our balanced $p$-adic $L$-functions and the product of the central $L$-value and local Ichino integrals is given in \corref{C:Ichino.bal}. In \secref{S:local1}, we prepare the tools for the computation of local Ichino integrals and carry out the calculations at the \padic place, and in \secref{S:local2}, we elaborate the calculation of local Ichino integrals at ramified primes. In particular, we show in \secref{SS:6.6} that the local Ichino integrals at ramified places can be interpolated into a unit in  the ring $\cR$ of three-variable Iwasawa functions. In \secref{S:interpolation}, we prove the main results (\thmref{T:main.7}) and show that the canonical periods of a primitive Hida family and its Dirichlet twists are equal up to a unit in $\bfI$ by the method of level-raising.  Finally, we prove the factorization of anticyclotomic $p$-adic $L$-functions and give applications in \secref{S:application}.

\begin{thank} Part of this work was done during the author's multiple visits to Tohoku University supported by the program \emph{Advancing Strategic International Networks to Accelerate
the Circulation of Talented Researchers} during 2015--2017. The author would like to thank Masataka Chida, Shinichi Kobayashi and Nobuo Tsuzuki for their hospitality during the period of this program. The author also thanks Shih-Yu Chen and Yao Cheng fo the very help discussions during the preparation of this article. Finally, the author is grateful to the referees for the suggestions and comments on the improvement of the manuscript.
\end{thank}

\subsection*{Notation} The following notations will be used frequently  throughout the paper.  Let $\A$ be the ring of adeles of $\Q$. If $v$ is a place of $\Q$, let $\Q_v$ be the completion of $\Q$ with respect to $v$, and for $a\in \A^\x$, let $a_v\in\Q_v^\x$ be the $v$-component of $a$. Denote by $\Abs_v$  (or simply $\Abs$ if there is no fear of confusion) the absolute value on $\Q_v$ normalized so that $\Abs$ is the usual absolute value on $\R$ if $v=\infty$ and $\abs{\ell}_\ell=\ell^{-1}$ if $v=\ell$ is finite. Let $\Abs_\A$ be the absolute value on $\A^\x$ given by $\abs{a}_\A=\prod_v\abs{a_v}_v$.  Let $\zeta_v(s)$ be the usual local zeta function of $\Q_v$. Namely,
\[\zeta_\infty(s)=\pi^{-\frac{s}{2}}\Gamma(\frac{s}{2});\quad
\zeta_\ell(s)=(1-\ell^{-s})^{-1}.\]
Define the global zeta function $\zeta_\Q(s)$ of $\Q$ by $\zeta_\Q(s)=\prod_v\zeta_v(s)$. In particular, $\zeta_\Q(2)=\pi^{-1
}\cdot \zeta(2)=\pi/6$.

For a prime $\ell$, let $\val_\ell:\Q_\ell^\x\to\C^\x$ be the valuation normalized so that $\val_\ell(\ell)=1$. We shall regard $\Q_\ell$ and $\Q_\ell^\x$ as subgroups of $\A$ and $\A^\x$ in a natural way. To avoid possible confusion, denote $\uf_\ell=(\uf_{\ell,v})\in\A^\x$ by the idele defined by $\uf_{\ell,\ell}=\ell$ and $\uf_{\ell,v}=1$ if $v\not =\ell$. 

Let $\addchar_\Q:\A/\Q\to\C^\x$ be the additive character with the archimedean component $\addchar_\R(x)=\exp(2\pi\sqrt{-1}x)$ and let $\addchar_{\Q_\ell}:\Q_\ell\to\C^\x$ be the local component of $\addchar_\Q$ at $\ell$.

If $R$ is a commutative ring and $G=\GL_2(R)$, we denote by $\rho$ the right translation of $G$ on the space of $\C$-valued functions on $G$: $\rho(g)f(g')=f(g'g)$ and by $\bfone:G\to\C$ the constant function $\bfone(g)=1$.  For a function $f:G\to\C$ and a character $\chi:R^\x\to\C^\x$, let $f\ot\chi:G\to\C$ denote the function $f\ot\chi(g)=f(g)\chi(\det g)$.

Let $G_\Q=\Gal(\Qbar/\Q)$ be the absolute Galois group of $\Q$ and if $\chi:(\Z/N\Z)^\x\to\C^\x$ is Dirichlet character modulo $N$, denote by $c_\ell(\chi)\leq\val_\ell(N)$ the $\ell$-exponent of the conductor of $\chi$. We shall identify $\chi$ with the Galois character $\chi:G_\Q\to\C^\x$ via class field theory. 

If $\om:\Q^\x\bksl\A^\x\to\Qbar^\x$ is a finite order Hecke character, we denote by $\om_\ell:\Q_\ell^\x\to\C^\x$ the local component of $\om$ at $\ell$. On the other hand, we write $\om=\om_{(\ell)}\om^{(\ell)}$, where $\om_{(\ell)}$ and $\om^{(\ell)}$ are finite order Hecke characters of conductor $\ell$-power and of prime-to-$\ell$ conductor respectively. With every Dirichlet character $\chi$ of conductor $N$, we can associate a Hecke character $\chi_\A$, called the \emph{adelization} of $\chi$, which is the unique finite order Hecke character $\chi_\A:\Q^\x\bksl \A^\x/\R_+(1+N\wh\Z)^\x\to\C^\x$ of conductor $N$ such that $\chi_\A(\uf_\pmq)=\chi(\pmq)^{-1}$ for any prime $\pmq\ndivides N$. We often identify Dirichlet characters with their adelization whenever no confusion arises. Then $\chi_\pmq(\pmq)=\chi(\pmq)^{-1}$ for $\pmq\ndivides N$.

%!TEX root = TRIPLE3.tex

\section{Classical modular forms and automorphic forms}\label{S:modularforms}
In this section, we recall basic definitions and facts about classical elliptic modular forms and automorphic forms on $\GL_2(\A)$. The main purpose of this section is to set up the notation and introduce some Hecke operators on the space of automorphic forms which will be frequently used in the construction of \padic $L$-functions.

\subsection{Classical modular forms}\label{SS:classical}
Let $C^\infty(\frakH)$ be the space of $\C$-valued smooth functions on the upper half complex plane $\frakH$. Let $k$ be any integer. Let $\gamma=\pMX{a}{b}{c}{d}\in\GL^+_2(\R)$ act on $z\in\frakH$ by $\gamma(z)=\frac{az+b}{cz+d}$, and for $f=f(z)\in C^\infty(\frakH)$, define \[f|_k \gamma(z):=f(\gamma(z))(cz+d)^{-k}(\det\gamma)^\frac{k}{2}.\]
Recall that the Maass-Shimura differential operators $\delta_k$ and $\varepsilon$ on $C^\infty(\frakH)$ are given by 
\[\delta_k=\frac{1}{2\pi\sqrt{-1}}(\frac{\partial }{\partial z}+\frac{k}{2\sqrt{-1}y})\text{ and }\varepsilon=-\frac{1}{2\pi\sqrt{-1}}y^2\frac{\partial}{\partial \ol{z}}\quad(y=\Im(z))\]
(\cf\cite[(1a,\,1b) page 310]{Hida93Blue}). Let $N$ be a positive integer and $\chi:(\Z/N\Z)^\x\to\C^\x$ be a Dirichlet character modulo $N$. Let $m$ be a non-negative integer. Denote by $\cN^{[m]}_k(N,\chi)$ the space of nearly holomorphic modular forms of weight $k$, level $N$ and character $\chi$, consisting of slowly increasing functions $f\in C^\infty(\frakH)$ such that $\varepsilon^{m+1} f=0$ and 
\[f|_k\pMX{a}{b}{c}{d} =\chi(d)f\quad\text{ for }\pMX{a}{b}{c}{d}\in\Gamma_0(N)\]
 (\cf\cite[page 314]{Hida93Blue}). Let $\cN_k(N,\chi)=\bigcup_{m=0}^\infty \cN^{[m]}_k(N,\chi)$.(\cf\cite[(1a), page 310]{Hida93Blue}) 
By definition, $\cN^{[0]}_k(N,\chi)=\cM_k(N,\chi)$ is the space of classical holomorphic modular forms of weight $k$, level $N$ and character $\chi$.  Denote by $\sS_k(N,\chi)$ the space of holomorphic cusp forms in $\cM_k(N,\chi)$. Let $\delta_k^m=\delta_{k+2m-2}\cdots\delta_{k+2}\delta_k$. If $f\in \cN_k(N,\chi)$ is a nearly holomorphic modular form of weight $k$, then $\delta_k^mf\in\cN_{k+2m}(N,\chi)$ has weight $k+2m$ (\cite[page 312]{Hida93Blue}. For a positive integer $d$, define \[\LR_df(z)=d\cdot f(dz);\quad \bfU_df(z)=\frac{1}{d}\sum_{j=0}^{d-1} f(\frac{z+j}{d}),\]
and recall that the classical Hecke operators $T_\ell$ for primes $\ell\ndivides N$ are given by
\[T_\ell f=\bfU_\ell f+\chi(\ell)\ell^{k-2}\LR_\ell f.\]
We say $f\in \cN_k(N,\chi)$ is a \emph{Hecke eigenform} if $f$ is an eigenfunction of all the Hecke operators $T_\ell$ for $\ell\ndivides N$ and the operators $\bfU_\ell$ for $\ell\divides N$. 

If $f\in\cM_k(N,\chi)$, let 
\[f(q)=\sum_{n\geq 0}\bfa(n,f)q^n\]
be the $q$-expansion (at the infinity cusp). If $\kappa$ is a Dirichlet character modulo $M$,  define $f|[\kappa]\in\cM_k(NM^2,\chi\kappa^2)$ the twist of $f$ by $\kappa$ to be the unique modular form with the $q$-expansion 
\[f|[\kappa](q)=\sum_{n\geq 0,\,(n,M)=1}\bfa(n,f)\kappa(n)q^n.\]

\subsection{Automorphic forms on $\GL_2(\A)$}\label{SS:auto}
Let $N$ be a positive integer. Define open-compact subgroups of $\GL_2(\wh\Z)$ by
\begin{align*}
\opcpt_0(N)=&\stt{g\in\GL_2(\wh\Z)\mid g\con \pMX{*}{*}{0}{*}\pmod{N\wh \Z}},\\
\opcpt_1(N)=&\stt{g\in \opcpt_0(N)\mid g\con \pMX{*}{*}{0}{1}\pmod{N\wh\Z}}.
\end{align*}
Let $\om:\Q^\x\bksl \A^\x\to\C^\x$ be a finite order Hecke character of level $N$. We extend $\om$ to a character of $\opcpt_0(N)$ defined by $\om(\pMX{a}{b}{c}{d})=\prod_{\pmq\mid N}\om_\pmq(d_\pmq)$ for $\pMX{a}{b}{c}{d}\in \opcpt_0(N)$, where $\om_\ell:\Q_\ell^\x\to\C^\x$ is the $\ell$-component of $\om$. Denote by $\cA(\om)$ the space of automorphic forms on $\GL_2(\A)$ with central character $\om$. For any integer $k$, let $\cA_k(N,\om)\subset\cA(\om)$ be the space of automorphic forms on $\GL_2(\A)$ of weight $k$, level $N$ and character $\om$. Namely, $\cA_k(N,\om)$ consists of automorphic forms $\varphi:\GL_2(\A)\to\C$ such that 
\begin{align*}\varphi(\alpha gu_\infty u_{\rm f})=&\varphi(g)e^{\sqrt{-1}k\theta}\om(u_{\rm f})\\
(\al\in\GL_2(\Q), u_\infty=&\pMX{\cos\theta}{\sin\theta}{-\sin\theta}{\cos\theta},\,u_{\rm f}\in \opcpt_0(N)).
\end{align*}
Let $\cA^0_k(N,\om)$ be the space of cusp forms in $\cA_k(N,\om)$.

Next we introduce important local Hecke operators on automorphic forms. At the archimedean place, let $\LR_{\pm}:\cA_k(N,\om)\to\cA_{k\pm 2}(N,\om)$ be the normalized weight raising/lowering operator in \cite[page 165]{JacquetLanglands70} given by 
\beq\label{E:diffop}%\begin{aligned}
\LR_{\pm}=\frac{1}{(-8\pi)}\left(\pDII{1}{-1}\ot 1\pm\pMX{0}{1}{1}{0}\ot \sqrt{-1}\right)\in\Lie(\GL_2(\R))\ot_\R\C.
%\LR_{\pm}:=&\frac{1}{(-8\pi)}V_\pm.\end{aligned}
\eeq
The level-raising operator $\LR_\ell:\cA_k(N,\om)\to\cA_k(N\ell,\om)$ at a finite prime $\ell$ by 
\[\LR_\ell\varphi(g):=\rho(\pDII{\uf_\ell^{-1}}{1})\varphi.\]
If $d=\prod_\ell \ell^{\val_\ell(d)}$ is an positive integer, define $\LR_d:\cA_k(N,\om)\to\cA_k(Nd,\chi)$ by 
\[\LR_d=\prod_\ell\LR_\ell^{\val_\ell(d)}.\]
Define the operator $\bfU_\ell$ on $\varphi\in\cA_k(N,\om)$ by 
\[\bfU_\ell\varphi=\sum_{x\in\Z_\ell/\ell\Z_\ell}\rho(\pMX{\uf_\ell}{x}{0}{1})\varphi.\]
Note that $\bfU_\ell V_\ell\varphi=\ell\varphi$ and that if $\ell\divides N$, then $\bfU_\ell\in\End_\C\cA_k(N,\om)$. For each prime $\ell\ndivides N$, let $T_\ell\in\End_\C\cA_k(N,\om)$ be the usual Hecke operator defined by
\[T_\ell=\bfU_\ell+\om(\uf_\ell)V_\ell.\] 
We introduce the twisting operator $\theta_\ell^\kappa$ attached to a Dirichlet character $\kappa$ of modulo $\ell^s$ for some $s>0$. Let $\ell^n$ be the conductor of $\kappa$. If $n>0$, define the Gauss sum $\frakg(\kappa)$ by 
\[\frakg(\kappa)=\sum_{x\in(\Z/\ell^n\Z)^\x}\kappa^{-1}(x)e^{\frac{-2\pi \sqrt{-1}x}{\ell^n}}.\] For $\varphi\in \cA_k(N,\om)$, we define $\theta_\ell^\kappa\varphi:\GL_2(\A)\to\C$ by 
\beq\label{E:twisting1}\theta_\ell^{\kappa}\varphi=\begin{cases}\varphi-\ell^{-1}\LR_\ell \bfU_\ell\varphi&\text{ if }n=0,\\
\frakg(\kappa)^{-1}\sum\limits_{x\in (\Z/\ell^n\Z)^\x}\kappa^{-1}(x)\rho(\pMX{1}{x/\uf_\ell^n}{0}{1})\varphi&\text{ if }n>0.\end{cases}
\eeq
\subsection{}We briefly recall a well-known connection between modular forms and automorphic forms. With each nearly holomorphic modular form $f\in\cN_k(N,\chi)$, we associate a unique automorphic form $\itPhi(f)\in\cA_k(N,\chi_\A^{-1})$ defined by the equation
\beq\label{E:MA2}\itPhi(f)(\al g_\infty u):= (f|_k g_\infty)(\sqrt{-1})\cdot \chi_\A^{-1}(u)\eeq
for $\al \in\GL_2(\Q)$, $g_\infty\in \GL^+_2(\R)$ and $u\in \opcpt_0(N)$ (\cf \cite[\S 3]{Casselman73MA}). We call $\itPhi(f)$ the \emph{adelic lift} of $f$. Conversely, we can recover the form $f$ from $\itPhi(f)$ by 
\beq\label{E:MA1}f(x+\sqrt{-1} y)=y^{-\frac{k}{2}}\itPhi(f)(\pMX{y}{x}{0}{1}).\eeq

The weight raising/lowering operators are the adelic avatar of the Maass-Shimura differential operators $\delta_k^m$ and $\varepsilon$ on the space of automorphic forms. A direct computation shows that the map $\itPhi$ is equivariant for the Hecke action in the sense that
\beq\label{E:diffop1}\itPhi(\delta_k^mf)=\LR_+^m\itPhi(f),\quad
\itPhi(\varepsilon f)=\LR_-\itPhi(f),
\eeq
for a positive integer $d$, \beq\label{E:diffop2}\itPhi(V_d f)=d^{1-\frac{k}{2}}\LR_d\itPhi(f),\eeq  and for a finite prime $\ell$\beq\label{E:heckeop}
\itPhi(T_\ell f)=\ell^{\frac{k}{2}-1}T_\ell\itPhi(f);\quad \itPhi(\bfU_\ell f)=\ell^{\frac{k}{2}-1}\bfU_\ell\itPhi(f).
\eeq
In particular, $f$ is holomorphic if and only if $V_-\itPhi(f)=0$. For $f\in \cM_k(N,\chi)$ and $\kappa$ a Dirichlet character modulo a $\ell$-power, one verifies that
\beq\label{E:diffop3}\itPhi(f|[\kappa])=\theta_\ell^\kappa\itPhi(f)\ot\kappa_\A^{-1}.\eeq

\subsection{Preliminaries on irreducible representations of $\GL_2(\Q_v)$}
\subsubsection{Measures}\label{SS:measure}
We shall normalize the Haar measures on $\Q_v$ and $\Q_v^\x$ as follows. If $v=\infty$, $\rmd x$ or $\rmd y$ denotes the usual Lebesgue measure on $\R$ and the measure $\rmd^\x y$ on $\R^\x$ is $\abs{y}^{-1}\rmd y$. If $v=\ell$ is a finite prime, denote by $\rmd x$ the Haar measure on $\Q_\ell$ with $\vol(\Z_\ell,\rmd x)=1$ and by $\rmd^\x y$ the Haar measure on $\Q^\x_\ell$ with $\vol(\Z^\x_\ell,\rmd^\x y)=1$. Define the compact subgroup $\bfK_v$ of $\GL_2(\Q_v)$ by $\bfK_v={\rm O}(2,\R)$ if $v=\infty$ and $\bfK_v=\GL_2(\Z_v)$ if $v$ is finite. Let $\rmd k_v$ be the Haar measure on $\bfK_v$ so that $\vol(\bfK_v,\rmd k_v)=1$. Let $\rmd g_v$ be the Haar measure on $\PGL_2(\Q_v)$ given by $\rmd g_v=\abs{y_v}^{-1}\rmd x_v\rmd^\x y_v\rmd k_v$ for $g_v=\pMX{y_v}{x_v}{0}{1}k_v$ with $y_v\in\Q_v^\x$, $x_v\in\Q_v$ and $k_v\in \bfK_v$.

\subsubsection{Representations}
Denote by $\Prin{\chi}{\upsilon}$ the irreducible principal series representation of $\GL_2(\Q_v)$ attached to two characters $\chi,\upsilon:\Q_v^\x\to\C^\x$ such that $\chi\upsilon^{-1}\not=\Abs^\pm$. If $v=\infty$ is the archimedean place and $k\geq 1$ is an integer, denote by $\cD_0(k)$ the discrete series of lowest weight $k$ if $k\geq 2$ or the limit of discrete series if $k=1$ with central character $\sgn^k$ (the $k$-the power of the sign function).
If $v$ is finite, denote by ${\rm St}$ the Steinberg representation and by $\chi{\rm St}$ the special representation ${\rm St}\ot\chi\circ\det$. 

\subsubsection{$L$-functions and $\varepsilon$-factors}
 For a character $\chi:\Q_v^\x\to\C^\x$, let $L(s,\chi)$ be the complex $L$-function  and $\varepsilon(s,\chi):=\varepsilon(s,\chi,\addchar_{\Q_v})$ be the $\varepsilon$-factor (\cf\cite[Section 1.1]{Schmidt02RJ}). Define the $\gamma$-factor 
\beq\label{E:gamma1}\gamma(s,\chi):=\varepsilon(s,\chi)\cdot\frac{L(1-s,\chi^{-1})}{L(s,\chi)}.\eeq

If $\pi$ is an irreducible admissible generic representation of $\GL_2(\Q_v)$, denote by $L(s,\pi)$ the $L$-function and by $\varepsilon(s,\pi):=\varepsilon(s,\pi,\addchar_{\Q_v})$ the $\varepsilon$-factor defined in \cite[Theorem 2.18]{JacquetLanglands70}. Let $\Contra{\pi}$ denote the contragradient representation of $\pi$. Denote by $L(s,\pi,\Ad)$ the adjoint $L$-function of $\pi$ determined in \cite{GJ78}.

\subsubsection{Conductors and new vectors}Let $\ell$ be a prime. Let $(\pi,\cV_\pi)$ be an irreducible admissible infinite dimensional representation of $\GL_2(\Q_\ell)$, where $\cV_\pi$ a realization of $\pi$. For a non-negative integer $n$, let 
\[%\cU_0(\ell^n)=\GL_2(\Z_\ell)\cap \pMX{\Z_\ell}{\Z_\ell}{\ell^n\Z_\ell}{\Z_\ell};\quad  
\cU_1(\ell^n)=\GL_2(\Z_\ell)\cap \pMX{\Z_\ell}{\Z_\ell}{\ell^n\Z_\ell}{1+\ell^n\Z_\ell}.\] 
Let $\cond{\pi}$ be the exponent of the conductor of $\pi$. By definition, $\cond{\pi}$ is the smallest integer such that $\cV_\pi^{\cU_1(\ell^{\cond{\pi}})}$ the space of $\cU_1(\ell^{\cond{\pi}})$-fixed vectors is non-zero. Define the subspace $\cV^{\rm new}_\pi$ by  
\[\cV_\pi^{\rm new}=\stt{ \xi\in \cV_\pi\mid \pi(\pMX{a}{b}{c}{d})\xi=\xi\text{ for all }\pMX{a}{b}{c}{d}\in \cU_1(\ell^{\cond{\pi}})}.\]
\begin{prop}[Multiplicity one for new vectors]\label{P:new} We have $\dim_\C \cV_\pi^{\rm new}=1$.
\end{prop}
\begin{proof}This is \cite[Theorem 1]{Casselman73MA}.\end{proof}
In the sequel, we call $\cV_\pi^{\rm new}$ \emph{the new line} of $\pi$.

\subsubsection{Whittaker models}\label{SS:Wnew}
Every admissible irreducible infinite dimensional representation $\pi$ of $\GL_2(\Q_v)$ admits a realization of the Whittaker model $\sW(\pi)=\sW(\pi,\addchar_{\Q_v})$associated with the additive character $\addchar_{\Q_v}$. Recall that $\sW(\pi)$ is a subspace of smooth functions $W:\GL_2(\Q_v)\to\C$ such that
\begin{itemize}
\item $W(\pMX{1}{x}{0}{1}g)=\addchar_{\Q_v}(x)W(g)$ for all $x\in\Q_v$,
\item if $v=\infty$ is the archimedean place, there exists an integer $M$ such that 
\[W(\pDII{a}{1})=O(\abs{a}^M)\text{ as }\abs{a}\to\infty.\] \end{itemize}
%If $\Q_\pmq$ is non-archimedean and $N$ is a positive integer, we put
%\[K_1(N)_\pmq=\stt{g=\pMX{a}{b}{c}{d}\in\GL_2(\Z_\pmq)\mid c\in N\Z_\pmq,\,d\con 1\pmod{NZ_\pmq}}.\]
The group $\GL_2(\Q_v)$ (or the Hecke algebra of $\GL_2(\Q_v)$) acts on $\sW(\pi)$ via the right translation $\rho$. We introduce the (normalized) \emph{local Whittaker newform} $W_{\pi}$ in $\cW(\pi)$ in the following cases. If $v=\infty$ and $\pi=\cD_0(k)$, then the Whittaker local newform $\newW_{\pi}\in\sW(\pi)$ is defined by
\beq\label{E:Winfty.1}\begin{aligned}\newW_{\pi}(z\pMX{y}{x}{0}{1}\pMX{\cos\theta}{\sin\theta}{-\sin\theta}{\cos\theta})&=\bbI_{\R_+}(y)\cdot y^\frac{k}{2}e^{-2\pi y}\cdot\sgn(z)^k\addchar_\R(x)e^{\sqrt{-1}k\theta}\\
&\quad(y,z\in\R^\x,\,x,\theta\in\R).\end{aligned}
\eeq
Here $\bbI_{\R_+}(a)$ denotes the characteristic function of the set of positive real numbers. 
If $v=\ell$ is a finite prime, then the local Whittaker newform  $\newW_{\pi}$ is the unique function in $\cW(\pi)^{\rm new}$ such that $\newW_{\pi}(1)=1$.

\subsection{Ordinary lines in irreducible representations of $\GL_2(\Qp)$}
Let $p$ be a prime. Let $(\pi,\cV_\pi)$ be an irreducible admissible generic representation of $\GL_2(\Qp)$ with central character $\om:\Qp^\x\to\C^\x$. Let $N(\Z_p)=\stt{\pMX{1}{x}{0}{1}\mid x\in\Zp}$. Define the local $\bfU_p$-operator and the local level-raising operator $V_p$ in $\End_\C(\cV_\pi^{N(\Zp)})$ by 
\beq\label{E:localUV}
\bfU_p\xi:=\sum_{x\in\Zp/p\Zp}\pi(\pMX{p}{x}{0}{1})\xi;\quad V_p\xi=\pi(\pDII{p^{-1}}{1})\xi.
\eeq
For a Dirichlet character $\kappa$ of conductor $p^n$, we define the local twisting operator $\theta_p^\kappa\in\End\cV_\pi$ by
\beq\label{E:deftheta2}
\theta_p^{\kappa}\xi=\begin{cases}\xi-p^{-1}\LR_p \bfU_p\xi&\text{ if }n=0,\\
\frakg(\kappa)^{-1}\sum\limits_{x\in (\Z/p^n\Z)^\x}\kappa^{-1}(x)\pi(\pMX{1}{x/p^n}{0}{1})\xi&\text{ if }n>0.\end{cases}
\eeq

For a character $\chi:\Qp^\x\to\C^\x$, define the subspace $\cV_\pi^\ord(\chi)$ by
%\beq\label{E:ordline}
\[\cV_\pi^\ord(\chi):=\stt{\xi\in \cV_\pi^{N(\Zp)}\mid \bfU_p\xi=\chi\Abs^{-\onehalf}(p)\cdot \xi,\,\,\pi(\pDII{t}{1})\xi=\chi(t)\xi,\,t\in\Zp^\x}.\]
%\eeq
\begin{prop}[Multiplicity one for ordinary vectors]\label{P:ordline} The space $\cV_\pi^\ord(\chi)$ is non-zero if and only if $\pi$ is either the principal series $\Prin{\chi}{\chi^{-1}\om}$ or the special representation $\chi\Abs^{-\onehalf}{\rm St}$. In this case, \[\dim_\C \cV_\pi^\ord(\chi)=1.\]
\end{prop}
\begin{proof} Replacing $\pi$ by $\pi\ot\chi^{-1}\Abs^\onehalf$, we may assume $\chi=\Abs^\onehalf$. For each $n$, let 
\[\cV_\pi^{[n]}[\bfU_p-1]=\stt{\xi\in\cV_{\pi}\mid \bfU_p\xi=\xi;\quad \pi(u)\xi=\xi\text{ for all }u\in \cU_1(p^n)}.\]
Let $\cV^\ord_\pi=\cV^\ord_\pi(\Abs^\onehalf)$. Let $\cond{\om}$ be the exponent of the conductor of $\om$ and $c^*:=\max\stt{1,c(\om)}$. Then it is easy to see that
\[\cV^\ord_\pi=\bigcup_{n\geq c^*}^\infty\cV_\pi^{[n]}[\bfU_p-1] .\] 
Suppose that $\pi=\Prin{\Abs^\onehalf}{\om\Abs^{-\onehalf}}$ or the Steinberg representation ${\rm St}$. We claim that $\cV_\pi^{[n]}[\bfU_p-1]$ is non-zero for some $n$. If $\om$ is ramified or $\pi$ is Steinberg, then $c(\pi)\geq c^*$ and the new line $\cV_\pi^{\rm new}=\cV^{[\cond{\pi}]}[\bfU_p-1]$ is not zero. If $\om$ is unramified, then $\pi$ is sphercial, and it is well known that $\dim_\C\cV_\pi^{[1]}=2$ and the characteristic polynomial of $\bfU_\ell$ on $\cV_\pi^{[1]}$ is given by $(X-1)(X-\om(p)p)$, so $\cV_\pi^{[1]}[\bfU_\ell-1]$ is non-zero. 

Now suppose that $\cV^\ord_\pi\not =0$. Then $\pi$ must be a principal series or special representation since $\bfU_p$ is a unipotent operator on $\cV_\pi^{[n]}$ if $\pi$ is supercuspidal. For any $u\in \cU_1(p^m)$ with $m\geq 1$ and $\xi\in \cV_\pi$, a straightforward calculation shows that  
\[\pi(u)\bfU_p\xi=\sum_{x\in\Zp/p\Zp}\pi(\pMX{p}{x}{0}{1}u_x'z_x)\xi \text{ for some }u'_x\in \cU_1(p^{m+1}),\,z_x\in 1+p^m\Zp.\]
It follows that if $\xi\in \cV_\pi^{[m+1]}[\bfU_p-1]$, then $\xi\in \cV_\pi^{[m]}[\bfU_p-1]$ whenever $m\geq c^*$. This implies that $\cV_\pi^\ord=\cV_\pi^{\cU_1(p^{c^*})}\not =0$, and hence $c^*\geq c(\pi)\geq c(\om)$. If $c^*=c(\om)>0$, then $c(\om)=c(\pi)$, and it follows that $\cV_\pi^\ord=\cV_\pi^{\rm new}$ is the new line in $\cV_\pi$ and $\pi=\Prin{\mu}{\mu^{-1}\om}$ with unramified character $\mu$. Since any new vector in $\Prin{\mu}{\mu^{-1}\om}$ is an eigenvector of $\bfU_p$ with the eigenvalue $\mu\Abs^{-\onehalf}$, we thus conclude that $\pi=\Prin{\Abs^\onehalf}{\om\Abs^{-\onehalf}}$. If $c(\om)=0$, then $c^*=1$ and $\cV_\pi^{\ord}=\cV_\pi^{[1]}[\bfU_\ell-1]$. It follows that $\pi$ is a unramified principal series or the Steinberg representation ${\rm St}$. If $\pi={\rm St}$, then $\cV^\ord_\pi$ is the new line. If $\pi$ is a unramified principal series, then the two dimensional vector space $\cV_\pi^{\cU_1(p)}$ has a basis $\xi^0\in \cV_\pi^{\rm new}=\cV_\pi^{\GL_2(\Zp)}$ and $V_p\xi^0$. Since $\bfU_p V_p\xi^0=p\xi^0$, $\bfU_p$ is not a scalar, and thus $\dim_\C\cV_\pi^{\ord}=\dim_\C\cV_\pi^{[1]}[\bfU_p-1]=1$.
\end{proof}
We shall call $\cV_\pi^\ord(\chi)$ \emph{the ordinary line} of $\pi$ with respect to $\chi$ whenever it is non-zero. 
\begin{cor}\label{C:Word}If $\pi$ is either the irreducible principal series $\Prin{\chi}{\chi^{-1}\om}$ or the special representation $\chi\Abs^{-\onehalf}{\rm St}$, then the ordinary line $\cW(\pi)^\ord(\chi)$ in the Whittaker model is generated by the normalized ordinary Whittaker function $W^\ord_\pi$ characterized by \[W^\ord_\pi(\pDII{y}{1})=\chi\Abs^\onehalf(y)\bbI_{\Zp}(y)\quad (y\in\Qp^\x).\] Here $\bbI_{\Zp}$ is the characteristic function of $\Zp$. 
\end{cor}
\begin{proof} The proof of \propref{P:ordline} actually gives the recipe to construct the ordinary line. Indeed, let $W=W_{\pi\ot\chi^{-1}}$ be the Whittaker local newform of $\pi\ot\chi^{-1}$. Define $W^\dagger\in \sW(\pi\ot\chi^{-1})$ as follows: $W^\dagger=W$ if $\pi\ot\chi^{-1}$ is not spherical and $W^\dagger=W-\chi^{-2}\om\Abs^\onehalf(p)\rho(\pDII{p^{-1}}{1})W$ if $\pi\ot\chi^{-1}$ is spherical. An elementary calculation shows that $W^\dagger\ot\chi$ belongs to $\cW_\pi^\ord(\chi)$. By using the explicit formulas of Whittaker newforms (\cite[Section 2.4]{Schmidt02RJ}), we find that $W^\dagger\ot\chi(\pDII{y}{1})=\chi\Abs^\onehalf(y)\bbI_{\Zp}(y)$ as desired.
\end{proof}
\subsection{$p$-stabilized newforms}\label{SS:Whittaker}
Let $\pi$ be a cuspidal automorphic representation of $\GL_2(\A)$ and let $\cA(\pi)$ be the $\pi$-isotypic part in the space of automorphic forms on $\GL_2(\A)$. For $\varphi\in\cA(\pi)$, the Whittaker function of $\varphi$ (with respect to the additive character $\addchar_\Q:\A/\Q\to\C^\x$) is given by 
\[W_\varphi(g)=\int_{\A/\Q}\varphi(\pMX{1}{x}{0}{1}g)\addchar_\Q(-x)\rmd x\quad (g\in\GL_2(\A)),\]
where $\rmd x$ is the Haar measure with $\vol(\A/\Q,\rmd x)=1$. We have the Fourier expansion:
\[\varphi(g)=\sum_{\al\in\Q^\x}W_\varphi(\pDII{\al}{1}g)\]
(\cf\cite[Theorem 3.5.5]{Bump97Grey}). Let $f(q)=\sum_{n}a(n,f)q^n\in\sS_k(N,\chi)$ be a normalized Hecke eigenform, we shall denote by $\pi_f=\ot_v'\pi_{f,v}$ the cuspidal automorphic representation of $\GL_2(\A)$ generated by the adelic lift 
$\itPhi(f)$ of $f$. Then $\pi_f$ is irreducible and unitary with the central character $\chi^{-1}$. If $f$ is newform, then the conductor of $\pi_f$ is $N$, its adelic lift $\itPhi(f)$ is the normalized new vector in $\cA_0(\pi_f)$ and the Mellin transform 
\[Z(s,\itPhi(f))=\int_{\A^\x/\Q^\x}\itPhi(f)(\pDII{y}{1})\abs{y}^{s-\onehalf}\rmd^\x y=L(s,\pi_f)\]
is the automorphic $L$-function of $\pi_f$. Here $\rmd^\x y$ is the product measure $\prod_v\rmd^\x y_v$.

 \begin{defn}[$p$-stabilized newform]Let $p$ be a prime and fix an isomorphism $\iota_p:\C\iso\Qbarp$. We say that a normalized Hecke eigenform $f\in \sS_k(Np,\chi)$ is a (ordinary) \emph{$p$-stabilized newform} (with respoect to $\iota_p$) if $f$ is a new outside $p$ and the eigenvalue of $\bfU_p$, \ie the $p$-th Fourier coefficient $\iota_p(a(p,f))$, is a \padic unit. The prime-to-$p$ part of the conductor of $f$ is called \emph{the tame conductor} of $f$.\end{defn} 
 
 \begin{Remark}\label{R:WhittakerPordinary}Let $f$ be a $p$-stabilized newform. By the multiplicity one for new and ordinary vectors, the Whittaker function of the adelic lift $\itPhi(f)$ is a product of local Whittaker functions in $\cW(\pi_{f,v})$. To be precise,  
\[W_{\itPhi(f)}(g)=W_{\pi_{f,p}}^\ord(g_v)\prod_{v\not =p} W_{\pi_{f,v}}(g_v)\quad(g=(g_v)\in \GL_2(\A)).\]
Comparing the Fourier expansions of $\itPhi(f)$ and $f$ via \eqref{E:MA1}, we find that
\beq\label{E:Whittaker1}W_{\pi_{f,\ell}}(\pDII{\ell}{1})=\bfa(\ell,f)\ell^{-\frac{k}{2}} \text{ if }\ell\not =p;\quad W_{\pi_{f,p}}^\ord(\pDII{p}{1})=\bfa(p,f)p^{-\frac{k}{2}}.\eeq
By \corref{C:Word}, $W_{\pi_{f,p}}^\ord\in \cW(\pi_{f,p})^\ord(\al_{f,p})$, where $\al_{f,p}$ is the unramified character with $\al_{f,p}(p)=\bfa(p,f)p^{\frac{1-k}{2}}$.
\end{Remark}

\subsection{The bilinear form}
Let $\cA^0(\om)$ be the space of cusp forms in $\cA(\om)$. Let $\pairing$ denote the $\GL_2(\A)$-equivariant pairing between $\cA^0(\om)$ and $\cA^0(\om^{-1})$ defined by 
\[\pair{\varphi}{\varphi'}=\int_{\A^\x\GL_2(\Q)\bksl\GL_2(\A)}\varphi(g)\varphi'(g)\rmd^{\tau} g\]
 for $\varphi\in\cA^0(\om),\varphi'\in\cA^0(\om^{-1})$, where $\rmd ^{\tau} g$ is the Tamagawa measure of $\PGL_2(\A)$. The following lemma is well-known (\cf\cite[page 217]{Wald85}), and we omit the proof.
\begin{lm}\label{L:bilinear}For cusp forms $\varphi\in\cA^0_k(N,\om)$ and $\varphi'\in\cA^0_{-k}(N,\om^{-1})$, we have
\begin{align*}
\pair{X\varphi}{\varphi'}=&-\pair{\varphi}{X\varphi'} \text{ for }X\in\Lie(\GL_2(\R)),\\
\pair{\varphi}{\bfU_\pmq\varphi'}=&\pmq\pair{\LR_\pmq \varphi}{\varphi'}\text{ for }\ell\divides N,\\
\pair{T_\pmq\varphi}{\varphi'}=&\om(\ell)\pair{\varphi}{T_\pmq\varphi'} \text{ for }\ell\ndivides N.
\end{align*}
\end{lm}
 Let $\pi=\ot_v'\pi_v$ be an irreducible unitary cuspidal automorphic reprensentation on $\GL_2(\A)$ with central character $\om$. Denote by $\Contra{\pi}$ the contragredient representation of $\pi$. By the multiplicity one theorem, the pairing $\pairing$ gives rise to the equality $\cA(\Contra{\pi})=\cA(\pi)\ot\om^{-1}$. For a place $v$ of $\Q$, define the non-degenerate $\GL_2(\Q_v)$-equivariant pairing $\pairing$ between $\sW(\pi_v)$ and $\sW(\Contra{\pi}_v)$ by 
\beq\label{E:Wpair}\pair{W}{W'}=\int_{\Q_v^\x}W(\pDII{y}{1})W'(\pDII{-y}{1})\rmd^\x y\eeq
for $W\in\sW(\pi_v)$ and $\sW(\Contra{\pi}_v)$. This integral converges absolutely as $\pi_v$ is unitarizable. 
% it is well known that for any $W\in\cW_{\addchar_v}(\pi_v)$, \[\abs{W(\pDII{y}{1})}=O(\abs{y}^{\epsilon})\text{ when }\abs{y}\to 0\] for any $\epsilon>0$. 
\begin{prop}\label{P:Petersson}
Let $\varphi\in\cA(\pi)$ and $\varphi'\in\cA(\Contra{\pi})$. Suppose that $W_\varphi=\prod_v W_v$ and $W_{\varphi'}=\prod_v W'_v$ such that $W_v(1)=W'_v(1)=1$ for all but finitely many $v$. Then we have \[\pair{\varphi}{\varphi'}=\frac{2L(1,\pi,\Ad)}{\zeta_\Q(2)}\prod_v\frac{\zeta_v(2)}{\zeta_v(1)L(1,\pi_v,\Ad)}\pair{W_v}{W'_v}.\]
\end{prop}
\begin{proof}
This is \cite[Proposition 6]{Wald85}. Note that $W_v=W_{\pi_v}$ and $W'_v=W_{\Contra{\pi}_v}$ are the normalized local Whittaker newforms for all but finitely many $v$, and if $\pi_v$ is spherical, then 
\[\pair{W_{\pi_v}}{W_{\Contra{\pi}_v}}=\frac{\zeta_v(1)L(1,\pi_v,\Ad)}{\zeta_v(2)},\]
 so the right hand side of the equation in the proposition is indeed a finite product. \end{proof}
 
We give the formula of the local pairing of ordinary Whittaker functions.
\begin{lm}\label{L:ordlocalnorm}Let $p$ be a prime. Suppose that $\pi_p$ is a principal series $\Prin{\chi}{\upsilon}$ or a special representation $\chi\Abs^{-\onehalf}{\rm St}$. Let $W_{\pi_p}^\ord\in\sW(\pi_p)^\ord(\chi)$ be the normalized ordinary Whittaker function in \corref{C:Word}. If $n\geq \max\stt{1,c(\pi_p)}$, then we have  
\begin{align*}
\pair{\rho(\pMX{0}{p^{-n}}{-p^n}{0})W_{\pi_p}^\ord}{W_{\pi_p}^\ord\ot\om_p^{-1}}=&\chi(-1)\chi\upsilon^{-1}\Abs(p^{n})\cdot \gamma(0,\upsilon\chi^{-1})\zeta_p(1).
\end{align*}
Here $\upsilon=\chi^{-1}\om_p$ and $\gamma(s,-)$ is the $\gamma$-factor defined in \eqref{E:gamma1}.
\end{lm}
\begin{proof}Let $W=W_{\pi_p}^\ord$ and $\bftn=\pMX{0}{p^{-n}}{-p^n}{0}$. We first note that  if $n\geq \max\stt{1,c_p(\pi)}$, then $W(\aone{y}\bftn)=0$ if $y\not\in\Zp$. Then we have
\begin{align*}\pair{\rho(\bftn)W}{W\ot\om_p^{-1}}=&\int_{\Qp^\x}W(\aone{y}\bftn)W(\aone{-y})\om_p^{-1}(-y)\rmd^\x y\\
=&\int_{\Qp^\x}W(\pDII{y}{1}\bftn)\chi\om_p^{-1}(-y)\Abs^{s-\onehalf}(y)\rmd^\x y|_{s=1}.
%=&\upsilon(-1)\Psi(\bft_r,s,W^\ord\ot\upsilon^{-1})|_{s=1}\\
%=&\chi(-1)\gamma(1-s,\pi\ot\chi^{-1})\Psi(\pMX{0}{1}{-1}{0}\pMX{0}{p^{-r}}{-p^r}{0},1-s,W^\ord\ot\chi^{-1})|_{s=1}\\
%=&\upsilon(-1)\chi^{-1}\upsilon(p^r)\abs{p^{2r}}^{s-\onehalf}\zeta_p(1-s)\gamma(1-s,\pi\ot\chi^{-1})|_{s=1}
%=\upsilon(-1)\chi^{-1}\upsilon\Abs(p^r)\cdot \gamma(0,\upsilon\chi^{-1})\zeta_p(1)
%=&\chi_2(p^{C_\Pi})\varepsilon(1/2,\upsilon_2\chi_2^{-1})(1-p^{-1})^{-1}\\
\end{align*}
By the local functional equation for $\GL(2)$ (\cf\cite[Theorem 4.7.5]{Bump97Grey}), the last integral equals
\begin{align*}
&\om_p(p^{-n})\varepsilon(1-s,\pi\ot\chi^{-1})\frac{L(s,\pi\ot\chi\om_p^{-1})}{L(1-s,\pi\ot\chi^{-1})}\chi\om_p(-1)\int_{\Q_p^\x}W(\pDII{y}{1}\pMX{0}{1}{-1}{0}\pMX{0}{1}{-p^{2n}}{0})\chi^{-1}\Abs^{1/2-s}(y)\rmd^\x y|_{s=1}\\
=&\om_p(p^{-n})\chi(-1)\chi(p^{2n})\cdot\abs{p^{2n}}^{s-\onehalf}\varepsilon(1-s,\pi\ot\chi^{-1})\frac{L(s,\pi\ot\chi\om_p^{-1})}{L(1-s,\pi\ot\chi^{-1})} \zeta_p(1-s)|_{s=1}
%=&\chi\om(-1)\chi(p^n)\abs{p}^\frac{n}{2}\cdot \begin{cases}\gamma(0,\upsilon\chi^{-1})\zeta_p(1)&\text{ if }\pi_p=\Prin{\chi}{\upsilon},\\
%-\abs{p}^{-1}\zeta_p(2)&\text{ if }\pi_p=\chi\Abs^{-\onehalf}{\rm St}.\end{cases}
\end{align*}
Using the formula
\[\varepsilon(1-s,\pi\ot\chi^{-1})\frac{L(s,\pi\ot\chi\om_p^{-1})}{L(1-s,\pi\ot\chi^{-1})}=\begin{cases}\varepsilon(1-s,\upsilon\chi^{-1})\frac{\zeta_p(s)L(s,\chi\upsilon^{-1})}{\zeta_p(1-s)L(1-s,\upsilon\chi^{-1})}&\text{ if }\pi_p=\Prin{\chi}{\upsilon},\\
-\abs{p}^{-s}\frac{\zeta_p(s+1)}{\zeta_p(1-s)}&\text{ if }\pi_p=\chi\Abs^{-\onehalf}{\rm St},
\end{cases}\]
we see that
\[\pair{\rho(\bftn)W}{W\ot\om_p^{-1}}=\chi(-1)\om_p(p^{-n})\chi^2\Abs(p^n)\begin{cases}\gamma(0,\upsilon\chi^{-1})\zeta_p(1)&\text{ if }\pi_p=\Prin{\chi}{\upsilon},\\
-\abs{p}^{-1}\zeta_p(2)&\text{ if }\pi_p=\chi\Abs^{-\onehalf}{\rm St}.
\end{cases}\]
Finally, we note that if $\pi=\chi\Abs^{-\onehalf}{\rm St}$, then $\upsilon=\chi\Abs^{-1}$ and $\gamma(0,\upsilon\chi^{-1})\zeta_p(1)=-\abs{p}^{-1}\zeta_p(2)$. This finishes the proof.
\end{proof}

\subsection{Root numbers and Petersson norms}
Let $f\in\sS_k(N,\chi)$ be a normalized cuspidal newform of weight $k$ and conductor $N$. Put $f_c(z):=\ol{f(-\ol{z})}$. Then it is a classical result that 
\beq\label{E:rootnumbers}f|_k\pMX{0}{-1}{N}{0}=w(f)\cdot f_c\eeq
for some $w(f)\in\C^\x$ with the modulus $\abs{w(f)}=1$ (\cf\cite[Theorem 4.6.15]{Miyake06book}). This complex number $w(f)$ is called the \emph{root number} of $f$. %On the other hand, for each prime $\ell$, define the local root number 
%\beq\label{E:localroot}w_\ell(f):=\varepsilon(1/2,\pi_{f,\ell}).\eeq
By \cite[page 38]{Hida88Fourier}, we have 
\[w(f)=\prod_{\ell<\infty}\varepsilon(1/2,\pi_{f,\ell}).\]
Recall that the Petersson norm of $f$ is defined by
\[\norm{f}^2_{\Gamma_0(N)}=\int_{\Gamma_0(N)\bksl \frakH}\abs{f(x+\sqrt{-1}y)}^2y^{k}\frac{\rmd x\rmd y}{y^2}.\]
For each integer $M$, define the matrix $\Tau_{M}=(\Tau_{M,v})\in\GL_2(\A)$ by \beq\label{E:atkin}
\begin{aligned}
\Tau_{M,\infty}=&\pDII{-1}{1},\quad\Tau_{M,\ell}=1\text{ if }\ell\ndivides M
;\\
\Tau_{M,\ell}=&\pMX{0}{1}{-\ell^{\val_\ell(M)}}{0}\in\GL_2(\Q_\ell)\text{ if }\ell\divides M.
\end{aligned}\eeq
Let $\pi=\pi_f$ be the cuspidal automorphic representation generated by $\varPhi(f)$ with central character $\om(=\chi_\A^{-1})$. Define the local norm of the normalized Whittaker newform $W_{\pi_v}$  by 
\beq\label{E:localnormnew}
B_{\pi_v}=\frac{\zeta_v(2)}{\zeta_v(1)L(1,\pi_v,\Ad)}\pair{\rho(\Tau_{N,v})W_{\pi_v}}{W_{\pi_v}\ot\om_v^{-1}}.\eeq
It is straightforward to verify that
\[B_{\pi_\infty}=2^{-1-k},\, B_{\pi_\pmq}=1\text{ if }\pmq\ndivides N.
%B_{\pi_\ell}=&\frac{\varepsilon(1/2,\pi_\ell)\zeta_\ell(2)}{\zeta_\ell(1)L(1,\pi_\ell,\Ad)}\text{ if }L(s,\pi_\pmq)=1\end{aligned}\eeq 
\]
By \propref{P:Petersson} and \eqref{E:rootnumbers}, we have
\beq\label{E:normformula}\norm{f}_{\Gamma_0(N)}^2=\frac{[\SL_2(\Z):\Gamma_0(N)]}{2^k\cdot w(f)}\cdot L(1,\pi,\Ad)\cdot\prod_{q\divides N}B_{\pi_q}.\eeq

%!TEX root = TRIPLE3.tex

\def\caseII{{(\rm II)}}
\def\irr{{\rm (b)}}
\def\red{{\rm (a)}}
\def\caseI{{(\rm I)}}
\section{The unbalanced $p$-adic triple product $L$-functions}\label{S:unb}
\subsection{Ordinary $\Lam$-adic modular forms}\label{S:nc.unb}Let $p>2$ be a prime and let $\cO$ be the ring of integers of a finite extension of $\Qp$.  Let $\bfI$ be a normal domain finite flat over $\Lam=\cO\powerseries{1+p\Zp}$. A point $Q\in\Spec \bfI(\Qbar_p)$, a ring homomorphism $Q:\bfI\to\Qbar_p$ is said to be locally algebraic if $Q|_{1+p\Zp}$ is a locally algebraic character in the sense that $Q(z)=z^{k_Q}\ep_Q(z)$ with $k_Q$ an integer and $\ep_Q(z)\in\mu_{p^\infty}$. We shall call $k_Q$ the \emph{weight} of $Q$ and $\ep_Q$ the \emph{finite part} of $Q$. Let $\frakX_\bfI$ be the set of locally algebraic points $Q\in\Spec\bfI(\Qbarp)$ of weight $k_Q\geq 1$. A point $Q\in\frakX_\bfI$ is called \emph{arithmetic} if the weight $k_Q\geq 2$ and let $\frakX_\bfI^\ari$ be the set of arithmetic points. Let $\wp_Q=\Ker Q$ be the prime ideal of $\bfI$ corresponding to $Q$ and $\cO(Q)$ be the image of $\bfI$ under $Q$.

Fix an isomorphism $\iota_p:\Cp\iso\C$ once and for all. Denote by $\Om:(\Z/p\Z)^\x\to\mu_{p-1}$ the \padic \Teich character. Let $N$ be a positive integer prime to $p$ and let $\chi:(\Z/Np\Z)^\x \to\cO^\x$ be a Dirichlet character modulo $Np$. Denote by $\bfS(N,\chi,\bfI)$ the space of $\bfI$-adic cusp forms of tame level $N$ and (even) branch character $\chi$, consisting of  formal power series $\bdsf(q)=\sum_{n\geq 1}\bfa(n,\bdsf)q^n\in \bfI\powerseries{q}$ with the following property: there exists an integer $a_\bdsf$ such that for arithmetic points $Q\in\frakX^+_\bfI$ with $k_Q\geq a_\bdsf$, the specialization $\bdsf_Q(q)$ is the $q$-expansion of a cusp form $\bdsf_Q\in \sS_{k_Q}(Np^e,\chi\Om^{2-k_Q}\ep_Q)$. The character $\chi$ is called the \emph{branch character} of $\bdsf$.

The space $\bfS(N,\chi,\bfI)$ is equipped with the action of the usual Hecke operators $T_\ell$ for $\ell\ndivides Np$ as in \cite[page 537]{Wiles88} and the operators $\bfU_\ell$ for $\ell\divides pN$ given by $\bfU_\ell(\sum_n \bfa(n,\bdsf)q^n)=\sum_n\bda(n\ell,\bdsf)q^n$. For a positive integer $d$ prime to $p$, define $V_d:\bfS(N,\chi,\bfI)\to \bfS(Nd,\chi,\bfI)$ by $V_d(\sum_n\bda(n,\bdsf)q^n)=d\sum_n\bda(n,\bdsf)q^{dn}$. Recall that Hida's ordinary projector $\eord$ is defined by 
\[\eord:=\lim_{n\to\infty}\bfU_p^{n!}.\]
This ordinary projector $\eord$ has a well-defined action on the space of classical modular forms preserving the cuspidal part as well as on the space $\bfS(N,\chi,\bfI)$ of $\bfI$-adic cusp forms (\cf\cite[page 537 and Prop. 1.2.1]{Wiles88}). The space $\eord\bfS(N,\chi,\bfI)$ is called the space of ordinary $\bfI$-adic forms defined over $\bfI$. A key result in Hida's theory of ordinary $\bfI$-adic cusp forms is that if $\bdsf\in \eord\bfS(N,\chi,\bfI)$, then for \emph{every} arithmetic points $Q\in\frakX_\bfI$, we have $\bdsf_Q\in \eord\sS_{k_Q}(Np^e,\chi\Om^{2-k_Q}\ep_Q)$. 
We say $\bdsf\in\eord\bfS(N,\chi,\bfI)$ is a \emph{primitive Hida family} if for every arithmetic points $Q\in\frakX_\bfI$, $\bdsf_Q$ is a $p$-stabilized cuspidal newform of tame conductor $N$. 
Let $\frakX_\bfI^\cls$ be the set of classical points (for $\bdsf$) given by \[\frakX_\bfI^\cls:=\stt{Q\in\frakX_\bfI^\cls\mid \bdsf_Q\text{ is the $q$-expansion of a classical modular form}}.\]
Note that $\frakX_\bfI^\cls$ contains the set of arithmetic points $\frakX_\bfI^\ari$ but may be strictly larger than $\frakX_\bfI^\ari$ as we allow the possibility of weight one points.

\subsection{Galois representation attached to Hida families}\label{SS:GaloisRep}
Let $\Dmd{\cdot}:\Zp^\x\to 1+p\Zp$ be character defined by $\Dmd{x}=x\Om^{-1}(x)$ and write $z\mapsto [z]_\Lam$ for the inclusion of group-like elements $1+p\Zp\to \cO\powerseries{1+p\Zp}^\x=\Lam^\x$. For $z\in\Zp^\x$, denote by $\Dmd{z}_\bfI\in\bfI^\x$ the image of $[\Dmd{z}]_\Lam$ in $\bfI$ under the structure morphism $\Lam\to\bfI$. By definition, $Q(\Dmd{z}_\bfI)=Q(\Dmd{z})$ for $Q\in\frakX_\bfI$. Let $\cyc:G_\Q\to\Zp^\x$ be the \padic cyclotomic character and let $\Dmd{\cyc}_\bfI:G_\Q\to\bfI^\x$ be the character $\Dmd{\cyc}_\bfI(\sg)=\Dmd{\cyc(\sigma)}_\bfI$. For each Dirichlet character $\chi$, we define $\chi_\bfI:G_\Q\to \bfI^\x$ by $\chi_\bfI:={\boldsymbol \sigma}_\chi\Dmd{\cyc}^{-2}\Dmd{\cyc}_\bfI$, where ${\boldsymbol \sigma}_\chi$ is the Galois character which sends the geometric Frobenious element $\Frob_\ell$ at $\ell$ to $\chi(\ell)^{-1}$.

If $\bdsf\in\eord\bfS(N,\chi,\bfI)$ is a primitive Hida family of tame conductor $N$, we let $\rho_{\bdsf}:G_\Q\to\GL_2(\Frac\bfI)$ be the $\bfI$-adic Galois representation attached to $\bdsf$ characterized by \[\Tr(\rho_\bdsf(\Frob_\ell))=\bfa(\ell,\bdsf);\quad\det\rho_\bdsf(\Frob_\ell)=\chi\Om^2(\ell)\Dmd{\ell}_\bfI\ell^{-1}\quad (\ell\ndivides pN)%=\brchf\Om^2\cyc\Dmd{\cyc}_\bfI^{-1}(\Frob_\ell)
.\]
Note that $\det\rho_\bdsf=\chi_\bfI^{-1}\cdot\cyc^{-1}$. We have a complete knowledge of the description of the restriction of $\rho_{\bdsf}$ to the local decomposition group $G_{\Q_\ell}$. For $\ell=p$,  according to \cite[Theorem 2.2.1]{Wiles88}, 
\[\rho_\bdsf|_{G_{\Qp}}\sim \pMX{\al_p}{*}{0}{\al_p^{-1}\chi_\bfI^{-1}\cyc^{-1}}\]
where $\al_p:G_{\Qp}\to\bfI^\x$ is the unramified character with $\al_p(\Frob_p)=\bfa(p,\bdsf)$.\footnote{Our representation  $\rho_\bdsf$ is the dual of $\rho_{\sF}$ considered in \cite{Wiles88}.}
For $\ell\not =p$, enlarging $\bfI$ if necessary, we have the following list of $\rho_{\bdsf}|_{G_{\Q_\ell}}$.
\begin{enumerate}
\item (Principal series) $\rho_\bdsf|_{G_{\Q_\ell}}$ is reducible and isomorphic to 
\[\al_\ell\xi\cyc^{1/2}\Dmd{\cyc}_\bfI^{-1/2}\oplus \al_\ell^{-1}\xi'\cyc^{1/2}\Dmd{\cyc}_\bfI^{-1/2}\]
with a unramified characters $\al_\ell:G_{\Q_\ell}\to\bfI^\x$ and a finite order characters $\xi,\xi':G_{\Q_\ell}\to\Qbar^\x$ with $\xi\xi'=\chi^{-1}\Om^{-2}$. 

\item (Special) $\rho_\bdsf|_{G_{\Q_\ell}}$ is indecomposable and 
\[\rho_{\bdsf}|_{G_{\Q_\ell}}\sim \pMX{\xi\cyc\Dmd{\cyc}_\bfI^{-1/2}}{*}{0}{\xi\Dmd{\cyc}_\bfI^{-1/2}}\]
with a finite order character $\xi:G_{\Q_\ell}\to\Qbar^\x$ such that $\xi^2=\chi^{-1}\Om^{-2}$.
\item (Supercuspidal) $\rho_\bdsf|_{G_{\Q_\ell}}$ is irreducible and $\rho_\bdsf\iso\rho_0\ot\Dmd{\cyc}_{\bfI}^{-1/2}$ with $\rho_0:G_{\Q_\ell}\to\GL_2(\Qbar)$ irreducible representation of finite image
\end{enumerate}
(\cf \cite[page 689]{SU06}).
% The proof uses Grothendieck's monodromy theorem for Galois representation and the fact that a non-constant element in $\bfI$ can takes infinitely many values.
\begin{Remark}[Rigidity of automorphic types]\label{R:rigidity}
We recall the \emph{rigidity of automorphic types} for a primitive Hida family $\bdsf$ in \cite[Lemma 2.14]{OF12Crelle}. Let $\ell\not =p$ be a prime. If for some arithmetic point $Q$ the associated cuspidal automorphic representation $\pi_{\bdsf_Q,\ell}$ is principal series (resp. special, supercuspidal) of conductor $\ell^n$, then for \emph{any} arithmetic point $Q'$, $\pi_{\bdsf_{Q'},\ell}$ is also principal series (resp. special, supercuspidal) of the same conductor $\ell^n$. This is a consequence of the above description of $\rho_{\bdsf}|_{G_{\Q_\ell}}$, the Langlands correspondence and the Ramanujan conjecture for elliptic modular forms (only needed in the case (Special)). 

In addition, if $\pi_{\bdsf_Q,\ell}$ is a discrete series at any arithmetic point $Q\in\frakX_\bfI^+$, then the Weil-Deligne representaion associated with the specialization of $\rho_{\bdsf}\ot\Dmd{\cyc}^{1/2}_\bfI |_{G_{\Q_\ell}}$ at $Q$ is independent of $Q$.

\end{Remark}
\def\rmT{\bbT}

 \subsection{Hecke algebras and congruence numbers}\label{SS:Heckealgebra}
If $N$ is a positive integer and $\chi$ is a Dirichlet character modulo $N$, we let $\rmT_k(N,\chi)$ be the $\cO$-subalgebra in $\End_\C\eord\sS_k(N,\chi)$ generated over $\cO$ by the Hecke operators $T_\ell$ for $\ell\ndivides Np$ and the operators $\bfU_\ell$ for $\ell\divides Np$. Suppose that $N$ is \emph{prime to} $p$. Let $\Delta=(\Z/Np\Z)^\x$ and $\wh \Delta$ be the group of Dirichlet characters modulo $Np$. Enlarging $\cO$ if necessary, we assume that every $\chi\in\wh\Delta$ takes value in $\cO^\x$. We are going to consider the Hecke algebra $\bfT(N,\bfI)$ acting on the space of ordinary $\Lam$-adic cusp forms of tame level $\Gamma_1(N)$ defined by 
\[\bfS(N,\bfI)^\ord:=\bigoplus_{\chi\in\wh\Delta}\eord\bfS(N,\chi,\bfI).\]
In addition to the action of Hecke operators, denote by $\sg_d$ the usual diamond operator for $d\in \Delta$ acting on $\bfS(N,\bfI)^\ord$ by $\sg_d(\bdsf)_{\chi\in\wh\Delta}=(\chi(d)\bdsf)_{\chi\in\wh \Delta}$. Then the ordinary $\bfI$-adic cuspidal Hecke algebra $\bfT(N,\bfI)$ is defined to be the $\bfI$-subalgebra of $\End_{\bfI}\bfS(N,\bfI)^\ord$ generated over $\bfI$ by $T_\ell$ for $\ell\divides Np$, $\bfU_\ell$ for $\ell\divides Np$ and the diamond operators $\sg_d$ for $d\in \Delta$. Let $Q\in\frakX_\bfI^\ari$ be an arithmetic point. Every $t\in \bfT(N,\bfI)$ commutes with the specialization: $(t\cdot \bdsf)_Q=t\cdot \bdsf_Q$. For $\chi\in \wh\Delta_{Np}$, let $\wp_{Q,\chi}$ be the ideal of $\bfT(N,\bfI)$ generated by $\wp_Q$ and $\stt{\sg_d-\chi(d)}_{d\in \Delta}$. A classical result \cite[Theorem 3.4]{Hida88Annals} in Hida theory asserts that \[\bfT(N,\bfI)/\wp_{Q,\chi}\iso \rmT_{k_Q}(Np^e,\chi\Om^{2-k_Q}\ep_Q)\ot_\cO\cO(Q).\]

Let $\bdsf\in\eord\bfS(N,\chi,\bfI)$ be a primitive Hida family of tame level $N$ and character $\chi$ and let $\lam_\bdsf: \bfT(N,\bfI)\to\bfI$ be the corresponding homomorphism defined by $\lam_\bdsf(T_\ell)=\bfa(\ell,\bdsf)$ for $\ell\ndivides Np$, $\lam_\bdsf(\bfU_\ell)=\bfa(\ell,\bdsf)$ for $\ell\divides Np$ and $\lam_\bdsf(\sg_d)=\chi(d)$ for  $d\in\Delta$. Let $\frakm_\bdsf$ be the maximal of $\bfT(N,\bfI)$ containing $\Ker\lam_\bdsf$ and let $\bfT_{\frakm_\bdsf}$ be the localization of $\bfT(N,\bfI)$ at $\frakm_\bdsf$. It is the local ring of $\bfT(N,\bfI)$ through which $\lam_\bdsf$ factors. Recall that the congruence ideal $C(\bdsf)$ of the morphism $\lam_\bdsf:\bfT_{\frakm_\bdsf}\to \bfI$ is defined by 
\[C(\bdsf):=\lam_\bdsf(\Ann_{\bfT_{\frakm_\bdsf}}(\Ker\lam_\bdsf))\subset\bfI.\]
The Hecke algebra $\bfT_{\frakm_\bdsf}$ is a local finite flat $\Lam$-algeba, and by the primitiveness of $\bdsf$, there is an algebra direct sum decomposition 
\beq\label{E:HeckeDecom}\lam:\bfT_{\frakm_\bdsf}\ot_{\bfI}\Frac\bfI\iso\Frac\bfI\oplus \sB,\,\,t\mapsto \lam(t)=(\lam_\bdsf(t),\lam_\sB(t)),\eeq where $\sB$ is some finite dimensional $(\Frac\bfI)$-algebra (\cite[Corollaty 3.7]{Hida88Annals}). By definition we have \[C(\bdsf)=\lam_\bdsf(\bfT_{\frakm_\bdsf}\cap \lam^{-1}(\Frac\bfI\oplus\stt{0})).\]
Now we impose the following \begin{hypothesis*}[CR] The residual Galois representation $\ol{\rho}_{\bdsf}$ of $\rho_{\bdsf}$ is absolutely irreducible and $p$-distiniguished. \end{hypothesis*} Under the above hypothesis, $\bfT_{\frakm_\bdsf}$ is Gorenstein by \cite[Corollay 2, page 482]{Wiles95}, and with this Gorenstein property of $\bfT_{\frakm_\bdsf}$, Hida in \cite{Hida88AJM} proved that the congruence ideal $C(\bdsf)$ is generated by a non-zero element $\eta_\bdsf\in\bfI$, called the congruence number for $\bdsf$. Let $1^*_\bdsf$ be the unique element in $\bfT_{\frakm_\bdsf}\cap \lam^{-1}(\Frac\bfI\oplus\stt{0})$ such that $\lam_\bdsf(1^*_\bdsf)=\eta_\bdsf$. Then $1_\bdsf:=\eta_\bdsf^{-1}1_\bdsf^*$ is the idempotent in $\bfT_{\frakm_\bdsf}\ot_\bfI\Frac\bfI$ corresponding to the direct summand $\Frac\bfI$ of \eqref{E:HeckeDecom} and $1_\bdsf$ does not depend on any choice of a generator of $C(\bdsf)$. Moreover, for each arithmetic point $Q$, it is also shown by Hida that the specialization $\eta_\bdsf(Q)\in\cO(Q)$ is the congruence number for $\bdsf_Q$ and \[1_f:=\eta_{\bdsf}^{-1}1^*_{\bdsf}\pmod{\wp_{\chi,Q}}\in \rmT^\ord_{k_Q}(Np^r,\chi\Om^{2-k_Q}\ep_Q)\ot_\cO\Frac\cO(Q)\] is the idempotent with $\lam_f(1_f)=1$.

There is a unique decomposition $\chi=\chi^{(p)}\chi_{(p)}$ of Dirichlet characters, where $\chi^{(p)}$ and $\chi_{(p)}$ are Dirichlet characters modulo $N$ and $p^r$ respectively.  We call $\chi_{(p)}$ the $p$-primary component of $\chi$. Let $\ol{\chi}=\chi^{-1}$ be the complex conjugation of $\chi$. Denote by
$\breve{\bdsf}\in \eord\bfS(N,\chi_{(p)}\ol{\chi}^{(p)},\bfI)$ the primitive Hida family corresponding to the twist $\bdsf|[\ol{\chi}^{(p)}](q)=\sum_{(n,N)=1}\ol{\chi}^{(p)}(n)\bfa(n,\bdsf)q^n$ (\cf\cite[Lemma 6.1]{Dim14}). To be precise, the Fourier coefficients of $\breve\bdsf$ are given by 
\[\bfa(\ell,\breve\bdsf)=\begin{cases}\ol{\chi}^{(p)}(\ell)\bfa(\ell,\bdsf)&\text{ if }\ell\ndivides N,\\
\bfa(\ell,\bdsf)^{-1}\chi_{(p)}\Om^2(\ell)\ell^{-1}\Dmd{\ell}_\bfI&\text{ if }\ell\divides N.\end{cases}
\]
by \cite[Theorem 4.6.16]{Miyake06book}. For every arithmetic point $Q\in\frakX^\ari$, $\breve\bdsf_Q$ is the $p$-stabilized newform attached to $\bdsf_Q|[\ol{\chi}^{(p)}]$. Moreover, the Atkin-Lehner involution $\eta_p'$ introduced in \cite[(4.6.21), page 168]{Miyake06book}) induces an isomorphism  $\eta'_p:\sS_k(Np^r,\chi\Om^{2-k_Q}\ep_Q)\iso \sS_k(Np^r,\ol{\chi}^{(p)}\chi_{(p)}\Om^{2-k_Q}\ep_Q)$ such that $T_\ell\eta'_p=\ol{\chi}^{(p)}(\ell) \eta_p'T_\ell$ for $\ell\ndivides N$ (\cite[(4.6.23)]{Miyake06book}). We thus obtain a $\Lam$-algebra isomorphism $[\ol{\chi}^{(p)}]:\bfT_{\frakm_\bdsf}\iso\bfT_{\frakm_{\breve\bdsf}}$ such that $[\ol{\chi}^{(p)}](T_\ell)= T_\ell\cdot\ol{\chi}^{(p)}(\ell)$ for $\ell\ndivides N$ and $\lam_{\breve\bdsf}\circ[\ol{\chi}^{(p)}]=\lam_\bdsf$. It follows that \beq\label{E:idempotent}1^*_{\breve\bdsf}=[\ol{\chi}^{(p)}](1^*_{\bdsf})\text{ and }\eta_{\breve\bdsf}=\eta_{\bdsf}.\eeq

\subsection{The adjustment of levels for a triple of modular forms}\label{SS:auxiliary}For any positive integer $M$, let $\supp(M)$ denote the support of $M$, \ie the set of prime factors of $M$. If $f$ is a $p$-stabilized newform of tame conductor $\condf$, let $c_\ell(f):=c(\pi_{f,\ell})$ be the exponent of the $\ell$-component of $\condf$ for each prime $\ell\not =p$ and set\begin{align*}\Sigma_f^1&=\stt{\ell\text{ : prime}\mid \pi_{f,\ell} \text{ is a principal series}};\\
\Sigma_f^0&=\stt{\ell\text{ : prime}\mid \pi_{f,\ell} \text{ is a discrete series}}.
\end{align*}
%If $(f,g)$ is a pair of $p$-stabilized newforms, then we put
%\begin{align*}
%\Sigma^{\red}_{fg}&=\stt{\ell\in\Sigma_f^0\cap\Sigma_g^0\mid L(s,\pi_{f,\pmq}\ot\pi_{g,\pmq})\not =1};\\
%\Sigma^{\irr}_{fg}&=\stt{\ell\in\Sigma_f^0\cap\Sigma_g^0\mid L(s,\pi_{f,\pmq}\ot\pi_{g,\pmq})=1}.
%\end{align*} We let \[c_\pmq(fg)=\val_\pmq(\lcm(\condf,\condg))=\max\stt{c_\ell(f),c_\ell(g)}.\] 

To a triple $(f,g,h)$ of $p$-stabilized newforms of tame conductors $(\condf,\condg,\condh)$, we are going to associate a set of auxiliary integers $(\Bd_f,\Bd_g,\Bd_h)$, which we call \emph{the adjustment of levels} for $(f,g,h)$. This adjustment of levels is crucial for the construction of our test $\Lam$-adic modular forms (\defref{D:testunb} and \defref{D:testbal}) in order to obtain the optimal value of the local zeta integrals in Ichino's formula at bad places, and it is defined according to the choice of good local test vectors in the space of product of local representations of $\pi_{f,\ell}\times\pi_{g,\ell}\times\pi_{h,\ell}$ (\cf\subsecref{SS:61}) at bad primes $\ell\not =p$. Inevitably, the definition is very ad-hoc and may seem to be artificial at the first sight. The readers are advised to skip the precise definition in this subsection at the first reading and come back until \subsecref{SS:61}. To begin with, let $N_{fgh}=\gcd(\condf,\condg,\condh)$ and $N=\lcm(\condf,\condg,\condh)$. Put \[c^{\min}_\ell:=\val_\ell(N_{fgh});\quad c_\ell(fg)=\max\stt{c_\ell(f),c_\ell(g)}.\] Let $\Sigma_{fgh}=\Sigma_f^0\cap\Sigma_g^0\cap\Sigma_h^0$. We introduce several disjoint subsets of $\supp(N)$:\begin{align*}
\Sigma_{fg}^{\caseI}&=\stt{
\ell\in\Sigma_f^1\cup\Sigma_g^1\cup \Sigma_{fgh}\mid c_\ell(h)<\min\stt{c_\ell(f),c_\ell(g)}},\\
\Sigma_f^{\rm (IIa)}&=\stt{\ell\in\Sigma_g^0\cap\Sigma_h^0\mid L(s,\pi_{g,\ell}\ot\pi_{h,\ell})\not=1,\,c_\ell(f)=0},\\
\Sigma_f^{\rm (IIb)}&=\stt{\ell\in\Sigma_g^0\cap\Sigma_h^0\mid L(s,\pi_{g,\ell}\ot\pi_{h,\ell})=1,\,\,\ell\in\Sigma_f^1,\,\,c_\ell(f)<\min\stt{c_\ell(g),c_\ell(h)}},\\
\Sigma_{f}^{\max}&=\stt{\ell\text{ : prime factor of }\condf\mid c_\ell(g)=c_\ell(h)=c_\ell^{\min}<c_\ell(f)}.
\end{align*}
Define $\Sigma_{fh}^{\caseI},\,\Sigma_{gh}^{\caseI},\Sigma_g^{\rm (IIa)}\,,\Sigma_g^{\rm (IIb)},\,\Sigma_g^{\max},\ldots,$ in the same manner.  We set \begin{align*}\Bd_f^\caseI&=\prod_{\ell\in\Sigma_{fg}^{\caseI}}\ell^{c_\ell(fg)-c_\ell(f)}\cdot\prod_{\ell\in\Sigma_{fh}^{\caseI}}\ell^{c_\ell(fh)-c_\ell(f)},\\
\Bd^{\caseII}_{f}&=\prod_{\ell\in\Sigma_f^{\rm (IIa)}}\ell^{\lceil \frac{c_\ell(gh)}{2}\rceil}\cdot\prod_{\ell\in\Sigma_f^{\rm (IIb)}}\ell^{c_\ell(gh)-c_\ell(f)},\\
 \Bd_f^{\max}&=\prod_{\ell\in\Sigma_f^{\max}}\ell^{c_\ell(f)-c_\ell^{\min}}.
\end{align*}
%\begin{align*}\Bd_f^\caseI&=\prod_{\substack{\ell\divides N_{fg},\,
%\ell\in\Sigma_f^1\cup\Sigma_g^1\cup \Sigma_{fgh}^0,\\ c_\ell(h)=c^{\min}_\ell}}\ell^{c_\ell(fg)-c_\ell(f)}\cdot\prod_{\substack{\ell\divides N_{fh},\,\Sigma_f^1\cup\Sigma_h^1\cup\Sigma_{fgh}^0,\\ c^{\min}_\ell=c_\ell(g)}}\ell^{c_\ell(fh)-c_\ell(f)},\\
%\Bd^{\caseII}_{f}&=\prod_{\substack{\ell\in\Sigma_{gh}^{\red},\,\ell\in\Sigma_f^1,\\ c_\ell(f)=c_\ell^{\min}}}\ell^{\lceil \frac{c_\ell(gh)}{2}\rceil}\cdot\prod_{\substack{\ell\in\Sigma_{gh}^{\irr},\ell\in\Sigma_f^1,\\c_\ell(f)=c_\ell^{\min}}}\ell^{c_\ell(gh)-\val_\ell(N_1)} .
%\Bd^{\red}_{gh}&=\prod_{\ell\in\Sigma_{gh}^{\red},\,\ell\ndivides \condf}\ell^{\lceil \frac{c_\ell(gh)}{2}\rceil}
%\end{align*}
Likewise we define $\Bd_g^\caseI,\Bd^{\caseII}_{g},\Bd_g^{\max},\Bd_h^\caseI,\Bd^{\caseII}_{h}$ and $\Bd_h^{\max}$. Finally, put
\beq\label{E:defBd}\Bd_f=\Bd_f^\caseI\Bd_{f}^{\caseII},\quad \Bd_g=\Bd_g^\caseI \Bd_f^{\max}\Bd_h^{\max}\cdot\Bd_{g}^{\caseII}\text{ and }\Bd_h=\Bd_h^\caseI\Bd_{g}^{\max}\cdot\Bd_{h}^{\caseII}.\eeq
By definition, we have 
\beq\label{E:supp1}\begin{aligned} \Bd_f\mid N/\condf,\quad \Bd_g\mid N/\condg,\quad \Bd_h\mid N/\condh.
%\supp N=\supp \Bd_f\disjoint \supp \Bd_g\disjoint \supp \Bd_h.
\end{aligned}\eeq
\subsection{Definitions of good test $\Lambda$-adic modular forms}\label{S:levelraising}
Let $\cO=\cO_F$ for some finite extension $F$ of $\Qp$. Fixing a topological generator $\gamma_0$ of $1+p\Zp$, we let $\Lam=\cO\powerseries{1+p\Zp}=\cO\powerseries{T}$ with $T=\gamma_0-1$. For $i=1,2,3$, let $\bfI_i$ be a normal domain finite flat over $\Lam$ and let $\psi_i:(\Z/pN_i\Z)^\x\to \cO^\x$ be Dirichlet characters with $\psi_i(-1)=1$. Throughout this paper, we fix a triplet of primitive Hida families 
\[\bdsF:=(\bdsf,\bdsg,\bdsh)\in  \eord\bfS(\condf,\brchf,\bfI_1)\times  \eord\bfS(\condg,\brchg,\bfI_2)\times  \eord\bfS(\condh,\brchh,\bfI_3)\]
of tame conductors $\ulN=(\condf,\condg,\condh)$ and branch characters $\ul{\psi}=(\brchf,\brchg,\brchh)$. We shall impose the following running hypotheses
\beqcd{ev}\brchf\brchg\brchh=\Om^{2a}\text{ for some }a\in\Z;\eeqcd 
\beqcd{sf}\gcd(\condf,\condg,\condh)\text{ is square-free.}\eeqcd
 \begin{lm}\label{L:1.ub}Let $(\Qx,\Qy,\Qz)\in\frakX_{\bfI_1}^\cls\times \frakX_{\bfI_2}^\cls\times\frakX_{\bfI_3}^\cls$ and $(f,g,h)=\bdsF_\ulQ=(\bdsf_\Qx,\bdsg_\Qy,\bdsh_\Qz)$ be the specialization of $\bdsF$ at $\ulQ$. The adjustment of levels $\Bd_{f}^{\bullet},\Bd_g^\bullet$ and $\Bd_h^\bullet$ for $\bullet\in\stt{\rm (I), (II), max}$ are independent of the choice of any arithmetic point $\ulQ$.
\end{lm}
\begin{proof} 
 The lemma is clear from the rigidity of automorphic types, the description of the restriction of $\rho_{\bdsf}|_{G_{\Q_\ell}}$ given in \subsecref{SS:GaloisRep} and the Langlands correspondence for $\GL(2)$. \end{proof}

\begin{defn}[Test $\Lam$-adic forms]\label{D:testunb} 
Let $N=\lcm(\condf,\condg,\condh)$. Put \[\Sigma_{?,0}^{\rm (IIb)}=\stt{\ell\in\Sigma_{?}^{\rm (IIb)}\mid c_\ell(?)=0}\text{ for }?\in\stt{f,g,h}.\] For each $\ell\in\Sigma^{\rm (IIb)}_{f,0}$ (resp. $\Sigma^{\rm (IIb)}_{g,0},\,\Sigma^{\rm (IIb)}_{h,0}$), we fix once and for all a root $\beta_\ell(\bdsf)\in\bfI_1^\x$ (resp. $\beta_\ell(\bdsg)\in\bfI_2^\x,\,\beta_\ell(\bdsh)\in \bfI_3^\x$) of the Hecke polynomial $H_{\bdsf,\ell}(X):=X^2-\bfa(\ell,\bdsf)X+\brchf\Om^{2}(\ell)\ell^{-1}\Dmd{\ell}_{\bfI_1}$ (resp. $H_{\bdsg,\ell}(X),\,H_{\bdsh,\ell}(X))$. With the above notation in the previous subsection, we define the pair $(\bdsg^\star,\bdsh^\star)$ in $ \eord\bfS(N,\brchg,\bfI_2)\times\eord\bfS(N,\brchh,\bfI_3)$ of the ordinary $\Lam$-adic cusp forms by
\begin{align*}
%\bdsf^\star(q)&=\sum_{I\subset \Sigma_{gh}^{\irr}}(-1)^{\abs{I}}\beta_I(\bdsf)^{-1}n_I^{-1}\bdsf(q^{\Bd_f/n_I}),\\
\bdsg^\star(q)&=\sum_{I\subset \Sigma_{g,0}^{\rm (IIb)}}(-1)^{\abs{I}}\beta_I(\bdsg)^{-1}\LR_{\Bd_g/n_I}\bdsg,\\
\bdsh^\star(q)&=\sum_{I\subset \Sigma_{h,0}^{\rm (IIb)}}(-1)^{\abs{I}}\beta_I(\bdsh)^{-1}\LR_{\Bd_h/n_I}\bdsh,
\end{align*}
where $n_I=\prod_{\ell\in I}\ell$, $\beta_I(?)=\prod_{\ell\in I}\beta_\ell(?)$ for $?=\bdsf,\bdsg,\bdsh$. 
\end{defn}

\subsection{The construction of the \padic $L$-function in the unblanced case}\label{SS:36} 
We let  \[\cR=\bfI_1\wh\ot_\cO\bfI_2\wh\ot_\cO\bfI_3\] be a finite extension over the three variable Iwasawa algebra \begin{align*}\cR_0:&=\Lam\wh\ot_\cO\Lam\wh\ot_\cO\Lam=\cO\powerseries{T_1,T_2,T_3},\\
(T_1=T\ot 1\ot 1,&\,\,T_2=1\ot T\ot 1,\,\,T_3=1\ot 1\ot T).\end{align*} Define the multiplicative map $\Theta:\Z_{(p)}^\x\to \cR^\x$ by
\[\Theta(n):=\psi_{1,(p)}\Om^{-a-1}(n)\Dmd{n}^{1/2}_{\bfI_1}\Dmd{n}_{\bfI_2}^{-1/2}\Dmd{n}_{\bfI_3}^{-1/2}.\]
Define the $\cR$-adic twisting operator $|[\Theta]:\cR[[q]]\to \cR[[q]]$ by 
\[(\sum_{n\geq 0} \bfa(n)q^n)|[\Theta]=\sum_{n\geq 0,\,p\ndivides n} \Theta(n)\cdot\bfa(n)q^n.\]
Here $\psi_{1,(p)}$ is the restriction of the branch character $\brchf$ of $\bdsf$ to $(\Z/p\Z)^\x$. Define the power series $\bdsH$ by 
\[\bdsH:=\bdsg^\star\cdot \bdsh^\star|[\Theta]\in \cR\powerseries{q}.\]
\begin{lm}\label{L:bdsH1}The power series $\bdsH$ belongs to $\bfS(N,\psi_{1,(p)}\ol{\brchf}^{(p)},\bfI_1)\wh\ot_{\bfI_1}\cR$. \end{lm}
\begin{proof}The following proof is taken from Hida's blue book \cite{Hida93Blue}. Put
\[\frakX_\cR^0=\stt{\ulQ=(\Qx,\Qy,\Qz)\in\frakX_{\bfI_1}^\ari\times\frakX_{\bfI_2}^\ari\times\frakX_{\bfI_3}^\ari\mid k_\Qx=k_\Qy+k_\Qz,\,\,k_\Qx\geq k_\Qy+2}.\]
For $\ulQ\in \frakX_\cR^0$, put
\[\Bkappa_0=\psi_{1,(p)}\Om^{-a-1}\ep_\Qx^{1/2}\ep_\Qy^{-1/2}\ep_\Qz^{-1/2}.\]
Here $\ep_?^{1/2}$ is the unique square root of $\ep_?$ taking value in $1+p\Zp$. We verify that $(\bdsh^\star|[\Theta])_\ulQ=\bdsh_\Qz|[\Bkappa_0]\in \sS_{k_\Qz}(N,\psi_{1,(p)}^2\brchf^{-1}\brchg^{-1}\ep_\Qx\ep_\Qy^{-1})$, and hence we find that for every $\ulQ\in\frakX_\cR^0$,
\beq\label{E:L341.ub}\bdsH_\ulQ=\bdsg^\star_\Qy\cdot\bdsh^\star_\Qz|[\Bkappa_0]\in \sS_{k_\Qx}(N,\psi_{1,(p)}\ol{\brchf}^{(p)}\Om^{2-k_\Qx}\ep_\Qx).\eeq
We have $\cR_0=\cO\powerseries{T_1,T_2,Z}$ with $Z=(1+T_1)^{-1}(1+T_2)(1+T_3)-1$. Let $L_0=\Frac\cR_0$ and $L=\Frac\cR$ be a finite extension of $L_0$. Let $\al_1,\cdots,\al_n$ be a basis of $\cR$ over $\cR_0$ and write $\bdsH=\sum_{j=1}^n \bdsH^{(j)}\al_j$ with $\bdsH^{(j)}\in \cR_0\powerseries{q}$. On the other hand, letting $\stt{\al^*_j}_{j=1,\ldots,n}$ be the dual basis of $\stt{\al_j}_{j=1,\ldots,n}$ with respect to the trace map $\Tr:L\to L_0$, we have  $\bdsH^{(j)}=\Tr(\bdsH\al_j^*)$. Let $\bfu=1+p$. By \eqref{E:L341.ub}, we can write $\bdsH^{(j)}=\bdsH^{(j)}(T_1,T_2,Z)\in \cO\powerseries{T_1,T_2,Z}\powerseries{q}$ so that
\beq\label{E:L342.ub}\bdsH^{(j)}(\bfu^{k_1}\zeta_1-1,\bfu^{k_2}\zeta_2-1,\zeta_3-1)=\Tr(\bdsH_\ulQ\al_i(\ulQ))\in \sS_{k_1}(N,\psi_{1,(p)}\ol{\brchf}^{(p)}\Om^{2-k_1})\eeq
for all but finite many positive integers $k_1,k_2$ with $k_1\geq k_2+2$ and $\zeta_i\in \mu_{p^\infty}$ ($i=1,2,3$), where 
$\ulQ=(\Qx,\Qy,\Qz)$ are some arithmetic points of weights $(k_1,k_2,k_1-k_2)$ and finite parts $(\ep_\Qx,\ep_\Qy,\ep_\Qz)$, $\ep_{Q_i}(z)$ is the finite order character with $\ep_{Q_i}(u)=\zeta_i$.

To prove the lemma, it suffices to show that
\label{E:L343.ub}\beq\bdsH^{(j)}(T_1,T_2,Z)\in \bdS\wh\ot_{\cO}\cO\powerseries{T_2,Z},\quad \bfS:=\bfS(N,\psi_{1,(p)}\ol{\brchf}^{(p)},\cO\powerseries{T_1}),\eeq 
which in turn, by \cite[Lemma 1 in page 328]{Hida93Blue}, is equivalent to showing that $\bdsH^{(j)}(T_1,T_2,\zeta-1)\in \bfS\wh\ot_\cO\cO[\zeta]\powerseries{T_2}$ for every $\zeta\in\mu_{p^\infty}$. Now we repeat the arguments in \cite[page 226-227]{Hida93Blue}. Let $a$ be a positive integer such that $\bdsg_Q$ is a classical modular form for all $Q\in\frakX_\bfI$ with $k_Q=a$. For $m=1,2,\ldots$, we define the power series inductively \begin{align*}H_0(T_1,T_2)&=\bdsH^{(j)}(T_1,T_2,\zeta-1),\quad Y_m=T_2-(\bfu^{m+a-1}-1)\in \cO\powerseries{T_2},\\ 
H_{m}(T_1,T_2)&=\frac{H_{m-1}(T_1,T_2)-H_{m-1}(T_1,\bfu^{m+a-1}-1)}{Y_{m}}\in \cO\powerseries{T_1,T_2}\powerseries{q}\end{align*} 
Then \eqref{E:L342.ub} implies that $H_0(T_1,\bfu^a-1)\in \bfS\ot_\cO\cO[\zeta]$ and by induction, we find easily that $H_{m}(T_1,\bfu^{m+a}-1)\in \bfS\ot_\cO\cO[\zeta]$ for all $m=0,1,\ldots$. On other hand, by construction we have
\[\bdsH^{(j)}(T_1,T_2,\zeta-1)=\sum_{m=0}^\infty H_m(T_1,\bfu^{m+a}-1)\prod_{i=1}^m{Y_i}.\]
It is clear that the right hand side is a convergent power series and belongs to $\bfS\wh\ot_\cO\cO[\zeta]\powerseries{T_2}$.
  \end{proof}
  
Define the auxiliary $\cR$-adic form $\bdsH^{\rm aux}$ by \beq\label{E:bdsH}\bdsH^{\rm aux}:=\sum_{I\subset \Sigma_{f,0}^{\rm (IIb)}} (-1)^I\frac{\psi_{1,(p)}(n_I/\Bd_f)\Dmd{n_I/\Bd_f}_{\bfI_1}\Bd_f}{\beta_I(\bdsf)n_I}\cdot \bfU_{\Bd_f/n_I}(\bdsH).\eeq
By the above \lmref{L:bdsH1}, we have $\bdsH^{\rm aux}\in  \bfS(N,\psi_{1,(p)}\ol{\brchf}^{(p)},\bfI_1)\wh\ot_{\bfI_1}\cR$. We can thus apply the ordinary projector $\eord$ to $\bdsH^{\rm aux}$ and obtain $\eord\bdsH^{\rm aux}\in \eord\bfS(N,\psi_{1,(p)}\ol{\brchf}^{(p)},\bfI_1)\wh\ot_{\bfI_1}\cR$ an ordinary $\Lam$-adic modular form with coefficients in $\cR$. With these preparations, we are ready to define the $p$-adic $L$-function following the construction in \cite[(4.6)]{Hida85Inv}. Denote by $\Tr_{N/\condf}:\eord\bfS(N,\psi_{1,(p)}\ol{\brchf}^{(p)},\bfI_1)\to \eord\bfS(\condf,\psi_{1,(p)}\ol{\brchf}^{(p)},\bfI_1)$ the usual trace map (\cf\cite[page 14]{Hida88Fourier}). 
\begin{defn}\label{D:padicL}
The unbalanced \padic triple product $L$-function $\sL^\bdsf_{\bdsF}$ is defined by  \[\sL^\bdsf_{\bdsF}:=\bfa(1,\eta_\bdsf\cdot 1_{\breve{\bdsf}}\Tr_{N/\condf}(\eord\bdsH^{\rm aux}))\in\cR.\]
\end{defn}
\subsection{Global trilinear period integrals}\label{S:gtp}Define the weight space for the triple $(\bdsf,\bdsg,\bdsh)$ in the $\bdsf$-dominated \emph{unbalanced range} by 
%\beq\label{E:wtsp}
\[\frakX_\cR^\bdsf:=\stt{\ulQ=(Q_1,Q_2,Q_3)\in \frakX_{\bfI_1}^\ari\times \frakX_{\bfI_2}^\cls\times\frakX_{\bfI_3}^\cls\mid k_\Qx\geq k_\Qy+k_\Qz,\,\,k_\Qx\con k_\Qy+k_\Qz\pmod{2}}.\]
%\eeq
In this subsection, we relate the value of $\sL^\bdsf_p(\ulQ)$ at a point $\ulQ=(Q_1,Q_2,Q_3)\in\frakX_\cR^\bdsf$ to a global trilinear period integral of a test triple of modular forms.  To this end, it is necessary to work in the framework of automorphic forms. Let $(k_1,k_2,k_3)=(k_\Qx,k_\Qy,k_\Qz)$ and let $r$ be an integer greater than $\max\stt{1,c_p(\ep_\Qx),c_p(\ep_\Qy),c_p(\ep_\Qz)}$. Recall that the specialization 
\[(f,g,h):=\bdsF_\ulQ=(\bdsf_\Qx,\bdsg_\Qy,\bdsh_\Qy)\in \sS_{k_1}(\condf p^r,\chi_{\fQx})\times \sS_{k_2}(\condg p^r,\chi_{\gQy})\times \sS_{k_3}(\condh p^r,\chi_{\hQz})\] are $p$-stabilized cuspidal newforms with characters modulo $Np^r$
\[\chi_{\fQx}=\brchf\ep_\Qx\Om^{2-k_1},\,\chi_{\gQy}=\brchg\ep_\Qy\Om^{2-k_2}\text{ and }\chi_{\hQz}=\brchh\ep_\Qz\Om^{2-k_3}.\] 

 Let $\varphi_{\fQx}=\itPhi(\fQx)$, $\varphi_{\gQy}=\itPhi(\gQy)$ and $\varphi_{\hQz}=\itPhi(\hQz)$ be the associated adelic lifts as in \eqref{E:MA2}. Then 
\[(\varphi_{\fQx},\varphi_{\gQy},\varphi_{\hQz})\in \cA^0_{k_1}(\condf p^r,\om_{\fQx})\times \cA^0_{k_2}(\condg p^r,\om_{\gQy})\times  \cA^0_{k_3}(\condh p^r,\om_{\hQz}),\]
and the central characters $\om_{\fQx},\om_{\gQy},\om_{\hQz}$ are the adelizations
\[\om_{\fQx}=(\chi_{\fQx}^{-1})_\A,\,\om_{\gQy}=(\chi_{\gQy}^{-1})_\A,\,\om_{\hQz}=(\chi_{\hQz}^{-1})_\A.\]
Denote by $(\beta_\ell(\fQx),\beta_\ell(\gQy),\beta_\ell(\hQz))$ the specialization $(\beta_\ell(\bdsf)(\Qx),\beta_\ell(\bdsg)(\Qy),\beta_\ell(\bdsh)(\Qz))$. For each finite prime $\ell$, define the polynomial $\cQ_{f,\ell}\in\cO[X]$ by \[\cQ_{f,\ell}(X)=X^{\val_\ell(\Bd_f)}\begin{cases}1&\text{ if }\ell\not\in\Sigma_{gh}^{\irr},\\
(1-\beta_\ell(\fQx)^{-1}\ell^{\frac{k_1}{2}-1}X^{-1})&\text{ if }\ell\in\Sigma_{gh}^{\irr}.\end{cases}\]
We define $\cQ_{g,\ell}(X)$ and $\cQ_{h,\ell}(X)$ likewise. Set 
\beq\label{E:phistar}\varphi_{\fQx}^\star=\prod_\ell\cQ_{f,\ell}(\LR_\ell)\varphi_{\fQx},\quad
\varphi_{\gQy}^\star=\prod_\ell\cQ_{\gQy,\ell}(\LR_\ell)\varphi_{\gQy}\text{ and }
\varphi_{\hQz}^\star=\prod_\ell\cQ_{\hQz,\ell}(\LR_\ell)\varphi_{\hQz}.
\eeq
By \eqref{E:diffop2}, we see that 
\[\varphi_{\gQy}^\star=\Bd_g^{\frac{k_\Qy}{2}-1}\itPhi(\bdsg^\star_\Qy)\text{ and }\varphi_{\hQz}^\star=\Bd_h^{\frac{k_\Qz}{2}-1}\itPhi(\bdsh^\star_\Qz).\]
Decompose $\om_{\fQx}=\om_{\fQx,(p)}\om_{\fQx}^{(p)}$, where $\om_{\fQx,(p)}$ and $\om_{\fQx}^{(p)}$ are finite order Hecke characters of $p$-power conductor and prime-to-$p$ conductor respectively. By definition, $\om_{\fQx,(p)}$ is the adelization of the $p$-primary component $\chi_{\fQx,(p)}^{-1}$ of $\chi^{-1}_{\fQx}$. 
%Let $\om_{\fQx,(p)}$ be the adelization of the Dirichlet character $\chi_{\fQx,(p)}^{-1}=\chi_{\fQx}^{-1}\brchf^{(p)}$ modulo $p^r$ given by  \[\om_{\fQx,(p)}=\om_{\fQx}(\brchf^{(p)})_\A.\] 
Let $\breve\bdsf$ be the primitive Hida family corresponding to the twist $\bdsf|[\ol{\brchf}^{(p)}]$ and put \[\breve\varphi_{\fQx}=\itPhi(\breve\fQx)\in\cA^0_{k_1}(\condf p^r,\om_{\fQx}^{-1}\om_{\fQx,(p)}^2).\]

We introduce the modified $p$-Euler factor $\cE_p(f,\Ad)$ for the adjoint motive attached to the $p$-stabilized newform $\fQx$. Let $\al_{\fQx,p}:\Qp^\x\to\C^\x$ be the unramified character as in \remref{R:WhittakerPordinary}. Let $\beta_{f,p}:=\al_{f,p}^{-1}\om_{\fQx,p}$. Hence the local component $\pi_{\fQx,p}$ is either the principal series $\Prin{\al_{f,p}}{\beta_{f,p}}$ or the special representation $\al_{f,p}\Abs^{-\onehalf}{\rm St}$. Define the modified $p$-Euler factor $\cE_p(f,\Ad)$ by \beq\label{E:EulerAd}\begin{aligned}\cE_p(f,\Ad)&=\varepsilon(1,\beta_{f,p}\al_{f,p}^{-1})L(0,\beta_{f,p}\al_{f,p}^{-1})^{-1}L(1,\beta_{f,p}\al_{f,p}^{-1})^{-1}\\
&=\bfa(p,\fQx)^{-c_p(\pi_f)}\cdot p^{c_p(\pi_f)(\frac{k_1}{2}-1)}\varepsilon(1/2,\pi_{f,p})\\
&\times\begin{cases} (1-\al_{f,p}^{-2}\om_{f,p}(p))(1-\al_{f,p}^{-2}\om_{f,p}(p)p^{-1})&\text{ if }c(\pi_{f,p})=0,\\
1&\text{ if }c(\pi_{f,p})>0.\end{cases}\end{aligned}
  \eeq
%By definition, we have
%\[\cE_p(f,\Ad)= \bfa(p,\fQx)^{-c_p(\pi_f)}p^{c_p(\pi_f)(\frac{k_1}{2}-1)}\varepsilon(1/2,\pi_{f,p})\begin{cases} (1-\beta_{f,p}\al_{f,p}^{-1}(p))(1-\beta_{f,p}\al_{f,p}^{-1}(p)p^{-1})&\text{ if }c_p(f)=0,\\
%1&\text{ if }c_p(f)>0.\end{cases}\]
Define  $\cJ_\infty$ and $\bftr\in\GL_2(\A)$ for a positive integer $n$ by 
\beq\label{E:spelt}\cJ_\infty=\pDII{-1}{1}\in\GL_2(\R),\quad \bftr=\pMX{0}{p^{-n}}{-p^n}{0}\in\GL_2(\Qp)\hookto\GL_2(\A).\eeq
\begin{lm}\label{L:norm}Let notation be as above.  For $n\geq \max\{c(\pi_{f,p}),1\}$, we have
\begin{align*}&\pair{\rho(\cJ_\infty \bftr)\varphi_{\fQx}}{\breve\varphi_{\fQx}\ot \om_{\fQx,(p)}^{-1}}
=\frac{\zeta_\Q(2)^{-1}}{[\SL_2(\Z):\Gamma_0(\condf)]}\cdot\norm{f^\circ}^2_{\Gamma_0(N_f^\circ)}\cdot\cE_p(f,\Ad)\cdot \frac{\om_{f,p}^{-1}\al_{f,p}^2\Abs_p(p^n)\zeta_p(2)}{\zeta_p(1)}.
\end{align*}
\end{lm}
\begin{proof} Write $\pi$ for $\pi_f$ the irreducible automorphic cuspidal representation on $\GL_2(\A)$ generated by $\varphi_{\fQx}=\itPhi(\fQx)$ and let $\om=\om_{\fQx}$ be the central character of $\pi$. Let $\varphi'=\rho(\cJ_\infty\bftr)\varphi_{\fQx}\in\cA^0_{-k_\Qx}(Np^r,\om)$ and $\varphi''=\breve\varphi_{\fQx}\ot\om_{\fQx,(p)}^{-1}\in \cA^0_{k_\Qx}(Np^r,\om^{-1})$. Then $\varphi'\in\cA(\pi)$ and $\varphi''\in\cA(\Contra{\pi})$. Since $\varphi_{\fQx}$ and $\breve\varphi_{\fQx}$ are automorphic forms attached to $p$-stabilized cuspidal newforms $\fQx$ and $\breve\fQx$, and $\om_{\fQx,(p)}$ is unramified outside $p$, according to \remref{R:WhittakerPordinary}, the Whittaker functions $W_{\varphi'}$ and $W_{\varphi''}$ have the factorizations
\[W_{\varphi'}=\rho(\bftr)W^\ord_{\pi_p}\cdot \rho(\cJ_\infty) W_{\pi_\infty}\prod_{v\not =p,\infty} W_{\pi_v},\quad W_{\varphi''}=W^\ord_{\pi_p}\ot\om_p^{-1}\cdot W_{\pi^\vee_\infty}\prod_{v\not =p,\infty} W_{\pi^\vee_v},\]
where $W^\ord_{\pi_p}\in\cW_{\pi_p}^\ord(\al_p)$ is the ordinary Whittaker functions attached to the character $\al_p$. On the other hand, let $\varphi^\circ=\itPhi(\fQx^\circ)$ be the normalized newform in $\cA(\pi)$ and let $\ol{\varphi^\circ}\in \cA(\Contra{\pi})$ be the complex conjugation of $\varphi^\circ$. Then $\rho(\cJ_\infty)\ol{\varphi^\circ}$ is the normalized newform in $\cA(\Contra{\pi})$. 

Let $\al=\al_{f,p}$, $\beta=\beta_{f,p}$ be the characters defined as above. Combining \propref{P:Petersson}, \lmref{L:ordlocalnorm} and the formula \[\varepsilon(1/2,\pi_p)=\begin{cases}\varepsilon(1/2,\beta)&\text{ if }\pi_p=\Prin{\al}{\beta}\\
-\al\Abs_p^{-\onehalf}(p)&\text{ if }\pi_p=\al\Abs_p^{-\onehalf}{\rm St},\end{cases}\] we find that
\begin{align*}\frac{\pair{\varphi'}{\varphi''}}{\pair{\varphi^\circ}{\ol{\varphi^\circ}}}&=\frac{\pair{\rho(\bftr)W^\ord_{\pi_p}}{W^\ord_{\pi_p}\ot\om_p^{-1}}}{\pair{W_{\pi_p}}{W_{\pi^\vee_p}}}\\
&=\al\beta^{-1}\Abs_p(p^n)\cdot\varepsilon(1/2,\pi_p)\begin{cases} (1-\beta\al^{-1}\Abs_p(p))(1-\beta\al^{-1}(p))(1+p^{-1})^{-1}&\text{ if }c(\pi_p)=0,\\
%\al_p^{-1}\abs{p}^{-1}&\text{ if $\pi_p$ is Steinberg}\\
\al^{-1}\Abs_p^{-\onehalf}(p^{c(\pi_p)})&\text{ if }c(\pi_p)>0.
\end{cases}\end{align*}
From above equation together with the following equation (\cite[page 1403]{II10GAFA}) \begin{align*}\pair{\varphi^\circ}{\ol{\varphi^\circ}}&=\frac{\zeta_\Q(2)^{-1}}{[\SL_2(\Z):\Gamma_0(\condf p^{c_p(\pi)})]}\norm{\fQx^\circ}^2_{\Gamma_0(N_{f^\circ})}\\
&=\norm{\fQx^\circ}^2_{\Gamma_0(N_{f^\circ})}\frac{\zeta_\Q(2)^{-1}}{[\SL_2(\Z):\Gamma_0(\condf)]}\begin{cases}1&\text{ if }c(\pi_p)=0\\
\abs{p}_p^{c_p(\pi)}(1+p^{-1})^{-1}&\text{ if }c(\pi_p)>0,
\end{cases}\end{align*}
we can directly deduce the lemma.
\end{proof}

We may regard $F:=\bdsF_\ulQ=(f,g,h)$ as the modular form on $\frakH^3$ of weight $(k_1,k_2,k_3)$ given by $F(z_1,z_2,z_3)=f(z_1)g(z_2)h(z_3)$. Let $\om_F$ be the central character of $F|_{\frakH}$ given by
\[\om_F=\om_f\om_g\om_h.\] Let $\Bkappa$ be the Dirichlet character modulo $p^r$ defined by 
\beq\label{E:Bkap1}\Bkappa=\psi_{1,(p)}\Om^{-a-1+\frac{k_2+k_3-k_1}{2}}\ep^{1/2}_\Qx \ep^{-1/2}_\Qy\ep^{-1/2}_\Qz.\eeq%=\brchf^{-1}\sqrt{\om_{\fQx}^{-1}\om_{\gQy}\om_{\hQz}}=\brchf^{-1}\om_{\fQx}^{-1}\om_F^{1/2}.\] 
By definition, $\Bkappa^2=\chi_{f,(p)}^2\chi_{\fQx}^{-1}\chi_{\gQy}^{-1}\chi_{\hQz}^{-1}$. Define the character $\om_F^{1/2}$ by \beq\label{E:central}\quad  \om_F^{1/2}=\om_{\fQx,(p)}\Bkappa_\A=\Om^{-a+\frac{k_1+k_2+k_3}{2}-3}\ep_\Qx^{-1/2}\ep_\Qy^{-1/2}\ep_{\Qz}^{-1/2}.\eeq  Then $\om_F^{1/2}$ is a finite order Hecke character unramified outside $p$, and  \[(\om_F^{1/2})^2=\om_{\fQx}\om_{\gQy}\om_{\hQz}=\om_F\]
as the notation suggests. Let $E=\Q\oplus \Q\oplus \Q$ be the split cubic \etale algebra over $\Q$. Let \[m=\frac{k_1-k_2-k_3}{2}.\]
%Let $\wh d_f=\prod_{\ell}\uf_\ell^{\val_\ell(d_f)}\in\wh\Q^\x\hookto\A^\x$. Since the product of central characters of $\pi_f,\pi_g$ and $\pi_h$ is not trivial, it is necessary to introduce certain twist $\varphi^\star_{f^\dagger}$ of $\varphi^\star_f$ in the trilinear period integral. Define $\varphi^\star_{\fQx^\dagger}:\GL_2(\A)\to\C$ by \[\varphi^\star_{\fQx^\dagger}(x)=\om_F^{1/2}(\wh \Bd_f)\cdot \varphi^\star_{\fQx}(x)\om_F^{-1/2}(\det x).\]
Define the automorphic cusp form $\phi_F^\star$ on $\GL_2(\A_E)$ by  
\beq\label{E:Defphi1}
\begin{aligned}
\phi_F^\star:&=(\rho(\cJ_\infty)\varphi^\star_{\fQx}\ot\om_F^{-1/2})\boxtimes \varphi_{\gQy}^\star\boxtimes \LR_+^m\theta_p^{\Bkappa}\varphi_{\hQz}^\star,\\
\phi_F^\star(x_1,x_2,x_3)&=\varphi_{\fQx}^\star(x_1\cJ_\infty)\cdot \varphi_{\gQy}^\star(x_2)\cdot \LR_+^m\theta_p^{\Bkappa}\varphi_{\hQz}^\star(x_3)\cdot\om_F^{-1/2}(\det x_1).\end{aligned}\eeq
Here $\theta_p^{\Bkappa}$ is the twisting operator as in \eqref{E:twisting1}. Put
\[\bft_n=(\bftr,1,1)\in \GL_2(E_p).\]
We shall relate the the valuation of our \padic $L$-function $\sL_p(\ulQ)$ at $\ulQ$ to the global trilinear period $I(\rho(\bft_n)\phi_F^\star)$ defined by
\[I(\rho(\bft_n)\phi_F^\star):=\int_{\A^\x\GL_2(\Q)\bksl \GL_2(\A)}\phi_F^\star(x\bftr,x,x)\rmd^{\tau}x.\]
Put \beq\label{E:BdF}\Bd_F^{\ul{\kappa}/2}:=\Bd_f^{\frac{k_1-2}{2}}\Bd_g^{\frac{k_2-2}{2}}\Bd_h^{\frac{k_3-2}{2}}.\eeq
\begin{prop}\label{P:inter1}For $n\geq r\geq \max\stt{c(\pi_{f,p}),c(\pi_{g,p}),c(\pi_{h,p}),1}$, we have 
\[\sL^\bdsf_{\bdsF}(\ulQ)=\frac{\zeta_\Q(2)[\SL_2(\Z):\Gamma_0(N)]}{\eta_{\fQx}^{-1}\norm{f^\circ}^2_{\Gamma_0(N_{f^\circ})}\cE_p(f,\Ad)}\cdot I(\rho(\bft_n)\phi_F^\star)\cdot\frac{\zeta_p(1)}{\om_{f,p}^{-1}\al_{f,p}^2\Abs_p(p^n)\zeta_p(2)}\cdot \frac{1}{\Bd_F^{\ul{\kappa}/2}}.\]
%$(-1)=(-1)^{1+a+\frac{k_1+k_2+k_3}{2}}$
\end{prop}
\begin{proof} First of all, since $\breve{\bdsf}_\Qx$ is a $p$-stabilized ordinary newform, by the multiplicity one for new and ordinary vectors together with \eqref{E:idempotent}, we have \[\sL_p^\bdsf(\ulQ)\cdot \breve{\bdsf}_\Qx=\eta_f\cdot 1_{\breve{\bdsf}_\Qx}\Tr_{N/\condf}(\eord\bdsH^{\rm aux}_\ulQ).\]
Taking the adelic lifts of both sides, we obtain that \beq\label{E:341.I}
\pair{\rho(\cJ_\infty\bftr)\varphi_{\fQx}\ot\om_{\fQx,(p)}^{-1}}{\breve\varphi_{\fQx}}\cdot \sL_p^\bdsf(\ulQ)=\eta_f\cdot \pair{\rho(\cJ_\infty\bftr)\varphi_{\fQx}\ot\om_{\fQx,(p)}^{-1}}{\Tr_{N/\condf}\itPhi(1^*_{\breve{\bdsf}_\Qx}e\bdsH^{\rm aux}_\ulQ)}.
\eeq
We set \[H=\bdsg^\star_\Qy\cdot\delta^m_{k_\Qz}\bdsh^\star_\Qz|[\Bkappa],\] where $\delta_{k_\Qz}^m$ is the Maass-Shimura differential operator. Since $\Theta(n)(\ulQ)=\Bkappa(n)n^m$ for $n\in\Z_{(p)}^\x$, from \cite[equation (2), page 330]{Hida93Blue}, we deduce that
\beq\label{E:342.I}\eord\bdsH_\ulQ=\eord(\bdsg^\star_\Qy{\rm d}^m(\bdsh^\star_\Qz|[\Bkappa]))=\eord{\rm Hol}(\bdsg^\star_\Qy\delta^m_{k_\Qz}(\bdsh^\star_\Qz|[ \Bkappa]))=\eord{\rm Hol}(H),\eeq
where ${\rm d}=q\frac{d}{dq}$ is Serre's \padic differential operator and ${\rm Hol}$ is the holomorphic projection as in \cite[(8a), page 314]{Hida93Blue}. Using \eqref{E:diffop1}, \eqref{E:diffop2} and \eqref{E:diffop3}, we see that
\[\varphi_H:=\itPhi(H)=\Bd_g^{1-\frac{k_2}{2}}\Bd_h^{1-\frac{k_3}{2}}\cdot \varphi_{\gQy}^\star\cdot\LR_+^m\theta_p^{\Bkappa}\varphi_{\hQz}^\star\ot\Bkappa_\A^{-1}.\]
Then $H$ is a nearly holomorphic cusp form of weight $k_\Qx$ and $\varphi_H\in\cA^0_{k_1}(Np^r,\om_{\fQx}^{-1}\om_{\fQx,(p)}^2)$ has a decomposition
\[\varphi_H={\rm Hol}(\varphi_H)+\LR_+\varphi'_1+\LR_+^2\varphi'_2+\cdots+\LR_+^n\varphi_n',\]
where ${\rm Hol}(\varphi_H)$ and $\stt{\varphi'_j}_{j=1,\ldots,n}$ are holomorphic automorphic forms. It follows that ${\rm Hol}(\varphi_H)=\itPhi({\rm Hol}(H))$.

Let $1^*_{\fQx}\in \rmT^\ord(\condf p^r,\chi_{\fQx})$ be the specialization of $1^*_\bdsf$. As a consequence of strong multiplicity one theorem for modular forms, the idempotent $1_{\fQx}=\eta_{\fQx}^{-1}1^*_{\fQx}\in \rmT^\ord(\condf p^r,\chi_{\fQx})\ot_\cO\Frac\cO(\Qx)$ is generated by the Hecke operators $T_\ell$ for $\ell\ndivides Np$, so we see that $1_{\fQx}$ is the left adjoint of $1_{\breve\bdsf_\Qx}$ for the pairing $\pair{-\ot\om_{\fQx,(p)}^{-1}}{-}$ by \lmref{L:bilinear}, and hence the right hand side of \eqref{E:341.I} equals
\beq\label{E:343.I}\begin{aligned}
&\eta_f \cdot \pair{\Tr_{N/\condf}\left(1_{\fQx}\cdot \rho(\cJ_\infty\bftr)\varphi_{\fQx}\ot\om_{\fQx,(p)}^{-1}\right)}{\itPhi(\eord\bdsH^{\rm aux}_\ulQ)}\\&=\eta_\fQx[K_0(\condf):K_0(N)]\cdot \pair{\rho(\cJ_\infty\bftr)\varphi_{\fQx}\ot\om_{\fQx,(p)}^{-1}}{\itPhi(\eord\bdsH^{\rm aux}_\ulQ)}.\end{aligned}
\eeq
Note that for any prime $\ell\not =p$, $\om_{\fQx,(p)}(\uf_\ell)=\chi_{f,(p)}(\ell)$ is the specialization of $\psi_{1,(p)}(\ell)\Dmd{\ell}_{\bfI_1}$ at $\Qx$. 
From the definition \eqref{E:bdsH}, \eqref{E:342.I} and \lmref{L:bilinear}, we find that the pairing in the right hand side of \eqref{E:343.I} equals
\begin{align*}
&\Bd_f^{-\frac{k_1}{2}}\sum_{I\subset\Sigma_{gh}^{\irr}}(-1)^I \frac{n_I^\frac{k_1}{2}}{\beta_I(\fQx)}\chi_{f,(p)}(n_I/\Bd_f)\cdot \pair{\rho(\cJ_\infty\bftr)\varphi_{\fQx}\ot\om_{\fQx,(p)}^{-1}}{\bfU_{\Bd_f/n_I} \varPhi(\eord{\rm Hol}(H))}\\
&=\Bd_f^{1-\frac{k_1}{2}}\pair{\rho(\cJ_\infty\bftr)\varphi_{\fQx}^\star\ot\om_{\fQx,(p)}^{-1}}{\eord{\rm Hol}(\varphi_H)}.
\end{align*}
On the other hand, it is straightforward to verify by \lmref{L:bilinear} that \begin{align*}\pair{\rho(t_n)\bfU_p\varphi}{\varphi'}&=\pair{\varphi}{\bfU_p\varphi'},\\
\pair{\rho(\cJ_\infty)\varphi}{\LR_+\varphi'}&=-\pair{\rho(\cJ_\infty)\LR_-\varphi}{\varphi'}\end{align*}
(\cf\cite[(5.4)]{Hida85Inv}), and together with \eqref{E:central}, it follows that
\begin{align*}\Bd_f^{1-\frac{k_1}{2}}\pair{\rho(\cJ_\infty\bftr)\varphi_{\fQx}^\star\ot\om_{\fQx,(p)}^{-1}}{\eord{\rm Hol}(\varphi_H)}
&=\Bd_f^{1-\frac{k_1}{2}}\pair{\rho(\cJ_\infty\bftr)\varphi_{\fQx}^\star\ot\om_{\fQx,(p)}^{-1}}{\varphi_H}\\
=\Bd_F^{-\ul{\kappa}/2}\pair{\rho(\cJ_\infty\bftr)\varphi^\star_{\fQx}\ot\om_F^{-1/2}}{\varphi^\star_{\gQy}\cdot \LR_+^m\theta_p^{\Bkappa}\varphi^\star_{\hQz}}
&=\Bd_F^{-\ul{\kappa}/2} I(\rho(\bft_n)\phi_F^\star).
\end{align*}
Combining the above equation with \eqref{E:341.I} and \eqref{E:343.I}, we find that
\begin{align*}
\pair{\rho(\cJ_\infty\bftr)\varphi_{\fQx}}{\breve\varphi_{\fQx}\ot\om_{\fQx,(p)}^{-1}}\cdot \sL_p^\bdsf(\ulQ)=\eta_{\fQx}[\Gamma_0(\condf):\Gamma_0(N)]\Bd_F^{-\ul{\kappa}/2}\cdot I(\rho(\bft_n)\phi_F^\star).
\end{align*}
Now the lemma follows from the formula of the pairing in the left hand side given in \lmref{L:norm}.
\end{proof}

%!TEX root = TRIPLE.tex
\def\pmq{q}
%\section{Ichino's formula}

\subsection{Ichino's period integral formula for triple products}\label{SS:Ichino}
\subsubsection{The setting}\label{SSS:localfac}
In this subsection, we apply Ichino's formula to express $I(\rho(\bft_n)\phi_F^\star)$ as a product of the central value of the triple product $L$-function attached to $F$ and normalized local trilinear integrals. We retain the notation in the previous subsection. Let
\[\pi_1=\pi_{\fQx}\ot\om_F^{-1/2},\quad\pi_2=\pi_{\gQy}\text{ and }\pi_3=\pi_{\hQz}\]
with central characters $\om_1=\om_{\gQy}^{-1}\om_{\hQz}^{-1}$, $\om_2=\om_{\gQy}$ and $\om_3=\om_{\hQz}$ respectively. Let \[\itPi_\ulQ=\pi_1\times\pi_2\times\pi_3\] be an irreducible unitary cuspidal automorphic representation of $\GL_2(\A_E)$ and let $\cA(\itPi_\ulQ)=\cA(\pi_1)\ot\cA(\pi_2)\ot\cA(\pi_3)$ be the unique automorphic realization of $\itPi_\ulQ$. For brevity of notation, we simply write $\itPi$ for $\itPi_\ulQ$. For each place $v$, let $\cV_{\itPi_v}=\cV_{\pi_{1,v}}\ot\cV_{\pi_{2,v}}\ot\cV_{\pi_{3,v}}$ denote a realization of $\itPi_v$, where $\cV_{\pi_{i,v}}$ is a realization of $\pi_{i,v}$ for $i=1,2,3$. Then we have the factorizations  
\[\itPi\iso\bigot_v\itPi_v,\quad \cA(\itPi)\iso\bigot_v\cV_{\itPi_v}.\]
We let $\phi_F=\varphi_1\boxtimes\varphi_2\boxtimes\varphi_3\in\cA(\itPi)$, where \[\varphi_1=\varphi_{\fQx}\ot\om_F^{-1/2}, \quad\varphi_2=\varphi_{\gQy}\text{ and }\varphi_3=\varphi_{\hQz}.\] 
Then we have a factorization $\phi_F=\bigot_v\phi_v$ via the above isomorphism. %Define the new line in $\cV_{\itPi_v}$ by \[\cV^{\rm new}_{\itPi_v}:=\cV^{\rm new}_{\pi_{1,v}}\ot\cV^{\rm new}_{\pi_{2,v}}\ot\cV^{\rm new}_{\pi_{3,v}},\]  and define the ordinary line in $\cV_{\itPi_p}$ by 
%\beq\label{E:ordline}\cV^\ord_{\itPi_p}:=\cV^\ord_{\pi_{1,p}}(\chi_{1,p})\ot\cV^\ord_{\pi_{2,p}}(\chi_{2,p})\ot \cV^\ord_{\pi_{3,p}}(\chi_{3,p}),\eeq where $\chi_{i,p}:\Qp^\x\to\C^\x$ are the characters 
%\[ \chi_{1,p}=\al_{\fQx,p}\om_{F,p}^{-1/2},\,\,\chi_{2,p}=\al_{\gQy,p}\text{ and }\chi_{3,p}=\al_{\hQz,p}.\]
%($\al_{?,p}$ is the character attached to a $p$-stabilized newform $?$ defined in \remref{R:WhittakerPordinary}).
Since $\varphi_f,\varphi_g$ and $\varphi_h$ are $p$-stabilized newforms and $\om_F^{1/2}$ is unramified outside $p$, we find that $\phi_v=\varphi_{1,v}\ot\varphi_{2,v}\ot\varphi_{3,v}\in\cV^{\rm new}_{\itPi_v}$ if $v\not =p$ and $\phi_p=\varphi_{1,p}\ot\varphi_{2,p}\ot\varphi_{3,p}\in \cV^\ord_{\itPi_p}$.  \begin{itemize}\item $\varphi_{i,v}\in \cV_{\pi_{i,v}}^{\rm new}$ is a new vector if $v\not =p$,
\item $\varphi_{i,p}\in \cV^\ord_{\pi_{i,p}}(\chi_{i,p})$ is an ordinary vector attached to the character $\chi_{i,p}:\Qp^\x\to\C^\x$, where
\beq\label{E:ordchar}\chi_{1,p}=\al_{\fQx,p}\om_{F,p}^{-1/2},\,\,\chi_{2,p}=\al_{\gQy,p}\text{ and }\chi_{3,p}=\al_{\hQz,p}\eeq
($\al_{?,p}$ is the character attached to a $p$-stabilized newform $?$ defined in \remref{R:WhittakerPordinary}). \end{itemize}
For each finite prime $\ell$, define the polynomial $\cQ_{1,\ell}(X)\in\cO[X]$ by \beq\label{E:defQ}\cQ_{1,\ell}(X)=X^{\val_\ell(\Bd_f)}\begin{cases}1&\text{ if }\ell\not\in\Sigma_{f,0}^{\rm (IIb)},\\
(1-\om_F^{1/2}(\uf_\ell)\beta_\ell(\fQx)^{-1}\ell^{\frac{k_1}{2}-1}X^{-1})&\text{ if }\ell\in\Sigma_{f,0}^{\rm (IIb)}.\end{cases}\eeq
Set $\cQ_{2,\ell}(X)=\cQ_{g,\ell}(X)$ and $\cQ_{3,\ell}(X)=\cQ_{h,\ell}(X)$. Let $\wh \Bd_f=\prod_{\ell}\uf_\ell^{\val_\ell(\Bd_f)}\in \wh\Q^\x$. We put
\begin{align*}\varphi_{1}^\star:&=\prod_\ell\cQ_{1,\ell}(\LR_\ell)\varphi_1=\om_F^{1/2}(\wh \Bd_f)\cdot \varphi_{\fQx}^\star\ot\om_F^{-1/2},\\
&\varphi^\star_2=\varphi^\star_g;\quad\varphi^\star_3=\varphi^\star_h.\end{align*}
We give the factorization of the automorphic form $\phi_F^\star$ defined in \eqref{E:Defphi1}. By definition, 
\[\phi_F^\star=C_1\cdot \rho(\cJ_\infty)\varphi_1^\star\boxtimes\varphi_2^\star\boxtimes\LR_+^m\theta_p^{\Bkappa}\varphi_3^\star\quad (C_1:=\om_{F,\infty}^{-1/2}(-1)\om_F^{-1/2}(\wh\Bd_f)).\] In view of \eqref{E:phistar}, we find that  that $\phi_F^\star=C_1\cdot \bigot_v\phi^\star_v$, where 
\beq\label{E:factorization1}\begin{aligned}
\phi^\star_v=\begin{cases}\pi_{1,\infty}(\cJ_\infty)\varphi_{1,\infty}\ot \varphi_{2,\infty}\ot \LR_+^m\varphi_{3,\infty}&\text{ if }v=\infty,\\[1em]
\varphi_{1,p}\ot \varphi_{2,p}\ot \theta_p^{\Bkappa}\varphi_{3,p}&\text{ if }v=p,\\[1em]
\cQ_{1,\ell}(V_\ell)\varphi_{1,\ell}\ot \cQ_{2,\ell}(V_\ell)\varphi_{2,\ell}\ot \cQ_{3,\ell}(V_\ell)\varphi_{3,\ell}&\text{ if }v=\ell\ndivides p.
\end{cases}
\end{aligned}\eeq
Here $\theta_p^\Bkappa$ is the local twisting operator attached to $\Bkappa$ as in \eqref{E:deftheta2} and $V_\ell$ is the level-raising operator as in \eqref{E:localUV}. Note that $\phi^\star_\ell=\phi_\ell$ is a new vector in $\cV_{\Pi_\ell}$ for $\ell\ndivides pN$.

Next we consider the contragredient representation $\Contra{\itPi}=\Contra{\pi}_1\ot\Contra{\pi}_2\ot\Contra{\pi}_3$.  We put
 \[\wtd\varphi_i=\varphi_i\ot\om_i^{-1}\text{ and }\wtd\varphi^\star_i=\varphi^\star_i\ot\om_i^{-1},\,i=1,2,3.\]
Define $\wtd\phi_F$ and $\wtd\phi_F^\star\in\cA(\Contra{\itPi})$ by 
\begin{align*}\wtd\phi_F&=\wtd\varphi_1\boxtimes\wtd\varphi_2\boxtimes\wtd\varphi_3,\\
\wtd\phi_F^\star&=\rho(\cJ_\infty)\wtd\varphi^\star_1\boxtimes \wtd\varphi^\star_2\boxtimes \LR_+^m\theta_p^{\Bkappa}\wtd\varphi^\star_3.\end{align*}
Recall that $N_i$ is the tame conductor of $\pi_i$. Take an isomorphism  $\cA(\Contra{\itPi})\iso\bigot_v\cV_{\Contra{\itPi}_v}$ with $\cV_{\Contra{\itPi}_v}=\cV_{\Contra{\pi}_{1,v}}\ot \cV_{\Contra{\pi}_{2,v}}\ot\cV_{\Contra{\pi}_{3,v}}$. We have a factorization $\wtd\phi_F=\bigot_v\wtd\phi_v$, where $\wtd\phi_v=\wtd\varphi_{1,v}\ot\wtd\varphi_{2,v}\ot\wtd\varphi_{3,v}$, 
\begin{align*}\wtd\phi_{i,\infty}\in \cV^{\rm new}_{\Contra{\pi}_{i,\infty}},\quad\wtd\phi_{i,p}\in \cV^\ord_{\Contra{\pi}_{i,p}}(\chi_{i,p}\om_{Q_i}^{-1});\\
 \wtd\phi_{i,v}\in \Contra{\pi}_{i,v}(\pMX{0}{1}{-N_i}{0})\cV^{\rm new}_{\Contra{\pi}_{i,v}}\text{ if }v\not =p\,\,\infty.\end{align*}Moreover, $\wtd\phi_F^\star=\bigot_v\wtd\phi^\star_v$, where
\beq\label{E:factorization2}\begin{aligned}
\wtd\phi^\star_v=\begin{cases}\pi_{1,\infty}(\cJ_\infty)\wtd\varphi_{1,\infty}\ot \wtd\varphi_{2,\infty}\ot \LR_+^m\wtd\varphi_{3,\infty}&\text{ if }v=\infty,\\[1em]
\wtd\varphi_{1,p}\ot \wtd\varphi_{2,p}\ot \theta_p^{\Bkappa}\wtd\varphi_{3,p}&\text{ if }v=p,\\[1em]
\Contra{\cQ}_{1,\ell}(V_\ell)\wtd\varphi_{1,\ell}\ot \Contra{\cQ}_{2,\ell}(V_\ell)\wtd\varphi_{2,\ell}\ot \Contra{\cQ}_{3,\ell}(V_\ell)\wtd\varphi_{3,\ell}&\text{ if }v=\ell\ndivides p.
\end{cases}
\end{aligned}\eeq
Here $\Contra{\cQ}_{i,\ell}(X)=\cQ_{i,\ell}(\om_i^{-1}(\uf_\ell)X)$ for $i=1,2,3$.

\subsubsection{Ichino's formula}\label{SS:Ichino.unb}
For $\ulN=(N_1,N_2,N_3)$, we put  
\[\Tau_{\ulN}=(\Tau_{\condf},\Tau_{\condg},\Tau_{\condh})\in\GL_2(\A_E).\]
Here $\Tau_{N_i}$ is the matrix defined as in \eqref{E:atkin}.
For each place $v$ of $\Q$, we choose a $\GL_2(E\ot\Q_v)$-equivariant map $\bfb_v :\cV_{\itPi_v}\ot \cV_{\Contra{\itPi}_v}\to\C$ such that $\bfb_v(\phi_v,\phi_v)=1$ for all but finitely many $v$. We introduce certain local zeta integrals that appear in our application of Ichino's formula. For each place $v$, we define the local zeta integral 
\beq\label{E:localzeta}
I_v(\phi^\star_v\ot \wtd\phi^\star_v):=\frac{L(1,\itPi_v,\Ad)}{\zeta_v(2)^2L(1/2,\itPi_v)}\int_{\PGL_2(\Q_v)}\frac{\bfb_v(\itPi_v(g_v)\phi^\star_v,\wtd\phi^\star_v)}{\bfb_v(\itPi_v(\Tau_{\ulN,v})\phi_v,\wtd\phi_v)}\rmd g_v.\eeq
Here $\rmd g_v$ is the Haar measure as in \subsecref{SS:measure}. At the place $p$, we will consider the local integral 
\beq\label{E:unbpzeta}
I_p^\ord(\phi^\star_p\ot\wtd\phi^\star_p,\bft_n):=\frac{L(1,\itPi_p,\Ad)}{\zeta_p(2)^2L(1/2,\itPi_p)}\int_{\PGL_2(\Q_p)}\frac{\bfb_p(\itPi_p(g_p\bft_n)\phi^\star_p,\Contra{\itPi}_p(\bft_n)\wtd\phi^\star_p)}{\bfb_v(\itPi_p(\bftr)\phi_p,\wtd\phi_p)}\rmd g_p.
\eeq
 \begin{Remark}The integrals $I_v(\phi^\star_v\ot \wtd\phi^\star_v)$ and $I_p^\ord(\phi^\star_p\ot\wtd\phi^\star_p,\bft_n)$ do not depend on any choice of the realizations $\cV_{\itPi_v},\cV_{\Contra{\itPi}_v}$, the pairing $\bfb_v$ and the new or ordinary vector $\phi_v$ in virtue of the irreducibility of $\itPi_v$ and the multiplicity one for new vectors and ordinary vectors \propref{P:ordline}. This allows us to evaluate these local integrals by choosing favourable realizations of $\cV_{\Pi_v}$.
\end{Remark}

\begin{defn}\label{D:root}Define the set \[\Sigma^-_{fgh}=\stt{\ell\in\Sigma_f^0\cap\Sigma_g^0\cap\Sigma_h^0\mid \varepsilon(1/2,\itPi_\ell)=-1}.\]
From the rigidity of automorphic types in \remref{R:rigidity}, we can deduce that there is a subset $\Sigma^-$ of primes dividing $N$ such that 
\begin{align*}\Sigma^-&=\Sigma_{\bdsf_\Qx\bdsg_\Qy\bdsh_\Qz}^-
=\stt{\ell:\text{ prime factos of }N\mid \varepsilon(\WD_\ell(\bfV^\dagger_\ulQ))=-1}
\end{align*}
for any arithmetic point $\ulQ\in \frakX_\cR^\bdsf$. \end{defn}

\begin{prop}\label{P:Ichino1}Suppose that $\Sigma^-=\emptyset$. Then \[\frac{I(\rho(\bft_n)\phi_F^\star)^2}{\prod\limits_{i=1}^3\pair{\rho(\Tau_{N_i}\bftr)\varphi_i}{\wtd\varphi_i}}=\frac{(-1)^{k_1}\zeta_\Q(2)}{8L(1,\itPi,\Ad)}\cdot L(\onehalf,\itPi)\cdot I_p^\ord(\phi^\star_p\ot\wtd\phi^\star_p,\bft_n)\cdot \prod_{v\not =p}I_v(\phi^\star_v\ot \wtd\phi^\star_v)\om_{F,\pmq}^{-1}(\Bd_f).\] 
\end{prop}
\begin{proof} Note that \[I(\rho(\bft_n)\phi_F^\star)^2=\om_{1,\infty}(-1)I(\rho(\bft_n)\phi_F^\star)\cdot I(\rho(\bft_n)\wtd\phi_F^\star).\]
Applying \cite[Theorem 1.1, Remark 1.3]{Ichino08Duke},
we obtain the proposition immediately in view of the decomposition of $\phi_F^\star$ and $\wtd\phi_F^\star$ into pure tensors. We remark that $\om_{1,\infty}(-1)=(-1)^{k_1}$ and the constant $C$ in Remark 1.3 \loccit equals $\zeta_\Q(2)^{-1}$ since the product measure $\prod_v\rmd g_v=\zeta_\Q(2)\cdot \rmd^{\tau}g$ (\cf\cite[page 1403]{II10GAFA}).\end{proof}

%Recall that $\pi_{i,\infty}=\Contra{\pi}_{i,\infty}=\cD_0(k_i)$ is a discrete series or limit of discrete series of $\GL_2(\R)$ of the lowest weight $k_i$  Let $\cV_{\itPi_\infty}=\sW(\pi_{1,\infty})\ot\sW(\pi_{2,\infty})\ot\sW(\pi_{3,\infty})$ be the Whittaker realization of $\itPi_\infty$. Recall that $V_+\in\text{Lie}(\GL_2(\R))\otimes_{\R}\C$ is the weight raising operator in \eqref{E:diffop}. 
%Then we choose \begin{align*}\phi_\infty&=W_{k_1}\ot W_{k_2}\ot W_{k_3}\in \cV_{\itPi_\infty}\intertext{ to be the Whittaker function of the minimal $\SO(2)$-type with $W_{k_i}(1)=1$ and} 
%\phi^\star_\infty&=\rho(\pDII{-1}{1})W_{k_1}\ot W_{k_2}\ot \LR_+^\frac{k_1-k_2-k_3}{2} W_{k_3}.\end{align*}
%Let $\eta=\pDII{-1}{1}$. Note that \[\bfL^{[n]}_\infty=(\eta,1,1),\quad \tau^{[n]}_\infty=(\eta,\eta,\eta)\in\GL_2(E_\infty).\]
%Then we have
%\begin{align*}
%I_\infty
%:&=(8\pi)^{-2m}\int_{\PGL_2(\R)}\pair{\rho(g\cJ_\infty)\phi_\infty}{\rho(\cJ_\infty)\phi_\infty}\rmd g\\
%&I^*(\pi_1(\begin{pmatrix}1&0\\0&-1\end{pmatrix})v_1,\pi_1(\begin{pmatrix}1&0\\0&-1\end{pmatrix})v_1;
%v_2,v_2;(-\frac{1}{8\pi}V_+)^m v_3,(-\frac{1}{8\pi}V_+)^m v_3)\\
%&=(8\pi)^{-2m}4^{-1-k_1+3m}\pi^{1+2m}\Gamma_\C(\frac{k_1+k_2+k_3-2}{2})\Gamma_\C(\frac{k_1-k_2-k_3+2}{2})\Gamma_\C(\frac{k_1-k_2+k_3}{2})\Gamma_\C(\frac{k_1+k_2-k_3}{2})\\
%&=\frac{\pi}{4^{k_1+1}}L(1/2,\itPi_\infty).
%\end{align*} 

\begin{lm}\label{L:local.I}We have the following equalities:
\begin{mylist}
\item 
If $\pmq\ndivides N$ is a finite prime, then $I_\pmq(\phi^\star_\pmq\ot\wtd\phi^\star_\pmq)=1$;
\item $I_\infty(\phi^\star_\infty\ot \wtd\phi^\star_\infty)=2^{k_2+k_3-k_1+1}.$
\end{mylist}
\end{lm}
\begin{proof}
Part (1) is \cite[Lemma 2.2]{Ichino08Duke}. Note that $\phi^\star_\pmq=\phi_\pmq$ is a new vector in $\cV_{\Pi_\pmq}$ for a finite prime $\pmq\ndivides N$. The formula of the archimedean zeta integral in part (2) is proved in \cite{ChenYao16}. For the reader's convenience, we sketch the proof. For $i=1,2,3$, let $W_{k_i}=W_{\pi_i,\infty}$ be the Whittaker newform of the discrete series $\pi_{i,\infty}=\cD_0(k_i)$ in \eqref{E:Winfty.1}. Define the matrix coefficient $\Phi_\infty:\GL_2(\R)\to\C$ by
\[\Phi_\infty(g):=\frac{\pair{\rho(g\cJ_\infty)W_{k_1}}{\rho(\cJ_\infty)W_{k_1}}}{\pair{\rho(\cJ_\infty)W_{k_1}}{W_{k_1}}}\cdot \frac{\pair{\rho(g)W_{k_2}}{W_{k_2}}}{\pair{\rho(\cJ_\infty)W_{k_2}}{W_{k_2}}}\cdot \frac{(8\pi)^{2m}\pair{\rho(g)V_+^mW_{k_3}}{V_+^mW_{k_3}}}{\pair{\rho(\cJ_\infty)W_{k_3}}{W_{k_3}}}\]
(recall that $m=\frac{k_1-k_2-k_3}{2}$). Note that $\Phi$ is right $\SO(2)(\R)$-invariant, and a lengthy computation shows that
\begin{align*}
\Phi_\infty(\pMX{y}{x}{0}{1})
&=\bbI_{\R_+}(y)\cdot \frac{4^{k_1}\Gamma(k_3+m)^2}{\Gamma(k_3)}\sum_{i,j=0}^m(-2)^{i+j}{m\choose i}{m\choose j}\frac{\Gamma(k_3+i+j)}{\Gamma(k_3+i)\Gamma(k_3+j)}\\
&\times \frac{(-y)^{k_1-m+i}}{((1-y)-\sqrt{-1}x)^{k_1}((1-y)+\sqrt{-1}x)^{k_1-2m+i+j}}.
\end{align*}
By definition, 
\[I_\infty(\phi^\star_\infty\ot \wtd\phi^\star_\infty)=\frac{L(1,\itPi_\infty,\Ad)}{\zeta_\infty(2)^2L(1/2,\itPi_\infty)}\cdot (8\pi)^{-2m}I(\Phi_\infty),\]
where
\[I(\Phi_\infty):=\int_{\PGL_2(\R)}\Phi_\infty(g)\rmd g=\int_{\R}\int_{\R^\x}\Phi_\infty(\pMX{y}{x}{0}{1})\frac{\rmd y}{\abs{y}}\rmd x.\]
By a direct computation, we obtain \beq\label{E:3.i}\begin{aligned}
I(\Phi_\infty)
&=\frac{4^{k_1}\Gamma(k_3+m)^2}{\Gamma(k_3)}\sum_{i,j=0}^m(-2)^{i+j}{m\choose i}{m\choose j}\frac{\Gamma(k_3+i+j)}{\Gamma(k_3+i)\Gamma(k_3+j)}\\
&\times 2^{2-2k_1+2m-i-j}\cdot\pi\cdot\frac{\Gamma(k_1-m+i-1)\Gamma(k_3-m+j)}{\Gamma(k_1-2m+i+j)\Gamma(k_1)}\\
&=\frac{4^{m+1}\pi\cdot\Gamma(k_3+m)}{\Gamma(k_2)\Gamma(k_3)}\sum_{j=0}^m(-1)^j{m\choose j}\frac{\Gamma(k_1-m+j)}{\Gamma(k_3+j)}\cdot S_j,
\end{aligned}\eeq
where
\[S_j:=\Gamma(k_2+m)\sum_{i=0}^m(-1)^i{m\choose i}\frac{\Gamma(k_2+j+i)}{\Gamma(k_3+i)}\cdot\frac{\Gamma(k_1-m-1+i)}{\Gamma(k_1-2m_j+i)}.\]
Applying the combinatorial identity \cite[Lemma 3]{Orl87} to $S_j$, we find that 
\[S_j=(-1)^m\cdot\frac{\Gamma(k_3+j)\Gamma(k_1-m-1)}{\Gamma(k_1-m+j)}\cdot\frac{\Gamma(k_1-k_3-m)}{\Gamma(k_1-k_3-2m)}\cdot\frac{\Gamma(j+1)}{\Gamma(j-m+1)}.\]
Substituting the above expression to the last line of \eqref{E:3.i}, we find that
\[I(\Phi_\infty)=4^{m+1}\cdot\pi\cdot\frac{\Gamma(k_1-m-1)\Gamma(k_3+m)\Gamma(k_2+m)\Gamma(m+1)}{\Gamma(k_1)\Gamma(k_2)\Gamma(k_3)}.\]
Hence, part (2) follows from the above expression of $I(\Phi_\infty)$ and \[\frac{L(1,\itPi_\infty,\Ad)}{\zeta_\infty(2)^2L(1/2,\itPi_\infty)}=\frac{\pi^{-3}\Gamma_\C(k_1)\Gamma_\C(k_2)\Gamma_\C(k_3)}{\pi^{-2}\cdot \Gamma_\C(k_1-m-1)\Gamma_\C(k_3+m)\Gamma_\C(k_2+m)\Gamma_\C(m+1)}.\qedhere\]
\end{proof}
To distinguish the contributions from each term in the formula of $\sL_{\bdsF}^\bdsf(\ulQ)$, we introduce the normalized local zeta integrals. For each place $v$, define the local norm of Whittaker newforms for $\itPi_v$ by\beq\label{E:lNorm1}
B_{\itPi_v}:=B_{\pi_{1,v}}B_{\pi_{2,v}}B_{\pi_{3,v}}\eeq
with $B_{\pi_{i,v}}$ the local norm of $\pi_{i,v}$ defined as in \eqref{E:localnormnew}. To each positive integer $n$, we associate the local norm $B_{\itPi_p^\ord}^{[n]}$ of ordinary Whittaker functions for $\itPi_p$ given by \beq\label{E:lNorm2}  
B_{\itPi_p^\ord}^{[n]}:=\frac{\zeta_p(2)^3}{\zeta_p(1)^3L(1,\itPi_p,\Ad)}\prod_{i=1}^{3}\pair{\rho(\bftr)W^\ord_{\pi_{i,p}}}{W^\ord_{\pi_{i,p}}\ot\om_{i,p}^{-1}}.
\eeq
We define the normalized local zeta integrals\begin{align}
\label{E:Nunb}\sI^\unb_{\itPi_{\ulQ,p}}&=I_p^\ord(\phi^\star_p\ot\wtd\phi^\star_p,\bft_n)\cdot \frac{B_{\itPi^\ord_p}^{[n]}}{\om_{f,p}^{-1}\al_{f,p}^2\Abs_p(-p^{2n})}\cdot\frac{\zeta_p(1)^2}{\zeta_p(2)^2};\\
\label{E:Nq}\sI^\star_{\itPi_{\ulQ,\pmq}}&=I_\pmq(\phi^\star_\pmq\ot\wtd\phi^\star_\pmq)\cdot B_{\itPi_\pmq}\cdot  \frac{\zeta_\pmq(1)^2}{\abs{N}_\pmq^2\zeta_\pmq(2)^2}\cdot  \om^{-1}_{F,\pmq}(\Bd_f)|\Bd_F^{\ulk}|_\pmq\text{ for }\pmq\divides N.
\end{align}
\begin{defn}[The canonical periods of Hida families]\label{D:period1}Define the canonical period $\Omega_{\bdsf_Q}$ of the specialization $\bdsf_Q$ at an arithmetic point $Q$ by \[\Omega_{\bdsf_Q}:=(-2\sqrt{-1})^{k_Q+1}\cdot\norm{\bdsf_Q^\circ}^2_{\Gamma_0(N_Q)}\cdot\frac{\cE_p(\bdsf_Q,\Ad)}{\eta_{\bdsf_Q}},\]
where $\bdsf_Q^\circ$ is the normalized newform associated with $\bdsf_Q$ of conductor $N_Q$ and $\eta_{\bdsf_Q}$ is the specialization of $\eta_\bdsf$ at $Q$ and $\cE_p(\bdsf_Q,\Ad)$ is the modified Euler factor in \eqref{E:EulerAd}.
\end{defn}
We summarize our computation in the following
\begin{cor}\label{C:Ichino.imb}Assume that $\Sigma^-=\emptyset$. For every $\ulQ=(\Qx,\Qy,\Qz)\in\frakX^\bdsf_\cR$, we have the interpolation formula
\begin{align*}\left(\sL_{\bdsF}^\bdsf(\ulQ)\right)^2&=\psi_{1,(p)}(-1)(-1)^{k_\Qx+1}\cdot\frac{L(1/2,\itPi_\ulQ)}{\Omega_{\bdsf_\Qx}^2 }\cdot\sI_{\itPi_{\ulQ,p}}^\unb\cdot \prod_{\pmq\divides N}\sI^\star_{\itPi_{\ulQ,\pmq}}.
\end{align*}
\end{cor}
\begin{proof}
By Waldspurger's Petersson inner product formula (\propref{P:Petersson}) and the identities
\[B_{\itPi_\infty}=
2^{-(k_1+k_2+k_3)-3};\quad
B_{\itPi_\pmq}=1\text{ if }\pmq\ndivides N\]
with $k_i=k_{Q_i}$, we find that
\[\prod_{i=1}^3\pair{\rho(\Tau_{N_i}\bftr)\varphi_i}{\wtd\varphi_i}=\frac{8L(1,\itPi,\Ad)}{\zeta_\Q(2)^3} \cdot  2^{-(k_1+k_2+k_3)-3}B_{\itPi^\ord_p}^{[n]}\prod_{\pmq\divides N}B_{\itPi_\pmq} .\] 
Note that $\om_{f,p}(-1)=(-1)^{k_1}\psi_{1,(p)}(-1)$. Combining \propref{P:inter1}, \propref{P:Ichino1}, \lmref{L:local.I} and the equality
 \[[\SL_2(\Z):\Gamma_0(N)]=\prod_{\pmq\divides N}\frac{\zeta_\pmq(1)}{\abs{N}_\pmq\zeta_\pmq(2)},\]
 we get the corollary.
\end{proof}

%!TEX root = TRIPLE3.tex
\def\wt{\kappa}
\def\bfchi{\boldsymbol\chi}
\def\Dstar{{D\star}}
\def\ulk{\ul{\wt}}
\def\qtnf{f}
\def\bdsF{{\boldsymbol F}}
\section{The balanced $p$-adic triple product $L$-functions}\label{S:bal}
\subsection{Notation and conventions}\label{S:nc.bal}
Let $D$ be the definite quaternion algebra over $\Q$ with discriminant $N^-$. Let $\nu:D^\x\to\Q^\x$ be the reduced norm. For any commutative $\Q$-algebra $R$, put
\[\Qtn(R)=(D\ot_\Q R)^\x.\]
If $v$ is a place of $\Q$, let $D_v=D\ot_\Q\Q_v$. For $x\in \Qtn(\A)$, denote by $x_v\in D_v^\x$ the local component of $x$ at $v$. We fix an isomorphism $\Psi=\prod_{\pmq\ndivides N^-}\Psi_\pmq:\Qtn(\wh\Q^{(N^-)})\iso {\rm M}_2(\wh\Q^{(N^-)})$ once and for all.
Let $\cO_D$ be the maximal order of $D$ such that $\Psi_\pmq(\cO_D\ot\Z_\pmq)={\rm M}_2(\Z_\pmq)$ for all primes $\pmq\ndivides \infty N^-$. Let $N^+$ be a positive integer prime to $N^-$ and let \[N=N^+N^-.\] Denote by $R_{N}$ the Eichler order of level $N^+$ in $D$ with respect to $\Psi$. Put
\[\opcpt_1(N)=\stt{g=(g_\pmq)_\pmq\in \wh R_N^{\x}\mid \Psi_\pmq(b_\pmq)\con\pMX{*}{*}{0}{1}\pmod{N\Z_\pmq}\text{ for }\pmq\divides N^+}.\]
We shall frequently use the following notation in this section: let $\pMX{a}{b}{c}{d}\in\GL_2(\wh\Q^{(N^-)})$ act on $x\in\wh D^\x$ by
\[x\pMX{a}{b}{c}{d}:=x\cdot \Psi^{-1}(\pMX{a}{b}{c}{d}).\]

Let $\rmd^\tau x$ be the Tamagawa measure on $\A^\x\bksl D^\x(\A)$ with the volume $\vol(\A^\x D^\x\bksl D^\x(\A),\rmd^\tau x)=2$. There exists a positive rational number $\vol(\wh R_N^\x)$ such that for any $f\in L^1(D^\x\bksl D^\x(\A)/D^\x_\infty\wh R^\x_N)$, we have
\beq\label{E:meaD}\int_{\A^\x D^\x\bksl D^\x(\A)}f(x)\rmd^\tau x=\vol(\wh R_N^\x) \sum_{[x]\in D^\x\bksl \wh D^\x/\wh R^\x_N}
f(x)\cdot (\#\Gamma_{N,x})^{-1}, \eeq
where $[x]$ means the double coset $D^\x x \wh R^\x_N$ and $\Gamma_{N,x}:=(D^\x\cap x\wh R^\x_Nx^{-1})\Q^\x/\Q^\x$. By Eichler's mass formula, we have \beq\label{E:vN}
\begin{aligned}\vol(\wh R^\x_N)
=&\frac{48}{N}\prod_{\pmq \mid N^-}\zeta_\pmq(1)\prod_{\pmq\mid N^+}(1+\pmq^{-1})^{-1}\\
=&\frac{48}{[\SL_2(\Z):\Gamma_0(N)]}\prod_{\pmq\divides N^-}\frac{1+\pmq^{-1}}{1-\pmq^{-1}}.
\end{aligned}
%=&\frac{48}{N}\zeta_{N^-}(1)\cdot \frac{\zeta_{N^+}(2)}{\zeta_{N^+}(1)}
\eeq

For a non-negative integer $\wt$ and a commutative ring $A$, let $L_{\wt}(A):=A[X,Y]_{\deg =\kappa}$ be the space of two variable polynomials of degree $\wt$ over $A$.  Let $\rho_\wt:{\rm M}_2(A)\to \End_A L_{k}(A)$ be the morphism $\rho_\wt(g)P(X,Y)=P((X,Y)g)$.  Let $\pairing_{\wt}:L_\wt(A)\x L_\wt(A)\to A[\frac{1}{\kappa!}]$ be the pairing defined by 
 \[\pair{X^iY^{\kappa-i}}{X^jY^{\kappa-j}}=\begin{cases}
(-1)^i{\kappa\choose i}^{-1}&\text{ if }i+j=\kappa,\\
 0&\text{ if }i+j\not =\kappa.
 \end{cases}\]
 Let $g\mapsto g'$ be the main involution of ${\rm M}_2(A)$ given by 
 \[\pMX{a}{b}{c}{d}'=\pMX{d}{-b}{-c}{a}.\]
 It is well-known that
\beq\label{E:pairing}\pair{\rho_\wt(g)P_1}{P_2}_{\wt}=\pair{P_1}{\rho_\wt(g')P_2}_\wt.\eeq

\subsection{\padic modular forms on definite quaternion algebras}
In the rest of this section, we shall freely identity Dirichelet characters $\chi$ with their adelizations $\chi_\A$ when no confusion may arise. Let $\cO\subset \cO_{\Cp}$ be a finite flat extension of $\Zp$ containing all $\phi(N)$-th roots of unity.  For an $\cO$-algebra $A$ and a $A$-valued (even) Hecke character $\chi:\Q^\x\bksl \wh\Q^\x\to A^\x$ , we let $\sS^D_{\wt+2}(N,\chi,A)$ be the space of \padic modular forms on $\wh D^\x$ of weight $\kappa+2$, level $N$ and branch character $\chi$, consisting of vector-valued functions $\qtnf:\wh D^\x\to L_\wt(A)$ such that 
\[\qtnf(\al xuz)=\rho_{\wt,p}(u_p^{-1})\qtnf(x)z_p^{-\wt}\chi^{-1}(z)\text{ for all }\al\in D^\x,\,\,u\in \opcpt_1(N^+),z\in\wh\Q^\x.\]
Here $u_p$ is the $p$-component of $u$ and $\rho_{\wt,p}(u_p)=\rho_{\wt}(\Psi_p(u_p))$. For each integer $d$ prime to $pN^-$, define the level raising operator $\LR_d:\sS^D_{\wt+2}(N,\chi,A)\to \sS^D_{\wt +2}(Nd,\chi,A)$ by \[\LR_d\qtnf(x)=\qtnf(x\pDII{d^{-1}}{1}).\]
We recall the Hecke operators $T_\pmq$ and the operators $\bfU_\pmq$ acting on $\qtnf\in\sS^D_{\wt+2}(N,\chi,A)$. For each prime $\pmq\divides N^-$, let $\uf_{D_\pmq}\in R_\pmq^\x$with $\nu(\uf_{D_\pmq})=\pmq$. The Hecke operator $T_\pmq$ for $\pmq\ndivides Np$ is given by 
\[T_\pmq \qtnf(x)=\qtnf(x\pDII{1}{\uf_\pmq})+\sum_{b\in \Z_\pmq/\pmq\Z_\pmq}\qtnf(x\pMX{\uf_\pmq}{b}{0}{1}) \]
and the operator $\bfU_\pmq$ for $\pmq\divides MN^-p$ is given by
\begin{align*}
\bfU_\pmq \qtnf(x)=&\sum_{b\in \Z_\pmq/\pmq\Z_\pmq}\qtnf(x\pMX{\uf_\pmq}{b}{0}{1})\text{ for }\pmq\divides M,\,\pmq\not =p,\quad
\bfU_\pmq \qtnf(x)=\qtnf(x\uf_{D_\pmq})\text{ for }\pmq\divides N^-,\\
\bfU_p\qtnf(x)=&\sum_{b\in\Zp/p\Zp}\rho_{\wt,p}(\pMX{\uf_p}{b}{0}{1})\qtnf(x\pMX{\uf_p}{b}{0}{1}).\end{align*}
Here $\uf_\pmq=(\uf_{\pmq,\ell})\in\wh\Q^{(N^-)\x}$ is the idele $\uf_{\pmq,\pmq}=\pmq$ and $\uf_{\pmq,\ell}=1$ for $\ell\ndivides N^-q$. If $A$ is $p$-adically complete, then the ordinary projector $\eord=\lim_{n\to\infty}\bfU_p^{n!}$ converges to an idempotent in $\End_\cO\sS^D_{\wt+2}(N,\chi,A)$. 

\subsubsection*{Inner products}Denote by $\cyc:\Q_+\bksl\wh\Q^\x\to\Zp^\x$ the \padic cyclotomic character defined by $\cyc(a)=\abs{a}_\A a_p$. Assuming $6\cdot \wt!\in A^\x$, we have a perfect pairing \[(\cdot,\cdot)_N\colon\sS^D_{\wt+2}(N,\chi,A)\times \sS^D_{\wt+2}(N,\chi^{-1},A)\to A\] given by
\[(f_1,f_2)_N:=\sum_{[x]\in D^\x\bksl \wh D^\x/\wh R^\x_{N}}\pair{f_1(x)}{f_2(x)}_\wt\cdot \cyc(\nu(x))^\wt\cdot
(\#\Gamma_{N,x})^{-1}.
\]
Let $\Tau_N^D=(\Tau_{N,\pmq}^D)\in\wh D^\x$ be the element with $\Tau_{N,\pmq}^D=1$ if $\pmq\ndivides N$ and $\Tau_{N,\pmq}^D=\Psi_\pmq^{-1}(\pMX{0}{1}{-N}{0})$ for $\pmq\divides N^+$. Define the Atkin-Lehner involution $[\Tau_N^D]\colon\sS^D_{\wt+2}(N,\chi,A)\to \sS^D_{\wt+2}(N,\chi^{-1},A)$ by
\[[\Tau_N^D]f(x):=\rho_{\wt,p}(\Tau_{N,p}^D)f(x\Tau_N^D)\chi(\nu(x)).\]
We can define a new pairing $\pairing_N:\sS^D_{\wt+2}(N,\chi,A)\times \sS^D_{\wt+2}(N,\chi,A)\to A$ by
\[\pair{f_1}{f_2}_N=(f_1,[\Tau_N^D]f_2)_N.\]
It is well known that this new pairing is Hecke equivariant and perfect (\cf\cite[Lemma 3.5]{Hida06blue}).

\subsection{Automorphic forms on definite quaternion algebras}Fixing $\iota_p:\C_p\iso\C$ once and for all, we choose an imbedding $\Psi_\infty:D_\infty\hookto {\rm M}_2(\C)$ such that $\Psi_\infty(\al)=\iota_p(\Psi_p(\al))$ for $\al\in D^\x$. Define the unitarized representation
$\rho^{\rm u}_{\wt}:D_\infty^\x\to\Aut L_\wt(\C)$ by $\rho^{\rm u}_{\wt}(x)P=\abs{\nu(g)}_\A^{\wt/2}\rho_{\wt}(\Psi_\infty(g))P$ for $P\in L_\wt(\C)$.

For a finite order Hecke character $\om$ modulo $N^+$, let $\cA^D_{\wt+2}(N,\om)$ be the space of $L_\wt(\C)$-valued automorphic forms on $D^\x(\A)$ of weight $\wt+2$, level $N$ and character $\om$. In other words, $\cA^D_{\wt+2}(N,\om)$ consists of functions $\varphi:D^\x(\A)\to L_\wt(\C)$ such that 
\begin{align*}
\varphi(\al x u_\infty u_{\rm f}z)&=\rho^{\rm u}_{\wt}(u_\infty^{-1})\varphi(x)\om(z)\\
(\al\in D^\x,\,u_\infty\in D^\x_\infty,\,&u_{\rm f}\in \opcpt_1(N),\,z\in\A^\x).
\end{align*}
Here $x_{\rm f}$ denotes the finite part of $x$. To each $p$-adic modular form $\qtnf\in\sS^D_{\wt +2}(N,\chi,\cO)$, we associate the adelic lift $\varPhi(\qtnf)\in\cA^D_{\wt+2}(N,\chi^{-1})$ defined by 
\beq\label{E:MA3}\varPhi(\qtnf)(x):=\rho_{\wt}(\Psi_\infty(x_\infty^{-1}))\iota_p(\rho_{\wt,p}(x_p)\qtnf(x_{\rm f}))\cdot \abs{\nu(x)}_{\A}^{\wt/2},\quad x\in D_\A^\x.\eeq

Let $\cA^D(\om)$ be the space of (scalar-valued) automorphic forms on $D^\x(\A)$ with central character $\om$. For $\varphi,\varphi'\in\cA^D(\om)$, define
\[\pair{\varphi}{\varphi'}=\int_{\A^\x D^\x\bksl D^\x(\A)}\varphi(x)\varphi'(x)\om^{-1}(\nu(x))\rmd^\tau x.\]
Here $\rmd^\tau x$ is the Tamagawa measure on $\A^\x\bksl D^\x(\A)$. For $f\in \sS^D_{\wt+2}(N,\om^{-1},\Cp)$ and $\bfu\in L_\wt(\Cp)$, let $\varPhi(f)_\bfu\in\cA^D(\om)$ be the automorphic form given by the matrix coefficient $\varPhi(f)_{\bfu}(x):=\pair{\varPhi(f)(x)}{\bfu}_{\wt}$. By \eqref{E:meaD} and Schur's orthogonality relations, we have 
\beq\label{E:Schur}\pair{\rho(\Tau_{N}^D)\varPhi(f)_{\bfu}}{\varPhi(f)_{\bfv}}=
%N^{-\wt/2}(\wt+1)^{-1}\int_{\A^\x D^\x\bksl \Qtn(\A)}\pair{f^\sharp_1(x\pMX{0}{1}{-N}{0})}{f^\sharp_2(x)}_{\wt}\om^{-1}\Abs_\A^\wt(\nu(x))\rmd x.\\
\frac{\vol(\wh R_N^\x)}{(N^+)^{\wt/2}(1+\wt)}\cdot \pair{f}{f}_N\cdot\pair{\bfu}{\bfv}_{\wt}.\eeq

\subsection{Hida theory for quaternionic modular forms}\label{SS:Hida1} In this subsection, we recall Hida theory for modular forms on definite quaternion algebras following \cite{Hida88Annals}. Suppose that $p\ndivides N$. For each positive integer $\al$, let $X_\al$ be the finite set  \[X_\al= D^\x\bksl \wh D^\x/\opcpt_1(Np^\al)\]
and let $\cO[X_\al]=\bigoplus_{x\in X_\al}\cO x$ be the finitely generated $\cO$-module spanned by divisors of $X_\al$. Recall that $\Lam=\cO\powerseries{1+p\Zp}=\cO\powerseries{T}$, where $T=\Dmd{1+p}_\Lam-1$. For $z\in 1+p\Zp$, let $\Dmd{z}_\Lam$ act on $\cO[X_\al]$ by $\Dmd{z}_\Lam x:=x\pDII{z}{z}$. Let $\Delta=(\Z/pN^+\Z)^\x$.  For $d\in \Delta$, the diamond operator $\sg_d$ acts on $\cO[X_\al]$ as follows: decomposing $d=(d_1,d_2)\in (\Z/p\Z)^\x\times (\Z/N^+\Z)^\x$ and choosing an idele $\wtd d\in \wh\Z^\x$ such that the $p$-component $\wtd d_p=\Om(d_1)\in\Zp^\x$ is the \Teich lifting of $d_1$ and the prime-to-$p$ component $\wtd d^{(p)}\in\wh \Z^{(p)\x}$ is a lifting of $d_2$, we define $\sg_d\, x:=x\wtd d$. Thus $\cO[X_\al]$ is a finitely generated $\Lam[\Delta]$-module. Moreover, $\cO[X_\al]$ is equipped with the usual Hecke operators $T_\pmq$ for $\pmq\ndivides Np$ given by
\[T_\pmq\, x=x\pDII{1}{\uf_\pmq}+\sum_{b\in\Z_\pmq/\pmq\Z_\pmq} x\pMX{\uf_\pmq}{b}{0}{1},\]
the operator $\bfU_\pmq$ for $\pmq\divides Np$ defined by 
\[
\bfU_\pmq\, x=\sum_{b\in\Z_\pmq/\pmq\Z_\pmq} x\pMX{\uf_\pmq}{b}{0}{1}\text{ if }\pmq\divides N^+p;\quad \bfU_\pmq\, x=x\uf_{D_\pmq}\text{ if }\pmq\divides N^-.
\]
The ordinary projector $\eord=\prolim_n \bfU_p^{n!}$ converges to an idempotent in $\End_\Lam(\cO[X_\al])$. 

We introduce the space of $\Lam$-adic modular forms on definite quaternion algebras. Let $X_\infty:=D^\x\bksl \wh D^\x/\opcpt_1(Np^\infty)$, where
\[\opcpt_1(Np^\infty)=\stt{g\in \opcpt_1(N)\mid g_p=\pMX{a}{b}{0}{1},\,a\in\Zp^\x,b\in\Zp}.\]
We have a natural quotient map $X_\infty\to X_{\beta}\to X_\al$ for $\beta>\al$. Let $P_\al$ be the principal ideal of $\Lam$ generated by $(1+T)^{p^\al}-1$.
\begin{defn}Denote by $\bfS^D(N,\Lam)$ the space of functions $\bff\colon X_\infty\to \Lam$ such that 
\begin{itemize}
\item $\bff(xz)=\bff(x)\Dmd{z}^2\Dmd{z}_\Lam^{-1}\text{ for }z\in 1+p\Zp$;
\item for any $\al$ sufficiently large, the function $\bff\pmod{P_\al}\colon X_\infty\to \Lam/P_\al$ factors through $X_\al$.
\end{itemize}
We call $\bfS^D(N,\Lam)$ the space of $\Lam$-adic modular forms on $D^\x$ of level $N$. 
\end{defn}
By definition, we have  \beq\label{E:Hida1}\bfS^D(N,\Lam)=\prolim_\al\Hom_\Lam(\cO[X_\al],\Lam/P_\al)\ot_{\Lam,\iota_2}\Lam,\eeq where $\iota_2:\Lam\to\Lam$ is the $\cO$-algebra homomorphism given by $\iota_2(T)=(1+T)^{-2}(1+p)^2-1$. Hence $\bfS^D(N,\Lam)$ is a compact $\Lam$-module endowed with the natural Hecke action given by $t\bff(x)=\bff(tx)$ for $t=T_\pmq,\bfU_\pmq$ and the action of diamond operators $\sigma_d$. In addition, the ordinary projector $\eord=\prolim_n \bfU_p^{n!}$ converges in $\End_\Lam\bfS^D(N,\Lam)$. For a finite order Hecke character $\chi:\Q^\x\bksl\wh\Q^\x\to\cO^\x$ modulo $N^+p$, put
\begin{align*}\bfS^D(N,\chi,\Lam):=&\stt{\bff\in \bfS^D(N,\Lam)\mid \sg_d\bff=\chi^{-1}(d)\bff\text{ for }d\in\Delta^\x}\\
=&\stt{\bff\in\bfS^D(N,\Lam)\mid \bff(xz)=\bff(x)\cdot\chi^{-1}(z)\Dmd{\cyc(z)}^2\Dmd{\cyc(z)}_\Lam^{-1}\text{ for }z\in\wh\Q^\x}.
\end{align*}
%For each $Q\in\frakX_\bfI^+$, put $\chi_Q=\chi\Om^{2-k_Q}\ep_Q$. Thus $\chi_Q$ is the finite part of the specialization of $\chi_\bfI$ at $Q$ in the sense that $Q(\chi_\bfI(z))=\chi_Q(z)\cyc(z)^{k_Q-2}$. 
Let $\bfI$ be a normal domain finite flat over $\Lam$. We define $\bfS^D(N,\bfI)=\bfS^D(N,\Lam)\ot_\Lam\bfI$ and $\bfS^D(N,\chi,\bfI)=\bfS^D(N,\chi,\Lam)\ot_\Lam\bfI$.

\begin{thm}[Control Theorem]\label{T:HidaQ}
Let $\rmN_\chi:=\sum_{d\in \Delta}\chi(d)\sigma_d\in\cO[\Delta]$ and let $P_\chi$ be the ideal of $\Lam[\Delta]$ generated by $\stt{\chi(d)\cdot\sigma_d-1}_{d\in \Delta}$. Suppose that $p>3$. Then 
\begin{enumerate}\item $\bfS^D(N,\chi,\bfI)$ is a free $\bfI$-module, and the norm map $\rmN_\chi\colon \eord\bfS^D(N,\bfI)/P_\chi\iso\eord\bfS^D(N,\chi,\bfI)$ is an isomorphism. \item For every arithmetic point $Q\in\frakX_\bfI^\ari$, we have a Hecke equivariant isomorphism
\begin{align*}\eord\bfS^D(N,\chi,\bfI)\ot_\bfI\bfI/\wp_Q&\iso \eord\sS^D_{k_Q}(Np^\al,\chi\Om^{2-k_Q}\ep_Q,\cO(Q)),\\
\bff\,(\mathrm{ mod }\, \wp_Q)\mapsto  \bff_Q,
\end{align*}
where $\al=\max\stt{1,c_p(\ep_Q)}$ and $\bff_Q$ is the unique $p$-adic modular form such that 
\[Q(\bff(x))=\pair{\bff_Q(x)}{X^{k_Q-2}}_{k_Q-2}\text{ for all }x\in \wh D^\x.\]
\end{enumerate}
\end{thm}
\begin{proof} This is a reformulation of Hida's control theorems for definite quaternion algebra. We sketch proofs in \cite{Hida88Annals} for the reader's convenience.  We may assume $\bfI=\Lam$ and $\cO=\cO(Q)$. Let $\Delta_p$ be the $p$-Sylow subgroup of $\Delta$. We first show that $\eord\bfS^D(N,\Lam)$ is a free $\Lam[\Delta_p]$-module.  For any abelian group $A$, let $\rmH^0(X_\al,A)$ be the space of $A$-valued functions on $X_\al$. Let $\sV^\Ord(N):=\dirlim_{\al}\dirlim_\beta\eord \rmH^0(X_\al,p^{-\beta}\cO/\cO)$ be the discrete $\Lam$-module $\sV^\Ord_0(0;\opcpt_1(N^+))$ defined in \cite[Theorem 8.6]{Hida88Annals}. Let $V^\Ord(N):=\prolim_\al \eord\cdot \cO[X_\al]$ be the \pont dual of $\sV^\Ord(N)$. In virtue of \eqref{E:Hida1}, 
\[\eord\bfS^D(N,\Lam)=\Hom_\Lam(V^\Ord(N),\Lam)\ot_{\Lam,\iota_2}\Lam,\]
so it suffices to show that $V^\Ord(N)$ is a free $\Lam[\Delta_p]$-module.
For any positive integer $\al$ and character $\xi:(\Z/N^+p^{\al})^\x\to \cO_K^\x$ of $p$-power order with value in some finite extension $K$ of $\Frac\cO$, we define the $\cO_K$-module \[\rmH^0(X_\al,\xi,A):=\stt{f\in \rmH^0(X_\al,A)\mid f(xz)=\xi(z)f(x),\,x\in X_\al,z\in \wh \Z^\x},\, A=K/\cO_K\text{ or }\cO_K.\]
Since any finite order element in $D^\x$ has order only divisible by $2$ or $3$ and $p>3$, one verifies that the group $D^\x\cap x\opcpt_1(Np^\al)x^{-1}=\stt{1}$ for any $x\in \wh D^\x$ and that 
\[\rmH^0(X_\al,\xi,K/\cO_K)=\rmH^0(X_\al,\xi,\cO_K)\ot K/\cO_K.\] In particular, $\rmH^0(X_\al,\xi,K/\cO_K)$ is $p$-divisible. Hence, the $\Lam[\Delta_p]$-freeness of $V^\Ord(N)$ follows from \cite[Corollary 10.1]{Hida88Annals} (and the proof therein). From the $\Lam[\Delta_p]$-freeness of $\eord\bfS^D(N,\Lam)$, we deduce that the map $\bff\mapsto \rmN_\chi\bff$ induces an isomorphism
\[\rmN_\chi\colon\eord\bfS^D(N,\Lam)/P_\chi\iso \eord\bfS^D(N,\Lam)^{\rmN_\chi=1}=\eord\bfS^D(N,\chi,\Lam).\]
This proves part (1). We proceed to prove part (2). By \cite[Theorem 9.4]{Hida88Annals}, we see that
\[\eord\bfS^D(N,\chi,\Lam)/\wp_Q\iso\eord\bfS^D(N,\Lam)/(P_\chi,\wp_Q)\iso \eord\sS^D_{k_Q}(Np^n,\chi\Om^{2-k_Q}\ep_Q,\cO).\]
The above isomorphism $\bff\mapsto \bff_Q$ is given by the dual map to the one $\iota$ in \cite[(8.10)]{Hida88Annals}, whose explicit description is given in \cite[line 9-11, page 375]{Hida88Annals}. This finishes the proof of part (2).\end{proof}
\subsubsection*{A perfect paring on the space of ordinary $\Lam$-adic forms}
For each positive integer $\al$, put
\[X_0(Np^\al)= D^\x\bksl\wh D^\x/\wh R^\x_{Np^\al}.\]
To each finite order character $\chi:\Q^\x\bksl \wh \Q^\x\to\cO^\x$,  we associate a universal $\bfI$-adic deformation defined by 
\[\chi_\bfI:\Q^\x\bksl\wh \Q^\x\to\bfI^\x,\quad \chi_\bfI(z):=\chi(z)\Dmd{\cyc(z)}^{-2}\Dmd{\cyc(z)}_\bfI.\]
For $\bff,\bff'\in\eord\bfS^D(N,\chi,\bfI)$, put
\[\bfB_{N,\al}(\bff,\bff'):=\sum_{[x]\in X_0(Np^\al)}\bfU_p^{-\al}\bff(x\Tau^D_{Np^\al})\bff'(x)\chi_\bfI(\nu(x))\cdot(\#\Gamma_{Np^\al,x})^{-1}\pmod{P_\al}\in\bfI/P_\al.\]
One verifies that $\bfB_{N,\al+1}(\bff,\bff')\con \bfB_{N,\al}(\bff,\bff')\pmod{P_\al}$.
\begin{defn}\label{D:pairing.bal}
Let 
\[\bfB_N:\eord\bfS^D(N,\chi,\bfI)\times\eord\bfS^D(N,\chi,\bfI)\to\bfI\]
be the Hecke-equivariant $\bfI$-bilinear pairing defined by \[\bfB_N(\bff,\bff'):=\prolim_\al\bfB_{N,\al}(\bff,\bff')\in \prolim_\al\bfI/P_\al=\bfI.\]
For every $Q\in\frakX_\bfI^+$ with $k_Q=2$, we have 
 \begin{align*}\bfB_N(\bff,\bff')(Q)&=\pair{\bfU_p^{-\al}\bff_Q}{\bff'_Q}_{Np^\al}\end{align*}
for any $\al\geq \max\stt{1,c_p(\ep_Q)}$. This in particular implies that the pairing $\bfB_N$ is  perfect. \end{defn}
\begin{lm}\label{L:pairing.bal}For each arithmetic point $Q$ in $\frakX^\ari_\bfI$ and integer $\al\geq \max\stt{1,c_p(\ep_Q)}$, we have 
\[\bfB_N(\bff,\bff')(Q)=(-1)^{k_Q}\cdot \pair{\bfU_p^{-\al}\bff_Q}{\bff'_Q}_{Np^\al}.\]
\end{lm} 
 \begin{proof}To lighten the notation, we let $\wt=k_Q-2$ and let $f=\bff_Q,\,f'=\bff'_Q\in \eord\cS^D_{k_Q}(Np^{\al},\chi\Om^{-\wt}\ep_Q,\cO(Q))$. We first claim that the value $\pair{\bfU_p^{-\beta}f}{f'}_{Np^\beta}$ is independent of any integer $\beta\geq \al$. Choose a prime $\ell\ndivides Np$ such that $\ell+ 1\not \con 0\pmod{p}$ and $\ell$ is inert in $\Q(\sqrt{-1})$ and $\Q(\sqrt{-3})$. Then $D^\x\cap x\wh R_{N\ell p^\al}^\x x^{-1}=\stt{\pm 1}$ for all $x\in\wh D^\x$. Write $\chi_Q=\chi_\bfI\pmod{Q}=\chi\Om^{-\wt}\ep_Q\cyc^\wt$ for brevity. For $\ell$ as above, $(1+\ell)\cdot \pair{\bfU_p^{-\beta}f}{f'}_{Np^\beta}$ equals  \begin{align*}
&\sum_{[x]\in X_0(N\ell p^\beta)}\pair{[\Tau_{Np^\beta}^D]\bfU_p^{-\beta}f(x)}{f'(x)}_{\wt}\cdot\chi_Q(\nu(x))\\
 =&\sum_{[x]\in X_0(N\ell p^\al)}\sum_{b\in\Zp/p^{\beta-\al}\Zp}\pair{[\Tau_{Np^\beta}^D]\bfU_p^{-\beta}f(x\pMX{1}{0}{p^\al b}{1})}{\rho_{\wt,p}(\pMX{1}{0}{-p^\al b}{1})f'(x)}_\wt\cdot\chi_Q(\nu(x))\\
 =&\sum_{[x]\in X_0(N\ell p^\al)}\sum_{b\in\Zp/p^{\beta-\al}\Zp}\pair{\rho_{\wt,p}(\pMX{0}{1}{-p^{\beta}}{bp^\al})\bfU_p^{-\beta}f(x\Tau^D_N\pMX{0}{1}{-p^{\beta}}{bp^\al})}{f'(x)}_\wt\cdot\chi_Q(\nu(x))\\
 =&\sum_{[x]\in D^\x\bksl \wh D^\x/\wh R^\x_{N\ell p^\al}}\sum_{b\in\Zp/p^{\beta-\al}\Zp}\pair{\rho_{\wt,p}(\Tau^D_{p^\al}\pMX{1}{-p^{\beta-\al}b}{0}{1})\bfU_p^{-\beta}f(x\Tau^D_{Np^\al}\pMX{1}{-p^{\beta-\al}b}{0}{1}}{f'(x)}\cdot\chi_Q(\nu(x))\\
 =&\sum_{[x]\in X_0(N\ell p^\al)}\pair{[\Tau^D_{Np^\al}]\bfU_p^{-\al}f(x)}{f'(x)}_\wt\chi_Q(\nu(x))=(1+\ell)\cdot \pair{\bfU_p^{-\al}f}{f'}_{Np^\al}.
 \end{align*}
 This verifies the claim. For $x\in\wh D^\x$, we let $f^{[0]}(x)=\pair{f(x)}{X^\wt}_\wt$ be the specialization of $\bff(x)$ at $Q$. For any positive integer $m$, there exists a sufficiently larger $\beta>m+\val_p(\wt !)$ such that 
 \begin{align*}
 \bfB_N(\bff,\bff')(Q)\pmod{p^m}\con&(1+\ell)^{-1}\sum_{[x]\in X_0(N\ell p^\beta)}\bfU_p^{-\beta}f^{[0]}(x\Tau^D_{Np^{\beta}})f^{\prime[0]}(x)\chi_Q(\nu(x))\pmod{p^m}.
% \\ \con&(1+\ell)^{-1}\sum_{[x]\in X_0(N\ell p^{\beta})} \pair{\bfU_p^{-\beta}f(x\Tau^D_{Np^{\beta}})}{\rho_{\wt,p}(\pMX{0}{-1}{p^\beta}{0})f(x)}\cdot\chi_Q(\nu(x))\pmod{p^{m}}\\ 
  %\con&(1+\ell)^{-1}\sum_{[x]\in X_0(N\ell p^{\beta})} \pair{\bfU_p^{-\beta}f(x\Tau^D_{Np^{\beta}})}{\rho_{\wt,p}(\pMX{0}{-1}{p^\beta}{0}\pMX{p^\beta}{z}{0}{1})\bfU_p^{-\beta}f(x\pMX{\uf_p^\beta}{z}{0}{1})}\cdot\chi_Q(\nu(x))\pmod{p^{m}}\\
%=& \pair{\bfU_p^{-\beta}f}{f}_{Np^\beta}\pmod{p^m}=\pair{\bfU_p^{-\al}f}{f}_{Np^\al}\pmod{p^m}.
 \end{align*}
 On the other hand, we have
 \begin{align*}
  &\pair{\bfU_p^{-\al}f}{f'}_{Np^\al}\con \pair{\bfU_p^{-\beta}f}{f'}_{Np^\beta}\pmod{p^m}\\
  \con&(1+\ell)^{-1}\sum_{[x]\in X_0(N\ell p^{\beta})} \sum_{z\in\Zp/p^\beta\Zp}\pair{\bfU_p^{-\beta}f(x\Tau^D_{Np^{\beta}})}{\rho_{\wt,p}(\pMX{0}{-1}{p^\beta}{0}\pMX{p^\beta}{z}{0}{1})\bfU_p^{-\beta}f'(x\pMX{\uf_p^\beta}{z}{0}{1})}\cdot\chi_Q(\nu(x))\pmod{p^{m}}\\
  \con&(1+\ell)^{-1}\sum_{[x]\in X_0(N\ell p^{\beta})} \sum_{z\in\Zp/p^\beta\Zp}\pair{\bfU_p^{-\beta}f(x\Tau^D_{Np^{\beta}})}{\rho_{\wt,p}(\pMX{0}{-1}{0}{0})\bfU_p^{-\beta}f'(x\pMX{\uf_p^\beta}{z}{0}{1})}\cdot\chi_Q(\nu(x))\pmod{p^{m}}\\
   \con&(1+\ell)^{-1}\sum_{[x]\in X_0(N\ell p^{\beta})}\pair{\bfU_p^{-\beta}f^{[0]}(x\Tau^D_{Np^{\beta}})}{f^{\prime [0]}(x)}(-1)^\wt\cdot\chi_Q(\nu(x))\con(-1)^\wt\cdot \bfB_N(\bff,\bff')(Q)\pmod{p^m}.
 \end{align*}
 In the third equality, we have used the fact that $\pair{\bfU^n_p f(x)}{X^\wt}=\bfU^n_pf^{[0]}(x)$ for any $n\in\Z$.
 This proves the lemma.
 \end{proof}

\subsection{Hecke algebras and primitive $\Lam$-adic forms}\label{SS:Hida2}
Let $\bfT^D(N,\bfI)$ be the sub-algebra of $\End_\bfI(\eord\bfS^D(N,\bfI))$ generated by $T_\pmq$, $\bfU_\pmq$ and the diamond operators $\Dmd{d}$ over $\bfI$ and let $\bfT^D(N,\chi,\bfI)$ be the holomorphic image of $\bfT^D(N,\bfI)$ in $\End_{\Lam}(\eord\bfS^D(N,\chi,\bfI))$. Thanks to the Jacquet-Langlands correspondence, there is a surjective $\bfI$-algebra homomorphism $JL\colon\bfT(N,\bfI)\to \bfT^D(N,\bfI)$ such that $JL(T_\pmq)=T_\pmq$ for $\pmq\ndivides Np$, $JL(\bfU_\pmq)=\bfU_\pmq$ for $\pmq\divides N^+p$, $JL(\bfU_\pmq)=(-1)\bfU_\pmq$ for $\pmq\divides N^-$ and $JL(\sigma_d)=\sg_d$; moreover, for an ordinary $\Lam$-adic newform $\bdsf\in \eord\bfS(N,\chi,\bfI)$ of tame conductor $N$ with $\supp N^-\subset \Sigma_{\bdsf}^0$, the corresponding homomorphism $\lam_\bdsf:\bfT(N,\bfI)\to\bfI$ factors through $JL$. We denote by $\lam_{\bdsf}^D:\bfT^D(N,\bfI)\to\bfT^D(N,\chi,\bfI)\to\bfI$ the morphism such that $\lam_{\bdsf}=\lam^D_{\bdsf}\circ JL$. Put
\[ \eord\bfS^D(N,\bfI)[\lam_{\bdsf}^D]:=\stt{\bff\in \eord\bfS^D(N,\bfI)\mid t\cdot\bff=\lam_{\bdsf}^D(t)\bff\text{ for }t\in\bfT^D(N,\bfI)}.\]
The multiplicity one theorem for $\GL(2)$ implies that 
$\dim_{\Frac\Lam} \eord\bfS^D(N,\bfI)[\lam_{\bdsf}^D]\ot_{\Lam}\Frac\Lam=1$. Any non-zero element in $\eord\bfS^D(N,\bfI)[\lam_{\bdsf}^D]$ is called a \emph{Jacquet-Langlands lift} of $\bdsf$, but we do not have a notion of normalized eigenforms for quaternionic modular forms due to the lack of the $q$-expansion. Nonetheless, we have the following
\begin{thm}\label{T:freeness}Suppose that $\bdsf$ satisfies the Hypothesis $({\rm CR},\,\supp(N^-))$ in \subsecref{S:period.1}. Then the $\bfI$-module $\eord\bfS^D(N,\bfI)[\lam_\bdsf^D]$ is free of rank one. In this case, a generator $\bdsf^D$ of $\eord\bfS^D(N,\bfI)[\lam_\bdsf^D]$ is called the primitive Jacquet-Langlands lift of $\bdsf$. By definition, $\bdsf^D$ is unique up to a scalar in $\bfI^\x$.\end{thm}
\begin{proof}Let $\frakm$ be the maximal ideal of $\bfT^D(N,\bfI)$ containing $\Ker\lam_\bdsf^D$.
Under the Hypothesis (CR), we note that $\eord\bfS^D(N,\bfI)_\frakm$ is a free $\bfT^D(N,\bfI)_\frakm$-module of rank one in virtue of \cite[Theorem, 2.1]{Wiles95} and \cite[Corollary 8.11and Remark 8.12]{Helm07} and Hida's control theorem (\cf\cite[Proposition 6.4 and 6.5]{PW11CoM}). By \thmref{T:HidaQ} (1),  we find that $\eord\bfS^D(N,\chi,\bfI)_\frakm$ is also a free $\bfT^D(N,\chi,\bfI)_\frakm$-module of rank one  which in turn implies that $\bfT^D(N,\chi,\bfI)_\frakm$ is Gorenstein as $\eord\bfS^D(N,\chi,\bfI)_\frakm$ is equipped with a Hecke-equivariant perfect pairing $\bfB_N$. It follows that $\eord\bfS^D(N,\bfI)_\frakm[\lam_\bdsf^D]=
\eord\bfS^D(N,\chi,\bfI)_\frakm[\lam_\bdsf^D]\iso \bfT^D(N,\chi,\bfI)_\frakm[\lam_\bdsf^D]$ is a free of rank one $\bfI$-module.
\end{proof}

\subsection{Regularized diagonal cycles and theta elements}\label{SS:theta1}
Recall that $E=\Q\oplus\Q\oplus\Q$ is the totally split \etale cubic algebra over $\Q$. Let $D_E=D\oplus D\oplus D$. 
For each positive integer $n$, let \[\opcpt_{E,1}(Np^n):=\opcpt_1(Np^n)\times \opcpt_1(Np^n)\times \opcpt_1(Np^n)\] be an open-compact subgroup of $\wh D_E^\x$. Define the finite set 
\begin{align*}\bfX_n:
&=D_E^\x\bksl \wh D_E^\x/\opcpt_{E,1}(Np^n)\wh\Q^\x\\
&=(X_n\times X_n\times X_n)/\wh\Q^\x.
\end{align*}
The set $\bfX_n$ is a zero dimensional analogue of the triple product of modular curves. Consider the finitely generated $\Zp$-module $\Zp[\bfX_n]$ equipped with the operator $\bfU_{E,p}:=\bfU_p\ot\bfU_p\ot\bfU_p$ and the ordinary projector $\eord_{E}:=\eord\ot\eord\ot\eord$. For each $(x_1,x_2,x_3)\in \wh D_E^\x$, let $[(x_1,x_2,x_3)]$ denote the double coset $D_E^\x(x_1,x_2,x_3)\opcpt_{E,1}(Np^n)\wh \Q^\x$.   \begin{defn}[Regularized diagonal cycles]\label{D:Reg}Put $\tau_{p^n}:=\pMX{0}{1}{-p^n}{0}\in\GL_2(\Qp)$. Let $\Delta_n\in \Zp[\bfX_n]$ be the twisted diagonal cycle given by 
\[\Delta_n:=\sum_{[x]\in X_0(Np^n)}\sum_{\substack{b\in\Zp/p^n\Zp,\\
z\in (\Zp/p^n\Zp)^\x}}[(x\pMX{p^n}{b}{0}{1},x\pMX{p^n}{b+z}{0}{1},x\tau_{p^n}\pDII{1}{z})]\]
and define the regularized diagonal cycle $\Delta_n^\dagger$ by 
\[\Delta_n^\dagger:=\bfU_{E,p}^{-n}\,(\eord_E\Delta_n).\]
The following lemma allows us to define the $\Lam$-adic diagonal cycle
\[\Delta_\infty^\dagger:=\prolim_{n\to\infty}\Delta^\dagger_n\in \prolim_{n\to\infty}\Zp[\bfX_n],\]
where the inverse limit is taken with respect to the natural homomorphism
 $\rmN_{n+1,n}:\Zp[\bfX_{n+1}]\to \Zp[\bfX_n]$.
\end{defn}
\begin{lm}[Distribution property]For every $n\geq 1$,  \[\rmN_{n+1,n}(\Delta^\dagger_{n+1})=\Delta_n^\dagger.\]
\end{lm}
\begin{proof}It is equivalent to showing that 
\[\rmN_{n+1,n}(\Delta_{n+1})=\bfU_{E,p}\Delta_n.\]
Let $S_n:= (\Zp/p^n\Zp)\times (\Zp/p^n\Zp)^\x$. A direct computation shows that
\begin{align*}
&\rmN_{n+1,n}(\Delta_{n+1})\\
=&\sum_{[x]\in X_0(Np^{n+1})}\sum_{(b,z)\in S_n}\sum_{b_1,z_1\in\Zp/p\Zp}[(x\pMX{p^n}{b}{0}{1}\pMX{p}{b_1}{0}{1},x\pMX{p^n}{b+z}{0}{1}\pMX{p}{b_1+zz_1}{0}{1},x\tau_{p^{n+1}}\pDII{1}{z})]\\
=&\sum_{[x]\in X_0(Np^{n+1})}\sum_{(b,z)\in S_n}(\bfU_p\ot\bfU_p\ot{\rm Id})[(x\pMX{p^n}{b}{0}{1},x\pMX{p^n}{b+z}{0}{1},x\tau_{p^{n+1}}\pDII{1}{z})]\\
=&\sum_{[x]\in X_0(Np^{n})}\sum_{(b,z)\in S_n}\sum_{c\in \Zp/p\Zp}(\bfU_p\ot\bfU_p\ot{\rm Id})[(x\pMX{p^n}{b}{0}{1},x\pMX{p^n}{b+z}{0}{1},x\pMX{1}{0}{p^nc}{1}\tau_{p^{n+1}}\pDII{1}{z})]\\
=&\sum_{[x]\in X_0(Np^{n})}\sum_{(b,z)\in S_n}\sum_{c\in \Zp/p\Zp}(\bfU_p\ot\bfU_p\ot{\rm Id})[(x\pMX{p^n}{b}{0}{1},x\pMX{p^n}{b+z}{0}{1},x\tau_{p^n}\pMX{p}{-c}{0}{1}\pDII{1}{z})]\\
=&(\bfU_p\ot\bfU_p\ot\bfU_p)\Delta_n.\end{align*}
This proves the assertion.
\end{proof}

Following the notation in \subsecref{SS:36}, we let
$\cR=\bfI_1\wh\ot_\cO\bfI_2\wh\ot_\cO\bfI_3$ be a finite extension of $\cR_0=\cO\powerseries{T_1,T_2,T_3}$. For a triple of ordinary $\Lam$-adic quaternionic forms
\[(\bff,\bfg,\bfh)\in\eord\bfS^D(N,\brchf,\bfI_1)\times \eord\bfS^D(N,\brchg,\bfI_2)\times \eord\bfS^D(N,\brchh,\bfI_3),\] 
we let $\bfF=\bff\boxtimes\bfg\boxtimes\bfh:D_E^\x\bksl \wh D_E^\x\to\cR$ be the triple product given by 
\[\bfF(x_1,x_2,x_3)=\bff(x_1)\ot\bfg(x_2)\ot \bfh(x_3).\]
 Let $\chi_\cR^*:\Q^\x\bksl \wh\Q^\x\to \cR^\x$ be the reciprocal of a square root of the character $\psi_{1\bfI_1}\ot\psi_{2\bfI_2}\ot\psi_{3\bfI_3}$ defined by   
\begin{align*}
\chi_\cR^*(z)&=\Om^{a}(z)\Dmd{\cyc(z)}^{-3}\Dmd{\cyc(z)}_{\bfI_1}^{1/2}\Dmd{\cyc(z)}_{\bfI_2}^{1/2}\Dmd{\cyc(z)}_{\bfI_3}^{1/2}\in\cR^\x
\end{align*}
and set \[\bfF^*(x_1,x_2,x_3):=\bfF(x_1,x_2,x_3)\cdot \chi_\cR^*(\nu(x_3)).\]
Then $\bfF^*$ naturally induces a $\Zp\powerseries{T_1,T_2,T_3}$-linear map  
\[\bfF^*\colon\prolim_{n\to\infty}\Zp[\bfX_n]\to \cR.\]
The theta element $\Theta_\bfF$ attached to the triple product $\bfF$ is then defined by the evaluation of $\bfF^*$ at the $\Lam$-adic diagonal cycle. In other words,
\[\Theta_\bfF:=\bfF^*(\Delta^\dagger_\infty)\in\cR.\]

\subsection{The construction of $p$-adic $L$-functions in the balanced case}\label{SS:theta2}
We let $\bdsF=(\bdsf,\bdsg,\bdsh)$ be the triple of primitive Hida families of tame conductor $(\condf,\condg,\condh)$ in \subsecref{S:levelraising}. Recall that $\Sigma^-$ is the finite subset of prime factors of $N=\lcm(\condf,\condg,\condh)$ in \defref{D:root}. Let $N^-=\prod_{\ell\in\Sigma^-}\ell$. In the remainder of this section, we assume that 
\begin{itemize}
\item $\#(\Sigma^-)$ is odd,
\item  $\bdsf$, $\bdsg$ and $\bdsh$ satisfy the Hypothesis (CR, $\Sigma^-$);
 \item$N^-$ and $N/N^-$ are relatively prime.
\end{itemize}
 Let $D$ be the definite quaternion algebra over $\Q$ with the discriminant $N^-$ and let 
\[(\bdsf^D,\bdsg^D,\bdsh^D)\in\eord\bfS^D(\condf,\brchf,\bfI_1)\times \eord\bfS^D(\condg,\brchg,\bfI_2)\times \eord\bfS^D(\condh,\brchh,\bfI_3)\]
be the primitive Jacquet-Langlands lift of $(\bdsf,\bdsg,\bdsh)$ constructed in \thmref{T:freeness}.
\begin{defn}\label{D:testbal}Let $N_i^+=N_i/N^-$ for $i=1,2,3$ and let $N^+=\lcm(\condf^+,\condg^+,\condh^+)$. Then $N=N^+N^-$. Define 
\[(\bdsf^\Dstar,\bdsg^\Dstar,\bdsh^\Dstar)\in\eord\bfS^D(N,\brchf,\bfI_1)\times \eord\bfS^D(N,\brchg,\bfI_2)\times \eord\bfS^D(N,\brchh,\bfI_3)\]
by
\begin{align*}\bdsf^\Dstar&:=\sum_{I\subset \Sigma_{f,0}^{\rm (IIb)}}(-1)^{\abs{I}} \beta_I(\bdsf)^{-1}\cdot \LR_{\Bd_f/n_f}\bdsf^D,\\
\bdsg^\Dstar&:=\sum_{I\subset \Sigma_{g,0}^{{\rm (IIb)}}}(-1)^{\abs{I}} \beta_I(\bdsg)^{-1}\cdot \LR_{\Bd_g/n_g}\bdsg^D,\\
\bdsh^\Dstar&:=\sum_{I\subset \Sigma_{h,0}^{{\rm (IIb)}}}(-1)^{\abs{I}} \beta_I(\bdsh)^{-1}\cdot \LR_{\Bd_h/n_h}\bdsh^D.
\end{align*}
Define the triple product $\bdsF^\Dstar:D_E^\x\bksl \wh D_E^\x\to\cR$ by \[\bdsF^\Dstar:=\bdsf^\Dstar\boxtimes\bdsg^\Dstar\boxtimes\bdsh^\Dstar.\]
Then $\bdsF^\Dstar$ is an eigenfunction of the operator $\bfU_{E,p}$ with the eigenvalue $\al_p(\bdsF):=\bfa(p,\bdsf)\bfa(p,\bdsg)\bfa(p,\bdsh)$. We define the associated theta element $\Theta_{\bdsF^\Dstar}$ to be the $p$-adic $L$-functions attached to the triple $(\bdsf,\bdsg,\bdsh)$ in the balanced range.
\end{defn}

\subsection{Global trilinear period integrals}
\subsubsection{The setting}
In this subsection, we relate the evaluations of the $p$-adic $L$-function $\Theta_{\bdsF^\Dstar}$ at arithmetic points in the balanced range to certain global trilinear period integral on $D^\x_\A$. The set $\frakX_\cR^\bal$ of arithmetic points in the balanced range, consisting of arithmetic points $\ulQ=(Q_1,Q_2,Q_3)\in\frakX_{\bfI_1}^\ari\times\frakX_{\bfI_2}^\ari\times\frakX_{\bfI_3}^\ari$ such that
\[k_\Qx+k_\Qy+k_\Qz\con 0\pmod{2};\quad k_\Qx+k_\Qy+k_\Qz>2k_{Q_i}\text{ for all }i=1,2,3.\]
 Let $\ulQ=(Q_1,Q_2,Q_3)\in\frakX_\cR^\bal$. Put \[k_i=k_{Q_i}\text{ and }\wt_i=k_i-2\text{ for }i=1,2,3.\]
We keep the notation in \subsecref{SS:Ichino}. Thus $F=(f,g,h)$ denotes the specialization $\bdsF_\ulQ=(\bdsf_\Qx,\bdsg_\Qy,\bdsh_\Qz)$ of $\bdsF$ at $\ulQ$ and $\om_F^{1/2}$ is the square root of the central character $\om_F=\om_f\om_g\om_h$ defined in \eqref{E:central}. Let $\itPi=\itPi_\ulQ$ be the automorphic representation of $\GL_2(\A_E)$ defined by \[\itPi_\ulQ=\pi_f\ot\om_F^{-1/2}\times\pi_g \times\pi_h.\]
Let $(f^D,g^D,h^D)=(\bdsf^D_\Qx,\bdsg^D_\Qy,\bdsh^D_\Qz)$ be the specializations in the sense of \thmref{T:HidaQ} (2). We have \[(f^D,g^D,h^D)\in \sS^D_{\wt_1+2}(\condf p^n,\om_f^{-1},\cO(\ulQ))\times \sS^D_{\wt_2+2}(\condg p^n,\om_g^{-1},\cO(\ulQ))\times \sS^D_{\wt_3+2}(\condh p^n,\om_h^{-1},\cO(\ulQ)),\] where \[\om_f=\brchf^{-1}\Om^{\wt_1}\ep_{\Qx}^{-1},\quad\om_g=\brchg^{-1}\Om^{\wt_2}\ep_{\Qy}^{-1}\text{ and }\om_h=\brchh^{-1}\Om^{\wt_3}\ep_{\Qz}^{-1}.\]
Let $\varphi_{f^D}=\varPhi(f^D)$, $\varphi_{g^D}=\varPhi(g^D)\text{ and }\varphi_{h^D}=\varPhi(h^D)$ be the associated adelic lifts as in \eqref{E:MA3}. We have
\[(\varphi_{f^D},\varphi_{g^D},\varphi_{h^D})\in \cA_{\wt_1+2}^D(\condf p^n,\om_f)\times\cA_{\wt_2+2}^D(\condg p^n,\om_g)\times \cA_{\wt_3+2}^D(\condh p^n,\om_h).\]
Let $\cQ_{1,\ell}(X)$, $\cQ_{2,\ell}(X)$ and $\cQ_{3,\ell}(X)$ be the polynomials defined in \eqref{E:defQ} and put
\beq\varphi_1^\Dstar=\prod_\ell\cQ_{1,\ell}(\LR_\ell)(\varphi_{f^D}\ot\om_F^{-1/2}),\quad \varphi_2^\Dstar=\prod_\ell\cQ_{2,\ell}(\LR_\ell)\varphi_{g^D};\quad\varphi_3^\Dstar=\prod_\ell\cQ_{3,\ell}(\LR_\ell)\varphi_{h^D}.\eeq
Note that 
\beq\label{E:F2}\begin{aligned}\varphi^\Dstar_1&=\Bd_f^{\wt_1/2}\om_F^{1/2}(\wh \Bd_f)\cdot\varPhi(\bdsf^\Dstar_\Qx)\ot\om_F^{-1/2},\\
\varphi^\Dstar_2&=\Bd_g^{\wt_2/2}\cdot\varPhi(\bdsg^\Dstar_\Qy);\quad \varphi^\Dstar_3=\Bd_h^{\wt_3/2}\cdot\varPhi(\bdsh^\Dstar_\Qz).\end{aligned}\eeq

Let $L_{\ulk}(A):=L_{\wt_1}(A)\ot L_{\wt_2}(A)\ot L_{\wt_3}(A)$ for any commutative ring $A$ and $\rho_{\ulk}=\rho_{\wt_1}\ot\rho_{\wt_2}\ot\rho_{\wt_3}$.  Define $\rho^{\rm u}_{\ulk}$ and $\rho_{\ulk,p}$ likewise. 
For any $\Q$-algebra $R$, let $D_E^\x(R):=D^\x(R)\times D^\x(R)\times D^\x(R)$ and let $\nu_E^{\ulk}:D_E^\x(R)\to R^\x$ be the map $\nu_E^{\ulk}(x_1,x_2,x_3):=\prod_{i=1}^3\nu(x_i)^{\wt_i}$. Define the vector-valued automorphic form   
\begin{align*}
\vec\phi^\Dstar&:D_E^\x(\A)\to L_{\ulk}(\C),\\
\vec\phi^\Dstar(x_1,x_2,x_3)&=\varphi^\Dstar_1(x_1)\ot\varphi^\Dstar_2(x_2)\ot\varphi^\Dstar_3(x_3)\quad(x_i\in D^\x(\A)).
%&=C_F\cdot \rho_{\ulk}(x_\infty^{-1})F^\Dstar_\ulQ(x_{\rm f})\abs{\nu_E(x)^{\ulk}}_\A^{1/2}\om_F^{-1/2}(\nu(x_1)).
\end{align*}
Define $\bfP_{\ulk}\in L_{\ulk}(\Z)$ by 
\beq\label{E:Delta.bal}\begin{aligned}\bfP_{\ulk}=&(X_1Y_2-X_2Y_1)^{\wt^*_1}(X_3Y_1-X_1Y_3)^{\wt^*_2}(X_3Y_2-X_2Y_3)^{\wt^*_3},\\
&\wt^*_i:=\frac{\wt_1+\wt_2+\wt_3}{2}-\wt_i \quad(i=1,2,3).\end{aligned}\eeq
Then $\bfP_{\ulk}$ is a basis of the line $L_{\ulk}(\C)$ fixed by $D^\x_\infty$ under the action of $\rho_{\ulk}^{\rm u}$. Define the automorphic form 
\beq\label{E:defphiD}
\begin{aligned}\phi^\Dstar_F&:D_E^\x(\A)\to\C,\\
\phi^\Dstar_F(x_1,x_2,x_3)&=\pair{\vec\phi^\Dstar(x_1,x_2,x_3)}{\bfP_{\ulk}}_{\ulk},\end{aligned}\eeq
where $\pairing_{\ulk}=\pairing_{\wt_1}\ot\pairing_{\wt_2}\ot\pairing_{\wt_3}$. One verifies that 
\beq\label{E:F3}\phi^\Dstar_F(x_1u_\infty,x_2u_\infty,x_3u_\infty)=\phi^\Dstar_F(x_1,x_2,x_3)\text{ for }u_\infty\in D_\infty^\x.\eeq
\subsubsection{The global trilinear period integrals}
Let $n_\ulQ=\max\stt{\cond{\ep_{Q_1}},\cond{\ep_{Q_2}},\cond{\ep_{Q_3}},1}$ and let $n\geq n_\ulQ$ be a positive integer. Let $\breve\bft_n\in D_E^\x(\Qp)$ be the matrix given by
\[\breve\bft_n=(\pMX{1}{p^{-n}}{0}{1},\pDII{1}{1},\pMX{0}{p^{-n}}{-p^n}{0})\in\GL_2(E_p).\]
We shall relate the interpolation to the global trilinear period integral
\[I(\rho(\breve\bft_n)\phi^\Dstar_F)=\int\limits_{D^\x\A^\x\bksl D_\A^\x}\phi^\Dstar_F(x\pMX{1}{p^{-n}}{0}{1},x,x\pMX{0}{p^{-n}}{-p^n}{0})\rmd^\tau x.\]
Here $\rmd^\tau x$ is the Tamagawa measure on $\A^\x\bksl D^\x_\A$. %Next we prove the following interpolation formula of $\Theta_{\bdsF^{\Dstar}}(\ulQ)$ based on ideas in \cite{CH17Crelle}. 
\begin{prop}\label{P:formulaTheta}  For every $n\geq n_\ulQ$, we have \[\Theta_{\bdsF^\Dstar}(\ulQ)=\frac{1}{\vol(\wh R_{N}^\x)}\cdot I(\rho(\breve\bft_n)\phi^\Dstar_F)\cdot\frac{\om_{F,p}^{1/2}(p^n)\abs{p^n}^{-\frac{k_1+k_2+k_3}{2}}}{\al_p(F)^n\zeta_p(2)}\cdot \frac{1}{\om_F^{1/2}(\wh \Bd_f)\Bd_F^{\ulk/2}},
\]
where $\al_p(F)=\bfa(p,f)\bfa(p,g)\bfa(p,h)$ and $\Bd_F^{\ulk/2}=\Bd_f^{\wt_1/2}\Bd_g^{\wt_2/2}\Bd_h^{\wt_3/2}$ defined in \eqref{E:BdF}.

\end{prop}
\begin{proof}We begin with some notation. Let $\ulQ(\bdsF^\Dstar): D_E^\x\bksl \wh D_E^\x\to \cO_{\Cp}$ denote the value of $\bdsF^\Dstar$ at the point $\ulQ\in\Spec\cR(\Qbarp)$. Namely, 
\[\ulQ(\bdsF^\Dstar)(x_1,x_2,x_3)=\Qx(\bdsf^\Dstar(x_1))\Qy(\bdsg^\Dstar(x_2))\Qz(\bdsh^\Dstar(x_3)).\]
Let $(f^\Dstar,g^\Dstar,h^\Dstar)=(\bdsf^\Dstar_\Qx,\bdsg^\Dstar_\Qy,\bdsh^\Dstar_\Qz)$ denote the specialization of $(\bdsf^\Dstar,\bdsg^\Dstar,\bdsh^\Dstar)$ as in \thmref{T:HidaQ} (2). Put
\[\wh F^\Dstar:=f^\Dstar\boxtimes g^\Dstar\boxtimes h^\Dstar,\quad 
\wh F^\Dstar(x_1,x_2,x_3)=f^\Dstar(x_1)\ot g^\Dstar(x_2)\ot h^\Dstar(x_3).
\]
By definition, we have
\beq\label{E:F2.T}\ulQ(\bdsF^\Dstar)(x_1,x_2,x_3)
=\pair{\wh F^\Dstar(x_1,x_2,x_3)}{X_1^{\wt_1}X_2^{\wt_2}X_3^{\wt_3}}_{\ulk}.\eeq
Define the adelic lift $F^\Dstar:D_E^\x(\A)\to L_{\ulk}(\C)$ of $\wh F^{\Dstar}$ to be the function 
\[F^\Dstar(x)=\rho_{\ul{\wt},p}(x_p)\wh F^{\Dstar}(x)\quad(x\in D_E^\x(\A)).\]
Then one verifies that
\[\phi^\Dstar_F(x)=d_F^{\ulk/2}\cdot\pair{F^\Dstar(x)}{\bfP_{\ulk}}_{\ulk}\cdot \abs{\nu_E^{\ulk}(x)}_\A^{1/2}.\]

Let $m_k$ be the $p$-adic valuation of $(\kappa_1+\kappa_2+\kappa_3)!$ and let $m>m_k$ be a positive integer. For a number $A\in\Cp$, denote by $A\pmod{p^m}$ the residue class of $A$ in $\Cp$ modulo $p^{m}\cO_{\Cp}$. By definition, for any sufficiently large integer $s\gg  n+m+m_k\geq 1$, 
%Suppose that $\bfF$ is the eigenfunction of the operator $\bfU_{E,p}$ with the eigenvalue $\al_p(\bfF)$. 
%Let $\frakm_\cR$ be the maximal ideal of $\cR$. By definition, given a positive integer $m$, for each sufficiently large integer $n$, we have the equality
\beq\label{E:F1.T}
\begin{aligned}&\Theta_{\bdsF^\Dstar}(\ulQ)\pmod{p^m}\\%\con \ulQ(\bdsF^{\Dstar *})(\Delta_s^\dagger)\pmod{p^m}\\
&\con\al_p(F)^{-s}\sum_{[x]\in X_0(Np^s)}\sum_{\substack{b\in (\Zp/p^s\Zp),\\ z\in (\Zp/p^s\Zp)^\x}}\ulQ(\bdsF^\Dstar)(x\pMX{p^s}{b}{0}{1},x\pMX{p^s}{b+z}{0}{1},x\tau_{p^s})\\
&\times \Bkappa_h(z)z^{\wt_3^*}\cdot \chi^*_\ulQ(\nu(x))\pmod{p^m},\end{aligned}\eeq
where 
%Let $\Bkappa^*_\cR:\Zp^\x\to\cR^\x$ be the character defined by 
%\begin{align*}
%\Bkappa^*_\cR(b):=&(\brchh^{-1})_{\bfI_3}(b)\cdot\chi_\cR^*(b).\end{align*} 
$\Bkappa_h(z):=\om_F^{-1/2}\om_h(z)$ for $z\in\Zp^\x$ and $\chi^*_\ulQ$ is the specialization of $\chi^*_\cR$ at $\ulQ$ 
\begin{align*}\chi^*_\ulQ=&\om_F^{-1/2}\cdot \cyc^{r_\wt}\quad(r_\wt:=\frac{\wt_1+\wt_2+\wt_3}{2}=\frac{k_1+k_2+k_3}{2}-3)\end{align*} Putting \[W_s'=\stt{(b_1,b_2)\in (\Zp/p^s\Zp)^2\mid b_1-b_2\in(\Zp/p^s\Zp)^\x},\]
we see from \eqref{E:F1.T} and \eqref{E:F2.T} that
\beq\label{E:1.T}\begin{aligned}&\Theta_{\bdsF^\Dstar}(\ulQ)\pmod{p^m}\con
\al_p(F)^{-s}\sum_{\substack{x\in X_0(Np^s),\\(b_1,b_2)\in W'_s}}(b_1-b_2)^{\wt_3^*}\Bkappa_h(b_1-b_2)\chi^*_\ulQ(\nu(x))\\
&\times \ulQ(\bdsF^\Dstar)(x\pMX{p^s}{b_1}{0}{1},x\pMX{p^s}{b_2}{0}{1},x\tau_{p^s})
\pmod{p^m}\\
\con &\,\al_p(F)^{-s}\sum_{x\in X_0(Np^s)}\sum_{c\in p^n\Zp/p^s\Zp}\sum_{(b_1,b_2)\in W'_s}\Bkappa_h(b_1-b_2)\chi^*_\ulQ(\nu(x))\\
&\times \pair{\wh F^\Dstar(x\pMX{p^s}{b_1}{cp^s}{1+b_1c}),x\pMX{p^s}{b_2}{cp^s}{1+b_2c},x\pMX{0}{1}{-p^s}{c})}{(b_1-b_2)^{\wt^*_3}X_1^{\wt_1}X_2^{\wt_2}X_3^{\wt_3}}_{\ulk}\pmod{p^m}.
\end{aligned}\eeq
To simplify the above expression, we note that by \eqref{E:pairing}, \[\pair{\rho_{\ulk}(x_p^{-1})F^\Dstar(xg_1,xg_2,xg_3)}{\bfP_{\ulk}}_{\ulk}=\pair{\wh F^\Dstar(xg_1,xg_2,xg_3)}{\rho_{\ulk}(g_1'\ot g_2'\ot g_3')\bfP_{\ulk}}_{\ulk}\]
with 
\[g_1=\pMX{p^s}{b_1}{cp^s}{1+b_1c},\,g_2=\pMX{p^s}{b_2}{cp^s}{1+b_2c},\,g_3=\pMX{0}{1}{-p^s}{c},\]
we find the congruence relation 
\begin{align*}
&\pair{\rho_{\ulk}(x_p^{-1})F^\Dstar(xg_1,xg_2,xg_3)}{\bfP_{\ulk}}_{\ulk}\\
\con&\pair{\wh F^\Dstar(xg_1,xg_2,xg_3)}{\rho_{\ulk}(\pMX{1+b_1c}{-b_1}{0}{0}\ot \pMX{1+b_2c}{-b_2}{0}{0}\ot \pMX{c}{-1}{0}{0})\bfP_{\ulk}}_{\ulk}\pmod{p^m}\\
\con&\pair{\wh F^\Dstar(xg_1,xg_2,xg_3)}{\pMX{1+cb_1}{-b_1}{1+b_2c}{-b_2}^{\wt^*_3}\pMX{c}{-1}{1+b_1c}{-b_1}^{\wt^*_2}\pMX{c}{-1}{1+b_2c}{-b_2}^{\wt^*_1}X_1^{\wt_1}X_2^{\wt_2}X_3^{\wt_3}}_{\ulk}\pmod{p^m}\\
\con&\pair{\wh F^\Dstar(xg_1,xg_2,xg_3)}{(b_1-b_2)^{\wt^*_3}X_1^{\wt_1}X_2^{\wt_2}X_3^{\wt_3}}_{\ulk}\pmod{p^m}.
\end{align*}
Substituting the above to \eqref{E:1.T}, we see that $\Theta_{\bdsF^\Dstar}(\ulQ)\pmod{p^m}$ equals
\begin{align*}
&\al_p(F)^{-s}\sum_{x\in X_0(Np^n)}\sum_{c\in p^n\Zp/p^s\Zp}\sum_{(b_1,b_2)\in W'_s}\Bkappa_h(b_1-b_2)\chi^*_\ulQ(\nu( x))\\
&\times \pair{\rho_{\ulk}(x_p^{-1})F^\Dstar(x\pMX{p^s}{b_1}{cp^s}{1+b_1c},x\pMX{p^s}{b_2}{cp^s}{1+b_2c},x\pMX{0}{1}{-p^s}{c})}{\bfP_{\ulk}}_{\ulk}\pmod{p^m}\\
\con &\,\al_p(F)^{-s}\sum_{x\in X_0(Np^n)}\sum_{(b_1,b_2)\in W'_s}\Bkappa_h(b_1-b_2)\chi^*_\ulQ(\nu( x))\nu( x_p)^{-r_\wt}\\
&\times \sum_{c\in p^n\Zp/p^s\Zp}\pair{F^\Dstar(x\pMX{p^s}{b_1}{0}{1},x\pMX{p^s}{b_2}{0}{1},x\tau_{p^n}\pMX{p^{s-n}}{-p^{-n}c}{0}{1})}{\bfP_{\ulk}}_{\ulk}\pmod{p^m}\\
\con &\,\al_p(F)^{-n}\sum_{x\in X_0(Np^n)}\sum_{(b_1,b_2)\in W'_n}\Bkappa_h(b_1-b_2)\om_F^{-1/2}\Abs_\A^{r_\wt}(\nu(x))\\
&\times\pair{F^\Dstar(x\pMX{p^n}{b_1}{0}{1},x\pMX{p^n}{b_2}{0}{1},x\tau_{p^n})}{\bfP_{\ulk}}_{\ulk}\pmod{p^m}.
\end{align*}
The last congruence relation holds for any sufficiently large $m$, so we obtain the expression 
\beq\label{E:2.T}\begin{aligned}
\Theta_{\bdsF^\Dstar}(\ulQ)
%=&\al_p(F)^{-n}\sum_{x\in X_0(Np^n)}\sum_{(b_1,b_2)\in W_n^\circ}\Bkappa^*(b_1-b_2)\om_F^{-1/2}\Abs_\A^{r_\wt}(\nu(x))\cdot \pair{F^\Dstar(x\pMX{p^n}{b_1}{0}{1},x\pMX{p^n}{b_2}{0}{1},x\tau_{p^n})}{\bfP_{\ulk}}_{\ulk}\\
=&\,\al_p(F)^{-n}\sum_{x\in X_0(Np^n)}\sum_{\substack{b_1\in (\Zp/p^n\Zp)^\x,\\ b_2\in\Zp/p^n\Zp}}\Bkappa_h(b_1)\om_F^{-1/2}\Abs^{r_\wt}_\A(\nu(x))\\
&\times \pair{F^\Dstar(x\pMX{p^n}{b_1+b_2}{0}{1},x\pMX{p^n}{b_2}{0}{1},x\tau_{p^n})}{\bfP_{\ulk}}_{\ulk}.\end{aligned}\eeq
By the definition \eqref{E:F2}, 
\[\vec\phi^\Dstar(x_1,x_2,x_3)=\Bd_F^{\ulk/2}\cdot F^\Dstar(x_1,x_2,x_3)\om_F^{-1/2}(\nu(x_1))\abs{\nu_E^{\ulk}(x_1,x_2,x_3)}_\A^{1/2}\]
for $(x_1,x_2,x_3)\in \wh D_E^\x$, and using \eqref{E:F3}, we obtain  
\begin{align*}&\sum_{x\in X_0(Np^n)}\om_F^{-1/2}\Abs^{r_\wt}_\A(\nu(x))\cdot 
\pair{F^\Dstar(x\pMX{p^n}{b_1+b_2}{0}{1},x\pMX{p^n}{b_2}{0}{1},x\tau_{p^n})}{\bfP_{\ulk}}_{\ulk}\\
=& \frac{\om_{F,p}^{1/2}(p^n)\abs{p^n}^{-r_\wt}}{\vol(\wh R^\x_{Np^n})\Bd_F^{\ulk/2}}\int\limits_{\A^\x D^\x\bksl D^\x_\A}\phi^\Dstar_F(x\pMX{p^n}{b_1+b_2}{0}{1},x\pMX{p^n}{b_2}{0}{1},x\tau_{p^n})\rmd^\tau x.\end{align*}
Since $\Bkappa_h=\om_F^{-1/2}\om_h$, we find that the right hand side of the equation \eqref{E:2.T} equals
\begin{align*}
&\frac{\om_{F,p}^{1/2}(p^n)\abs{p^n}^{-r_\wt}}{\al_p(F)^{n}\vol(\wh R^\x_{Np^n})\Bd_F^{\ulk}/2}\int\limits_{\A^\x D^\x\bksl D_\A^\x}\sum_{\substack{b_1\in (\Zp/p^n\Zp)^\x,\\ b_2\in\Zp/p^n\Zp}}\Bkappa_h(b_1)\cdot \phi^\Dstar_F(x\pMX{p^n}{b_1+b_2}{0}{1},x\pMX{p^n}{b_2}{0}{1},x\tau_{p^n})\rmd^\tau x\\
=&\frac{p^{2n}(1-p^{-1})\om_{F,p}^{1/2}(p^n)\abs{p^n}^{-r_\wt}}{\al_p(F)^{n}\vol(\wh R_{Np^n}^\x)\Bd_F^{\ulk/2}}\int_{\A^\x D^\x\bksl D^\x_\A}\phi^\Dstar_F(x\pMX{p^n}{1}{0}{1},x\pMX{p^n}{0}{0}{1},x\tau_{p^n})\rmd^\tau x.
\end{align*}
Since $\vol(\wh R^\x_{N})=\vol(\wh R^\x_{Np^n})(1+p^{-1})p^n$, the proposition can be deduced from the last equation directly by making change of variable.
\end{proof}

\subsection{Ichino's formula}
We now apply Ichino's formula to relate the global trilinear period $I(\rho(\breve\bft_n)\phi^\Dstar_F)$ to a product of central $L$-values of triple $L$-functions, the local zeta integrals $I_\pmq(\phi^\star_\pmq\ot\wtd\phi^\star_\pmq)$ defined in \eqref{E:localzeta} at primes $\pmq\not =p$ and the following local zeta integral at $p$\beq\label{E:balpzeta}I_p^\ord(\phi_p\ot\wtd\phi_p,\breve\bft_n):=\frac{L(1,\itPi_p,\Ad)}{\zeta_p(2)^2L(1/2,\itPi_p)}\int_{\PGL_2(\Qp)}\frac{\bfb_p(\itPi_p(g_p\breve\bft_n)\phi_p\ot \wtd\itPi_p(\breve\bft_n)\wtd\phi_p)}{\bfb_p(\itPi_p(\bftn)\phi_p,\wtd\phi_p)}\rmd g_p.\eeq
Here we recall that $\phi_p$ is any non-zero vector in the ordinary line $\cV^\ord_{\pi_{1,p}}(\chi_{1,p})\ot\cV^\ord_{\pi_{2,p}}(\chi_{2,p})\ot \cV^\ord_{\pi_{3,p}}(\chi_{3,p})$ with characters $\chi_{i,p}$ defined in \eqref{E:ordchar} and $\bftn=\pMX{0}{p^{-n}}{-p^n}{0})\in D_p^\x\hookto D_E^\x(\Qp)$ for any integer $n\geq n_\ulQ$. For each positive integer $M$, we shall use the notation $\wh M\in\wh\Q^\x$ to denote the idele with $\wh M_\ell=\ell^{\val_\ell(M)}$ at each finite prime $\ell$. 

\begin{prop}\label{P:Ichino2}We have
\begin{align*}\frac{I(\rho(\breve\bft_n)\phi^\Dstar_F)^2}{\pair{F^D}{F^D}}=&2^{\#\Sigma^-+1}\vol(\wh\cO_D^\x)^2\cdot L(1/2,\itPi)\cdot\frac{\om_{F}^{-1/2}(\wh N_1^+)\cdot\om_{F,p}^{-1}(p^n)\al_p(F)^{2n}}{L(1,\itPi,\Ad)\prod_{i=1}^3[\SL_2(\Z):\Gamma_0(N_ip^n)] (N_i^+p^n)^{\wt_i/2}}\\
&\times
I_p^\ord(\phi_p\ot\wtd\phi_p,\breve\bft_n)\prod_{\pmq\in\Sigma^-}\frac{\zeta_\pmq(1)^3}{\zeta_\pmq(2)^3}\cdot\prod_{\pmq\divides N^+} I_\pmq(\phi^\star_\pmq\ot\wtd\phi^\star_\pmq),\end{align*}
\end{prop}
where
\begin{align*}\pair{F^D}{F^D}
=&\pair{\bfU_p^{-n}f^D}{f^D}_{\condf p^n}\pair{\bfU_p^{-n}g^D}{g^D}_{\condg p^n}\pair{\bfU_p^{-n}h^D}{h^D}_{\condh p^n}.
%=&\bfB_{\condf}(\bdsf^D,\bdsf^D)(\Qx)\cdot \bfB_{\condg}(\bdsg^D,\bdsg^D)(\Qy)\cdot\bfB_{\condh}(\bdsh^D,\bdsh^D)(\Qz)
\end{align*}
\begin{proof} We begin with the explanation of the representation theoretic factorization for the automorphic form $\phi^\Dstar_F$. Let $(\pi^D_f,\pi^D_g,\pi^D_h)$ be the image of $(\pi_f,\pi_g,\pi_h)$ under the Jacquet-Langlands correspondence and let
\[\pi^D_1=\pi^D_f\ot\om_F^{-1/2},\quad \pi^D_2=\pi^D_g\text{ and }\pi^D_3=\pi^D_h.\]
Let $\itPi^D=\pi^D_1\boxtimes \pi^D_2\boxtimes\pi^D_3$ be the Jacquet-Langlands transfer of $\itPi$ and let $\cA(\itPi^D)$ be the unique automorphic realization of $\itPi^D$. With the isomorphism $\Psi:\wh D^{(N^-)\x}\iso\GL_2(\wh\Q^{(N^-)})$, we have a factorization
\beq\label{E:facD1}\cA(\itPi^D)\iso \bigot_{v\in\Sigma_D}\cV_{\itPi^D_v}\bigot_{v\not\in\Sigma_D}\cV_{\itPi_v}.\eeq
Here $(\itPi^D_\infty,\cV_{\itPi^D_\infty})=(\rho^{\rm u}_{\ulk},L_{\ulk}(\C))$ and for finite prime $\ell\divides N^-$, $(\it\Pi^D_\ell,\cV_{\itPi^D_\ell})=(\mu_{E_\ell}\circ\nu,\C\, e_{\mu_{E_\ell}})$ is the one dimensional representation given by a unramified character $\mu_{E_\ell}=(\mu_{1,\ell},\mu_{2,\ell},\mu_{3,\ell}):E_\ell^\x\to\C^\x$ with a basis $e_{\mu_{E_\ell}}$. Consider $\vec\phi^D_F=\varphi^D_1\boxtimes\varphi^D_2\boxtimes\varphi^D_3\in\cA(\itPi^D)\ot L_{\ulk}(\C)$. Let $X^{\ulk}:=X_1^{\wt_1}X_2^{\wt_2}X_3^{\wt_3}\in L_{\ulk}(\C)$ and define $\phi^D_{X^{\ulk}}\in\cA(\itPi^D)$ by \[\phi^D_{X^{\ulk}}(x):=\pair{\vec\phi^D(x)}{X^{\ulk}}_{\ulk}\quad (x\in D_E^\x(\A)).\] 
 Under the isomorphism \eqref{E:facD1}, we have the factorization $\phi^D_{X^{\ulk}}=\ot_v\phi^D_v$, where
\begin{align*}\phi^D_{\infty}&=X_1^{\wt_1}X_2^{\wt_2}X_3^{\wt_3},\quad \phi^D_\ell=e_{\mu_{E_\ell}}\text{ for }\ell\divides N^-,\\
\phi^D_\ell&=\varphi_{1,\ell}\ot\varphi_{2,\ell}\ot\varphi_{3,\ell}\text{ for }\ell\not\in\Sigma_D\end{align*}as in \subsecref{SSS:localfac}. Recall that $\varphi_{i,\ell}\in\cV^{\rm new}_{\pi_{i,v}}$ for $\ell\not=p$ is a new vector and $\varphi_{i,p}\in \cV^\ord_{\pi_{i,p}}(\chi_{i,p})$ is an ordinary vector. In view of the definition of $\phi^\Dstar_F$ in \eqref{E:defphiD}, we obtain the factorization
$\phi^\Dstar_F=\ot_v\phi^\Dstar_v$, where
\beq
\phi^\Dstar_v=\begin{cases}
\bfP_{\ulk}\in L_{\ulk}(\C)&\text{ if }v=\infty,\\[1em]
e_{\mu_{E_\ell}}&\text{ if }v=\ell\in\Sigma^-,\\[1em]
\varphi_{1,p}\ot\varphi_{2,p}\ot\varphi_{3,p}(=\phi_p)&\text{ if }v=p,\\[1em]
\cQ_{1,\ell}(\LR_\ell)\varphi_{1,\ell}\ot\cQ_{2,\ell}(\LR_\ell)\varphi_{2,\ell}\ot\cQ_{3,\ell}(\LR_\ell)\varphi_{3,\ell})(=\phi^\star_\ell)&\text{ if }v=\ell\not\in\stt{p}\cup\Sigma^-.
\end{cases}
\eeq
Now consider the contragredient representation $\wtd\itPi^D$. Let $\wtd\varphi^D_i=\varphi^D_i\ot\om_i^{-1}$  and $\wtd\varphi^\Dstar_i=\varphi^\Dstar_i\ot\om_i^{-1}$ for $i=1,2,3$. Let $Y^{\ulk}=Y^{\wt_1}Y^{\wt_2}Y^{\wt_3}\in L_{\ulk}(\C)$. Define $\wtd\phi^D_{Y^{\ulk}}$ and $\wtd\phi^\Dstar_F\in \cA(\wtd\itPi^D)$ by \[\wtd\phi^D_{Y^{\ulk}}(x):=\pair{\wtd\varphi^D_1\boxtimes\wtd\varphi^D_2\boxtimes\wtd\varphi^D_3(x)}{Y^{\ulk}}_{\ulk};\quad \wtd\phi^\Dstar_F(x)=\pair{\varphi^\Dstar_1\boxtimes\wtd\varphi^\Dstar_2\boxtimes\wtd\varphi^\Dstar_3(x)}{\bfP_{\ulk}}_{\ulk}\]
for $x\in D_E^\x(\A)$. Fixing an isomorphism
\[\cA(\wtd\itPi^D)\iso\bigot_{v\in\Sigma_D}\cV_{\wtd\itPi^D_v}\bigot_{v\not\in\Sigma_D}\cV_{\wtd\itPi_v},\]
 we then have a similar description for the factorizations $\wtd\phi^D_{Y^{\ulk}}=\ot_v\wtd\phi^D_v$ and $\wtd\phi^\Dstar_F=\ot_v\wtd\phi^\Dstar_v$ likewise.

For $v\not\in\stt{\infty}\cup\Sigma^-$, let $\bfb_v:\cV_{\itPi_v}\times \cV_{\wtd\itPi_v}\to\C$ be a non-degenerate $\GL_2(E_v)$-equivariant pairing such that $\bfb_v(\wtd\phi^D_v,\phi^D_v)=1$ for all but finitely many $v$. For $v\in\stt{\infty}\cup\Sigma^-$, let $\bfb^D_v:\cV_{\itPi^D_v}\times \cV_{\wtd\itPi^D_v}\to\C$ be a $D_E^\x(\Q_v)$-equivariant pairing and define 
\[I_v(\phi^D_v\ot\wtd\phi^D_v)=\frac{L(1,\itPi_v,\Ad)}{\zeta_v(2)^2L(1/2,\itPi_v)}\int_{D_v^\x/\Q_v^\x}\frac{\bfb^D_v(\itPi^D_v(x_v)\phi^\Dstar_v\ot \wtd\phi^\Dstar_v)}{\bfb^D_v(\phi^D_v,\wtd\phi^D_v)}\rmd x_v.\]
Here $\rmd x_v$ is the Haar measure with $\vol(\cO_{D_v}^\x/\Z_v^\x,\rmd x_v)=1$. In the notation of \cite[page 282]{Ichino08Duke}, we have 
\begin{align*}I(\rho(\breve\bft_n)\phi^\Dstar_F)^2=&I(\rho(\breve\bft_n)\phi^\Dstar_F)\cdot I(\rho(\breve\bft_n)\wtd\phi^\Dstar_F).\end{align*}
Therefore, according to \cite[Theorem 1.1,\,Remark 1.3]{Ichino08Duke}, we obtain 
\[\begin{aligned}\frac{I(\rho(\breve\bft_n)\phi^\Dstar_F)^2}{\pair{\rho(\Tau_{\ul{N}^+}^D\bftn)\phi^D_{X^{\wt}}}{\wtd\phi^D_{Y^{\wt}}}}
=&\frac{\vol(\wh\cO^\x_D)}{8}\cdot\frac{\zeta_\Q(2)^2L(1/2,\itPi)}{L(1,\itPi,\Ad)}\\
&\times I_p^\ord(\phi_p\ot\wtd\phi_p,\breve\bft_n)\prod_{v\in\stt{\infty}\cup\Sigma^-}I_v(\phi^D_v\ot\wtd\phi^D_v) \prod_{\pmq\not\in\stt{p}\cup\Sigma^-} I_\pmq(\phi^\star_\pmq\ot\wtd\phi^\star_\pmq),
%=&C'\cdot \frac{\vol(\wh \cO_D^\x)}{8}\cdot\frac{\zeta_\Q(2)^2L(1/2,\itPi)}{L(1,\itPi,\Ad)}\cdot I_p(\itPi_p(\bftn')\phi_p\ot\wtd\itPi_p(\bftn')\wtd\phi_p)\cdot \prod_{\pmq\ndivides N^-p} I_\pmq(\phi^\star_\pmq\ot \wtd\phi^\star_\pmq)\cdot \pair{\bfP_{\ulk}}{\bfP_{\ulk}}_{\ulk}.
\end{aligned}\]
From \eqref{E:Schur} and \eqref{E:vN}, we find that $\pair{\rho(\Tau_{\ul{N}^+}^D\bftn)\phi^D_{X^{\wt}}}{\wtd\phi^D_{Y^{\wt}}}$ equals
\beq\label{E:1.bal}\begin{aligned}&\om_F^{-1/2}(\wh N^+_1)\om^{-1}_{F,p}(p^n)\cdot\pair{F^D}{F^D}\cdot\al_p(F)^{2n}\prod_{i=1}^3\frac{\vol(\wh R^\x_{N_ip^{2n}})}{(N^+_ip^{2n})^{\wt_i/2}(\wt_i+1)}\\
=&\om_F^{-1/2}(\wh N^+_1)\pair{F^D}{F^D}\cdot\frac{48^3\cdot\om^{-1}_{F,p}(p^n)\al_p(F)^{2n}}{\prod_{i=1}^3(N^+_ip^{2n})^{\wt_i/2}[\SL_2(\Z):\Gamma_0(N_i)]}\prod_{i=1}^3\frac{1}{(\wt_i+1)}\prod_{\pmq\in\Sigma^-}\frac{\zeta_\pmq(1)^6}{\zeta_\pmq(2)^3}
\end{aligned}\eeq

We now proceed to compute the local zeta integrals $I_v(\phi^D_v\ot\wtd\phi^D_v)$ for $v\in\stt{\infty}\cup \Sigma^-$. Recall that the archimedean $L$-factors are given by
\begin{align*}L(s,\itPi_\infty,\Ad)&=\pi^{-3}\Gamma_\C(s+\wt_1+1)\Gamma_\C(s+\wt_2+1)\Gamma_\C(s+\wt_3+1);\\\quad L(s,\itPi_\infty)&=\Gamma_\C(s+\frac{\wt_1+\wt_2+\wt_3+3}{2})
\Gamma_\C(s+\wt_1^*+\onehalf)\Gamma_\C(s+\wt_2^*+\onehalf)\Gamma_\C(s+\wt_3^*+\onehalf),\end{align*}
so we have
\begin{align*}
%I_\pmq(\phi^D_\pmq\ot\wtd\phi^D_\pmq)=&2\cdot\zeta_\pmq(1)^{-2};\\
I_\infty(\phi^D_\infty\ot\wtd\phi^D_\infty)&=\frac{L(1,\itPi_\infty,\Ad)}{\zeta_\infty(2)^2L(1/2,\itPi_\infty)}\int_{\Qtn(\R)/\R^\x}\frac{\pair{\rho^u_{\ulk}(x_\infty)\bfP_{\ulk}}{\bfP_{\ulk}}_{\ulk}}{\prod_{i=1}^3\pair{X^{\wt_i}}{Y^{\wt_i}}_{\wt_i}}\rmd x_\infty\\
%=&\frac{L(1,\itPi_\infty,\Ad)}{\zeta_\infty(2)^2L(1/2,\itPi_\infty)}\cdot\pair{\bfP_{\ulk}}{\bfP_{\ulk}}_{\ulk}\\
&=\frac{\Gamma(\wt_1+2)\Gamma(\wt_2+2)\Gamma(\wt_3+2)}{4\pi^2\Gamma(\frac{\wt_1+\wt_2+\wt_3}{2}+2)\Gamma(\wt^*_1+1)\Gamma(\wt^*_2+1)\Gamma(\wt^*_3+1)}\cdot\pair{\bfP_{\ulk}}{\bfP_{\ulk}}_{\ulk}\\
&=(4\pi^2)^{-1}(1+\wt_1)(1+\wt_2)(1+\wt_3).
\end{align*}
The last equality follows from \lmref{L:infty.bal} below. Now let $\pmq$ be a prime in $\Sigma^-$. According to \cite{Prasad90}, $\pi_{i,\pmq}=\mu_{i}{\rm St}$ for $i=1,2,3$ are unramified special representations with $\mu_{1}\mu_{2}\mu_{3}(\pmq)=1$. Since \[L(s,\itPi_\pmq,\Ad)=\zeta_\pmq(s+1)^3;\quad L(s,\itPi_\pmq)=\zeta_\pmq(s+1/2)^2\zeta_\pmq(s+3/2),\] we obtain \begin{align*}I_\pmq(\phi^D_\pmq\ot\wtd\phi^D_\pmq)=&\frac{L(1,\itPi_\pmq,\Ad)}{\zeta_\pmq(2)^2L(1/2,\itPi_\pmq)}\cdot (1+\mu_{1,\pmq}\mu_{2,\pmq}\mu_{3,\pmq}(\pmq))
=2\zeta_\pmq(1)^{-2}.
\end{align*}
Substituting \eqref{E:1.bal} and the above computation of $I_\pmq(\phi_\pmq^D\ot\wtd\phi_\pmq^D)$ into Ichino's formula, we obtain
\begin{align*}
&\frac{I(\rho(\breve\bft_n)\phi^\Dstar_F)^2}{\pair{F^D}{F^D}\om_F^{-1/2}(\wh N^+_1)}\\
=&\vol(\wh \cO_D^\x)^2\cdot \frac{N^-}{48}\cdot\frac{\zeta_\Q(2)^2\cdot 48^3}{8\cdot 4\pi^2}\cdot\frac{L(1/2,\itPi)}{L(1,\itPi,\Ad)}\cdot\frac{\om_{F,p}^{-1}(p^n)\al_p(F)^{2n}}{\prod_{i=1}^3[\SL_2(\Z):\Gamma_0(N_ip^n)](N_i^+p^n)^{\wt_i/2}}\prod_{\pmq\in\Sigma^-}\frac{2\zeta_\pmq(1)^3}{\zeta_\pmq(2)^3} \\
&\times I_p^\ord(\phi_p\ot\wtd\phi_p,\breve\bft_n)\prod_{\pmq\not\in\stt{p,\infty}\cup\Sigma^-}I_\pmq(\phi^\star_\pmq\ot\wtd\phi^\star_\pmq),
\end{align*}
and the proposition follows.
\end{proof}
\begin{lm}\label{L:infty.bal}We have \[\pair{\bfP_{\ulk}}{\bfP_{\ulk}}_{\ulk}=\frac{\Gamma(\frac{\wt_1+\wt_2+\wt_3}{2}+2)\Gamma(\wt_1^*+1)\Gamma(\wt_2^*+1)\Gamma(\wt_3^*+1)}{\Gamma(\wt_1+1)\Gamma(\wt_2+1)\Gamma(\wt_3+1)}.\]
\end{lm}
\begin{proof}
Let $v_1=X_1^{\wt_1}\ot Y_2^{\wt_2}\ot X_3^{\wt_1^*}Y_3^{\wt_2^*}$ and $v_2=Y_1^{\wt_1}\ot X_2^{\wt_2}\ot X_3^{\wt_2^*}Y_3^{\wt_1^*}$. Let $\rmd u$ be the Haar measure on $\SU(2)(\R)$ with the volume $\vol(\SU(2)(\R),\rmd u)=1$. More precisely, $\rmd u$ is given by 
\begin{align*}\int_{\SU(2)(\R)} \Phi(u)\rmd u=&\frac{1}{4\pi^2}\int_0^{2\pi}\int_0^{2\pi}\int_0^{\pi/2}\Phi(u)\sin 2\theta \,\rmd \theta\,\rmd \varphi\,\rmd\varrho,\\
(u=\pMX{\al}{\beta}{-\ol{\beta}}{\ol{\al}},\,&\al=\cos\theta e^{i\varphi},\,\beta=\sin\theta e^{i\varrho})
 \end{align*}
for $\Phi\in L^1(\SU(2)(\R))$. Write $\pairing=\pairing_{\ulk}$ for simplicity. Since $L_{\ulk}(\C)^{\SU(2)(\R)}=\C\cdot\bfP_{\ulk}$, we see that
\beq\label{E:7.bal}\int_{\SU(2)(\R)}\pair{\rho_{\ulk}(u)v_1}{v_2}\rmd u\cdot\pair{\bfP_{\ulk}}{\bfP_{\ulk}}=\pair{v_1}{\bfP_{\ulk}}\cdot\pair{\bfP_{\ulk}}{v_2}.\eeq
By definition, 
\begin{align*}\bfP_{\ulk}=&\sum_{n_1=0}^{\wt_1^*}\sum_{n_2=0}^{\wt_2^*}\sum_{n_3=0}^{\wt_3^*}{\wt_1^*\choose n_1}{\wt_2^*\choose n_2}{\wt_3^*\choose n_3}(-1)^{\wt_1^*+\wt_2^*+\wt_3^*-n_1-n_2-n_3}\\
&\quad\quad\times X_1^{\wt_2^*-n_2+n_3}Y_1^{\wt_3^*+n_2-n_3}\ot X_2^{\wt_3^*+n_1-n_3}Y_2^{\wt_1^*-n_1+n_3}\ot X_3^{\wt_1^*-n_1+n_2}Y_3^{\wt_2^*+n_1-n_2}.\end{align*}
Then 
\[\pair{v_1}{\bfP_{\ulk}}=(-1)^{\wt_1+\wt_2}{\wt_3\choose \wt_2^*}^{-1},\,\pair{\bfP_{\ulk}}{v_2}=(-1)^{\wt_1+\wt_1^*+\wt_3}{\wt_3\choose \wt_1^*}^{-1}.\]
Let $r=\frac{\wt_1+\wt_2+\wt_3}{2}=\wt^*_1+\wt^*_2+\wt^*_3$. A  direct computation shows that
\begin{align*}
\int_{\SU(2)(\R)}\pair{\rho_{\ulk}(u)v_1}{v_2}\rmd u
&=(-1)^{r}\binom{\wt_3}{\wt_1^*}^{-1}\sum_{j=0}^{\wt_1^*}{\wt^*_1\choose j}{\wt_2^*\choose j}(-1)^j\int_{\SU(2)(\R)}\abs{\al\ol{\al}}^{r-j}\abs{\beta\ol{\beta}}^j\rmd u
\\
&=2(-1)^r(2r+2)^{-1}\binom{\wt_3}{\wt_1^*}^{-1}\sum_{j=0}^{\wt_1^*}(-1)^j
\frac{\binom{\wt_1^*}{j}\binom{\wt_2^*}{j}}{\binom{r}{j}}\\
&=(-1)^{r}(r+1)^{-1}\binom{\wt_3}{\wt_1^*}^{-1}\frac{\Gamma(\wt_1+1)\Gamma(k^*_1+1)}{\Gamma(r+1)}\sum_{j=0}^{k^*_1}(-1)^j\binom{k^*_2}{j}\binom{r-j}{\wt_1}\\
&=(-1)^{r}(r+1)^{-1}\binom{\wt_3}{\wt_1^*}^{-1}\frac{\Gamma(\wt_1+1)\Gamma(k^*_1+1)}{\Gamma(r+1)}\cdot\binom{r-k^*_2}{\wt_1-\wt_2^*}.
\end{align*}
Substituting the above to \eqref{E:7.bal}, we obtain
\begin{align*}
\pair{\bfP_{\ulk}}{\bfP_{\ulk}}_{\ulk}=&\frac{\Gamma(r+2)}{\Gamma(\wt_1+1)\Gamma(k^*_1+1)}\cdot\frac{\Gamma(k^*_1+1)\Gamma(k^*_2+1)}{\Gamma(\wt_3+1)}\cdot \frac{\Gamma(k^*_3+1)\Gamma(k^*_1+1)}{\Gamma(\wt_2+1)},
\end{align*}
and the lemma follows.
%We use the following combinatorial identity: for $a\geq n$, 
%\[\sum_{j=0}^n(-1)^j\binom{a}{j}\binom{a+b+n-j}{a+b}=\binom{b+n}{b}.\]
%Now let $a=\wt_2^*$, $b=\wt_3^*$ and $n=\wt_1^*$. Then 
%\begin{align*}\sum_{j=0}^{n}(-1)^j
%\frac{\binom{n}{j}\binom{a}{j}}{\binom{a+b+n}{j}}=&\frac{n!(a+b)!}{(a+b+n)!}\sum_{j=0}^n (-1)^j\binom{a}{j}\binom{a+b+n-j}{a+b}\\
%=&\frac{n!(a+b)!}{(a+b+n)!}\binom{b+n}{b}.\end{align*}
\end{proof}

\begin{defn}[The Gross periods of Hida families]\label{D:period2}Suppose that $\cF$ is a primitive $\bfI$-adic Hida family which satisfies (CR, $\Sigma^-$).  Let $\cF^D$ be a primitive Jacquet-Langlands lift of $\cF$ with the tame conductor $N_\cF=N^-N_\cF^+$. Put
\[\eta_{\cF^D}:=\bfB_{N_\cF}(\cF^D,\cF^D)\in\bfI,\]
where $\bfB_{N_\cF}$ is the Hecke-equivariant perfect pairing defined in \defref{D:pairing.bal}. For each arithmetic point $Q\in\frakX_\bfI^+$, writing $\eta_{\cF_Q^D}$ for the specialization of $\eta_\cF$ at $Q$, define the Gross' period $\Omega_{\cF_Q^D}$ of $\cF_Q$ by\[\Omega_{\cF_Q^D}=(-2\sqrt{-1})^{k_Q+1}\norm{\cF_Q^\circ}^2_{\Gamma_0(N_{\cF_Q^\circ})}\cdot\frac{\cE_p(\cF_Q,\Ad)}{\eta_{\cF^D_Q}}\cdot \varepsilon^{\Sigma^-}(\cF_Q), \]where $\cE_p(\cF_Q,\Ad)$ is the modified $p$-Euler factor in \eqref{E:EulerAd} and  \[\varepsilon^{\Sigma^-}(\cF_Q):=\prod_{\ell\divides N_{\cF}^+}\varepsilon(1/2,\pi_{\cF_Q,\ell})\abs{N_{\cF}^+}_\ell^\frac{2-k_Q}{2}\in\Zbar_{(p)}^\x.\]
is the prime-to-$\Sigma^-$ part of the root number of $\cF_Q$.\footnote{Here $\Omega_{\cF_Q^D}$ is called the Gross period for $\cF_Q$ as it first appeared in the Gross' special value formula for modular forms over imaginary quadratic fields.} 
We will see from \remref{R:periodc} that the canonical period is an integral multiple of the Gross period in the sense that there exists a non-zero $u\in\bfI$ such that $\Omega_{\cF^D_Q}=u(Q)\cdot \Omega_{\cF_Q}$ for each arithmetic point $Q$.
\end{defn}

\begin{cor}\label{C:Ichino.bal}For each $\ulQ=(\Qx,\Qy,\Qz)\in\frakX^\bal_\cR$ in the balanced range, we have the interpolation formula
\[\left(\Theta_{\bdsF^\Dstar}(\ulQ)\right)^2=2^{\#(\Sigma^-)+4}N\cdot \frac{L(1/2,\itPi_\ulQ)}{(\sqrt{-1})^{k_\Qx+k_\Qy+k_\Qz-1}\Omega_{\bdsf_\Qx^D}\Omega_{\bdsg_\Qy^D}\Omega_{\bdsh_\Qz^D}}\cdot \sI^\bal_{\itPi_{\ulQ,p}}\cdot\prod_{\pmq\divides N^+}\sI^\star_{\itPi_{\ulQ,\pmq}} ,\]
where $\sI^\bal_{\itPi_{\ulQ,p}}$ is the normalized $p$-adic zeta integrals given by \beq\label{E:Nbal}
\sI^\bal_{\itPi_{\ulQ,p}}=I_p^\ord(\phi_p,\breve\bft_n)\cdot B^{[n]}_{\itPi^\ord_p}\cdot \frac{\om^{1/2}_{F,p}(-p^{2n})\abs{p}_p^{-n(k_1+k_2+k_3)}}{\al_p(F)^{2n}\zeta_p(2)^2}
%\\ \sI^\star_\pmq=&I_\pmq(\phi^\star_\pmq\ot\wtd\phi^\star_\pmq)\cdot B_{\itPi_\pmq}\cdot\frac{\zeta_\pmq(1)^2}{\abs{N}^2_\pmq\zeta_\pmq(2)^2}\cdot\abs{\Bd_F^{\ulk}}_\pmq\om_{F,\pmq}(\Bd_f),
\eeq
with $B_{\itPi_p^\ord}^{[n]}$ defined in \eqref{E:lNorm2}, and $\sI^\star_{\itPi_{\ulQ,\pmq}}$ are the local zeta integrals at $\pmq$ defined in \eqref{E:Nq}. 
\end{cor}
\begin{proof}To simplify our notation, we let $f_1=f$, $f_2=g$ and $f_3=h$. For a finite prime $\pmq$, we put $B_{\itPi_{F,\pmq}}=\prod_{i=1}^3B_{\pi_{f_i,\pmq}}$. By definition, we have $B_{\itPi_{F,\pmq}}=\om_{F,\pmq}^{1/2}(N^+_f)B_{\itPi_\pmq}$ if $\pmq\not =p$ and $B_{\itPi_{F,\pmq}}=1$ if $\pmq\ndivides pN$. At the place $p$, from \lmref{L:ordlocalnorm} and the definition of $\cE_p(f_i,\Ad)$ in \eqref{E:EulerAd}, we see that
\begin{align*}\frac{B^{[n]}_{\itPi_p^\ord}}{B_{\itPi_{F,p}}}%=&\prod_{i=1}^3\frac{\pair{\rho(\bftn)W^\ord_{\pi_{f_i,p}}}{\wtd W_{\pi_{f_i,p}}^\ord\ot\om_{f_i,p}^{-1}}}{\pair{W_{\pi_{f_i,p}}}{W_{\pi_{f_i,p}}\ot\om_{f_i,p}^{-1}}}\\
&=
\om_{F,p}^{1/2}(-p^{-2n})\prod_{i=1}^3\frac{\al_{f_i,p}\Abs_p^\onehalf(p^{2n})}{\varepsilon(1/2,\pi_{f_i,p})}\cdot\frac{ [\SL_2(\Z):\Gamma_0(p^{c_i})]}{(1+p^{-1})}\cdot\cE_p(f_i,\Ad).\end{align*}
Let $f_i^\circ$ be the associated newform of $f_i$ and $c_i=c(\pi_{f_i,p})$. Write $\norm{f_i^\circ}^2$ for the Petersson norm $\norm{f_i^\circ}^2_{\Gamma_0(N_{f_i^\circ})}$. From the above equation and the Petersson norm formula \eqref{E:normformula}, we find that
\begin{align*}&\frac{\om_F^{-1/2}(\wh N^+_1)\om_{F,p}(p^{-n})\al_p(F)^{2n}}{L(1,\itPi,\Ad)\prod_{i=1}^3[\SL_2(\Z):\Gamma_0(N_ip^{2n})](N^+_ip^{2n})^{\wt_i/2}}\\
=&\om_F^{-1/2}(\wh N^+_1)\om_{F,p}(p^{-n})\al_p(F)^{2n}
\prod_{\pmq\divides Np}B_{\itPi_{F,\pmq}}\prod_{i=1}^3\cdot\frac{[\SL_2(\Z):\Gamma_0(p^{c_i})]p^{-2n}}{2^{k_i}w(f^\circ_i)\norm{f^\circ_i}^2(1+p^{-1})(N^+_ip^{2n})^{\wt_i/2}},\\
=&\om_F^{1/2}(\wh N^-)\om_{F,p}^{1/2}(-1)\al_p(F)^{2n}\cdot B^{[n]}_{\itPi^\ord_p}\cdot\prod_{\pmq\divides N}B_{\itPi_{\pmq}}\prod_{i=1}^3\frac{\varepsilon(1/2,\pi_{f_i,p})}{w(f^\circ_i)(N_i^+)^{\wt_i/2}} \cdot\frac{1}{\al_{f_i,p}\Abs_p^{\frac{1-k_i}{2}}(p^{2n})2^{k_i}\norm{f_i^\circ}^2\cE_p(f_i,\Ad)}\\
=&\om_{F,p}^{1/2}(-1)B^{[n]}_{\itPi^\ord_p}\cdot\prod_{\pmq\divides N}B_{\itPi_{\pmq}}\prod_{i=1}^3\frac{1}{\varepsilon^{\Sigma^-}(f_i)2^{k_i}\norm{f_i^\circ}^2\cE_p(f_i,\Ad)}\prod_{\pmq\in\Sigma^-}\frac{\om_{F,\pmq}^{1/2}(\pmq)}{\varepsilon(1/2,\pi_{f_1,\pmq})\varepsilon(1/2,\pi_{f_2,\pmq})\varepsilon(1/2,\pi_{f_3,\pmq})}\\
=& (-1)^{\#(\Sigma^-)}\cdot\om_{F,p}^{1/2}(-1)\cdot B^{[n]}_{\itPi^\ord_p}\cdot\prod_{\pmq\divides N}B_{\itPi_{\pmq}}\prod_{i=1}^3\frac{\pair{\bfU_p^{-n}f_i^D}{f_i^D}_{N_ip^n}}{2^{-1}(\sqrt{-1})^{k_i+1}\Omega_{f_i^D}}.
\end{align*}
In the last equality, we have used \lmref{L:pairing.bal} and the fact that for $\pmq\in\Sigma^-$, 
\[\varepsilon(1/2,\itPi_\pmq)=\om_{F,\pmq}^{-1/2}(\pmq)\varepsilon(1/2,\pi_{f_1,\pmq})\varepsilon(1/2,\pi_{f_2,\pmq})\varepsilon(1/2,\pi_{f_3,\pmq})=-1.\]
Substituting the above equation and the definition of $\sI^*_{\itPi_\pmq}$ in \eqref{E:Nq} to \propref{P:Ichino2}, we deduce from \propref{P:formulaTheta} that
\begin{align*}
\left(\Theta_{\bdsF^\Dstar}(\ulQ)\right)^2
&=\frac{\vol(\wh\cO^\x_D)^2}{\vol(\wh R_N^\x)^2}\cdot \frac{(-2)^{\#\Sigma^-}2^4N^-}{(\sqrt{-1})^{k_1+k_2+k_3+3}}\cdot\frac{L(1/2,\itPi)}{\Omega_{f^D}\Omega_{g^D}\Omega_{h^D}}\cdot\sI^\bal_{\itPi_p}\prod_{\pmq\in\Sigma^-}B_{\itPi_\pmq}\cdot\frac{\zeta_\pmq(1)^3}{\zeta_\pmq(2)^3}\prod_{\pmq\divides N^+} \sI^\star_{\itPi_\pmq}\cdot\frac{\abs{N}_\pmq^2\zeta_\pmq(2)^2}{\zeta_\pmq(1)^2}.
\end{align*}
Therefore, we obtain the corollary by noting that 
\[\frac{\vol(\cO_D^\x)}{\vol(\wh R_N^\x)}=\prod_{q\divides N^+}\frac{\zeta_\pmq(1)}{\abs{N}_\pmq\zeta_\pmq(2)},\]
and that for $\pmq\in\Sigma^-$, 
\begin{align*}
B_{\itPi_\pmq}=&\prod_{i=1}^3\frac{\zeta_\pmq(2)\pair{\rho(\Tau_{\pmq})W_{\pi_i}}{\wtd W_{\pi_i}}}{\zeta_\pmq(1)L(1,\pi_i,\Ad)}=\prod_{i=1}^3\varepsilon(1/2,\pi_{i,\pmq})\frac{\zeta_\pmq(2)}{\zeta_\pmq(1)}
=(-1)\frac{\zeta_\pmq(2)^3}{\zeta_\pmq(1)^3}.
\end{align*}
This finishes the proof.
\end{proof}

%!TEX root = TRIPLE3.tex
\def\Mu{\chi}
\def\pmq{q}
\def\sB{\cB}
\section{The calculation of local zeta integrals (I)}\label{S:local1}
\subsection{Notation and conventions}\label{S:notation.5}
Let $q$ be a finite prime. Let $G=\GL_2(\Qq)$ and $Z=\Qq^\x$ be the center of $G$. Denote by $B$ the group of the upper triangular matrices of $G$ and by $N$ the unipotent radical of $B$. Let $\pi$ be an irreducible unitary generic admissible representation of $G$. Define a real number $\lam(\pi)$ by 
\[\lam(\pi)=\begin{cases}\abs{\lam}&\text{ if $\pi=\Prin{\chi_1\Abs^{\lam}}{\chi_2\Abs^{-\lam}}$ with $\chi_1,\chi_2$ unitary and $\lam\in\R$},\\
-\onehalf&\text{ if $\pi$ is a discrete series}.
\end{cases}\]
Recall that $\sW(\pi)=\sW(\pi,\addchar_{\Q_\pmq})$ is the Whittaker model of $\pi$ with respect to $\psi_{\Q_\pmq}$. It is well known that for any $W\in\cW(\pi)$ and $\ep>0$, there exists a $\Phi_\ep\in\cS(\Q_\pmq)$ with
\beq\label{E:estimate.loc}W(\pDII{y}{1})=\abs{y}^{\onehalf-\lam(\pi)-\ep}\Phi_\ep(y).\eeq

For characters $\chi,\upsilon:\Qq^\x\to\C^\x$, let $\sB(\Mu,\upsilon)$ denote the induced representation given by 
\[\sB(\Mu,\upsilon)=\stt{\text{ smooth functions }f: G\to\C\mid f(\pMX{a}{b}{0}{d}g)=\Mu(a)\upsilon(d)\abs{\frac{a}{d}}^\onehalf f(g)}.\]
Let $K=\GL_2(\Z_\pmq)$. We let $\pairing:\sB(\Mu,\upsilon)\times \sB(\Mu^{-1},\upsilon^{-1})\to\C$ be the $G$-invariant perfect pairing given by 
\[\pair{f}{f'}:=\int_Kf(k)f'(k)\rmd k,\]
where $\rmd k$ is the Haar measure with $\vol(K,\rmd k)=1$.  If $\chi\upsilon^{-1}\not =\Abs^{-1}$, then we let $\cB(\Mu,\upsilon)_0$ be the unique irreducible sub-representation of $\cB(\Mu,\upsilon)$ and let $\cB(\upsilon,\Mu)^0$ be the unique irreducible quotient of $\cB(\upsilon,\Mu)$. It is well known that 
$\cB(\Mu,\upsilon)_0=\cB(\Mu,\upsilon)$ and 
$\cB(\upsilon,\Mu)^0=\cB(\upsilon,\Mu)$ unless $\Mu\upsilon^{-1}=\Abs$. The above pairing $\pairing$ induces a $G$-invariant perfect paring $\pairing:\cB(\chi,\upsilon)_0\times \cB(\chi^{-1},\upsilon^{-1})^0\to\C$.
\subsubsection*{Intertwining operator}Define the normalized intertwining operator $M^*(\upsilon,\chi,s):\cB(\upsilon\Abs^s,\chi\Abs^{-s})\to \cB(\chi\Abs^{-s},\upsilon\Abs^s)$ by 
\[M^*(\upsilon,\chi,s)f:=\gamma(2s,\upsilon\chi^{-1})\int_{\Qp}f(\pMX{0}{1}{-1}{0}\pMX{1}{x}{0}{1}g)\rmd x\quad (g\in G).\]
Here $\gamma(s,-)$ is the $\gamma$-factor as in \eqref{E:gamma1}, and the integral in the right hand side is convergent absolutely for $\Re s$ sufficiently large and has analytic continuation to all $s\in\C$ (\cf\cite[Proposition 4.5.7]{Bump97Grey}). Let $\delta:G\to\R_+$ be the function given by $\delta(\pMX{a}{b}{0}{d}k)=\abs{ad^{-1}}$ for $k\in K$. If $\chi\upsilon^{-1}\not=\Abs^{-1}$, then $M^*(\upsilon,\chi,s)|_{s=0}$ factors through $\cB(\upsilon,\chi)^0$, and hence we have a well-defined map $M^*(\upsilon,\chi):\sB(\upsilon,\chi)^0\to \sB(\chi,\upsilon)_0$ given by 
\beq\label{E:intertwining.l}M^*(\upsilon,\chi)f:=M^*(\upsilon,\chi,s)(f\delta^s)|_{s=0}.\eeq

\subsubsection*{An integration formula}
The following integration formula will be used frequently in our computation. For $F\in L^1(ZN\bksl G)$, 
\beq\label{E:intformula}\begin{aligned}\int_{ZN\bksl G}F(g)\rmd g=&\int_K\int_{\Q_\pmq^\x}F(\pDII{y}{1}k)\frac{\rmd^\x y}{\abs{y}} \rmd k\\
=&\frac{\zeta_\pmq(2)}{\zeta_\pmq(1)}\int_{\Q_\pmq^\x}\int_{\Q_\pmq}F(\pDII{y}{1}\pMX{1}{0}{x}{1})\rmd x\frac{\rmd^\x y}{\abs{y}}\end{aligned}\eeq
 (\cf\cite[3.1.6, page 206]{MV10}).
\subsection{Local trilinear integrals and Rankin-Selberg integrals}

Let $\pi_1,\pi_2$ and $\pi_3$ be irreducible unitary generic admissible representation of $G$ with central characters $\om_1,\om_2$ and $\om_3$. Suppose that $\om_1\om_2\om_3=1$ and that $\pi_3$ is a constituent (an irreducible subquotient) of $\cB(\chi_3,\upsilon_3)$. Assume further that the following condition holds for $(\pi_1,\pi_2;\pi_3)$: \beqcd{Hb}\lam(\pi_1)+\lam(\pi_2)+\abs{\lam(\pi_3)}<1/2\text{ and }\abs{\lam(\pi_3)}\leq 1/2.\eeqcd 
 Put
\[\cJ=\pDII{-1}{1}\in\GL_2(\Q_\pmq).\]
For $(W_1,W_2,f_3)\in \sW(\pi_1)\times\sW(\pi_2)\times \sB(\Mu_3,\upsilon_3)$, define the local Rankin-Selberg integrals by 
\begin{align*}
\Psi(W_1,W_2,f_3)=&\int_{ZN\bksl G}W_1(g)W_2(\cJ g)f_3(g)\rmd g;\\
\wtd\Psi(\wtd W_1,\wtd W_2,\wtd f_3)=&\int_{ZN\bksl G}\wtd W_1(\cJ g)\wtd W_2(g)\wtd f_3(g)\rmd g.
\end{align*}
The above integrals converge absolutely under the assumption \eqref{Hb}.
For $\wtd W_1\in \sW(\Contra{\pi}_1)$, $\wtd W_2\in\sW(\Contra{\pi}_2)$ and $\wtd f_3\in \sB(\Mu_3^{-1},\upsilon_3^{-1})$,
define the local trilinear integral by 
\[\sJ_\pmq(W_1\ot W_2\ot f_3,\wtd W_1\ot\wtd W_2\ot\wtd f_3):=\int_{Z\bksl G}\pair{\rho(g)W_1}{\wtd W_1}\pair{\rho(g)W_2}{\wtd W_2}\pair{\rho(g)f_3}{\wtd f_3}\rmd g.\]
The following result is a generalization of \cite[Lemma 3.4.2]{MV10}. We provide a different and more elementary proof and replace the assumption on the temperedness with a much weaker hypothesis \eqref{Hb}. 
\begin{prop}\label{P:IRSintegral} With the assumption \eqref{Hb} for $(\pi_1,\pi_2;\pi_3)$, we have 
\[\sJ_\pmq(W_1\ot W_2\ot f_3,\wtd W_1\ot\wtd W_2\ot \wtd f_3)=\zeta_\pmq(1)\cdot \Psi(W_1,W_2,f_3)\cdot \wtd\Psi(\wtd W_1,\wtd W_2,\wtd f_3).\]
\end{prop}
\begin{proof}  Denote by $\Psi:ZN\to\C^\x$ the character $\om_2\boxtimes\psi_{\Qq}$. Let $\bbpair{}{}: L^2(ZN\bksl G,\Psi)\ot L^2(ZN\bksl G,\Psi^{-1})\to\C$ be the $G$-equivariant bilinear pairing given by
\[\bbpair{F}{F'}=\int_{ZN\bksl G}F(g)F'(g)\rmd g.\]
Let $\lam_1=\lam(\pi_1)$, $\lam_2=\lam(\pi_2)$ and $\lam_3=\abs{\lam(\pi_3)}$. By \eqref{Hb} and symmetry, we may assume $\lam_1+\lam_3<1/2$. Put
\begin{align*}
F_1(g)=&W_1(g)f_3(g),\,W_4(g)=\wtd W_2(g)\in L^2(ZN\bksl G,\Psi);\\
F_2(g)=&\wtd W_1(\cJ g)\wtd f_3(g),\,W_3(g)=W_2(\cJ g)\in L^2(ZN\bksl G,\Psi^{-1}).
\end{align*}
Then one verifies that 
\[\bbpair{\rho(g)F_1}{F_2}=\pair{\rho(g)W_1}{\wtd W_1}\pair{\rho(g)f_3}{\wtd f_3},\]
and hence, it is equivalent to showing that 
\beq\label{E:RSTriple}\sJ_\pmq(W_1\ot W_2\ot f_3,\wtd W_1\ot\wtd W_2\ot \wtd f_3)=\int_{ZN\bksl G}\bbpair{\rho(g)F_1}{F_2}\pair{\rho(g)W_3}{W_4}\rmd g=\zeta_\pmq(1)\bbpair{F_1}{W_3}\bbpair{F_2}{W_4}.\eeq
Put \[K_n=\bigcup_{i=0}^nK\pDII{\pmq^i}{1}K\subset G.\]
First we claim that if $y_1,y_2\in \Q_\pmq^\x$, then
\beq\label{E:511.L}\pMX{y_2^{-1}y_1}{y_2^{-1}x}{0}{1}\in ZK_n\iff \pmq^{-n}\leq \abs{y_2^{-1}y_1}\leq \pmq^n\text{ and }\abs{x^2}\leq \pmq^{n}\abs{y_1y_2}.\eeq
To see the claim, we note that if $\abs{x}\leq 1$ or $\abs{x}\leq\abs{y}$, then
\[\pMX{y}{x}{0}{1}\in K\pDII{y}{1}K,\]
and if $\abs{x}>1$ and $\abs{x}>\abs{y}$, then 
\[\pMX{y}{x}{0}{1}\in K\pDII{x^{-2}y}{1}K.\]
By the Cartan decomposition, we find $\pMX{y_2^{-1}y_1}{y_2^{-1}x}{0}{1}\in ZK_n$
if and only if $\abs{x}\leq \max\stt{\abs{y_1},\abs{y_2}}$ and $\pmq^{-n}\leq \abs{y_2^{-1}y_1}\leq \pmq^n$ or 
$\abs{x}>\max\stt{\abs{y_1},\abs{y_2}}$ and $\pmq^{-n}\leq \abs{x^{-2}y_1y_2}\leq 1$, and this proves the claim. 

Now we proceed to prove the equation \eqref{E:RSTriple}. Let $\bbI_{K_n}$ be the characteristic function of $ZK_n$ and set
\begin{align*}\cI_n=&\int_{Z\bksl G}\bbpair{\rho(g)F_1}{F_2}\pair{\rho(g)W_3}{W_4}\bbI_{K_{2n}}(g)\rmd g.
\end{align*}
By a formal computation, we find that
\begin{align*}
\cI_n=&\int_{Z\bksl G}\int_{ZN\bksl G}F_1(hg)F_2(h)\cdot \pair{\rho(g)W_3}{W_4}\bbI_{K_{2n}}(g)\rmd h\rmd g\\
=&\int_{ZN\bksl G}\int_{Z\bksl G}F_1(hg)F_2(h)\cdot \pair{\rho(g)W_3}{W_4}\bbI_{K_{2n}}(g)\rmd g\rmd h\\
=&\int_{(ZN\bksl G)^2}\int_F\psi_{\Q_\pmq}(x)F_1(g)F_2(h)\cdot \pair{h^{-1}\pMX{1}{x}{0}{1}gW_3}{W_4}\cdot \bbI_{K_{2n}}(h^{-1}\pMX{1}{x}{0}{1}g)\rmd x\rmd g\rmd h\\
=&\int_{K}\int_K\int_{\Q_\pmq^\x}\int_{\Q_\pmq^\x}\int_{\Q_\pmq} \psi_{\Q_\pmq}(x)F_1(\aone{y_1}k_1)F_2(\aone{y_2}k_2) \pair{\rho(\pMX{y_2^{-1}y_1}{y_2^{-1}x}{0}{1}k_1)W_3}{\rho(k_2)W_4}\\
&\times \bbI_{K_{2n}}(\pMX{y_2^{-1}y_1}{y_2^{-1}x}{0}{1})\,\rmd x\frac{\rmd^\x y_1}{\abs{y_1}}\frac{\rmd^\x y_2}{\abs{y_2}}\rmd k_1\rmd k_2.
\end{align*}
To justify the above computation, it suffices to show that the integral
\[\int_{\Q_\pmq^\x}\int_{\Q_\pmq^\x}\int_{\Q_\pmq}\psi_{\Q_\pmq}(x)F_1'(\aone{y_1})F_2'(\aone{y_2})\cdot\pair{\pMX{y_2^{-1}y_1}{y_2^{-1}x}{0}{1}W_3'}{W_4'}
 \bbI_{K_{2n}}(\pMX{y_2^{-1}y_1}{y_2^{-1}x}{0}{1})\,\rmd x\frac{\rmd^\x y_1}{\abs{y_1}}\frac{\rmd^\x y_2}{\abs{y_2}}\]
is absolutely convergent, where $F_1'=\rho(k_1)F_1$, $F_2'=\rho(k_2)F_2$, $W_3'=\rho(k_1)W_3$ and $W_4'=\rho(k_2)W_4$. From\eqref{E:estimate.loc} and \eqref{E:511.L}, we deduce that for any $\ep>0$ there exist constants $C_\ep$ and $M$ such that 
\begin{align*}&\int_{\Q_\pmq^\x}\int_{\Q_\pmq^\x}\int_{\Q_\pmq}\abs{F_1'(\aone{y_1}F_2'(\aone{y_2}}\abs{\pair{\rho(\pMX{y_2^{-1}y_1}{y_2^{-1}x}{0}{1})W_3'}{W_4'}}
\bbI_{K_{2n}}(\pMX{y_2^{-1}y_1}{y_2^{-1}x}{0}{1})\,\rmd x\frac{\rmd^\x y_1}{\abs{y_1}}\frac{\rmd^\x y_2}{\abs{y_2}}\\
<&\,C_\ep\iint_{\substack{\abs{y_1}\leq M,\abs{y_2}\leq M,\\ q^{-n}\leq\abs{y_2^{-1}y_1}\leq q^n}}\int_{\abs{x}^2\leq \abs{y_1y_2}\pmq^{2n}}\abs{y_1y_2}^{1-\lam_1-\lam_3-\ep}\abs{y_2^{-1}y_1}^{\onehalf-\lam_2-\ep}\,\rmd x\frac{\rmd^\x y_1}{\abs{y_1}}\frac{\rmd^\x y_2}{\abs{y_2}}\\
< &\,C_\ep \pmq^{n(3/2-\lam_2-\ep)} \int_{\abs{y_1}\leq M}\int_{\abs{y_2}\leq M}\abs{y_1y_2}^{\onehalf-\lam_1-\lam_3-\ep} \rmd^\x y_1\rmd^\x y_2<\infty.\end{align*}
For $(g,h)\in G\times G$, we put
\[\sA_n(g,h):=\int_{\Q_\pmq}\psi(x)\pair{\rho(\pMX{1}{x}{0}{1}g)W_3}{\rho(h)W_4}\bbI_{K_{2n}}(h^{-1}\pMX{1}{x}{0}{1}g)\rmd x.\]
Then we have 
\begin{align*}
\sA_n(\aone{y_1}k_1,\aone{y_2}k_2)=&\int_{\Q_\pmq}\psi_{\Q_\pmq}(x)\pair{\rho(\pMX{y_1}{x}{0}{1}k_1)W_3}{\rho(\aone{y_2}k_2)W_4}\bbI_{K_{2n}}(\pMX{y_1y_2^{-1}}{y_2^{-1}x}{0}{1})\rmd x\\
=&\int_{\Q_\pmq}\int_{\Q_\pmq^\x}\psi_{\Q_\pmq}((1-\nu)x)W_3(\aone{\nu y_1}k_1)W_4(\aone{\nu y_2}k_2)\bbI_{\pmq^{r-n}\Z_\pmq}(x) \rmd x\nu \rmd x\\
=&\int_{\Q_\pmq^\x}W_3(\aone{\nu y_1}k_1)W_4(\aone{\nu y_2}k_2)\bbI_{1+\pmq^{n-r}\Z_\pmq}(\nu)\abs{\pmq^{r-n}}\dx \nu,\end{align*}
where $r=\lfloor\frac{v_p(y_1y_2)}{2}\rfloor$. Therefore, there exists a positive integer $m_0$ such that if $v_p(y_1y_2)<2n-m_0$, then 
\[\sA_n(\aone{y_1}k_1,\aone{y_2}k_2)=\zeta_\pmq(1)W_3(\aone{y_1}k_1)W_4(\aone{y_2}k_2).\]
On the other hand, if $v_p(y_1y_1)\geq 2n-m_0$, then we have
\begin{align*}\abs{\sA_n(\aone{y_1}k_1,\aone{y_2}k_2)}<_{W_2,\wtd W_w,\ep}&\,\pmq^{n-r}\abs{y_1y_2}^{\onehalf-\lam_2-\ep} \int_{\Q_\pmq^\x}\abs{\nu}^{1-2\lam_2-2\ep}\bbI_{1+\pmq^{n-r}\Z_\pmq}(\nu)\rmd^\x\nu\\
<&C_\ep \cdot \pmq^{m_0/2}\cdot \abs{y_1y_2}^{\onehalf-\lam_2-\ep}.
\end{align*}
We thus obtain\begin{align*}
\cI_n
=&\int_K\int_K\int_{\Q_\pmq^\x}\int_{\Q_\pmq^\x}F_1(\aone{y_1}k_1)F_2(\aone{y_2}k_2)\cdot \sA_n(\aone{y_1}k_1,\aone{y_2}k_2)\frac{\rmd^\x y_1\rmd^\x y_2}{\abs{y_1y_2}}\rmd k_1\rmd k_2\\
=&\zeta_\pmq(1)\int_K\int_K\int\limits_{\substack{\pmq^{-2n}\leq \abs{y_1y_2^{-1}}\leq \pmq^{2n},\\
\abs{y_1y_2}>\abs{\pmq}^{2n-m_0}}}F_1\ot W_3(\aone{y_1}k_1)F_2\ot W_4(\aone{y_2}k_2)\frac{\rmd^\x y_1\rmd^\x y_2}{\abs{y_1y_2}}\rmd k_1\rmd k_2+B_n,\\
\end{align*}
where
\begin{align*}\abs{B_n}<&C_\ep'\int\limits_{\substack{\pmq^{-2n}\leq \abs{y_1y_2^{-1}}\leq \pmq^{2n},\\
\abs{y_1y_2}\leq \abs{\pmq}^{2n-m_0}}}\abs{y_1y_2}^{\onehalf-\lam_1-\lam_2-\lam_3-2\ep}\rmd^\x y_1\rmd^\x y_2\\
<&C_\ep''\abs{\pmq}^{2n(\onehalf-\lam_1-\lam_2-\lam_3-2\ep)}(4n+1).
\end{align*}
It follows that  \begin{align*}
\int_{ZN\bksl G}\pair{\rho(g)F_1}{F_2}\pair{\rho(g)W_3}{W_4}\rmd g=&\lim_{n\to\infty}\cI_n=\zeta_\pmq(1)\int_{ZN\bksl G}F_1(g)W_3(g)\rmd g\int_{ZN\bksl G}F_2(h)W_4(h)\rmd h.
\end{align*}
This finishes the proof of \eqref{E:RSTriple}.
 \end{proof}
\def\rmt{t}
\def\cpitwo{{c_2}}
\def\cpithree{{c_3}}
Denote by $L(s,\pi_1\ot\pi_2)$ the local $L$-factor and by $\varepsilon(s,\pi_1\ot\pi_2):=\varepsilon(s,\pi_1\ot\pi_2,\addchar_{\Q_\pmq})$ the $\varepsilon$-factors attached to $\pi_1\times \pi_2$ defined in \cite{GJ78}. Define the $\gamma$-factor
\beq\label{E:gamma2}\gamma(s,\pi_1\ot\pi_2):=\varepsilon(s,\pi_1\ot\pi_2)\frac{L(1-s,\Contra{\pi}_1\ot\Contra{\pi}_2)}{L(s,\pi_1\ot\pi_2)}.\eeq
The following corollary is the core of our calculations of local zeta integrals $I_v(\phi_v^\star\ot\wtd\phi_v^\star)$ at the non-archimedean places.
\begin{cor}\label{C:IchinoRS}Suppose that $(\pi_1,\pi_2;\pi_3)$ satisfies \eqref{Hb} and that $\chi_3\upsilon_3^{-1}\not =\Abs$. If $\wtd W_1=W_1\ot\om_1^{-1}$, $\wtd W_2=W_2\ot\om_2^{-1}$ and $\wtd f_3=M^*(\chi_3,\upsilon_3)f_3\ot\om_3^{-1}$, then 
\begin{align*}\sJ_\pmq(W_1\ot W_2\ot f_3,\wtd W_1\ot\wtd W_2\ot \wtd f_3)
=&\zeta_\pmq(1)\chi_3(-1)\cdot \gamma(1/2,\pi_1\ot\pi_2\ot\chi_3)\cdot \Psi(W_1,W_2,f_3)^2.\end{align*}
\end{cor}
\begin{proof}This is an immediate consequence of the local functional equation of $\GL(2)\times\GL(2)$ in\cite{Jacquet72Part2}. With the notation of \cite[page 12]{Jacquet72Part2}, we may assume that \[f_3(g)=\chi_3\Abs^{s+\onehalf}(\det g)\cdot z(\chi_3\upsilon_3^{-1}\Abs^{2s+1},\rho(g)\Phi)\cdot \frac{1}{L(2s+1,\chi_3\upsilon_3^{-1})}|_{s=0}\] is the Godement section attached to a \BS function $\Phi$ on $\Qq^2$. Since $\chi_3\upsilon_3^{-1}\not =\Abs$, one verifies that \[M^*(\chi_3,\upsilon_3)f_3(g)=\upsilon_3\Abs^{-s+\onehalf}(\det g)\cdot z(\upsilon_3\chi_3^{-1}\Abs^{-2s+1},\rho(g)\wh\Phi)\cdot\frac{1}{L(2s+1,\chi_3\upsilon_3^{-1})}|_{s=0},\]
where $\wh\Phi$ is the Fourier transform of $\Phi$ defined in \cite[Theorem 14.2 (3)]{Jacquet72Part2}. Under the hypothesis \eqref{Hb}, we have 
\begin{align*}\Psi(W_1,W_2,f_3)=&\frac{\Psi(s,W_1,W_2,\Phi)}{L(2s+1,\chi_3\upsilon_3^{-1})}|_{s=0},\\
\wtd\Psi(W_1,W_2,M^*(\chi_3,\upsilon_3)f_3)=&\frac{\wtd\Psi(s,\wtd W_1,\wtd W_2,\wh \Phi)}{L(2s+1,\chi_3\upsilon_3^{-1})}|_{s=0},
\end{align*}
where $\Psi(s,W_1,W_2,\Phi)$ and $\wtd\Phi(s,W_1,W_2,\wh\Phi)$ are defined in \cite[(14.5) and (14.6)]{Jacquet72Part2}.

Therefore, from \cite[Theorem 14.8]{Jacquet72Part2} we can deduce that
\begin{align*}
\wtd\Psi(\wtd W_1,\wtd W_2,\wtd f_3)=&\om_1(-1)\upsilon_3(-1)\Psi(W_1,W_2,M^*(\chi_3,\upsilon_3)f_3)\\
%=&\om_1(-1)\upsilon_3(-1)\Psi(s,W_1,W_2,\wh\Phi)\\
=&\om_1\om_2\upsilon_3(-1)\gamma(1/2,\pi_1\ot\pi_2\ot\chi_3)\Psi(W_1,W_2,f_3).\qedhere
\end{align*}
\end{proof}

\subsection{The calculation of the $p$-adic zeta integrals}
\subsubsection{Preliminaries}We follow the notation in \subsecref{S:gtp}. Let $(f,g,h)=(\bdsf_\Qx,\bdsg_\Qy,\bdsh_\Qz)$ be the specialization of the triple of Hida families at a classical point $\ulQ=(\Qx,\Qy,\Qz)\in\frakX_\cR^\cls:=\frakX_{\bfI_1}^\cls\times\frakX_{\bfI_2}^\cls\times\frakX_{\bfI_3}^\cls$. Let $\pi_1=\pi_{\fQx,p}\ot\om_{\fgh,p}^{-1/2}$, $\pi_2=\pi_{\gQy,p}$ and $\pi_3=\pi_{\hQz,p}$ of the central characters $\om_1=\om_{\gQy,p}^{-1}\om_{\hQz,p}^{-1}$, $\om_2=\om_{\gQy,p}$ and $\om_3=\om_{\hQz,p}$ respectively. Let $\itPi_{\ulQ,p}:=\pi_1\times\pi_2\times\pi_3$. For $i=1,2,3$, since $\pi_{i,p}$ contains a non-zero ordinary vector, by \propref{P:ordline}
$\pi_i$ must be a constituent of the induced representation $\cB(\upsilon_i,\chi_i)$ with $\cV_{\pi_i}^\ord(\chi_i)\not=\stt{0}$. In view of the discussion in \remref{R:WhittakerPordinary}, we have $\chi_1=\al_{f,p}\om_{\fgh,p}^{-1/2}$, $\chi_2=\al_{g,p}$ and $\chi_3=\al_{h,p}$ with $\al_{?,p}$ unramified characters defined there, and the ordinary assumption implies that $\chi_i\upsilon_i^{-1}\not =\Abs^{-1}$. Recall that if we let $\xi_i\in \cV_{\pi_i}^\ord(\chi_i)$ and $\wtd\xi_i\in \cV_{\Contra{\pi}_i}^\ord(\upsilon_i^{-1})$ be nonzero ordinary vectors for $i=1,2,3$, then \[\phi_p=\xi_1\ot\xi_2\ot\xi_3\text{ and }\wtd\phi_p=\wtd\xi_1\ot\wtd\xi_2\ot\wtd\xi_3.\] Put
\[w=\pMX{0}{1}{-1}{0};\quad \bftr=\pMX{0}{p^{-n}}{-p^n}{0}\in\SL_2(\Qp).\]  

We introduce the normalized ordinary section in the induced representations and compute its local pairing.
\begin{lm}\label{L:ordinarysection}Let $\pi$ be a constituent of the induced representation $\cB(\upsilon,\chi)$ of $\GL_2(\Qp)$ with the central character $\om$. Suppose that $\chi\upsilon^{-1}\not =\Abs^{-1}$. Let $f^\ord\in \cB(\upsilon,\chi)$ be the unique section such that (i) $f^\ord$ is supported in $BwN(\Zp)$ (ii) $f^\ord(g)=1$ for all $g\in wN(\Zp)$. Then \[f^\ord\in \cB(\upsilon,\chi)^\ord(\chi).\] We call $f^\ord$ the normalized ordinary section. Moreover, put \[\wtd f^\ord:=M^*(\upsilon,\chi)f^\ord\ot\om^{-1}\in\cB(\upsilon^{-1},\chi^{-1})^\ord(\upsilon^{-1}).\] For $n\geq \max\stt{1,c(\pi_p)}$, we have
 \[\pair{\rho(t_n)f^\ord}{\wtd f^\ord}=\frac{\om(p^{-n})\zeta_p(2)\chi\Abs^\onehalf(p^{2n})}{\zeta_p(1)}\cdot\gamma(0,\upsilon\chi^{-1}).\]
 In particular, if $W^\ord$ is the normalized ordinary Whittaker function in \corref{C:Word}, then
 \beq\label{E:3.imb} \frac{\pair{\rho(\bftr)W^\ord}{W^\ord\ot\om^{-1}}}{\pair{\rho(\bftr)f^\ord}{\wtd f^\ord}}=\frac{\chi(-1)\zeta_p(2)}{\zeta_p(1)}.\eeq
\end{lm}
\begin{proof}It is straightforward to verify that $f^\ord\in \cB^\ord(\upsilon,\chi)^\ord(\chi)$ is an $\bfU_p$-eigenfunction with eigenvalue $\chi\Abs^{-\onehalf}$. By the integration formula \cite[(3.2) page 207]{MV10}, $\pair{\rho(t_n)f^\ord}{\wtd f^\ord}$ equals \begin{align*}
%\pair{\rho(\pMX{0}{p^{-n}}{-p^n}{0})f_1^\ord}{\wtd f_1^\ord}=& 
\pair{f^\ord}{\rho(\pMX{0}{-p^{-n}}{p^n}{0})\wtd f^\ord}
&=\frac{\zeta_p(2)}{\zeta_p(1)}\int_{\Qp} f^\ord(w\pMX{1}{x}{0}{1})\wtd f^\ord(w\pMX{1}{x}{0}{1}\pMX{0}{-p^{-n}}{p^n}{0})\rmd x\\
&=\frac{\zeta_p(2)}{\zeta_p(1)}f^\ord(w)\wtd f^\ord(\pDII{p^{n}}{p^{-n}})\\
%=&\abs{p}^\frac{c_3}{2}(1+p^{-1})^{-1}\chi_3(p^{c_3})\gamma(0,\upsilon_3\chi_3^{-1})\\
&=\om(p^{-n})\frac{\zeta_p(2)}{\zeta_p(1)}\chi\Abs^\onehalf(p^{2n})\gamma(0,\upsilon\chi^{-1}).\end{align*}
The ratio of local pairings of ordinary Whittaker functions and ordinary sections is computed by the above and  \lmref{L:ordlocalnorm}. 
\end{proof}

\subsubsection{The unbalanced case}Suppose that $\ulQ$ is in the unbalanced range $ \frakX^\bdsf_\cR$. We apply \corref{C:IchinoRS} to calculate the normalized \padic zeta integral $\sI^\unb_{\itPi_{\ulQ,p}}$ in \eqref{E:Nunb}. \begin{prop}[$p$-adic zeta integral in the unbalanced case]\label{P:padic.imb}Put \begin{align*}\cE_\bdsf(\itPi_{\ulQ,p})&:=\gamma(1/2,\pi_2\ot\pi_3\ot\chi_1)^{-1}.%=\frac{1}{\varepsilon(1/2,\pi_2\ot\pi_3\ot\chi_1)}\cdot\frac{L(1/2,\pi_2\ot\pi_3\ot\chi_1)}{L(1/2,\pi_2\ot\pi_3\ot\upsilon_1)}.
\end{align*}
Then \[\sI_{\itPi_{\ulQ,p}}^\unb=\cE_\bdsf(\itPi_{\ulQ,p})\cdot\frac{1}{L(1/2,\itPi_{\ulQ,p})}.\]
\end{prop}
\begin{proof} We write $\itPi_p=\itPi_{\ulQ,p}$ for brevity. It is equivalent to proving that  \beq\label{E:5.imb}L(1/2,\itPi_p)\cdot I_p^\ord(\phi^\star_p\ot\wtd\phi^\star_p,\bft_n)
=\cE^f(\itPi_p)\cdot\frac{\chi_1\upsilon_1^{-1}\Abs(-p^{2n})}{B^{[n]}_{\itPi^\ord_p}}\cdot\frac{\zeta_p(2)^2}{\zeta_p(1)^2}
\eeq
for $n\geq \max\stt{c(\pi_1),c(\pi_2),c(\pi_3),1}$, where $I_p^\ord(\phi^\star_p\ot\wtd\phi^\star_p,\bft_n)$ is the local zeta integral defined in \eqref{E:unbpzeta}. We first treat the case where either (i) $\pi_1$ is principal series or (ii) $\pi_2$ or $\pi_3$ is discrete series. Then it is known that $(\pi_2,\pi_3;\pi_1)$ satisfies \eqref{Hb} since each $\pi_i$ is a local component of a cuspidal automorphic representation of $\GL_2(\A)$. Consider the realizations 
\[\cV_{\itPi_p}:=\sB(\upsilon_1,\chi_1)^0\boxtimes\sW(\pi_2)\boxtimes \sW(\pi_3);\quad \cV_{\Contra{\itPi}_p}:=\sB(\upsilon^{-1}_1,\chi^{-1}_1)_0\boxtimes\sW(\Contra{\pi}_2)\boxtimes \sW(\Contra{\pi}_3)\]of $\itPi_p$ and the contragredient representation $\Contra{\itPi}_p$.
For $i=1,2,3$, let $W^\ord_i=W^\ord_{\pi_i}\in\cW^\ord(\pi_i)(\chi_i)$ be the normalized ordinary Whittaker functions such that $W^\ord_{\pi_i}(\aone{y})=\chi_i\Abs^\onehalf(y)\bbI_{\Zp}(y)$ in \corref{C:Word}; let $f_i^\ord\in \sB(\upsilon_i,\chi_i)^\ord(\chi_i)$ be the normalized ordinary section in \lmref{L:ordinarysection} and $\wtd f^\ord_i:=M^*(\upsilon_i,\chi_i)f_i^\ord\ot\om_i^{-1}\in\cB(\upsilon_i^{-1},\chi_i^{-1})_0^\ord(\upsilon_i^{-1})$. First consider the case where $\pi_1$ is the principal series $\Prin{\chi_1}{\upsilon_1}$.  Let $(f_1^\ord)^0$ be the homomorphic image of $f_1^\ord$ in $\cB(\upsilon_1,\chi_1)^0$. In view of \eqref{E:factorization1}, we may take
\beq\label{E:ordinaryunbalanced}\begin{aligned}
\phi_p:=&(f^\ord_1)^0\ot W^\ord_2\ot W^\ord_3,\quad \wtd\phi_p:=\wtd f_1^\ord\ot \wtd W^\ord_2\ot\wtd W^\ord_3,\\
\phi^\star_p:=&(f^\ord_1)^0\ot W^\ord_2\ot \theta^{\Bkappa}_p W^\ord_3,\quad
\wtd\phi^\star_p=\wtd f^\ord_1\ot \wtd W^\ord_2\ot \theta^{\Bkappa}_p\wtd W^\ord_3,\end{aligned}\eeq
where $\Bkappa$ is the Dirichlet character defined in \eqref{E:Bkap1} and $\theta^\Bkappa_p$ is the twisting operator in \eqref{E:deftheta2}. According to the definition \eqref{E:unbpzeta} and \corref{C:IchinoRS}, we find that
\beq\label{E:1.imb}\begin{aligned}
I_p^\ord(\phi^\star_p\ot\wtd\phi^\star_p,\bft_n)&=\frac{L(1,\itPi_p,\ad)}{\zeta_p(2)^2L(1/2,\itPi_p)}\cdot\frac{\sJ_p(\rho(\bftr)W_2^\ord\ot\theta_p^{\Bkappa}W_3^\ord\ot f^\ord_1,\rho(\bftr)\wtd W_2^\ord\ot\theta_p^{\Bkappa}\wtd W^\ord_3\ot\wtd f^\ord_1)}{\pair{\rho(\bftr)W_2^\ord}{\wtd W_2^\ord}\pair{\rho(\bftr)W^\ord_3}{\wtd W^\ord_3}\pair{\rho(\bftr)f_1}{\wtd f^\ord_1}}\\
&=\frac{I^*_p}{B^{[n]}_{\itPi^\ord_p}}\cdot \frac{\pair{\rho(\bftr)W^\ord_1}{\wtd W^\ord_1}}{\pair{\rho(\bftr)f^\ord_1}{\wtd f^\ord_1}}\cdot \frac{\zeta_p(2)^3}{\zeta_p(1)^3}\cdot\frac{1}{L(1/2,\itPi_p)},\end{aligned}\eeq
where
\begin{align*}I^*_p&=\frac{\zeta_p(1)\upsilon_1(-1)\gamma(1/2,\pi_2\ot\pi_3\ot\upsilon_1)}{\zeta_p(2)^2}\cdot \Psi(W_2^\ord,\theta_p^{\Bkappa}W_3^\ord,\rho(\bftr)f^\ord_1)^2.\end{align*}
Note that the adelization $\Bkappa_\A=\breve\om_f^{-1}\om_{\fgh}^{1/2}$; hence
 \[\Bkappa|_{\Zp^\x}=\beta_p^{-1}\om_{\fgh}^{1/2}|_{\Zp^\x}=\upsilon_1^{-1}|_{\Zp^\x},\] 
and a simple calculation shows that $\theta_p^{\Bkappa}W^\ord_3(\pDII{a}{1})=\upsilon_1^{-1}(a)\bbI_{\Zp^\x}(a)$. We proceed to calculate the local Rankin-Selberg integral
\begin{align*}
&\Psi(W_2^\ord,\theta_p^{\Bkappa}W^\ord_3,\rho(\bftr)f^\ord_1)\\
=&\frac{\zeta_p(2)}{\zeta_p(1)}\int_{\Qp^\x}\int_{\Qp}W_2^\ord(\pDII{y}{1}\pMX{1}{0}{x}{1})\theta_p^{\Bkappa}W^\ord_3(\pDII{-y}{1}\pMX{1}{0}{x}{1})\upsilon_1\Abs^\onehalf(y)f^\ord_1(\pMX{0}{p^{-n}}{-p^n}{0}\pMX{1}{p^{-2n}x}{0}{1})\rmd x\frac{\rmd^\x y}{\abs{y}}\\
=&\frac{\zeta_p(2)\chi_1\upsilon_1^{-1}\Abs(-p^n)}{\zeta_p(1)}\int_{\Qp^\x}W_2^\ord(\pDII{y}{1})\theta_p^{\Bkappa}W^\ord_3(\pDII{-y}{1})\upsilon_1\Abs^{-\onehalf}(y)\rmd^\x y\\
=&\frac{\zeta_p(2)\chi_1\upsilon_1^{-1}\Abs(-p^n)}{\zeta_p(1)}\int_{\Zp^\x}W_2^\ord(\pDII{y}{1})\rmd^\x y=\frac{\zeta_p(2)\chi_1\upsilon_1^{-1}\Abs(-p^n)}{\zeta_p(1)}.
\end{align*}
We thus obtain
\beq\label{E:2.imb}\begin{aligned}I^*_p&=\frac{\zeta_p(1)\upsilon_1(-1)}{\zeta_p(2)^2}\cdot\frac{\chi_1\upsilon_1^{-1}\Abs(p^{2n})\zeta_p(2)^2\gamma(1/2,\pi_2\ot\pi_3\ot\upsilon_1)}{\zeta_p(1)^2}\\
&=\frac{\chi_1\upsilon_1^{-1}\Abs(p^{2n})\upsilon_1(-1)}{\zeta_p(1)}\cdot\frac{\varepsilon(1/2,\pi_2\ot\pi_3\ot\upsilon_1)L(1/2,\pi_2\ot\pi_3\ot\chi_1)}{L(1/2,\pi_2\ot\pi_3\ot\upsilon_1)}.\end{aligned}\eeq
Substituting \eqref{E:3.imb} and \eqref{E:2.imb} to \eqref{E:1.imb} and noting that \[\varepsilon(1/2,\pi_2\ot\pi_3\ot\upsilon_1)\varepsilon(1/2,\pi_2\ot\pi_3\ot\chi_1)=1,\] we immediately obtain  \eqref{E:5.imb}.

Now we treat the remaining case, \ie $\pi_1=\chi_1\Abs^{-\onehalf}{\rm St}$ is special, and $\pi_2$ and $\pi_3$ are principal series. Thus $(\pi_1,\pi_3;\pi_2)$ satisfies \eqref{Hb}. Consider the realizations 
\[\cV_{\itPi_p}:=\sW(\pi_1)\boxtimes \sW(\pi_3)\boxtimes \sB(\upsilon_2,\chi_2);\quad \cV_{\Contra{\itPi}_p}:=\sW(\Contra{\pi}_1)\boxtimes \sW(\Contra{\pi}_3)\boxtimes\sB(\upsilon^{-1}_2,\chi^{-1}_2).\]
By \corref{C:IchinoRS}, we have
\beq\label{E:4.imb}\begin{aligned}L(1/2,\itPi_p)\cdot I_p^\ord(\phi^\star_p\ot
\wtd\phi^\star_p,\bft_n)&=
\frac{ \zeta_p(1)\upsilon_2(-1)\gamma(1/2,\pi_1\ot\pi_3\ot\upsilon_2)}{\zeta_p(2)^2\cdot B^{[n]}_{\itPi^\ord_p}}\cdot \Psi(\rho(\bftr)W_1^\ord, \theta_p^{\Bkappa}W^\ord_3,f^\ord_2)^2\\
&\times \frac{\pair{\rho(\bftr)W^\ord_2}{\wtd W^\ord_2}}{\pair{\rho(\bftr)f^\ord_2}{\wtd f^\ord_2}}\cdot \frac{\zeta_p(2)^3}{\zeta_p(1)^3}.\end{aligned}\eeq
We calculate the local Rankin-Selberg integral in the right hand side
\begin{align*}&\Psi(\rho(\bftr)W_1^\ord, \theta_p^{\Bkappa}W^\ord_3,f^\ord_2)\\
=&\frac{\zeta_p(2)}{\zeta_p(1)}
\int_{\Qp^\x}\int_{\Qp}W^\ord_1(\pDII{y}{1}w\pMX{1}{x}{0}{1}\bftr) \theta_p^{\Bkappa}W^\ord_3(\pDII{-y}{1}w\pMX{1}{x}{0}{1})\upsilon_2\Abs^\onehalf(y)\bbI_{\Zp}(x)\rmd x\frac{\rmd^\x y}{\abs{y}}\\
=&\frac{\zeta_p(2)\chi_1\upsilon_1^{-1}\Abs(p^n)\upsilon_1\upsilon_2(-1)}{\zeta_p(1)}\int_{\Qp^\x} \theta_p^{\Bkappa}W^\ord_3(\pDII{y}{1}w)\chi_1\upsilon_2(y)\rmd^\x y\\
=&\frac{\zeta_p(2)\chi_1\upsilon_1^{-1}\Abs(p^n)\chi_1(-1)}{\zeta_p(1)}\gamma(1/2,\pi_3\ot\upsilon_1\chi_2)\int_{\Qp^\x} \theta_p^{\Bkappa}W^\ord_3(\pDII{y}{1})\upsilon_1\chi_2(y)\rmd^\x y\\
=&\frac{\zeta_p(2)\chi_1\upsilon_1^{-1}\Abs(p^n)\chi_1(-1)}{\zeta_p(1)}\gamma(1/2,\pi_3\ot\upsilon_1\chi_2).\end{align*}
Substituting the above equation and \eqref{E:3.imb} into \eqref{E:4.imb} and using the formulae of the local $L$-factor and $\varepsilon$-factor of $\pi_1\ot\pi_3\ot\upsilon_2$ in \cite[Proposition 1.4 (1.4.2)]{GJ78}, we find that $L(1/2,\itPi_p)\cdot I_p(\phi^\star_p\ot\wtd\phi^\star_p)$ equals
\begin{align*}&\frac{\chi_1\upsilon_1^{-1}\Abs(p^{2n})\zeta_p(2)^2}{B^{[n]}_{\itPi^\ord_p}\zeta_p(1)^2}\cdot\frac{\varepsilon(1/2,\pi_1\ot\pi_3\ot\upsilon_2)L(1/2,\pi_1\ot\pi_3\ot\chi_2)}{L(1/2,\pi_1\ot\pi_3\ot\upsilon_2)}\cdot\varepsilon(1/2,\pi_3\ot\upsilon_1\chi_2)^2\frac{L(1/2,\pi_3\ot\chi_1\upsilon_2)^2}{L(1/2,\pi_3\ot\upsilon_1\chi_2)^2}\\
=&\frac{\zeta_p(2)^2}{\zeta_p(1)^2}\cdot\frac{\chi_1\upsilon_1^{-1}\Abs(p^{2n})\om_2\om_3(-1)}{B^{[n]}_{\itPi^\ord_p}}\cdot\frac{\varepsilon(1/2,\pi_2\ot\pi_3\ot\upsilon_1)L(1/2,\pi_2\ot\pi_3\ot\chi_1)}{L(1/2,\pi_2\ot\pi_3\ot\upsilon_1)}.\end{align*}
This proves \eqref{E:5.imb} in the remaining case.
\end{proof}
\begin{Remark}\label{R:improved}Replacing $\phi^\star_p\ot\wtd\phi^\star_p$ with $\phi_p\ot\wtd\phi_p$ in \eqref{E:unbpzeta}, we define the improved \padic zeta integral 
\begin{align*}\sI^*_{\itPi_{\ulQ,p}}:=&I_p^\ord(\phi_p\ot\wtd\phi_p,\bft_n)\cdot \frac{B^{[n]}_{\itPi^\ord_p}}{\chi_1\upsilon_1^{-1}\Abs(-p^{2n})}
\cdot\frac{\zeta_p(1)^2}{\zeta_p(2)^2}.
\end{align*}
If $\pi_1$ is principal series, then $\upsilon_1\chi_2\chi_3\not =\Abs^{-\onehalf}$, and 
\[\sI^*_{\itPi_{\ulQ,p}}=\frac{1}{\varepsilon(1/2,\pi_2\ot\pi_3\ot\chi_1)}\cdot\frac{L(1/2,\pi_2\ot\pi_3\ot\chi_1)}{L(1/2,\pi_2\ot\pi_3\ot\upsilon_1)}\cdot L(1/2,\upsilon_1\chi_2\chi_3)^2;\]
if $\pi_1$ is special and $\upsilon_1\chi_2\chi_3=\Abs^{-\onehalf}$, then \[\sI^*_{\itPi_{\ulQ,p}}= \frac{1}{ \varepsilon(1/2,\pi_2\ot\pi_3\ot\upsilon_1)}\cdot\frac{(-1)}{L(1/2,\upsilon_1\chi_2\upsilon_3)L(1/2,\upsilon_1\upsilon_2\chi_3)}.\]
%\[\Psi(\rho(\bftr)W_1^\ord,W_3^\ord,f^\ord_2)=\frac{\zeta_p(2)\chi_1\upsilon_1^{-1}\Abs(p^n)\chi_1(-1)}{\zeta_p(1)}\cdot\varepsilon(1/2,\pi_3\ot\upsilon_1\chi_2)\cdot\frac{L(1/2,\pi_3\ot\chi_1\upsilon_2)}{L(1/2,\upsilon_1\chi_2\upsilon_3)},\] 
These equations will be used later for the interpolation formula of improved $p$-adic $L$-functions. It can be obtained by the same computation in the above proposition. We omit the details.
\end{Remark}
\subsubsection{The balanced case}Now suppose that $\ulQ$ is in the balanced range $\frakX^\bal_\cR$. We shall compute the normalized $p$-adic zeta integral  $\sI_{\itPi_p}^\bal$ in \eqref{E:Nbal}.  Put
\[u_n=\pMX{1}{p^{-n}}{0}{1}\in\SL_2(\Qp);\quad \breve\bft_n=(u_n,1,\bftn)\in \GL_2(E_p)\]
for $n\geq \max\stt{c(\pi_1),c(\pi_2),c(\pi_3),1}$. Observe that if $L:\pi_1\ot\pi_2\ot\pi_3\to\C$ is any $\GL_2(\Qp)$-invariant trilinear form, then \begin{align*}
L(\pi_1(u_n)\xi_1,\xi_2,\pi_3(\bftn)\xi_3)=&L(\pi_1(u_n)\xi_1,\pi_2(\bftn)\xi_2,\xi_3)\\
=&L(\xi_1,\pi_2(u_n)\xi_2,\pi_3(\bftn)\xi_3).
\end{align*}
Thus we may assume that 
\beqcd{Hb$'$}\text{either $\pi_3=\Prin{\chi_3}{\upsilon_3}$ is principal series or each of $\pi_1$, $\pi_2$ and $\pi_3$ is special.}\eeqcd

\begin{prop}[$p$-adic zeta integral in the balanced case]\label{P:padic.bal}Under the assumption \eqref{Hb$'$}, we put  \[
\cE_\bal(\itPi_{\ulQ,p}):=%\frac{\varepsilon(1/2,\upsilon_1\upsilon_2\chi_3)}{\varepsilon(1/2,\chi_1\chi_2\upsilon_3)} \cdot\left(\frac{L(1/2,\chi_1\chi_2\upsilon_3)}{L(1/2,\upsilon_1\upsilon_2\chi_3)}\right)^2\cdot\frac{L(1/2,\pi_1\ot\pi_2\ot\chi_3)}{\varepsilon(1/2,\pi_1\ot\pi_2\ot\chi_3)L(1/2,\pi_1\ot\pi_2\ot\upsilon_3)}.
\gamma(1/2,\pi_1\ot\pi_2\ot\chi_3)^{-1}\gamma(1/2,\chi_1\chi_2\upsilon_3)^{-2}.
\]
Then we have 
\[\sI_{\itPi_{\ulQ,p}}^\bal=\cE_\bal(\itPi_{\ulQ,p})\cdot\frac{1}{L(1/2,\itPi_{\ulQ,p})}.\]
\end{prop}
\begin{proof}We write $\itPi_p=\itPi_{\ulQ,p}$ as before. By definition, this is equivalent to proving \begin{align*}
I_p^\ord(\phi_p\ot\wtd\phi_p,\breve\bft_n)
=&\al_p(F)^{2n}\om_{F,p}^{-1/2}(-p^{2n})\abs{p^n}^{k_1+k_2+k_3}\cdot\cE_\bal(\itPi_p)\cdot \frac{\zeta_p(2)^2}{B^{[n]}_{\itPi^\ord_p}}\cdot \frac{1}{L(1/2,\itPi_p)}\\
=&\chi_1\chi_2\chi_3(-p^{2n})\abs{p}^{3n}\cdot\cE_\bal(\itPi_p)\cdot \frac{\zeta_p(2)^2}{B^{[n]}_{\itPi^\ord_p}}\cdot \frac{1}{L(1/2,\itPi_p)},
\end{align*}
where $I_p^\ord(\phi_p\ot\wtd\phi_p,\breve\bft_n)$ is the local zeta integral in \eqref{E:balpzeta}. The assumption \eqref{Hb$'$} implies that $(\pi_1,\pi_2;\pi_3)$ satisfies \eqref{Hb}, so we consider the realizations
\[\cV_{\itPi_p}=\cW(\pi_1)\boxtimes\cW(\pi_2)\boxtimes \cB(\upsilon_3,\chi_3)^0;\quad \cV_{\Contra{\itPi}_p}=\cW(\Contra{\pi}_1)\boxtimes\cW(\Contra{\pi}_2)\boxtimes \cB(\upsilon^{-1}_3,\chi^{-1}_3)_0.\]
Let $W^\ord_i=W^\ord_{\pi_i}$ and $\wtd W^\ord_i=W^\ord_i\ot\om_i^{-1}$ be the normalized ordinary Whittaker functions for $i=1,2$. Let $f^\ord_3$ be the normalized ordinary section in $\cB(\upsilon_3,\chi_3)^\ord(\chi_3)$ in \lmref{L:ordinarysection} and let $\wtd f^\ord_3:=M^*(\upsilon_3,\chi_3)f^\ord_3\ot\om_3^{-1}$.
Letting $(f^\ord_3)^0$ be the holomorphic image of $f^\ord_3$ in $\cB(\upsilon_3,\chi_3)^0$ as before,  we may take \[
\phi_p=W^\ord_1\ot W^\ord_2\ot (f^\ord_3)^0;\quad \wtd\phi_p=\wtd W^\ord_1\ot\wtd W^\ord_2\ot \wtd f^\ord_3.
\]
From the definition \eqref{E:balpzeta}, \corref{C:IchinoRS} and \eqref{E:3.imb}, we deduce that
\beq\label{E:6.bal}
\begin{aligned}&I_p^\ord(\phi_p\ot\wtd\phi_p,\breve\bft_n)\\
=&\frac{\sJ_p(\rho(u_n)W^\ord_1\ot W^\ord_2\ot \rho(\bftn)f^\ord_3,\rho(u_n)\wtd W^\ord_1\ot\wtd W^\ord_2\ot\rho(\bftn)\wtd f^\ord_3)}{\zeta_p(2)^2L(1/2,\itPi_p)\cdot B^{[n]}_{\itPi^\ord_p}}\cdot\frac{\pair{\rho(t_n)W^\ord_3}{\wtd W^\ord_3}}{\pair{\rho(t_n)f^\ord_3}{\wtd f^\ord_3}}\cdot\frac{\zeta_p(2)^3}{\zeta_p(1)^3}\\
=&\frac{I^*_p}{B^{[n]}_{\itPi^\ord_p}}\cdot\frac{\chi_3(-1)\zeta_p(1)^2}{\zeta_p(2)}\cdot\frac{\zeta_p(2)^3}{\zeta_p(1)^3}\cdot\frac{1}{L(1/2,\itPi_p)},\end{aligned}\eeq
where 
\begin{align*}
I^*_p
=&\frac{\zeta_p(1)\upsilon_3(-1)\gamma(1/2,\pi_1\ot\pi_2\ot\upsilon_3)}{\zeta_p(2)^2}\cdot \Psi(\rho(u_n)W^\ord_1,W^\ord_2,\rho(\bftn)f^\ord_3)^2.
\end{align*}
The local Rankin-Selberg integral $\Psi(\rho(u_n)W^\ord_1,W^\ord_2,\rho(\bftn)f^\ord_3)$ equals
\begin{align*}
&\int_{ZN\bksl G}W^\ord_1(g\pMX{1}{p^{-n}}{0}{1})W^\ord_2(\cJ g)f^\ord_3(g\pMX{0}{p^{-n}}{-p^n}{0})\rmd g\\
=&\frac{\zeta_p(2)}{\zeta_p(1)}\int_{\Qp^\x}\int_{\Qp}W^\ord_1(\pDII{y}{1}\pMX{1}{p^{-n}}{x}{1+xp^{-n}})W^\ord_2(\pMX{-y}{0}{x}{1})
f^\ord_3(\pDII{p^{-n}y}{p^n}w\pMX{1}{-p^{-2n}x}{0}{1})\frac{\rmd^\x y}{\abs{y}}\rmd x\\
=&\frac{\zeta_p(2)\abs{p^{2n}}\chi_3\upsilon_3^{-1}(p^n)\Abs^\onehalf(p^{-2n})}{\zeta_p(1)}\int_{\Qp^\x}\int_{\Zp}W^\ord_1(\pMX{y}{yp^{-n}(1+xp^n)^{-1}}{0}{1})W^\ord_2(\aone{-y})\upsilon_3\Abs^{-\onehalf}(y)\rmd x\rmd^\x y\\
=&\frac{\zeta_p(2)\abs{p^n}\chi_3\upsilon_3^{-1}(p^{n})\chi_2(-1)}{\zeta_p(1)}\int_{\Qp^\x}\psi(yp^{-n})\bbI_{\Zp}(y)\chi_1\chi_2\upsilon_3\Abs^\onehalf(y)\rmd^\x y\\
=&\frac{\zeta_p(2)\abs{p^n}\chi_3\upsilon_3^{-1}(p^{n})\chi_2\upsilon_3(-1)}{\zeta_p(1)}\cdot\frac{L(1/2,\chi_1\chi_2\upsilon_3)}{\varepsilon(1/2,\chi_1\chi_2\upsilon_3)L(1/2,\chi_1^{-1}\chi_2^{-1}\upsilon_3^{-1})}\int_{\Qp^\x}\bbI_{p^{-n}(-1+p^n\Zp)}(y)\chi_1^{-1}\chi_2^{-1}\upsilon^{-1}_3\Abs^\onehalf(y)\rmd^\x y\\
=&\frac{\zeta_p(2)\abs{p^n}\chi_1\chi_2\chi_3(p^n)\chi_2\upsilon_3(-1)}{\zeta_p(1)}\cdot\gamma(1/2,\chi_1^{-1}\chi_2^{-1}\upsilon_3^{-1})\cdot\frac{\abs{p}^\frac{n}{2}}{1-\abs{p}},
\end{align*}
so we find that
\begin{align*}
I^*_p=&\upsilon_3(-1)\varepsilon(1/2,\pi_1\ot\pi_2\ot\upsilon_3)\frac{L(1/2,\pi_1\ot\pi_2\ot\chi_3)}{L(1/2,\pi_1\ot\pi_2\ot\upsilon_3)}\left(\zeta_p(2)\abs{p}^\frac{3n}{2}\chi_1\chi_2\chi_3(p^n)\cdot\gamma(1/2,\chi_1^{-1}\chi_2^{-1}\upsilon^{-1}_3)\right)^2\\
=&\upsilon_3(-1)\zeta_p(2)^2\cdot \chi_1\chi_2\chi_3(p^{2n})\abs{p}^{3n}\cdot \chi_1\chi_2\upsilon_3(-1)\cE_{{\rm bal}}(\itPi_p).
\end{align*}
Substituting the above equation to \eqref{E:6.bal}, we obtain the desired formula.
\end{proof}
\begin{Remark}\label{R:padic}Keep the notation in \subsecref{S:mod.1}. For $\bullet\in\stt{\bdsf,\bal}$, we put $U_\ulQ:=\WD_p(\Fil_\bullet^+\bfV^\dagger_\ulQ)\ot_{\Qbarp,\iota_p}\C$ be the Weil-Deligne representation of $W_{\Qp}$ associated with $\Fil_\bullet^+\bfV^\dagger_\ulQ$ by Fontaine \cite[(4.2.3)]{F94}. It is not difficult to show that 
\[\cE_\bullet(\itPi_{\ulQ,p})=\frac{L(0,U_\ulQ)}{\varepsilon(U_\ulQ)L(1,U_\ulQ^\vee)},\]
and hence \[\sI^\bullet_{\itPi_\ulQ,p}=\cE_p(\Fil_\bullet^+\bfV_\ulQ^\dagger).\] For example, if $\bullet=\bal$ and $\pi_i=\chi_i\Abs^{-\onehalf}{\rm St}$ are special for $i=1,2,3$, then $\dim U_\ulQ^{\rm N=0}=3$, where ${\rm N}$ is the monodromy operator, and one verifies that $L(s-\onehalf,U_\ulQ)=L(s,\chi_1\chi_2\chi_3)L(s,\chi_1\chi_2\upsilon_3)^2$, $L(s+\onehalf,U^\vee_\ulQ)=L(s,\upsilon_1\upsilon_2\chi_3)^3$ and $\varepsilon(U_\ulQ)=\lim_{s\to 0}L(\onehalf-s,\chi_1^{-1}\chi_2^{-1}\chi_3^{-1})/L(s+\onehalf,\chi_1\chi_2\upsilon_3)=-\chi_1\chi_2\chi_3\Abs^{-1/2}(p)$.
\end{Remark}

\section{The calculation of local zeta integrals (II)}\label{S:local2}
\subsection{Setting}\label{SS:61}We continue to let $F=(f,g,h)$ be the specialization of $\bdsF=(\bdsf,\bdsg,\bdsh)$ at a classical point $\ulQ=(\Qx,\Qy,\Qz)$. In this section, we assume the following \emph{minimal} hypothesis for the unitary automorphic representations $(\pi_f,\pi_g,\pi_h)$ attached to $(f,g,h)$
\begin{hypothesis}\label{H:ram}For each prime $\pmq\divides N$, there exists a rearrangement $\stt{f_1,f_2,f_3}$ of $\stt{f,g,h}$ such that \begin{enumerate}\item $c_\pmq(\pi_{f_1})\leq\min\stt{c_\pmq(\pi_{f_2}),c_\pmq(\pi_{f_3})}$, 
\item the local components $\pi_{f_1,\pmq}$ and $\pi_{f_3,\pmq}$ are minimal, 
\item either $\pi_{f_3,\pmq}$ is a principal series or $\pi_{f_2,\pmq}$ and $\pi_{f_3,\pmq}$ are both discrete series.\end{enumerate}
\end{hypothesis}
Recall that an irreducible admissible representation $\pi$ of $\GL_2(\Qq)$ is minimal if the conductor $c(\pi)$ is minimal among the twists $\pi\ot\chi$ for all characters $\chi:\Qq^\x\to\C^\x$. 

\begin{Remark}\label{R:hyp}Note that if the above hypothesis holds for $(f,g,h)$, then it also holds for specializations of $(\bdsf,\bdsg,\bdsh)$ at any classical point by \remref{R:rigidity}. Moreover, we observe that one can always find Dirichlet characters $\chi_1$, $\chi_2$ and $\chi_3$ modulo some $M$ with $M^2\divides N$ such that $\chi_1\chi_2\chi_3=1$ and $(\pi_f\ot\chi_1,\pi_g\ot\chi_2,\pi_h\ot\chi_3)$ satisfies \hypref{H:ram}.
\end{Remark}
 As before, we let $\pi_1=\pi_{\fQx,\pmq}\ot\om_{\fgh,\pmq}^{-1/2}$, $\pi_2=\pi_{\gQy,\pmq}$ and $\pi_3=\pi_{\hQz,\pmq}$; let  $\itPi_\pmq=\itPi_{\ulQ,\pmq}=\pi_1\times\pi_2\times\pi_3$. Let $\pmq$ be a prime factor of $N$. Suppose that 
\[ \varepsilon(1/2,\itPi_\pmq)=+1\quad (\pmq\not\in\Sigma^-).\]
The purpose of this section is to evaluate the local zeta integral defined in \eqref{E:localzeta}
\[I_\pmq(\phi^\star_\pmq\ot\wtd\phi^\star_\pmq)=
\frac{L(1,\itPi_\pmq,\ad)}{\zeta_\pmq(2)^2L(1/2,\itPi_\pmq)}\int_{\PGL_2(\Q_\pmq)}\frac{\bfb_\pmq(\itPi_\pmq(g_\pmq)\phi^\star_\pmq,\wtd\phi^\star_\pmq)}{\bfb_\pmq(\itPi_\pmq(\Tau_{\ulN,\pmq})\phi_\pmq,\wtd\phi_\pmq)}\rmd g_\pmq\]  under \hypref{H:ram}. For $i=1,2,3$, let $c_i=c(\pi_i)$ be the exponent of the conductors. Note that $\om_{\fgh,\pmq}^{1/2}$ is unramified, so under \hypref{H:ram} and the condition \eqref{sf},
 we may assume by symmetry that 
\[c_1\leq \min\stt{c_2,c_3,1};
 \quad \pi_3\text{ is minimal},\]
and that $\stt{\pi_1,\pi_2,\pi_3}$ satisfies one of the following conditions:
\begin{itemize}
\item Case (Ia):  $\pi_3=\Prin{\Mu_3}{\upsilon_3}$ is a principal series with $\chi_3$ unramified character of $\Q_\pmq^\x$.
\item Case (Ib):  $\pi_1,\pi_2$ and $\pi_3$ are discrete series.
\item Case (IIa): $\pi_1$ is a principal series; $\pi_2$ and $\pi_3$ are discrete series with $L(s,\pi_2\ot\pi_3)\not=1$. 
%{\color{red} $(\iff \pmq\divides \Bd^{\irr}_{gh}\Bd^{\irr}_{fg}\Bd^{\irr}_{fh})$.}
\item Case (IIb): $\pi_1$ is a principal series; $\pi_2$ and $\pi_3$ are discrete series with $L(s,\pi_2\ot\pi_3)=1$. 
%{\color{red} $(\iff \pmq\divides \Bd^{\red}_{gh}\Bd^{\red}_{fg}\Bd^{\red}_{fh})$.}
\end{itemize}
For $i=1,2,3$, let $\xi_i\in \cV_{\pi_i}^{\rm new}$ and $\wtd\xi_i\in \Contra{\pi}_i(\tau_{c_i})\cV_{\Contra{\pi}_i}^{\rm new}$ be new vectors. Set \[c^*=\max\stt{c_2,c_3}>0.\] We recall the following choices of local test vectors $\phi^\star_\pmq\in\cV_{\itPi_\pmq}$ and $\wtd\phi^\star_\pmq\in\cV_{\Contra{\itPi}_\pmq}$ in \eqref{E:factorization1} and \eqref{E:factorization2} according to the polynomials $\cQ_{i,\pmq}(X)$ for $i=1,2,3$ in \eqref{E:defQ}. Put
\[w=\pMX{0}{1}{-1}{0},\quad \eta=\pDII{\pmq^{-1}}{1}\text{ and }\tau_n=\pMX{0}{1}{-\pmq^n}{0}\text{ for }n\in\Z.\]
\begin{itemize}
\item Case (Ia) and (Ib):
\begin{align*}\phi^\star_\pmq=&\xi_1\ot\pi_2(\eta^{c^*-c_2})\xi_2\ot\pi_3(\eta^{c^*-c_3})\xi_3,\\
\wtd\phi^\star_\pmq=&\om_2(\pmq^{c_2-c^*})\om_3(\pmq^{c_3-c^*})\cdot \wtd\xi_1\ot\Contra{\pi}_2(\eta^{c^*-c_2})\wtd\xi_2\ot\Contra{\pi}_3(\eta^{c^*-c_3})\wtd \xi_3.
\end{align*}
\item Case (IIa): Let $r={\lceil\frac{c^*}{2}\rceil}$. Then
\begin{align*}\phi^\star_\pmq=\pi_1(\eta^r)\xi_1\ot\xi_2\ot\xi_3,\quad
\wtd\phi^\star_\pmq=\om_1(\pmq^{-r})\cdot\Contra{\pi}_1(\eta^r)\wtd\xi_1\ot\wtd\xi_2\ot\wtd\xi_3.
\end{align*}
\item Case (IIb): 
If $c_1=0$, then let $\upsilon_1:\Qq^\x\to\C^\x$ be the unramified character with $\upsilon_1(\pmq)=\beta_\pmq(f)\abs{\pmq}^\frac{k_1-1}{2}$, where $\beta_\pmq(f)$ is the specialization of $\beta_\pmq(\bdsf)$ at $\Qx$ in \defref{D:testunb} and we have
\begin{align*}\phi^\star_\pmq=&(\pi_1(\eta^{c^*})\xi_1-\upsilon_1^{-1}\Abs^{\onehalf}(\pmq)\pi_1(\eta^{c^*-1})\xi_1)\ot\xi_2\ot\xi_3,\\
\wtd\phi^\star_\pmq=&\om_1(\pmq^{-c^*})\cdot(\Contra{\pi}_1(\eta^{c^*})\wtd\xi_1-\om_1\upsilon_1^{-1}\Abs^{\onehalf}(\pmq)\Contra{\pi}_1(\eta^{c^*-1})\wtd\xi_1)\ot\wtd\xi_2\ot\wtd\xi_3.\\
\end{align*}
If $c_1>0$, then 
\[\phi^\star=\pi_1(\eta^{c^*-c_1})\xi_1\ot\xi_2\ot\xi_3,\quad \wtd\phi^\star_\pmq=\om_1(\pmq^{c_1-c^*})\cdot\Contra{\pi}_1(\eta^{c^*-c_1})\wtd\xi_1\ot\wtd\xi_2\ot\wtd\xi_3.\]
\end{itemize}
In what follows, we let $W_i=W_{\pi_i}\in\cW(\pi_i)^{\rm new}$ be the normalized Whittaker newforms and let $\wtd W_i=W_{\pi_i}\ot\om_i^{-1}$ for $i=1,2,3$. For a non-negative integer $n$, put
\[\cU_0(\pmq^n)=\GL_2(\Z_\pmq)\cap\pMX{\Z_\pmq}{\Z_\pmq}{\pmq^n\Z_\pmq}{\Z_\pmq}.\]

\subsection{The ramified case (Ia)}\label{S:Ia}
In the case (Ia), $\pi_3=\Prin{\chi_3}{\upsilon_3}$ is a principle series with $\cond{\Mu_3}=0$.

\begin{prop}\label{P:ramifiedI} In case (Ia), we have 
\begin{align*}
I_\pmq(\phi^\star_\pmq\ot\wtd\phi^\star_\pmq)
=& \varepsilon(1/2,\pi_1\ot\pi_2\ot\Mu_3)\cdot \chi_3^{-2}\Abs(\pmq^{c^*})\om_3(-1)\varepsilon(1/2,\pi_3)^2\cdot \frac{1}{B_{\itPi_\pmq}}\cdot\frac{\zeta_\pmq(2)^2}{\zeta_\pmq(1)^2}.
\end{align*}
\end{prop}
\begin{proof}  In this case, $c_3=\cond{\om_3}=\cond{\om_1\om_2}\leq c_2$, so $c^*=c_2$. We use the realizations
\[\cV_{\itPi_\pmq}=\cW(\pi_1)\boxtimes\cW(\pi_2)\boxtimes \cB(\chi_3,\upsilon_3);\quad \cV_{\Contra{\itPi}_\pmq}=\cW(\Contra{\pi}_1)\boxtimes\cW(\Contra{\pi}_2)\boxtimes \cB(\chi_3^{-1},\upsilon_3^{-1}).\]
Let $f_3\in \sB(\Mu_3,\upsilon_3)^{\rm new}$ be the new section normalized so that $f_3(1)=1$ and $\wtd f_3=M^*(\chi_3,\upsilon_3)f_3\ot\om_3^{-1}$. Let \[f^\star_3=\rho(\pDII{\pmq^{c_3-c_2}}{1}f_3;\quad \wtd f^\star_3=\om_3(\pmq^{c_3-c_2})\cdot\rho(\pDII{\pmq^{c_3-c_2}}{1}\wtd f_3=M^*(\Mu_3,\upsilon_3)f^\star_3\ot\om_3^{-1}.\] We thus have
\[\phi^\star_\pmq=W_1\ot W_2\ot f^\star_3;\quad \wtd\phi^\star_\pmq=\wtd W_1\ot\wtd W_2\ot \wtd f^\star_3.\]
By \corref{C:IchinoRS}, 
\beq\label{E:1.local2}\begin{aligned}I_\pmq(\phi^\star_\pmq\ot\wtd\phi^\star_\pmq)=&\frac{\sJ_\pmq(W_1\ot W_2\ot f^\star_3,\wtd W_1\ot \wtd W_2\ot\wtd f^\star_3)}{\zeta_\pmq(2)^2L(1/2,\itPi_\pmq)\cdot B_{\itPi_\pmq}}\cdot\frac{\pair{\rho(\tau_{\cpithree})W_3}{\wtd W_3}}
{\pair{\rho(\tau_{\cpithree})f_3}{\wtd f_3}}\cdot\frac{\zeta_\pmq(2)^3}{\zeta_\pmq(1)^3}\\
=&\frac{I^*_\pmq}{B_{\itPi_\pmq}}\cdot\frac{\pair{\rho(\tau_{\cpithree})W_3}{\wtd W_3}}
{\pair{\rho(\tau_{\cpithree})f_3}{\wtd f_3}}\cdot\frac{\zeta_\pmq(2)^3}{\zeta_\pmq(1)^3}
,\end{aligned}\eeq
where \[I^*_\pmq=
\frac{\zeta_\pmq(1)\cdot\chi_3(-1)\gamma(1/2,\pi_1\ot\pi_2\ot\chi_3)}{\zeta_\pmq(2)^2L(1/2,\itPi_\pmq)}\cdot \Psi(W_1,W_2,f_3)^2
\]
There are three subcases:
\begin{itemize}
\item[(a)] $\upsilon_3$ is ramified,
\item[(b)] $\upsilon_3$ is unramified and $L(s,\pi_2)=L(s,\chi_2)$ for some unramified character $\chi_2$,
\item[(c)] $\upsilon_3$ is unramified and $L(s,\pi_2)=1$.
\end{itemize}

Subcase (a): In this case, $f_3\in\sB(\Mu_3,\upsilon_3)$ is given by
\[f_3(\pMX{1}{0}{x}{1})=\bbI_{q^{c_3}\Z_\pmq}(x)\]
by \cite[Prop.\,2.1.2]{Schmidt02RJ}. We have $W_2(\pDII{y}{1})=\bbI_{\Z_\pmq^\x}(y)$ if $L(s,\pi_2)=1$ and $W_2(\pDII{y}{1})=\Mu_2\Abs^\onehalf(y)\bbI_{\Z_\pmq}(y)$ if $L(s,\pi_2)=L(s,\Mu_2)$ for some unramified character $\Mu_2$. In any case, the integral $\Psi(W_1,W_2,f^\star_3)$ equals 
\begin{align*}
&\frac{\zeta_\pmq(2)\Mu_3\Abs^\onehalf(\pmq^{c_3-c_2})}{\zeta_\pmq(1)}\int_{\Qq^\x}\int_{\Qq}W_1(\pDII{y}{1}\pMX{1}{0}{x}{1})W_2(\pDII{-y}{1}\pMX{1}{0}{x}{1})\Mu_3\Abs^{-\onehalf}
(y)f_3(\pMX{1}{0}{\pmq^{c_3-c_2}x}{1})\rmd x\rmd^\x y\\
=&\frac{\zeta_\pmq(2)\Mu_3\Abs^\onehalf(\pmq^{c_3-c_2})}{\zeta_\pmq(1)}\abs{\pmq}^{c_2}\int_{\Qq^\x}W_1(\pDII{y}{1})W_2(\pDII{-y}{1})\Mu_3\Abs^{-\onehalf}(y)\rmd^\x y.\end{align*}
Note that $\pi_1$ and $\pi_2$ can not be both unramified special representations as $c(\pi_1)\leq 1$ and $\upsilon_3$ is ramified. A standard calculation together with the recipe of local $L$-factors for $\GL(2)\times\GL(2)$ in \cite[Proposition 1.4]{GJ78} shows that
\[\int_{\Qq^\x}W_1(\pDII{y}{1})W_2(\pDII{-y}{1})\Mu_3\Abs^{-\onehalf}(y)\rmd^\x y=L(1/2,\pi_1\ot\pi_2\ot\Mu_3).\]
We obtain
\[\Psi(W_1,W_2,f^\star_3)=\frac{\zeta_\pmq(2)\Mu_3\Abs^\onehalf(\pmq^{c_3-c_2})}{\zeta_\pmq(1)}\abs{\pmq}^{c_2}L(1/2,\pi_1\ot\pi_2\ot\Mu_3),\]
and hence
\[I^*_\pmq=\Mu_3^{-2}\Abs(\pmq^{c_2})\cdot\Mu_3^2\Abs(\pmq^{c_3})\cdot\varepsilon(1/2,\pi_1\ot\pi_2\ot\Mu_3).\]
Substituting the above equation and the formula \lmref{L:basic3} below to \eqref{E:1.local2}, we obtain the expression of $I_\pmq(\phi^\star_\pmq\ot\wtd\phi^\star_\pmq)$ as claimed in this subcase.

Subcase (b) and (c): Next we consider the case $\upsilon_3$ is unramified, so $\pi_1$ and $\pi_3$ are spherical ($c_1=c_3=0$). Note that in Subcase (b) where $L(s,\pi_2)=L(s,\chi_2)$ for $\chi_2$ an unramified character, we must have $\pi_2=\chi_2\Abs^{-\onehalf}{\rm St}$ is an unramified special representation. Define the function $\sF:ZN\bksl G/K_0(\pmq^{c_2})\to\C$ by
\[\sF(g)=W_1(g)W_2(\pDII{-1}{1}g)f_3(g\pDII{\pmq^{-c_2}}{1}).\]
We have
\begin{align*}
\Psi(W_1,W_2,f^\star_3)=&\frac{\zeta_\pmq(2)}{\zeta_\pmq(1)}\int_{\Q_\pmq^\x}\int_{\Q_\pmq}\sF(\pDII{y}{1}\pMX{1}{0}{x}{1})\rmd x\frac{\rmd^\x y}{\abs{y}}\\
=&\frac{\zeta_\pmq(2)}{\zeta_\pmq(1)}\cdot (J_0^-+J_{c_2}^++\sum_{n=1}^{c_2-1}J_n),\end{align*}
where
\begin{align*}
J_0^-=&\int_{\Q_\pmq^\x}\int_{\abs{x}\geq 1}\sF(\pDII{y}{1}\pMX{1}{0}{x}{1})\rmd x\frac{\rmd^\x y}{\abs{y}},\\
J_n=&\int_{\Q_\pmq^\x}\int_{\pmq^n\Z_\pmq^\x}\sF(\pDII{y}{1}\pMX{1}{0}{x}{1})\rmd x\frac{\rmd^\x y}{\abs{y}},\\
J_{c_2}^+=&\int_{\Q_\pmq^\x}\int_{\abs{x}\leq \abs{\pmq}^{c_2}}\sF(\pDII{y}{1}\pMX{1}{0}{x}{1})\rmd x\frac{\rmd^\x y}{\abs{y}}.\end{align*}
Using the identity \[\pDII{y}{1}\pMX{1}{0}{x}{1}=(-x)\cdot \pMX{yx^{-2}}{x^{-1}}{0}{1}w\pMX{1}{x^{-1}}{0}{1}\]
and the formula \[W_2(\pDII{y}{1}w)=
\begin{cases} -\abs{\uf}\chi_2\Abs^\onehalf(y)\bbI_{\Z_\pmq}(y)&\text{ in subcase (b)},\\
\varepsilon(1/2,\pi_2)\cdot\om_2(\pmq^{-c_2})\bbI_{\pmq^{-c_2}\Z_\pmq^\x}(y)&\text{ in subcase (c)},\end{cases}\]
we find that
\begin{align*}
J_0^-=&\int_{\Q_\pmq^\x}\sF(\pDII{y}{1}w)\frac{\rmd^\x y}{\abs{y}}\\
=&\upsilon_3^{-1}\Abs^\onehalf(\pmq^{c_2})\int_{\Q_\pmq^\x}W_1(\pDII{y}{1})W_2(\pDII{y}{1}w)\chi_3\Abs^{-\onehalf}(y)\rmd^\x y\\
=&-\abs{\pmq}\cdot\upsilon_3^{-1}\Abs^\onehalf(\pmq^{c_2})\cdot \begin{cases} 
L(1/2,\pi_1\ot\chi_2\chi_3)&\text{ in subcase (b),}\\
0&\text{ in subcase (c)}.
\end{cases}
\end{align*}
On the other hand, it is easy to see that
\begin{align*}
J_{c_2}^+=&\abs{\pmq}^{c_2}\int_{\Q_\pmq^\x}\sF(\pDII{y}{1})\frac{\rmd^\x y}{\abs{y}}=\chi_3^{-1}\Abs^\onehalf(\pmq^{c_2})L(1/2,\pi_1\ot\pi_2\ot \chi_3).\end{align*}
It remains to calculate $J_n$ in subcase (c). We have
\begin{align*}
J_n
=&\sum_{m\in\Z} (1-\abs{\pmq})\abs{\pmq}^{n}\chi_3\Abs^\onehalf(\pmq^{-c_2})W_1(\pDII{\pmq^m}{1})\chi_3\Abs^{-\onehalf}(\pmq^m)f_3(\pMX{1}{0}{\pmq^{n-c_2}}{1})A^{(m)}_n(\bfone),
\end{align*}
where
\[A^{(m)}_n(\bfone)=\int_{\Z_\pmq^\x}W_2(\pDII{\pmq^m y}{1}\pMX{1}{0}{\pmq^{n}}{1})\rmd^\x y.\]
By \lmref{L:Kirillov.local} below, we find that $J_n=0$ unless $n=c_2-1$ and \[J_{c_2-1}=\chi_3^{-1}\Abs^\onehalf(\pmq^{c_2})\cdot\chi_3\upsilon_3^{-1}\Abs(\pmq).\] 
Combining the above calculations, in either subcase (b) or subcase (c), we obtain 
\begin{align*}\Psi(W_1,W_2,f_3)=&\frac{\zeta_\pmq(2)}{\zeta_\pmq(1)}\cdot (J_0^-+J_{c_2}^++\sum_{n=1}^{c_2-1}J_n)=
\frac{\zeta_\pmq(2)\chi_3^{-1}\Abs^\onehalf(\pmq^{c_2})}{\zeta_\pmq(1)}\cdot\frac{L(1/2,\pi_1\ot\pi_2\ot\chi_3)}{L(1,\Mu_3\upsilon_3^{-1})}.
\end{align*}
This shows that 
\[I^*_\pmq=\frac{\zeta_\pmq(1)\gamma(1/2,\pi_1\ot\pi_2\ot\chi_3)\Psi(W_1,W_2,f^\star_3)^2}{\zeta_\pmq(2)^2L(1/2,\itPi_\pmq)}=
\frac{\chi_3^{-2}\Abs(\pmq^{c_2})\varepsilon(1/2,\pi_1\ot\pi_2\ot\chi_3)}{\zeta_\pmq(1)L(1,\chi_3\upsilon_3^{-1})^2}.
\]
The above equation with \lmref{L:basic3} below and \eqref{E:1.local2} yield that
\[I_\pmq(\phi^\star_\pmq\ot\wtd\phi^\star_\pmq)= \chi_3^{-2}\Abs(\pmq^{c_2})\varepsilon(1/2,\pi_1\ot\pi_2\ot\Mu_3)\cdot\frac{1}{B_{\itPi_\pmq}}\cdot\frac{\zeta_\pmq(2)^2}{\zeta_\pmq(1)^2}.\]
This completes the proof.
\end{proof}

\begin{lm}\label{L:basic3}Let $\pi$ be a constituent of $\cB(\chi,\upsilon)$ of central character $\om$. Suppose that $\chi$ is unramified. Let $c=c(\pi)$ be the exponent of the conductor. Let $W_\pi$ be the new vector in $\sW(\pi)^{\rm new}$ with $W_\pi(1)=1$ and $\wtd W_\pi=W_\pi\ot\om^{-1}$. Let $f\in\cB(\upsilon,\chi)$ and $\wtd f=M^*(\chi,\upsilon)f\ot\om^{-1}$. \begin{enumerate}
\item Suppose that $\pi$ is a principal series and $f\in\cB(\chi,\upsilon)^{\rm new}$ is the new section with $f(1)=1$. Then 
\[\frac{\pair{\rho(\tau_c)W_\pi}{\wtd W_\pi}}{\pair{\rho(\tau_{c})f}{\wtd f}}=\Mu^2\Abs(\pmq^{-c})\varepsilon(1/2,\pi)^2\om(-1)\cdot L(1,\chi\upsilon^{-1})^2\cdot \frac{\zeta_\pmq(1)^2}{\zeta_\pmq(2)}.\]
\item Suppose that $\pi$ is an unramified special representation with $\chi\upsilon^{-1}=\Abs^{-1}$, \ie $\pi=\upsilon\Abs^{-\onehalf}{\rm St}$. Let $f$ be the section in $\cB(\chi,\upsilon)^{\cU_0(\pmq)}$ with $f(w)=1$. Then
\[\frac{\pair{\rho(\tau_c)W_\pi}{\wtd W_\pi}}{\pair{\rho(\tau_c)f}{\wtd f}}=\frac{\zeta_\pmq(1)^2}{\zeta_\pmq(2)}.\]\end{enumerate}
\end{lm}
\begin{proof}We first consider the case $\pi$ is a principal series. Suppose that $c=0$. Then we have  \begin{align*}
\pair{f}{M^*(\Mu,\upsilon)f\ot\om^{-1}}=&\gamma(0,\Mu\upsilon^{-1})\frac{L(0,\Mu\upsilon^{-1})}{L(1,\Mu\upsilon^{-1})}
=\frac{L(1,\Mu^{-1}\upsilon)}{L(1,\Mu\upsilon^{-1})},\\
\pair{W_\pi}{W_\pi}=&\frac{L(1,\pi,\Ad)\zeta_\pmq(1)}{\zeta_\pmq(2)},
\end{align*}
and hence
\[\frac{\pair{W_\pi}{\wtd W_\pi}}{\pair{f}{\wtd f}}=L(1,\Mu\upsilon^{-1})^2\cdot\frac{\zeta_\pmq(1)^2}{\zeta_\pmq(2)}.\]
Suppose that $c>0$. Then $\upsilon$ is ramified and $f$ is supported in $B\cU_0(\pmq^c)$ (\cf\cite[Proposition 2.1.2]{Schmidt02RJ}), and hence $\pair{\rho(\tau_c)f}{\wtd f}=\pair{f}{\rho(\tau_c^{-1})\wtd f}$ equals \begin{align*}
\int_{K}f(k)\wtd f(k\tau_c^{-1})\rmd k
=&\vol(\cU_0(\pmq^c))\cdot \wtd f(\tau_c^{-1})\\
=&\vol(\cU_0(\pmq^c))\cdot \om(\pmq^c)\cdot M^*(\Mu,\upsilon)f(\tau_c^{-1})\\
=&\vol(\cU_0(\pmq^c))\cdot \om(\pmq^c)\cdot\gamma(0,\Mu\upsilon^{-1})f(\pDII{1}{\pmq^{-c}})\\
=&\abs{\pmq}^c(1+\abs{\pmq})^{-1}\varepsilon(1/2,\Mu\upsilon^{-1})\cdot \Mu(\pmq^c).
\end{align*}
In addition, $\pair{\rho(\tau_c)W_{\pi}}{\wtd W_{\pi}}=\varepsilon(1/2,\pi)\zeta_\pmq(1)$, so we obtain  that
\[\frac{\pair{\rho(\tau_c)W_{\pi}}{\wtd W_{\pi}}}{\pair{\rho(\tau_c)f}{\wtd f}}=\frac{\zeta_\pmq(1)^2}{\zeta_\pmq(2)}\Mu^2\Abs(\pmq^{-c})\varepsilon(1/2,\pi)^2\om(-1).\]

Now we consider the case $\pi$ is an unramified special representation. Then $c=1$ and we may assume $f(w)=1$, \ie $f$ is supported in $Bw\cU_0(\pmq)$. An elementary computation shows that 
\[M^*(\chi,\upsilon)f(1)=\zeta_\pmq(2)(1-\abs{\pmq}^{-1});\quad M^*(\chi,\upsilon)f(w)=\zeta_\pmq(2).\]
Then $\pair{\rho(\tau_1)f}{\wtd f}$ equals
\begin{align*}
\int_{w\cU_0(\pmq)}f(k\tau_1)\wtd f(k)\rmd k=vol(\cU_0(\pmq))\cdot f(\tau_1)M^*(\chi,\upsilon)f(1)=(-\upsilon\Abs^{-\onehalf}(\pmq))\cdot\frac{\zeta_\pmq(2)^2}{\zeta_\pmq(1)^2}.
\end{align*}
Combined with the formulas
\[\pair{\rho(\tau_1)W_\pi}{\wtd W_\pi}=\varepsilon(1/2,\pi)\zeta_\pmq(2)=(-\upsilon\Abs^{-\onehalf}(\pmq))\cdot\zeta_\pmq(2),\]
the lemma in this case follows.
\end{proof}

\begin{lm}\label{L:Kirillov.local} Let $\pi$ is an irreducible admissible generic representation of $\GL_2(\Q_\pmq)$ and let $W_\pi\in\sW(\pi)^{\rm new}$ be the normalized Whittaker newform with $W_\pi(1)=1$. Let $\Mu:\Q_\pmq^\x\to\C^\x$ with $\Mu(q)=1$. Suppose that $L(s,\pi)=L(s,\pi\ot\chi)=1$. Put
\[A_n^{(m)}(\chi):= \int_{\Z_\pmq^\x}W_\pi(\pDII{\pmq^m y}{1}\pMX{1}{0}{\pmq^n}{1})\Mu(y)\rmd^\x y.\]
 If $\chi\not =1$, then $A_n^{(m)}(\chi)=0$ unless $m=\cond{\pi}-\cond{\pi\ot\Mu}$ and $n=\cond{\pi}-\cond{\Mu}$; in this case 
\[A_{c(\pi)-c(\Mu)}^{(c(\pi)-c(\pi\ot\Mu))}(\chi)=\varepsilon(1,\Mu)\cdot\frac{\varepsilon(1/2,\pi)}{\varepsilon(1/2,\pi\ot \Mu)}\cdot \Mu(-1)\zeta_\pmq(1).\]
If $\chi=\bfone$ is the trivial character, then $A_n^{(m)}(\bfone)=0$ unless $m=0$ and $n\geq c(\pi)-1$; in this case,
\[A_{c(\pi)-1}^{(m)}(\bfone)=-\abs{\pmq}\zeta_\pmq(1)\text{ and } A_n^{(m)}(\bfone)=1\text{ if }n\geq c(\pi).\]
\end{lm}
\begin{proof}
Let $A_n^{(m)}=A_n^{(m)}(\chi)$ and $c=\cond{\pi}$. Let $\varphi_n(a):=W_\pi(\pDII{a}{1}\pMX{1}{0}{\pmq^n}{1})$ for $a\in \Qq^\x$. Then $\varphi_n$ belongs to the Krillov model $\cK(\pi)$ of $\pi$ with respect to $\psi_{\Qq}$. Since $L(s,\pi)=1$, $\varphi:=\bbI_{\Z_\pmq^\x}$ is a new vector in $\cK(\pi)$ and $\cK(\Contra{\pi})$ (\cf\cite[\S 2.4]{Schmidt02RJ}). Then $\pi(\pMX{0}{1}{-\pmq^c}{0})\varphi(a)\om^{-1}(a)=\al\cdot \varphi(a)$ for some $\al\in\C^\x$. By the functional equation, we have
\[\int_{\Q_\pmq^\x}\varphi_n(a)\Mu(a)\abs{a}^{s-\onehalf}\rmd^\x a=\gamma(s,\pi\ot\Mu)^{-1}\int_{\Q_\pmq^\x}\pi(w)\varphi_n(a)\om^{-1}(a)\Mu^{-1}(a)\abs{a}^{\onehalf-s}\rmd^\x a,\]
where
\[\gamma(s,\pi\ot\Mu)^{-1}=\frac{L(s,\pi\ot\Mu)}{L(1-s,\Contra{\pi}\ot\Mu^{-1})\varepsilon(s,\pi\ot\Mu)}.\]
By the relation \begin{align*}\pi(w)\varphi_n(a)=&\pi(\pMX{1}{-\pmq^n}{0}{1}w)\varphi(a)
=\psi_{\Qq}(-a\pmq^n)\pi(\pDII{1}{\pmq^{-c}}\pMX{0}{1}{-\pmq^{c}}{0})\varphi(a)\\
=&\al\cdot \psi_{\Qq}(-a\pmq^n)\cdot\varphi(\pmq^{c}a)\om(a),
\end{align*}
we find that $\al=\varepsilon(1/2,\pi)$ and
\begin{align*}&\sum_{m\in\Z}\int_{\Z_\pmq^\x}\varphi_n(\pmq^my)\Mu(y)\rmd^\x y\cdot \Mu(\pmq^m)\abs{\pmq^m}^{s-\onehalf}\\
=&
\gamma(s,\pi\ot\Mu)^{-1}\varepsilon(1/2,\pi)\cdot\abs{\pmq^c}^{s-\onehalf}\cdot\int_{\Qq^\x}\psi_{\Qq}(-\frac{a}{\pmq^{c-n}})\Mu^{-1}\Abs^{\onehalf-s}(a)\varphi(a)\rmd^\x a.
\end{align*}
Let $t=\abs{\pmq}^s$. From the above equation, we deduce that
\begin{align*}&\sum_{m\in\Z}A_m^{(n)}\cdot \Mu(\pmq^m)\abs{\pmq^m}^{-\onehalf}\cdot t^m\\
=&\gamma(s,\pi\ot\Mu)^{-1}\cdot \varepsilon(1/2,\pi)\Mu(-1)\abs{\pmq^c}^{-\onehalf}\cdot t^c\cdot \begin{cases}
0&\text{ if }c-n\not=c(\Mu)>0\text{ or }c-n\geq 2,\,c(\Mu)=0,\\
\Mu(\pmq^{-c(\Mu)})\varepsilon(1,\Mu)\zeta_\pmq(1)&\text{ if }c-n=c(\Mu)>0,\\
1&\text{ if }c-n\leq 0,\,c(\Mu)=0,\\
-\abs{\pmq}\zeta_\pmq(1)&\text{ if }c-n=1,\,c(\Mu)=0.
\end{cases}\end{align*}
 Since $L(s,\pi)=L(s,\pi\ot\chi)=1$, we have
 \[\gamma(s,\pi\ot\Mu)^{-1}=\varepsilon(0,\pi\ot\Mu)^{-1} t^{-c(\pi\ot\Mu)}.\]
Comparing the coefficients of $t^m$, if $\Mu\not =\bfone$, we find that $A^{(m)}_n\not =0$ only when $c-n=c(\Mu)$, and $m=c-c(\pi\ot\Mu)$. In this case
\[A^{(m)}_n=\Mu(-\pmq^{-m-c(\Mu)})\abs{\pmq}^{\frac{m-c}{2}}\varepsilon(1/2,\pi)\cdot\frac{\varepsilon(1,\Mu)}{\varepsilon(0,\pi\ot\Mu)}\zeta_\pmq(1). \]
If $\chi=\bfone$, and $A^{(m)}_n=0$ unless $m=0$, and 
\[A^{(0)}_n=\begin{cases}1,&\text{ if }c-n\leq 0\\
-\abs{\pmq}\zeta_\pmq(1),&\text{ if }c-n=1\\
0&\text{ if }c-n\geq 2.
\end{cases}\]
This completes the proof.
\end{proof}

%%%%%%%%%%%%

\subsection{The case (Ib)}\label{S:Ib}
In this case $\pi_1=\chi_1\Abs^{-\onehalf}{\rm St}$ is an unramified special representation, and $\pi_2$ and $\pi_3$ are discrete series with the local root number $\varepsilon(1/2,\itPi_\pmq)=1$. We first remark that if $L(s,\pi_2\ot\pi_3)\not =1$, then by the minimality of $\pi_3$ combined with \cite[Proposition (1.2)]{GJ78}, this implies that $\pi_3=\Contra{\pi}_2\ot\sigma$ for some unramified character $\sigma$ of $\Q_\pmq^\x$ and $\pi_2$ is also minimal. Hence, in view of \cite[Proposition 8.5]{Prasad90} $\pi_2$ and $\pi_3$ must be unramified special in case (Ib) if $L(s,\pi_2\ot\pi_3)\not =1$.

\begin{prop}\label{P:ram4}In case (Ib),\begin{mylist}\item if $L(s,\pi_2\ot\pi_3)=1$, then we have \[I_\pmq(\phi_\pmq^\star\ot\wtd\phi_\pmq^\star)=\chi_1^2\Abs(\pmq^{c^*})\varepsilon(1/2,\pi_2\ot\pi_3\ot\upsilon_1)\varepsilon(1/2,\pi_2)^2\varepsilon(1/2,\pi_3)^2\cdot\frac{1}{B_{\itPi_\pmq}}\cdot \frac{\zeta_\pmq(2)^2}{\zeta_\pmq(1)^2};\]
\item if $L(s,\pi_2\ot\pi_3)\not =1$, then $c_1=c_2=c_3=1$ and 
\[I_\pmq(\phi_\pmq^\star\ot\wtd\phi_\pmq^\star)=I_\pmq(\phi_\pmq\ot\wtd\phi_\pmq)=\frac{2\abs{\pmq}}{B_{\itPi_\pmq}}\cdot\frac{\zeta_\pmq(2)^2}{\zeta_\pmq(1)^2}.\]\end{mylist}
\end{prop}
\begin{proof}
Now we suppose that $\pi_1=\chi_1\Abs^{-\onehalf}{\rm St}$ is unramified special. Let $\upsilon_1=\chi_1\Abs^{-1}$. We use the realizations
\[\cV_{\itPi_\pmq}=\cB(\upsilon_1,\chi_1)^0\boxtimes\cW(\pi_2)\boxtimes\cW(\pi_3) ;\quad \cV_{\Contra{\itPi}_\pmq}=\cB(\upsilon_1^{-1},\chi_1^{-1})_0\boxtimes\cW(\Contra{\pi}_2)\boxtimes\cW(\Contra{\pi}_3).\]
Here $\cB(\upsilon_1,\chi_1)^0$ is the unique irreducible quotient space of $\cB(\upsilon_1,\chi_1)$ and $\cB(\upsilon_1^{-1},\chi_1^{-1})_0$ is the unique irreducible sub-representation of $\cB(\upsilon_1^{-1},\chi_1^{-1})$ as in \subsecref{S:notation.5}.
Let $f_1\in\cB(\upsilon_1,\chi_1)^{\cU_0(\pmq)}$ be the unique function supported in $BwN(\Z_\pmq)$ with $f_1(1)=1$. Then the holomorphic image $f_1^0$ of $f_1$ in $\cV_{\pi_1}=\cB(\upsilon_1,\chi_1)^0$ is a new vector. Let $\wtd f_1=M^*(\upsilon_1,\chi_1)f\ot \om_1^{-1}$. We may assume that $c_2\geq c_3$ (so $c^*=c_2$). Let $W_3^\star=\rho(\pDII{\pmq^{c_3-c_2}}{1})W_3$ and $\wtd W_3^\star=W_3^\star\ot\om_3^{-1}$. Then 
\[\phi_\pmq^\star=f_1^0\ot W_2\ot W^\star_3;\quad \wtd\phi_\pmq=\wtd f_1\ot \wtd W_2\ot\wtd W_3^\star.\]
By \corref{C:IchinoRS} and \lmref{L:basic3} (2), we obtain
\beq\label{E:4.local2}\begin{aligned}I_\pmq(\phi^\star_\pmq\ot\wtd\phi^\star_\pmq)=&\frac{\sJ_\pmq(W_2\ot W^\star_3\ot f_1,\wtd W_2\ot \wtd W^\star_3\ot\wtd f_1)}{\zeta_\pmq(2)^2L(1/2,\itPi_\pmq)\cdot B_{\itPi_\pmq}}\cdot\frac{\pair{\rho(\tau_{c_1})W_1}{\wtd W_1}}
{\pair{\rho(\tau_{c_1})f_1}{\wtd f_1}}\cdot\frac{\zeta_\pmq(2)^3}{\zeta_\pmq(1)^3}\\
=&\frac{\gamma(1/2,\pi_2\ot\pi_3\ot\upsilon_1)\Psi(W_2,W^\star_3,f_1)^2}{L(1/2,\itPi_\pmq)}\cdot\frac{1}{B_{\itPi_\pmq}}.
\end{aligned}\eeq
In what follows, if $L(s,\pi_2\ot\pi_3)\not =1$, then we write $\pi_2=\chi_2\Abs^{-\onehalf}{\rm St}$ and $\pi_3=\chi_3\Abs^{-\onehalf}{\rm St}$ with $\chi_2,\chi_3$ unramified.  Using the integration formula \eqref{E:intformula}, we find that $\Psi(W_2,W^\star_3,f_1)$ equals
\begin{align*}&\frac{\zeta_\pmq(2)}{\zeta_\pmq(1)}\int_{\Qq^\x}\int_{\Qq}W_2(\pDII{y}{1}w\pMX{1}{x}{0}{1})W_3(\pDII{y}{1}w\pMX{\pmq^{c_3-c_2}}{x}{0}{1})\upsilon_1\Abs^\onehalf(y)\bbI_{\Z_\pmq}(x)\rmd x\frac{\rmd^\x y}{\abs{y}}\\
=&\frac{\zeta_\pmq(2)\chi_1\Abs^\onehalf(\pmq^{c_2})}{\zeta_\pmq(1)}\int_{\Qq^\x}W_2(\pDII{y}{1}\tau_{c_2})W_3(\pDII{y}{1}\tau_{c_3})\upsilon_1\Abs^\onehalf(y)\bbI_{\Z_\pmq}(x)\rmd x\frac{\rmd^\x y}{\abs{y}}\\
=&\frac{\zeta_\pmq(2)}{\zeta_\pmq(1)}\cdot\chi_1\Abs^\onehalf(\pmq^{c_2})\cdot\varepsilon(1/2,\pi_2)\varepsilon(1/2,\pi_3)\begin{cases}
1&\text{ if }L(s,\pi_2\ot\pi_3)=1,\\
L(-1/2,\chi_1\chi_2\chi_3)&\text{ if }L(s,\pi_1\ot\pi_3)\not =1.
\end{cases}
\end{align*}
If $L(s,\pi_2\ot\pi_3)=1$, then one verifies easily that $\gamma(1/2,\pi_2\ot\pi_3\ot\upsilon_1)=\varepsilon(1/2,\pi_2\ot\pi_3\ot\upsilon_1)$ and $L(s,\itPi_\pmq)=1$, so we obtain the claimed expression of $I_\pmq(\phi_\pmq^\star\ot\wtd\phi_\pmq^\star)$ in this case by substituting the above equation into \eqref{E:4.local2}.

Suppose that $L(s,\pi_2\ot\pi_3)\not =1$. Then $c_1=c_2=c_3=1$ and $\varepsilon(1/2,\pi_i)=-\chi_i\Abs^{-\onehalf}(\pmq)$ for $i=1,2,3$. Hence, $W_3^\star=W_3$ and  
\[\Psi(W_2,W_3^\star,f_1)^2=\Psi(W_2,W_3,f_1)^2=\abs{\pmq}^2\cdot \frac{\zeta_\pmq(2)^2}{\zeta_\pmq(1)^2}\cdot L(-1/2,\chi_1\chi_2\chi_3)^2.\]
On the other hand, by \cite[Proposition 8.6]{Prasad90}, $\varepsilon(1/2,\itPi_\pmq)=1$ implies that 
\[\chi_1\chi_2\chi_3(\pmq)=-\abs{\pmq}^{\frac{3}{2}}.\]
By \cite[Proposition 1.4]{GJ78}, $\varepsilon(1/2,\pi_2\ot\pi_3\ot\upsilon_1)=\abs{\pmq}^{-1}$ and 
\[\frac{L(1/2,\pi_2\ot\pi_3\ot\chi_1)}{L(1/2,\pi_2\ot\pi_3\ot\upsilon_1)}=\frac{L(1/2,\chi_1\chi_2\chi_3)}{L(-3/2,\chi_1\chi_2\chi_3)}=2L(1/2,\chi_1\chi_2\chi_3),\]
and a simple computation of the Langlands parameter for $\itPi_\pmq$ shows
\[L(s,\itPi_\pmq)=L(s,\chi_1\chi_2\chi_3)L(s-1,\chi_1\chi_2\chi_3)^2.\]
We thus obtain
\[\gamma(1/2,\pi_2\ot\pi_3\ot\upsilon_1)=2\abs{\pmq}^{-1}\cdot \frac{L(1/2,\itPi_\pmq)}{L(-1/2,\chi_1\chi_2\chi_3)^2}.\]
The desired formula of $I_\pmq(\phi_\pmq^\star\ot\wtd\phi_\pmq^\star)=I_\pmq(\phi_\pmq\ot\wtd\phi_\pmq)$ in this case can be deduced immediately by combining \eqref{E:4.local2} with the above formulae of $\Psi(W_2,W_3^\star,f_1)$ and the $\gamma$-factor.
\end{proof}
\begin{Remark}In the case where $L(s,\pi_2\ot\pi_3)\not =1$, \ie $\pi_i$ are special unramified, the integral $I_\pmq(\phi_\pmq^\star\ot\wtd\phi_\pmq^\star)$ was computed in \cite[page 1405-1406] {II10GAFA}, from which we have $I_\pmq(\phi_\pmq^\star\ot\wtd\phi_\pmq^\star)=2\abs{\pmq}(1+\abs{\pmq})$. Our computation agrees with the result therein (note that $B_{\itPi_\pmq}=\zeta_\pmq(2)^3\zeta_\pmq(1)^{-3}$). 
\end{Remark}
\subsection{The ramified case (IIa)}\label{S:IIa}
In this case, $\pi_2,\pi_3$ are discrete series and $L(s,\pi_2\ot\pi_3)\not =1$. As we have remarked in the previous subsection, $\pi_3\simeq\Contra{\pi_2}\ot\sigma$ for some unramified character $\sigma$ of $\Q_\pmq^\x$ and $\pi_1$ must be spherical. Let $\tau_{\Q_{\pmq^2}}$ be the quadratic character associated with the unramified quadratic field extension $\Q_{\pmq^2}$ of $\Q_\pmq$. We say a discrete series $\pi$ is of type $\bfone$ if $\pi\iso\pi\ot\tau_{\Q_{\pmq^2}}$ and is of type $\mathbf 2$ if $\pi\not \simeq\pi\ot\tau_{\Q_{\pmq^2}}$. 

The following lemma for minimal supercuspidal representations should be well-known to experts. We include a proof here for the reader's convenience. \begin{lm}\label{L:basic1}Let $\pi$ be a minimal supercuspidal representation with central character $\om$.\begin{enumerate}
\item Let $\chi$ be a charatcer of $\Qq^\x$. Then we have the following conductor formula\[
c(\pi\ot\chi)=\begin{cases}c(\pi)&\text{ if }c(\pi)\geq 2c(\chi),\\
2c(\chi)&\text{ if }c(\pi)<2c(\chi).\end{cases}
\]
Here recall that $c(?)$ denotes the exponent of the conductor of $?$.
\item If $\pi$ is of type $\bfone$, then $\cond{\pi}$ is even and $L(s,\pi\ot\Contra{\pi})=\zeta_\pmq(2s)$. 
If $\pi$ is of type $\mathbf 2$, then $c(\pi)$ is odd and $L(s,\pi\ot\Contra{\pi})=\zeta_\pmq(s)$. 
\end{enumerate}
\end{lm}
\begin{proof}Let $c=c(\pi)\geq 2$. To prove the first assertion, we begin with an immediate consequence of \cite[Proposition 2.11 (i)]{JacquetLanglands70}. Let $\chi_0=\chi|_{\Z_\pmq^\x}$ and $\om_0=\om|_{\Z_\pmq^\x}$. If $\chi_0\om_0\not =1$, then there exists a character $\sigma$ such that 
 \beq\label{E:con}c(\pi\ot\sigma)=c+c(\pi\ot\chi)-2c(\chi\om)\eeq and if $\chi_0\not =1,\,\om_0^{-1}$, then either of the following condition holds:
\begin{enumerate} \item[(i)]
$\sigma|_{\Z_\pmq^\x}\not =1,\chi_0$ and \[c(\sigma)=c-c(\chi\om),\quad c(\sigma\chi^{-1})=c(\pi\ot\chi)-c(\chi\om),\]  
\item[(ii)] $\sigma|_{\Z_\pmq^\x}=1$, $c(\chi)=c(\pi\ot\chi)-c(\chi\om)$ and $c(\chi\om)-c\geq -1$;
\item[(iii)] $\sigma |_{\Z_\pmq^\x}=\chi_0$, $c(\chi)=c(\pi)-c(\chi\om)$ and $c(\chi\om)-c(\pi\ot\chi)\geq -1$.
\end{enumerate}
To see it, we set $\rho=\chi_0^{-1}\om_0^{-1}$, $\nu=\om_0^{-1}$, $m=c(\chi\om)$, $p=m-c(\pi\ot\chi)$ and $n=m-c(\pi)$ in the equality proved in \cite[Proposition 2.11 (i)]{JacquetLanglands70}, from which  we see immediately that the equality shows the existence of desired $\sigma$ by noting that $C_n(\rho^{-1}\om^{-1})\not =0$ if and only if $n=c(\pi\ot\rho)$. Note that \eqref{E:con} implies that
\[c(\pi\ot\chi)\geq 2c(\chi\om)\text{ for all }\chi\]
by the minimality of $\pi$. In particular, $c(\om)\leq c/2$. Suppose that $c(\chi)>c/2$. Then $c(\chi\om)=c(\chi)$ and $\sg$ satisfies either (i) or (ii). In case (ii), we have $c(\pi\ot\chi)=2c(\chi)$. In case (i), $c(\sigma)=c-c(\chi)<c/2$, and hence we also have $c(\pi\ot\chi)=c(\chi)+c(\sigma\chi^{-1})=2c(\chi)$. Now we suppose that  $c(\chi)\leq c/2$. If $\chi_0=\om_0^{-1}$, then $c(\pi\ot\chi)=c(\Contra{\pi})=c$, so we may assume $\chi_0\not =\om_0^{-1}$. It suffices to show $c(\pi\ot\chi)\leq c$. Note that $c(\chi\om)\leq c/2$. In case (iii), $c(\pi\ot\chi)\leq c(\chi\om)+1\leq c$, and in the case (ii), $c(\pi\ot\chi)=c(\chi)+c(\chi\om)\leq c$. We consider case (i). We have $c(\sigma)=c-c(\chi\om)\geq c/2$. If $c(\sigma)>c/2$, then \[c(\pi\ot\chi)=c(\chi\om)+c(\sigma^{-1}\chi)=c(\chi\om)+c(\sigma)=c.\]
If $c(\sigma)=c/2$, then we also have $c(\pi\ot\chi)\leq c/2+c/2=c$. This finishes the proof of the first assertion.

We proceed to show the second assertion. This is \cite[Proposition 6.1]{Hida90Michigan}. We give a more elementary proof. The local $L$-factor of $L(s,\pi\ot\Contra{\pi})$ is given in \cite[Corollary (1.3)]{GJ78}.  To see the parity of the conductor, we note that $\pi\iso\pi\ot\tau_{\Q_{q^2}}$ if and only if $\varepsilon(s,\pi\ot\chi)=\varepsilon(s,\pi\ot\chi\tau_{\Q_{q^2}})$ for all character $\chi:\Qq^\x\to\C^\x$ as $\pi$ is supercuspidal. Since $\tau_{\Q_{q^2}}$ is unramified, this is equivalent to saying $(-1)^{c(\pi\ot\chi)}=1$ for all $\chi$. It follows from part (1) that $\pi$ is of type $\bfone$ if and only if $c(\pi)$ is even.
\end{proof}

\begin{prop}\label{P:ramifiedIII}Let $r=\left\lceil{\frac{\cond{\pi_2}}{2}}\right\rceil$. We have
\[I_\pmq(\phi^\star_\pmq\ot\wtd\phi^\star_\pmq)
=\chi_1^{-2}\Abs(\pmq^r)\cdot\varepsilon(1/2,\pi_2\ot\pi_3\ot\chi_1)\cdot\frac{1}{B_{\itPi_\pmq}}\cdot\frac{\zeta_\pmq(2)^2}{\zeta_\pmq(1)^2}\cdot \begin{cases}(1+\abs{\pmq})^2&\text{ if $\pi_2$ is of type $\bfone$},\\
1&\text{ if $\pi_2$ is of type $\mathbf 2$}.\end{cases}\]
\end{prop}
\begin{proof}After an unramified twist, we may assume that $\pi_1=\Prin{\chi_1}{\upsilon_1}$ with $\chi_1=\Abs^{\bfs-\onehalf}$ and $\upsilon_1=\Abs^{\onehalf-\bfs}$ for some $\bfs\in\C$ and $\pi_3=\Contra{\pi}_2$. Let $\pi=\pi_2$ be a minimal discrete series. We use the realizations as in \eqref{E:realization1}. Let $f_1$ be the normalized new vector in $\cB(\Abs^{\bfs-\onehalf},\Abs^{\onehalf-\bfs})$ and let $f^\star_1=\rho(\pDII{\pmq^{-r}}{1})f_1$.  As in the previous cases, by \corref{C:IchinoRS} we obtain
\beq\label{E:super.1}\begin{aligned} I_\pmq(\phi^\star_\pmq\ot\wtd\phi^\star_\pmq)=&\frac{\zeta_\pmq(1)\sJ_\pmq(W_2\ot W_3\ot f^\star_1,\wtd W_2\ot\wtd W_3\ot\wtd f^\star_1)}{\zeta_\pmq(2)^2L(1/2,\itPi_\pmq)B_{\itPi_\pmq}}\cdot\frac{\pair{W_{\pi_1}}{\wtd W_{\pi_1}}}{\pair{f_1}{\wtd f_1}}\cdot\frac{\zeta_\pmq(2)^3}{\zeta_\pmq(1)^3}\\
=&\frac{\zeta_\pmq(1)\gamma(1/2,\pi_2\ot\pi_3\ot\Mu_1)\cdot\Psi(W_2,W_3,f^\star_1)^2}{\zeta_\pmq(2)^2L(1/2,\itPi_\pmq)}\cdot\frac{\pair{W_{\pi_1}}{\wtd W_{\pi_1}}}{\pair{f_1}{\wtd f_1}}\cdot\frac{1}{B_{\itPi_\pmq}}\cdot\frac{\zeta_\pmq(2)^3}{\zeta_\pmq(1)^3}.
\end{aligned}\eeq
 Define the function $\bfW:ZN\bksl G\to\C$ by
\[\bfW(g):=W_2(g)W_3(\pDII{-1}{1}g).\]
We compute $\Psi(W_2,W_3,f_1^\star)$ in the following two subcases. 

Subcase (a): $\pi_i=\chi_i\Abs^{-\onehalf}{\rm St}$ are unramified special for $i=2,3$. Then $\pi_2$ is of type $\mathbf 2$ and $r=1$. We have  
\begin{align*}
\Psi(W_2,W_3,f^\star_1)=&\vol(K_0(\pmq))(J_1+J_2),\end{align*}
where
\begin{align*}J_1=&\abs{\pmq}^{-\bfs}\int_{\Q_\pmq^\x}\bfW(\pDII{y}{1})\abs{y}^{\bfs-1}\rmd^\x y,\\
J_2=&\sum_{x\in\Z/\pmq\Z}\abs{\pmq}^{\bfs}\int_{\Q_\pmq^\x}\bfW(\pDII{y}{1}\pMX{1}{x}{0}{1}w)\abs{y}^{\bfs-1}\rmd^\x y.
\end{align*}
By a direct calculation, we find that
\begin{align*}
J_1=&\abs{\pmq}^{-\bfs}L(\bfs,\chi_2\chi_3),\\
J_2=&\abs{\pmq}^{\bfs}\cdot q\cdot \abs{\pmq}^2\cdot \chi_2^{-1}\chi_3^{-1}\Abs^{-\bfs}(\pmq)\cdot L(\bfs,\chi_1\chi_2)=\abs{\pmq}\chi_1\chi_2(\pmq^{-1})L(\bfs,\chi_2\chi_3).
\end{align*}
Note that $\om_2\om_3=\chi_2^2\chi_3^2\Abs^{-2}=1$. Hence\begin{align*}\Psi(W_2,W_3,f^\star_1)=&
\frac{1}{1+\pmq}\abs{\pmq}^{-\bfs}\cdot(1+\chi_2\chi_3\Abs^{\bfs-1}(\pmq))\cdot L(\bfs,\chi_2\chi_3)\\
=&\frac{\zeta_\pmq(2)}{\zeta_\pmq(1)}
\abs{\pmq}^{1-\bfs}\frac{L(\bfs,\chi_2\chi_3)L(\bfs-1,\chi_2\chi_3)}{\zeta_\pmq(2\bfs)}\\
=&\frac{\zeta_\pmq(2)}{\zeta_\pmq(1)}
\abs{\pmq}^{1-\bfs}\frac{L(1/2,\pi_2\ot\pi_3\ot\chi_1)}{L(1,\chi_1\upsilon_1^{-1})}.
\end{align*}

Subcase (b): $\pi_2$ and $\pi_3$ are supercuspidal. In this case,
\begin{align*}
\Psi(W_2,W_3,f^\star_1)=&\frac{\zeta_\pmq(2)}{\zeta_\pmq(1)}\int_{\Q_\pmq^\x}\int_{\Q_\pmq}
\bfW(\pDII{y}{1}\pMX{1}{0}{x}{1})\abs{y}^{\bfs-1}f_1(\pMX{1}{0}{x}{1})\pDII{\pmq^{-r}}{1})\rmd x\rmd^\x y
\\
=&\frac{\zeta_\pmq(2)}{\zeta_\pmq(1)}\abs{\pmq}^{-r\bfs}\sum_{n\in\Z}J_n,
%=&\int_{\Q_\pmq^\x}\int_{\Z_\pmq}W_1\ot W_2(\pDII{a}{1}w\pMX{1}{x}{0}{1})\Abs^{s-\onehalf}(a)f_3(w\pMX{1}{x}{0}{1})\rmd^\x a\rmd x\\
%=&\Abs^{s-\onehalf}(\pmq^{-c})=t^{-c}\abs{\pmq}^\frac{c}{2}.\\
%I_c=&\int_{\Q_\pmq^\x}\int_{\pmq^c\Zp}W_1\ot W_2(\pDII{a}{1}\pMX{1}{0}{x}{1})\Abs^{s-\onehalf}(a)f_3(\pMX{1}{0}{x}{1})\rmd^\x a\rmd x\\
\end{align*}
where
\begin{align*}
J_n=&\int_{\Qq^\x}\int_{\pmq^n\Z_\pmq^\x}\bfW(\pDII{y}{1}\pMX{1}{0}{\pmq^n}{1}))\abs{y}^{\bfs-1}f_3(\pMX{1}{0}{\pmq^{n-r}}{1})\rmd x\rmd^\x y\\
=&\abs{\pmq^n}(1-\abs{\pmq})\sum_{m\in\Z}\abs{\pmq^m}^{\bfs-1}\cdot f_1(\pMX{1}{0}{\pmq^{n-r}}{1})\int_{\Z_\pmq^\x}\bfW(\pDII{\pmq^m u}{1}\pMX{1}{0}{\pmq^n}{1})\rmd^\x u\\
=&\abs{\pmq^n}(1-\abs{\pmq})\sum_{m\in\Z}\abs{\pmq^m}^{\bfs-1}f_1(\pMX{1}{0}{\pmq^{n-r}}{1})\sum_{\chi\in\wh\Z_\pmq^\x}A_{\pi_2,n}^{(m)}(\chi)A_{\pi_3,n}^{(m)}(\chi^{-1})\chi(-1),
\end{align*}
where 
\[A_{\pi_i,n}^{(m)}(\chi):=\int_{\Z_\pmq^\x}W_i(\pDII{\pmq^mu}{1}\pMX{1}{0}{\pmq^n}{1})\rmd u.\]
In the case $\chi\not =\bfone$, by \lmref{L:Kirillov.local}, we have
\[A_{\pi_2,n}^{(m)}(\chi)A_{\pi_3,n}^{(m)}(\chi^{-1})\chi(-1)=\abs{\pmq}^{c-n}\zeta_\pmq(1)^2\]if  $n=c-c(\chi)$ and \[m=c-c(\pi\ot\chi)=\begin{cases}0&\text{ if }n\geq r,\\
2n-c&\text{ if }n<r,\end{cases} 
\] by \lmref{L:basic1} (2), and $A_{\pi_2,n}^{(m)}(\chi)A_{\pi_3,n}^{(m)}(\chi^{-1})=0$ otherwise.
If $\chi=\bfone$, then 
\[A_{\pi_2,n}^{(m)}(\bfone)A_{\pi_3,n}^{(m)}(\bfone)=\abs{\pmq}^2\zeta_\pmq(1)^2\]
if $c-n=1$. Therefore, if $n<r$, then 
\begin{align*}J_n=&(1-\abs{\pmq})\abs{\pmq}^n\abs{\pmq}^{(2n-c)(\bfs-1)}\abs{\pmq}^{2(r-n)\bfs}\cdot\abs{\pmq}^{c-n}\zeta_\pmq(1)^2\#\stt{\chi\mid \cond{\chi}=c-n}\\
=&(1-\abs{\pmq})\abs{\pmq}^{(2r-c)\bfs+c-n}.
\end{align*}
If $r\leq n <c-1$, then \begin{align*}
J_n=&(1-\abs{\pmq})\abs{\pmq}^n.
\end{align*}
If $n=c-1$, then 
\[J_{c-1}=(1-\abs{\pmq})\abs{\pmq}^{c-1}(\abs{\pmq}(q-1-1)+\abs{\pmq}^2)\zeta_\pmq(1)^2=(1-\abs{\pmq})\abs{\pmq}^{c-1}.\]
If $n\geq c$, then $J_{\geq c}=\abs{\pmq}^c$. Combining the above equations, we find that $\Psi(W_2,W_3,f_1^\star)$ equals
\begin{align*}
\frac{\zeta_\pmq(2)}{\zeta_\pmq(1)}\abs{\pmq}^{-r\bfs}
\sum_{n\in\Z}J_n=&\frac{\zeta_\pmq(2)}{\zeta_\pmq(1)}\abs{\pmq}^{-r\bfs}\left(J_{r-1}^-+\sum_{n=1}^{r}J_n+J_c^+\right)\\
=&\frac{\zeta_\pmq(2)}{\zeta_\pmq(1)}\abs{\pmq}^{-r\bfs}(\abs{\pmq}^{(2r-c)\bfs}\abs{\pmq}^{c+1-r}+\abs{\pmq}^r-\abs{\pmq}^c+\abs{\pmq}^c)\\
=&\frac{\zeta_\pmq(2)}{\zeta_\pmq(1)}\abs{\pmq}^{-r\bfs}\begin{cases}\abs{\pmq}^\frac{c}{2}(1+\abs{\pmq})&\text{ if $c$ is even ($\pi_2$ is of type $\bfone$),}\\
\abs{\pmq}^\frac{c+1}{2}(1+\abs{\pmq}^\bfs)&\text{ if $c$ is odd ($\pi_2$ is of type $\mathbf 2$).}\end{cases}
\end{align*}
On the other hand, when $\pi_2$ and $\pi_3$ are supercuspidal, it is easy to see that
\[\frac{L(1/2,\pi_2\ot\pi_3\ot\Mu_1)}{L(1,\Mu_1\upsilon_1^{-1})}=\begin{cases}1 &\text{ if $\pi_2$ is of type $\bfone$,}\\
1+\abs{\pmq}^\bfs&\text{ if $\pi_2$ is of type $\mathbf 2$}.\end{cases}
\]
We thus conclude that in either subcase (a) or subcase (b),  \[\Psi(W_2,W_3,f^\star_1)=\frac{\zeta_\pmq(2)}{\zeta_\pmq(1)}\abs{\pmq}^{r(1-\bfs)}\frac{L(1/2,\pi_2\ot\pi_3\ot\Mu_1)}{L(1,\Mu_1\upsilon_1^{-1})}\begin{cases}1+\abs{\pmq} &\text{ if $\pi_2$ is of type $\bfone$,}\\
1&\text{ if $\pi_2$ is of type $\mathbf 2$}.\end{cases}\]
Substituting the above equation and \lmref{L:basic3} into \eqref{E:super.1}, we obtain
\begin{align*}
I_\pmq(\phi^\star_\pmq\ot\wtd\phi^\star_\pmq)&=
\frac{\zeta_\pmq(2)\gamma(1/2,\pi_2\ot\pi_3\ot\Mu_1)\cdot\Psi(W_2,W_3,f^\star_1)^2}{\zeta_\pmq(1)^2L(1/2,\itPi_\pmq)B_{\itPi_\pmq}}\cdot \frac{\zeta_\pmq(1)^2L(1,\Mu_1\upsilon_1^{-1})^2}{\zeta_\pmq(2)}\\
=&\varepsilon(1/2,\pi_2\ot\pi_3\ot\chi_1)\chi_1^{-2}\Abs(\pmq^r)\cdot
\frac{1}{B_{\itPi_\pmq}}\cdot\frac{\zeta_\pmq(2)^2}{\zeta_\pmq(1)^2}\cdot \begin{cases}(1+\abs{\pmq})^2&\text{ if $\pi_2$ is of type $\bfone$},\\
1&\text{ if $\pi_2$ is of type $\mathbf 2$}\end{cases}
\end{align*}
by noting that $L(s,\itPi_\pmq)=L(s,\pi_2\ot\pi_3\ot\chi_1)L(s,\pi_2\ot\pi_3\ot\upsilon_1)$. This finishes the proof.
\end{proof}

\subsection{The ramified case (IIb)}\label{S:IIb}
Finally, we consider the case where $\pi_2$ and $\pi_3$ are discrete series, $\pi_3$ is minimal and $L(s,\pi_2\ot\pi_3)=1$. It is also assumed that $\pi_1=\Prin{\chi_1}{\upsilon_1}$ is a principal series with $c(\chi_1)=0$ and $c(\upsilon_1)\leq 1$. 

\begin{prop}\label{P:ramifiedII}Let $c^*=\max\stt{c_2,c_3}$. We have
\[I_\pmq(\phi^\star_\pmq\ot\wtd \phi^\star_\pmq)= \om_1(-1)\Mu_1^{-2}\Abs(\pmq^{c^*})\varepsilon(1/2,\pi_1)^2\cdot\varepsilon(1/2,\pi_2\ot\pi_3\ot\Mu_1)\cdot\frac{1}{B_{\itPi_\pmq}}\cdot\frac{\zeta_\pmq(2)^2}{\zeta_\pmq(1)^2}.\]

\end{prop}
\begin{proof}In this case, we use the realizations \beq\label{E:realization1}\cV_{\itPi_\pmq}=\cB(\chi_1,\upsilon_1)\boxtimes\cW(\pi_2)\boxtimes \cW(\pi_3);\quad \cV_{\Contra{\itPi}_\pmq}= \cB(\chi_1^{-1},\upsilon_1^{-1})\boxtimes\cW(\Contra{\pi}_2)\boxtimes\cW(\Contra{\pi}_3).\eeq
Let $f_1\in \sB(\Mu_1,\upsilon_1)^{\rm new}$ be the new vector with $f_1(1)=1$. Define the section $f^\star_1\in\sB(\Mu_1,\upsilon_1)^{\cU_0(\pmq^{c^*})}$ by 
\[f^\star_1=\rho(\pDII{\pmq^{-c^*}}{1})f_1-\upsilon_1^{-1}\Abs^\onehalf(\pmq)\rho(\pDII{\pmq^{1-c^*}}{1})f_1\text{ if }c_1=c(\upsilon_1)=0\]
and $f^\star_1=\rho(\pDII{\pmq^{1-c^*}}{1})f_1$ if $c_1=1$. Then $f^\star_1$ is the section supported in the $B\cU_0(\pmq^{c^*})$ with $f^\star_1(1)=\Mu_1\Abs^\onehalf(\pmq^{c_1-c^*})L(1,\Mu_1\upsilon_1^{-1})^{-1}$. Let $\wtd f_1=M^*(\Mu_1,\upsilon_1)f_1\ot\om_1^{-1}$. Then we have \begin{align*}\wtd f^\star_1=&\rho(\pDII{\pmq^{-c^*}}{1})\wtd f_1\cdot \om_1(p^{-c^*})-\upsilon_1^{-1}\Abs^\onehalf(\pmq)\rho(\pDII{\pmq^{1-c^*}}{1})\wtd f_1\cdot\om_1(p^{1-c^*})\\
=&M^*(\Mu_1,\upsilon_1)f^\star_1\ot \om_1^{-1}\text{ if }c_1=0.\end{align*}
A direct computation shows that \begin{align*}
\Psi(W_2,W_3,f^\star_1)=&\frac{\zeta_\pmq(2)}{\zeta_\pmq(1)}\int_{\Q_\pmq^\x}\int_{\Q_\pmq}W_2(\pDII{y}{1}\pMX{1}{0}{x}{1})W_3(\pDII{-y}{1}\pMX{1}{0}{x}{1})\Mu_1\Abs^\onehalf(y)f^\star_1(\pMX{1}{0}{x}{1})\rmd x\frac{\rmd^\x y}{\abs{y}}\\
=&\frac{\Mu_1\Abs^\onehalf(\pmq^{c_1-c^*})}{L(1,\Mu_1\upsilon_1^{-1})}\frac{\zeta_\pmq(2)\abs{\pmq}^{c^*}}{\zeta_\pmq(1)}\int_{\Q_\pmq}W_2(\pDII{y}{1})W_3(\pDII{-y}{1})\Mu_1\Abs^{-\onehalf}(y)\rmd^\x y\\
=&\frac{\zeta_\pmq(2)\Mu_1(\pmq^{c_1-c^*})\abs{\pmq}^\frac{c_1+c^*}{2}}{\zeta_\pmq(1)L(1,\Mu_1\upsilon_1^{-1})}.
\end{align*}
The last equality follows from the fact that either $L(s,\pi_2)=1$ or $L(s,\pi_3)=1$ in case (IIb). By \corref{C:IchinoRS}, the above equation and \lmref{L:basic3} (1), we obtain
\begin{align*}
I_\pmq(\phi_\pmq^\star\ot\wtd\phi_\pmq^\star)=&
\frac{\sJ_\pmq(W_2\ot W_3\ot f^\star_1,\wtd W_2\ot \wtd W_3\ot\wtd f^\star_1)}{\zeta_\pmq(2)^2 L(1/2,\itPi_\pmq)B_{\itPi_\pmq}}\cdot\frac{\pair{\rho(\tau_{c_1})W_{\pi_1}}{\wtd W_{\pi_1}}}{\pair{\rho(\tau_{c_1})f_1}{\wtd f_1}}\cdot\frac{\zeta_\pmq(2)^3}{\zeta_\pmq(1)^3}\\
=&\frac{\gamma(1/2,\pi_2\ot\pi_3\ot\chi_1)\Psi(W_2,W_3,f^\star_1)^2}{L(1/2,\itPi_\pmq)B_{\itPi_\pmq}}\cdot \chi_1^2\Abs(\pmq^{-c_1})\varepsilon(1/2,\pi_1)^2\om_1(-1)L(1,\chi_1\upsilon_1^{-1})^2\\
=&\frac{\Mu_1^{-2}(\pmq^{c^*})\abs{\pmq}^{c^*}\varepsilon(1/2,\pi_2\ot\pi_3\ot\Mu_1)}{B_{\itPi_\pmq}}\cdot\frac{\zeta_\pmq(2)^2}{\zeta_\pmq(1)^2}\cdot\varepsilon(1/2,\pi_1)^2\om_1(-1).
\end{align*}
The lemma follows.
\end{proof}
\def\Rec{{\rm Rec}}
\subsection{The \padic interpolation of normalized local zeta integrals $\sI^*_{\itPi_{\ulQ,\pmq}}$}\label{SS:6.6}
In this subsection, we compute the normalized local zeta integrals $\sI^*_{\itPi_{\ulQ,\pmq}}=\sI^*_{\itPi_\pmq}$ in \eqref{E:Nq} and show these integrals can be $p$-adically interpolated by an Iwasawa function in $\ulQ\in\frakX_\cR^\ari$. We begin with recalling some facts. If $\cF\in\bfI\powerseries{q}$ is a primitive Hida family of tame conductor $N$ and $Q\in \frakX^\ari_\bfI$ is a classical point, as in the introduction we denote by $V_{\cF_Q}$ the associated $p$-adic Galois representation, and for each prime $\ell$, let $\WD_\ell(V_{\cF_Q})$ be the representation of the Weil-Deligne group $W'_{\Q_\ell}$ attached to $V_{\cF_Q}$. Let $\ell\not =p$ be a prime. On the automorphic side, denote by $\Rec_{\Q_\ell}$ the local Langlands reciprocity map from the set of isomorphism classes of irreducible representations of $\GL_n(\Q_\ell)$ to the set of isomorphism classes of $n$-dimensional representations of Weil-Deligne group $W'_{\Q_\ell}$ over $\Qbarp$ (\cite{HarrisTaylor01}). Then \beq\label{E:root.local2}\Rec_{\Q_\pmq}(\pi_{\cF_Q,\ell}\ot\Abs_\ell^\frac{1-k_Q}{2})=\WD_\ell(V_{\cF_Q});\quad \varepsilon_\ell(1/2,\pi_{\cF_Q})=\abs{N}_\ell^{\frac{k_Q}{2}}\varepsilon(\WD_\ell(V_{\cF_Q})).\eeq
We recall the following standard fact for the $p$-adic interpolation of local constants in Hida families.
\begin{lm}\label{L:root.local2}There exists $\varepsilon_\ell(\cF)\in\bfI^\x$ such that 
\[\varepsilon_\ell(\cF)(Q)=\varepsilon(\WD_\pmq(V_{\cF_Q}))\]for every classical point $Q\in \frakX^+_\bfI$. Moreover, if $\cG\in\bfI\powerseries{q}$ is another primitive Hida family, then there exists $\varepsilon(\cF\ot\cG)\in(\bfI\wh\ot_\cO\bfI)^\x$ such that 
\[\varepsilon_\ell(\cF\ot\cG)(\Qx,\Qy)=\varepsilon(\WD_\ell(V_{\cF_{\Qx}}\ot V_{\cG_{\Qy}}))\]for every classical points $(\Qx,\Qy)\in \frakX_\bfI^+\times\frakX_\bfI^+$.
\end{lm}
\begin{proof}
This is a simple consequence of the description of $\rho_\cF|_{G_{\Q_\ell}}$  together with the rigidity of automorphic types of Hida families in \subsecref{SS:GaloisRep}. We can actually make explicit the construction of $\varepsilon_\ell(\cF)$ as follows. Let $Q\in\frakX_\bfI^\ari$ be any arithmetic point. If $\pi_{\cF_Q,\ell}$ is a principal series, then 
$\rho_{\cF,\ell}\ot\Dmd{\cyc}_\bfI^{1/2}|_{G_{\Q_\ell}}\iso\al_{\cF,\ell}\xi_1\cyc^{1/2}\oplus\al_{\cF,\ell}^{-1}\xi_2\cyc^{1/2}$ is reducible with $\xi_1,\xi_2:G_{\Q_\ell}\to\Qbar^\x$ finite order characters and $\al_{\cF,\ell}:G_{\Q_\ell}\to \bfI^\x$ unramified, and it is not difficult to see that \[\varepsilon_\ell(\cF)= \varepsilon(0,\xi)\varepsilon(0,\xi')\cdot \al_{\cF,\ell}(\Frob_\ell^{n_1-n_2})\Dmd{\cyc}_\bfI^\onehalf(\Frob_\ell^{n_1+n_2})\cdot\abs{\ell}_\ell^{n_1+n_2},\]
where $n_1=c(\xi_1)$ and $n_2=c(\xi_2)$. If $\pi_{\cF,Q}$ is special, then $\rho_{\cF,\ell}|_{G_{\Q_\ell}}\ot\Dmd{\cyc}_{\bfI}^{1/2}$ is a non-split extension of $\xi$ by $\xi\cyc$ for a finite order character $\xi:G_{\Q_\ell}\to\Qbar^\x$, and letting $n'=c(\xi)$, we have
\[\varepsilon(\cF)=\varepsilon(0,\xi)^2\Dmd{\cyc}^{-1}_{\bfI}\cyc(\Frob_\ell^{n'})\cdot\begin{cases}
-\Dmd{\cyc}^{1/2}_{\bfI}&\text{ if }n'=0,\\
1&\text{ if }n'>0.\end{cases}\]
If $\pi_{\cF_Q,\ell}$ is supercuspidal, then $\rho_{\cF,\ell}|_{G_{\Q_\ell}}=\rho_0\ot \Dmd{\cyc}_\bfI^{-1/2}$ for some irreducible representation $\rho_0:G_{\Q_\ell}\to\GL_2(\Qbar)$ of finite image and of conductor $\ell^{n''}$, and we have
\[\varepsilon_\ell(\cF)=\varepsilon(\WD_\ell(\rho_0))\cdot\Dmd{\cyc}_\bfI^\onehalf(\Frob_\ell^{n''}).\]
The case $\rho_\cF\ot\rho_\cG$ can be treated in the same manner by the formulae of $\ep$-factors in \cite{GJ78}. We omit the details.
\end{proof}

We recall that the finite set $\Sigma_{\rm exc}$ in \eqref{E:exp.1} is given by 
\[\Sigma_{\rm exc}=\stt{\pmq\in \Sigma^{\rm (IIa)} _{f}\disjoint\Sigma^{\rm (IIa)} _{g}\disjoint\Sigma^{\rm (IIa)} _{h}\mid \text{either of $\pi_{f,\pmq},\,\pi_{g,\pmq},\pi_{h,\pmq}$ is supercuspidal of type $\bfone$}}.\]
\begin{prop}\label{P:summary}With \hypref{H:ram}, for each $\pmq\divides N$ with $\pmq\not\in\Sigma^-$, there exists a unique element $\frakf_{\bdsF,\pmq}\in\cR^\x$, which we call the fudge factor at $\pmq$ such that \[\sI_{\itPi_{\ulQ,\pmq}}^*=\frakf_{\bdsF,\pmq}(\ulQ)\cdot \begin{cases}
(1+\pmq^{-1})^2&\text{ if }\pmq\in\Sigma_{\rm exc},\\
1&\text{ otherwise}.
\end{cases}\] for all $\ulQ\in\frakX_\cR^\ari$.
\end{prop}
\begin{proof}We shall express $\sI^*_{\itPi_\pmq}$ in terms of epsilon factors of Galois representation under the setting in \subsecref{SS:61}. As before, let $(\fQx,\gQy,\hQz)=(\bdsf_\Qx,\bdsg_\Qy,\bdsh_\Qz)$ be a triplet of $p$-stabilized newforms of weights $(k_1,k_2,k_3)$. Let $\chi_{\bdsF}\colon G_\Q\to\cR^\x$ be the unique character such that $\chi_{\bdsF}^{-2}=(\det\rho_\bdsf\ot\det\rho_{\bdsg}\ot\det\rho_{\bdsh})\cyc^{-1}$. Then $\chi_\bdsF$ is unramified at $\pmq$. If $\chi_F$ is the specialization of $\chi_\bdsF$ at $\ulQ$, then 
\[\Rec_{\Q_\pmq}(\om_F^{-1/2}\Abs_\pmq^{\frac{w_\ulQ+1}{2}})=\chi_{\bdsF_\ulQ}|_{W_{\Q_\pmq}}.\] 
As before, $c_2=c_\pmq(\pi_g)$, $c_3=c_\pmq(\pi_h)$ and $c^*=\max\stt{c_2,c_3}$. Write $\Abs$ for $\Abs_\pmq$. Recall that 
\[\sI^\star_{\itPi_{\ulQ,\pmq}}=I_\pmq(\phi^\star_\pmq\ot\wtd\phi^\star_\pmq)\cdot B_{\itPi_\pmq}\cdot  \frac{\zeta_\pmq(1)^2}{\abs{N}^2\zeta_\pmq(2)^2}\cdot  \om^{-1}_{F,\pmq}(\Bd_f)|\Bd_F^{\ulk}|.\]
Here $\Bd_F^{\ulk}=\Bd_f^{\kappa_1}\Bd_g^{\kappa_2}\Bd_h^{\kappa_3}$ is a product of the adjustment of levels defined in \subsecref{SS:auxiliary}. Let $\Frob_\pmq$ be the geometric Frobenius element in the Weil group $W_{\Qq}$. 

Case (Ia) and (Ib): Suppose we are in the situation of either \subsecref{S:Ia} or \subsecref{S:Ib}. Then we have  $\val_\pmq(\Bd_f)=0$, $\val_\pmq(\Bd_g)=c^*-c_2$ and $\val_\pmq(\Bd_h)=c^*-c_3$. Thus \[\om^{-1}_{F,\pmq}(\Bd_f)\abs{\Bd_F^{\ulk}}=\abs{\pmq}^{\wt_2(c^*-c_2)+\wt_3(c^*-c_3)}\quad (\kappa_i=k_i-2).\] 
In Case (Ia) with $c_3=0$, by \propref{P:ramifiedI} we obtain \[\sI^\star_{\itPi_\pmq}=\om_2\om_3(\pmq^{-c_2})\varepsilon(1/2,\pi_2)^2\abs{\pmq}^{(\wt_3-2)c_2}.\]
Hence, we find that $\frakf_{\bdsF,\pmq}=\det\rho_\bdsg\det\rho_{\bdsh}(\Frob_\pmq^{c^*})\abs{\pmq}_{-2c_2}\cdot\varepsilon(\bdsg)^2$. Consider Case (Ia) with $c_3>0$ ($c^*=c_2$). Let $\al_\pmq^*(\bdsh):W_{\Q_\pmq}\to\bfI^\x$ be the unramified character sending $\Frob_\pmq$ to $\bfa(\pmq,\bdsh)$ and let $\al_\pmq(h)=\bfa(\pmq,h):=\chi_3\Abs^\frac{1-k_3}{2}(\pmq)$. By local Langlands correspondence for $\GL(2)$, \[\varepsilon(\WD_\pmq(V_f\ot V_g))=\varepsilon(\frac{2-k_1-k_2}{2},\pi_f\ot\pi_g).\] This implies that
\[\varepsilon(1/2,\pi_1\ot\pi_2\ot\chi_3)=\varepsilon(\WD_\pmq(V_f\ot V_g)\ot\al^*_\pmq(h)\chi_F).\]
By \propref{P:ramifiedI} and \eqref{E:root.local2}, we thus obtain
\[\sI^\star_{\itPi_\pmq}=\varepsilon(\WD_\pmq(V_f\ot V_g)\ot\al^*_\pmq(h)\chi_{\bdsF_\ulQ})\cdot\al_\pmq(h)^{-2c^*}\abs{\pmq}^{2c^*}\cdot \det V_h({\rm Art}_\pmq(-1))\cdot\varepsilon(V_h)^2.\]
Here ${\rm Art}:\Q_\pmq^\x\to W^{ab}_{\Q_\pmq}$ is the Artin map. Therefore, by \lmref{L:root.local2} we find that \[\frakf_{\bdsF,\pmq}=\varepsilon(\bdsf\ot\bdsg)\cdot \al^*_\pmq(\bdsh)\chi_{\bdsF_\ulQ}(\Frob_\pmq^{c'})\cdot \al_\pmq^*(\bdsh)\cyc(\Frob_\pmq^{-2c^*}) \det\rho_\bdsh({\rm Art}_\pmq(-1))\cdot\varepsilon(\bdsh)^2,\]
where $c'$ is the exponent of the conductor of $\pi_{f,\pmq}\times \pi_{g,\pmq}$. In case (Ib) with $L(s,\pi_2\ot\pi_3)=1$, we see from \propref{P:ram4} that
\begin{align*}\sI^\star_{\itPi_\pmq}=&\varepsilon(\WD_p(V_g\ot V_h)\ot\al^*_{\pmq}(f)\chi_{F_\ulQ})\cdot \al_{\pmq}(f)^{2c^*}\chi_{F}(\pmq^{c^*})\\
&\times \varepsilon(\WD_\pmq(V_g))^2\varepsilon(\WD_\pmq(V_h))^2\cdot\abs{\pmq}^{2(c_2+c_3-2c^*)}.\end{align*}
It follows that 
\[\frakf_{\bdsF,\pmq}=\varepsilon(\bdsg\ot\bdsh)\cdot\al^*_\pmq(\bdsf)^{2}\chi_{\bdsF}(\Frob_\pmq^{c^*})\cdot\al_\pmq^*(\bdsf)\chi_{\bdsF}(\Frob_\pmq^{c''})\abs{\pmq}^{2(c_2+c_3-2c^*)},\]
where $c''$ is the exponent of the conductor of $\pi_{g,\pmq}\times\pi_{h,\pmq}$. If $L(s,\pi_{g,\pmq}\ot\pi_{h,\pmq})\not =1$, then $\sI^*_{\itPi_\pmq}=2\abs{\pmq^{-1}}$.

We proceed to treat Case(IIa) and (IIb). So $\pi_{f,\pmq}$ is principal series while $\pi_{g,\pmq}$ and $\pi_{h,\pmq}$ are discrete series.

Case (IIa): In the setting of \subsecref{S:IIa}, we have $\val_\pmq(\Bd_f)=r=\lceil\frac{c^*}{2}\rceil$ and $\val_\pmq(\Bd_g)=\val_\pmq(\Bd_h)=0$; then  \[\om_{F,\pmq}^{-1}(\Bd_f)\abs{\Bd_F^{\ulk}}_\pmq=\om_{\fgh,\pmq}(\pmq^{-r})\abs{\pmq}^{\wt_1r}.\] By \propref{P:ramifiedIII} and \eqref{E:root.local2}, we find that 
\[\sI^\star_{\itPi_\pmq}=\varepsilon(\WD_\pmq(V_{g}\ot V_{h})\ot\al^*_{f,\pmq}\chi_F)\cdot\al^*_{f,\pmq}(\Frob_\pmq^{-2r})\begin{cases}(1+\abs{\pmq})^2&\text{ if }\pi_2\text{ is of type $\bfone$},\\
1&\text{ if }\pi_2\text{ is of type $\mathbf 2$}.\end{cases}\]

Case (IIb): In the setting of \subsecref{S:IIb}, we have $\val_\pmq(\Bd_f)=c^*-c_1$ and $\val_\pmq(\Bd_g)=\val_\pmq(\Bd_h)=0$. Then \[\om^{-1}_{F,\pmq}(\Bd_f)\abs{\Bd_F^{\ulk}}_\pmq=\om_{\fgh,\pmq}(\pmq^{c_1-c^*})\abs{\pmq}^{\wt_1(c^*-c_1)}.\]
 If $c_1>0$, we set $\al_\pmq(\bdsf):=\bfa(\pmq,\bdsf)$. If $c_1=0$, then set $\al_\pmq(\bdsf):=\bfa(\pmq,\bdsf)-\beta_\pmq(\bdsf)$, where $\beta(\pmq,\bdsf)$ is a root of the Hecke polynomial of $\bdsf$ at $\pmq$ fixed in \defref{D:testunb}. Define $\al^*_{\bdsf,\pmq}:W_{\Q_\pmq}\to\bfI_1^\x$ to be the unramified character with $\al^*_{\bdsf,\pmq}(\Frob_\pmq)=\al_\pmq(\bdsf)$. By definition, $\Rec_{\Q_\pmq}(\chi_1\om_F^{1/2}\Abs^\frac{1-k_\Qx}{2})=\al^*_{f,\pmq}$ the specialization of $\al_{\bdsf,\pmq}^*$ at $\Qx$.
 From \propref{P:ramifiedII}, we obtain the following expression of $\sI^\star_{\itPi_\pmq}$:
\[\sI^\star_{\itPi_\pmq}=\varepsilon(\WD_\pmq(V_{g}\ot V_{h})\ot \al^*_{f,\pmq}\chi_F)\cdot\al^*_{f,\pmq}(\Frob_\pmq^{-2c^*}) \cdot \varepsilon(\WD_\pmq(V_f))^2\abs{\pmq}^{2c_1}\cdot\det V_f({\rm Art}_\pmq(-1)).\]
In either case, it is easy to see by \lmref{L:root.local2} that 
\[\frakf_{\bdsF,\pmq}=\varepsilon_\pmq(\bdsf\ot\bdsg)\cdot \al^*_{\bdsf,\pmq}\chi_\bdsF(\Frob_\pmq^{c'})\cdot\al^*_{\bdsf,\pmq}(\Frob_\pmq^{-2c^*})\cdot\varepsilon_\pmq(\bdsf)^2\det\rho_\bdsf({\rm Art}_\pmq(-1))\abs{\pmq}^{2c_1}, \]
where $c'$ is the exponent of the conductor of $\pi_{f,\pmq}\times\pi_{g,\pmq}$.
This completes the proof in all cases.
\end{proof}

%!TEX root = TRIPLE3.tex

\section{The interpolation formulae}\label{S:interpolation}
\subsection{Proof of the main results}We complete the proofs of the main results in this section. We retain the notation in the introduction. For $\ulQ=(\Qx,\Qy,\Qz)$, recall that $\om_{F_\ulQ}^{1/2}=\Om^{a-\frac{w_\ulQ-3}{2}}\ep_\Qx^{1/2}\ep_\Qy^{1/2}\ep_\Qz^{1/2}$ and that
\[\itPi_\ulQ=\pi_{\bdsf_\Qx}\times\pi_{\bdsg_\Qy}\times\pi_{\bdsh_\Qz}\ot\om_{F_\ulQ}^{-1/2}.\]
In terms of $L$-functions attached to Galois representations in the introduction, we have
\[L(s+\onehalf,\itPi_\ulQ)=\Gamma_{\bfV_\ulQ^\dagger}(s)\cdot L(\bfV_\ulQ^\dagger,s),\]
where $\Gamma_{\bfV_\ulQ^\dagger}(s)=L(s+\onehalf,\itPi_{\ulQ,\infty})$ is the $\Gamma$-factor of $\bfV_\ulQ^\dagger$ in \eqref{E:Gamma.1}. The set $\Sigma^-$ in \defref{D:root} is given by 
\[\Sigma^-=\stt{\ell\divides N\mid \varepsilon(\WD_\ell(\bfV_\ulQ^\dagger))=-1\text{ for some } \ulQ\in\frakX_\cR^+}.\]

\begin{thm}\label{T:main.7}Suppose that $p$ is an odd prime and that \eqref{ev} and \eqref{sf} hold. After we enlarge the coefficient ring $\cO$ to some finite unramified extension over $\cO$, the following statements hold. \begin{enumerate}
\item If $\Sigma^-=\emptyset$ and $\bdsf$ satisfies the Hypothesis $\rm (CR)$, then there exists an element $\cL^\bdsf_{\bdsF}\in\cR$ such that for every $\ulQ=(\Qx,\Qy,\Qz)\in\frakX_\cR^\bdsf$ in the unbalanced range dominated by $\bdsf$, we have
\begin{align*}(\cL_{\bdsF}^\bdsf(\ulQ))^2=&\Gamma_{\bfV_\ulQ^\dagger}(0)\cdot\frac{L(\bfV_\ulQ^\dagger,0)}{(\sqrt{-1})^{2k_{\Qx}}\Omega_{\bdsf_\Qx}^2}\cdot\cE_p(\Fil^+_\bdsf\bfV_\ulQ)\cdot\prod_{\ell\in\Sigma_{\rm exc}}(1+\ell^{-1})^2,\end{align*}
where $\Omega_{\bdsf_\Qx}$ is the canonical period attached to the $p$-stabilized form $\bdsf_\Qx$ as in \defref{D:period1}.
\item If $p>3$, $\#\Sigma^-$ is odd, $\bdsf,\bdsg$ and $\bdsh$ all satisfy Hypothesis $({\rm CR},\Sigma^-)$, and $N^-$ and $N/N^-$ are relatively prime, then there exists a unique element $\cL_{\bdsF}^\bal\in\cR$ such that for any arithmetic point $\ulQ\in\frakX_\bal$ in the balanced range, we have
\begin{align*}\left(\cL_{\bdsF}^\bal(\ulQ)\right)^2=&\Gamma_{\bfV_\ulQ^\dagger}(0)\cdot 
\frac{L(\bfV_\ulQ^\dagger,0)}{(\sqrt{-1})^{k_\Qx+k_\Qy+k_\Qz-1}\Omega_{\bdsf^D_\Qx}\Omega_{\bdsg^D_\Qy}\Omega_{\bdsh^D_\Qz}}\cdot\cE_p(\Fil^+_\bal\bfV_\ulQ)\cdot\prod_{\ell\in\Sigma_{\rm exc}}(1+\ell^{-1})^2,\end{align*}
where $\Omega_{\bdsf^D_\Qx}, \Omega_{\bdsg^D_\Qx}$ and $\Omega_{\bdsh^D_\Qz}$ are the Gross periods in \defref{D:period2}
\end{enumerate}
\end{thm}
\begin{proof}By the observation in \remref{R:hyp}, there exists Drichlete characters $\ul{\chi}=(\chi_1,\chi_2,\chi_3)$ modulo $M$ with $M^2\divides N$ such that \begin{itemize}\item $\chi_1\chi_2\chi_3=1$;
\item the triple $\bdsF'$ of primitive Hida families attached to the Dirichlet twists $(\bdsf|[\chi_1],\bdsg|[\chi_2],\bdsh|[\chi_3])$ given  by \[\bdsF'=(\bdsf\ot\chi_1,\bdsg\ot\chi_2,\bdsh\ot\chi_3)\] satisfies \hypref{H:ram} at all classical points. \end{itemize}Enlarging $\cO$ if necessary, we may choose a square root $\sqrt{\frakf_{\bdsF'}}\in\cR^\x$ of the fudge factor $\frakf_{\bdsF'}:=\prod_{\pmq\divides N/N^-}\frakf_{\bdsF',\pmq}$ defined in \propref{P:summary}. On the other hand, by \propref{P:period1} and \propref{P:period2} in the next subsection, there exist $u_1\in\bfI_1^\x$ and $u_2\in\cR^\x$ such that for all arithmetic points $\ulQ\in\frakX_\cR^+$, we have the equalities
\begin{align*}
\Omega_{(\bdsf\ot\chi_1)_\Qx}=&u_1(\Qx)\cdot \Omega_{\bdsf_\Qx};\\
\Omega_{\bdsf_\Qx^D\ot\chi_1}\Omega_{\bdsg_\Qy^D\ot\chi_2}\Omega_{\bdsh_\Qz^D\ot\chi_3}=&u_2^2(\ulQ)\cdot\Omega_{\bdsf_\Qx^D}\Omega_{\bdsg_\Qy^D}\Omega_{\bdsh_\Qz^D}.\end{align*}
Now we define \begin{align*}\cL_{\bdsF}^\bdsf&:=\sL_{\bdsF'}^{\bdsf\ot\chi_1}\cdot \sqrt{\psi_{1,(p)}(-1)(-1)}\cdot\sqrt{\frakf_{\bdsF'}}^{-1}\cdot u_1;\\
 \cL_{\bdsF}^\bal&:=\Theta_{\bdsF^{\prime\Dstar}}\cdot 2^{-\frac{\#\Sigma^-+4}{2}}\sqrt{N}^{-1}\sqrt{\frakf_{\bdsF'}}^{-1}\cdot u_2.\end{align*} Then we can verify directly that $\cL_{\bdsF}^\bdsf$ (resp. $\cL_{\bdsF}^\bal$) enjoys the desired interpolation formulae by \corref{C:Ichino.imb} (resp. \corref{C:Ichino.bal}) combined with \propref{P:summary}, the $p$-adic computation \propref{P:padic.imb} (resp. \propref{P:padic.bal}) and \remref{R:padic}.
\end{proof}
\begin{Remark}The reason for the appearance of the extra fudge factor $\prod_{\ell\in\Sigma_{\rm exc}}(1+\ell^{-1})^2$ is not clear to the author, but a similar factor $H_0$ appeared in \padic $L$-functions for adjoint representations \cite[Corollary 7.12]{Hida88AJM}.
\end{Remark}
\subsection{The comparison between the canonical periods of Hida families with twists}\label{SS:periods}
Let $\bdsf\in\eord\bfS(N,\psi,\bfI)$ be a primitive Hida family of the tame conductor $N$ and of the brach character $\psi$. We assume that $\bdsf$ satisfies (CR). Let $\pmq\not =p$ be a prime. We further suppose that $\bdsf$ is \emph{minimal} at $q$, \ie for some arithmetic point $Q\in\frakX_\bfI^+$, the unitary cuspidal automorphic representation $\pi:=\pi_{\bdsf_Q}$ of $\GL_2(\A)$ associated with the specialization $\bdsf_Q$ is minimal at $\pmq$. Note that this definition does not depend on the choice of arithmetic points by the rigidity of automorphic types for Hida families. Let $\chi$ be a Dirichlet character modulo a power of $\pmq$ and let $\bdsf^\sharp$ be the primitive Hida family corresponding to the twist $\bdsf|[\chi]$ and let $N^\sharp$ be the tame conductor of $\bdsf^\sharp$. %We assume that $\pi\ot\chi$ is not minimal at $\pmq$. Namely, \[N^\sharp=N\pmq^t\text{ for some }t>0.\]
The aim of this subsection is to use the method of level-raising to show the two periods $\Omega_{\bdsf_Q}$ and $\Omega_{\bdsf^\sharp_Q}$ defined in \defref{D:period1} are equal up to a unit in $\bfI$. We will also prove the same result for the Gross periods of the primitive Jacquet-Langlands lifts $\bdsf^D$ and the twist $\bdsf^{\sharp D}$.

\begin{Remark}\label{R:Cideal}We recall some generalities on congruence ideals following the discussion in \cite[page 363-366]{Hida88AJM}. Let $R$ be a domain. Let $T$ be a finite reduced $R$-algebra with a $R$-algebra homomorphism $\lam:T\to R$. For any $T$-module $M$, we denote 
\[M[\lam]:=\stt{x\in M\mid rx=0\text{ for all }r\in\Ker\lam}.\]  
Then \[C(\lam):=\lam(T[\lam])=\lam(\Ann_T(\Ker\lam)).\]
Let $H$ be a free $T$-module of rank $d$. Suppose that $T$ is Gorenstein, \ie $T\iso\Hom_R(T,R)$ as $T$-modules and that we have a perfect pairing $\pairing:H\times H\to R$ such that $\pair{tx}{y}=\pair{x}{ty}$ for $t\in T$. Then $T[\lam]$ is free $R$-module of rank one and hence $H[\lam]$ is free $R$-module of rank $d$ with a basis $\stt{e_1,\ldots,e_d}$. We have \[C(\lam)^d=(\det\pair{e_i}{e_j}).\] 
\end{Remark}
Let $\psi_{(q)}$ be the $q$-primary component of $\psi$. If $\chi=1$ or $\psi_{(q)}^{-1}$, then $N^\sharp=N$ and the Atkin-Lehner involution $\eta_\pmq$ at $\pmq$ (\cite[page 168]{Miyake06book}) induces the isomorphism $\eord\bfS(N,\bfI)_{\frakm_{\bdsf^\sharp}}\iso \eord\bfS(N,\bfI)_{\frakm_{\bdsf}}$, so we find that $C(\bdsf^\sharp)=C(\bdsf)$.
\begin{lm}\label{L:73}Suppose that $\chi\not =1,\psi_{(\pmq)}^{-1}$. Then $C(\bdsf^\sharp)=C(\bdsf)\cdot E_\pmq(\bdsf)$, where
\[E_\pmq(\bdsf)=\begin{cases}
(\pmq-1)(\bfa(\pmq,\bdsf)^2-\psi_\bfI(q)(1+\pmq)^2)&\text{ if }\pmq\ndivides N,\\
1-\pmq^{-1}&\text{ if $\pi_\pmq$ is a ramified principal series},\\
1-\pmq^{-2}&\text{ if $\pi_\pmq$ is unramified special,}\\
1&\text{ if $\pi_\pmq$ is supercuspidal}
\end{cases}\]
$($recall that $\psi_\bfI$ is the $\bfI$-adic character $\psi\Dmd{\cyc}^{-2}\Dmd{\cyc}_{\bfI})$.  \end{lm}
\begin{proof}We shall follow the notation in \subsecref{SS:Heckealgebra}. Let $\bfT^\sharp:=\bfT(N^\sharp,\bfI)$ and let $\frakm^\sharp$ be the maximal ideal  of $\bfT^\sharp$ containing the operator $U_\pmq$, $\stt{T_\pmq-\bfa(\pmq,\bdsf)}_{\pmq\ndivides Np\pmq}$ and $\stt{U_\pmq-\bfa(\pmq,\bdsf)}_{\pmq\divides Np,\,\pmq\not=\pmq}$. Since $\chi\not =1,\psi^{-1}_{(\pmq)}$, we have $\bfa(\pmq,\bdsf^\sharp)=0$, and the twisting morphism $|[\chi^{-1}]$ induces an isomorphism \[|[\chi^{-1}]:\eord\bfS(N^\sharp,\bfI)_{\frakm_{\bdsf^\sharp}}\iso\eord\bfS(N^\sharp,\bfI)_{\frakm^\sharp}[U_\pmq=0]\]
as $\bfT^\sharp$-modules.  Let $r_0=2$ if $\pmq\ndivides N$, $r_0=1$ if $\pmq\divides N$ and $\pi_\pmq$ is not supercuspidal and $r_0=0$ if $\pi_\pmq$ is supercuspidal. For brevity, we put
\[\bfS(N\pmq^r):=\eord\bfS(N\pmq^r,\bfI)_{\frakm^\sharp}\ot_{\bfI} \Frac\bfI\text{ for }r\in\Z_{\geq 0}.\]
According to the possible list of tame conductors of newforms in $\eord\bfS(N\pmq^r,\bfI)_{\frakm^\sharp}$ \cite[page 436]{DiaomondTaylor94}, all newforms in $\eord\bfS(N\pmq^r,\bfI)_{\frakm^\sharp}$ have tame conductor dividing $N\pmq^{r_0}$. It follows hat $U_\pmq=0$ on $\bfS(N\pmq^{r_0})$ and that
\[\bfS(N\pmq^{r+1})=\bfS(N\pmq^r)\oplus \LR_\pmq\bfS(N\pmq^r)\text{ if }r\geq r_0.\]
Here recall that $\LR_\pmq(\sum \bfa_nq^n)=\pmq\sum\bfa_nq^{\pmq n}$.
Combined with the relation $U_\pmq V_\pmq=\pmq$, the above facts implies that
\[\eord\bfS(N^\sharp,\bfI)_{\frakm_{\bdsf^\sharp}}\ot_\bfI\Frac\bfI\iso\bfS(N^\sharp)[U_\pmq=0]=\bfS(N\pmq^{r_0})[U_\pmq=0]=\bfS(N\pmq^{r_0}),\]
and hence
\beq\label{E:AZ1}\bfT(N^\sharp,\bfI)_{\frakm_{\bdsf^\sharp}}\iso\bfT^\sharp_{\frakm^\sharp}=\bfT(N\pmq^{r_0},\bfI)_{\frakm^\sharp}.\eeq

We are going to apply the discussion in \remref{R:Cideal} to compare the congruence ideals. For each positive integer $M$ not divisible by $p$, put
 \[\bfH_p(M)=\prolim_{n\to\infty}\rmH^1_\et(X_1(Mp^n)_{/\Qbar},\Zp)\ot_{\Zp}\cO.\]
 Let $\stt{,}_M\colon \bfH_p(M)\times \bfH_p(M)\to\Lam$ denote the Hecke-equivariant perfect pairing defined in \cite[Definition (4.1.17)]{Ohta95}. Let $\bfH_p(M)_\frakm:=(\bfH_p(M)\ot_\Lam\bfI)_{\frakm}$. By \cite[Corollary 1 and 2, page 482]{Wiles95}, $\bfH_p(M)_\frakm$ is a free $\bfT(M,\bfI)_\frakm$-module of rank two and $\bfT(M,\bfI)_\frakm$ is Gorenstein under the Hypothesis (CR). Let $\bfH=\bfH_p(N)_\frakm$ and $\bfH^\sharp=\bfH_p(N\pmq^{r_0})_{\frakm^\sharp}$. Suppose that we have an injective $\bfI$-linear map  $i_\pmq:\bfH\to\bfH^\sharp$ such that 
 \begin{itemize}\item[(i)] $i_\pmq(\bfH[\lam_\bdsf])\subset \bfH^\sharp[\lam_{\bdsf^\sharp}]$; \item[(ii)] the $\bfI$-submodule $i_\pmq(\bfH)$ is a direct summand of $\bfH^\sharp$.\end{itemize}
Let $i_\pmq^*$ be the adjoint map of $i$. Recall that $i^*\colon\bfH^\sharp\to \bfH$  is the unique map such that $\stt{i_\pmq(x),y}_{N\pmq^{r_0}}=\stt{x,i^*_\pmq(y)}_N$. We have 
\beq\label{E:con.int}C(\bdsf^\sharp)^2=C(\bdsf)^2\det (i^*_\pmq i_\pmq|_{\bfH[\lam_\bdsf]}).\eeq
We proceed to construct the map $i_\pmq$ and compute the composition $i_\pmq i^*_\pmq$. Let $\lam=\lam_\bdsf$. For an integer $d$ relatively prime to $Np$, $S_d$ denotes the Hecke operator $[\Gamma_N\pDII{d}{d}\Gamma_N]$. Then we have $S_d=\sg_d\Dmd{d}_\bfI \Dmd{d}^{-2}\in \bfT$, where $\sg_d$ is the diamond operator. 

Case $\pmq\ndivides N$ ($r_0=2$): Define $i_\pmq:\bfH\to\bfH^\sharp$ by
\[i_\pmq (x)=\pmq x -\LR_\pmq T_\pmq x-S_\pmq\LR_\pmq^2 x .\]
Then one verifies directly that $U_\pmq i_\pmq=0$, which implies (i). The property (ii) is a consequence of Ihara's lemma \cite[Theorem 4.1]{Ribet84Ihara}.  A direct computation shows that 
\[i_\pmq^*=\pmq [\Gamma_N\Gamma_{N\pmq}]-S_\pmq^{-1} T_\pmq[\Gamma_N\pDII{\pmq}{1}\Gamma_{N\pmq}]+S_\pmq^{-1}[\Gamma_N\pDII{\pmq^2}{1}\Gamma_{N\pmq}],\]
and hence $i^*_\pmq i_\pmq|_{\bfH[\lam]}$ is a scalar given by
\[i^*_\pmq i_\pmq|_{\bfH[\lam]}=\lam(S_\pmq)^{-1}\pmq(1-\pmq)(\lam(T_\pmq)^2-(1+\pmq)^2\lam(S_\pmq)).\]
Note that $\lam(S_\pmq)=\psi_\bfI(\pmq)$.

Case $\pmq \divides N$ ($r_0=1$):  Define $i_\pmq:\bfH\to\bfH^\sharp$ by
\[i_\pmq(x)=x-\pmq^{-1}\LR_\pmq \bfU_\pmq x.\] A direct computation shows that the adjoint map $i_\pmq^*$ is given by
\[i_\pmq^*=[\Gamma_N\pDII{\pmq}{1}\Gamma_{N\pmq}]-\pmq^{-1}\bfU_\pmq [\Gamma_N\Gamma_{N\pmq}]\]and that 
\[i_\pmq^*i_\pmq=-\pmq^{-1}\left([\Gamma_N\pDII{\pmq}{1}\Gamma_{N\pmq}]V_\pmq-\pmq^{-1}\bfU_\pmq[\Gamma_N\Gamma_{N\pmq}]\LR_\pmq\right)\bfU_\pmq.\]
Let $s=\val_\pmq(N)$ and $\tau_{\pmq^s}:=\pMX{0}{1}{-\pmq^s}{0}\in\GL_2(\Q_\pmq)$. It is easy to see that 
\[\tau_{\pmq^s}\cdot [\Gamma_N\Gamma_{N\pmq}] \LR_\pmq \cdot \tau_{\pmq^s}^{-1} = S_\pmq^{-1}\cdot \bfU_\pmq .\]The restriction of $[\Gamma_N\Gamma_{N\pmq}] \LR_\pmq $ to $\bfH[\lam]$ is given by 
\[\lam(S_\pmq)^{-1}\cdot \left(\tau_{\pmq^s}^{-1}\bfU_\pmq\tau_{\pmq^s}|_{\bfH[\lam_\bdsf]}\right)=\lam(\bfU_\pmq)^{-1}\cdot\begin{cases}\pmq &\text{ if $\pi_{f,\pmq}$ is a ramified principal series,}\\
1&\text{ if $\pi_{f,\pmq}$ is unramified special.}
\end{cases}\]
 We thus find that
\[i^*_\pmq i_\pmq|_{\bfH[\lam]}=-\pmq^{-1}\lam(\bfU_\pmq)(\pmq-\pmq^{-1}\lam(\bfU_\pmq)\lam(S_\pmq^{-1})(\tau_\pmq^{-1}\bfU_\pmq\tau_\pmq|_{\bfH[\lam]}))=-\lam(\bfU_\pmq) E_\pmq(1,\Ad\rho_\bdsf).\]
The assertion follows from \eqref{E:con.int} and the above computation of $i^*_\pmq i_\pmq|_{\bfH[\lam]}$.
\end{proof}

\begin{prop}\label{P:period1}There exists a unit $u\in\bfI^\x$ such that for any arithmetic point $Q$, we have \[\Omega_{\bdsf_Q}=u(Q)\cdot \Omega_{\bdsf^\sharp_Q}.\]\end{prop}
\begin{proof}Let $\bdsf_Q^\circ$ and $\bdsf_Q^{\sharp \circ}$ be the newforms corresponding to $\bdsf_Q$ and $\bdsf^\sharp_Q$ of conductors $Np^n$ and $N^\sharp p^n$ respectively. If $\chi=\psi_{(\pmq)}^{-1}$, then $N=N^\sharp$ and $\bdsf_Q^{\sharp \circ}$ is the image of $\bdsf_Q^\circ$ acted by the Atkin-Lehner involution at $\pmq$, from which we can deduce the assertion easily if $\chi=1$ or $\psi_{(\pmq)}^{-1}$. Suppose that $\chi\not =1,\psi_{(\pmq)}^{-1}$. From \eqref{E:normformula}, we see that
\[
\frac{\norm{\bdsf_Q^{\sharp \circ}}^2_{\Gamma_0(N^\sharp p^n)}}{\norm{\bdsf_Q^\circ}^2_{\Gamma_0(Np^n)}}= %\frac{[\SL_2(\Z):\Gamma_0(N^\sharp p^n)]}{[\SL_2(\Z):\Gamma_0(Np^n)]}\cdot\prod_{\ell\divides N^\sharp p^n}\frac{\varepsilon(1/2,\pi_\ell)}{\varepsilon(1/2,\pi_\ell\ot\chi_\ell)}\frac{B_{\pi_\ell\ot\chi_\ell}}{B_{\pi_\ell}}\\
\frac{[\SL_2(\Z):\Gamma_0(N^\sharp)]}{[\SL_2(\Z):\Gamma_0(N)]}\cdot \frac{\varepsilon(1/2,\pi_\pmq)B_{\pi_\pmq\ot\chi_\pmq}}{\varepsilon(1/2,\pi_\pmq\ot\chi_\pmq)B_{\pi_\pmq}}.
\]
A direct computation shows that if $\pmq\ndivides N$, then the right hand side equals
\[\frac{N^\sharp}{N}\cdot L(1,\pi_\pmq,\Ad)^{-1}=\pmq^{-3}\cdot (\psi_\bfI^{-1}(\pmq) E_\pmq(\bdsf))(Q),\]
and if $\pmq\divides N$, then it is equal to 
\[\frac{N^\sharp}{N}\begin{cases}
1-q^{-1}&\text{ if $\pmq\divides N$ and $\pi_\pmq$ is a ramified principal series},\\
1-q^{-2}&\text{ if $\pi_\pmq$ is special},\\
1&\text{ if $\pi_\pmq$ is supercuspidal}.
\end{cases}
\]
In any case, it is clear that there exists a unit $u'\in\bfI^\x$ such that
\[\frac{\norm{\bdsf_Q^{\sharp \circ}}^2_{\Gamma_0(N^\sharp p^n)}}{\norm{\bdsf_Q^\circ}^2_{\Gamma_0(Np^n)}}=u'(Q)\cdot E_\pmq(\bdsf)(Q)\]
for all arithmetic points $Q$. Therefore, the assertion follows from \defref{D:period1}, \lmref{L:73} and the fact that $\cE_p(\bdsf_Q,\Ad)=\cE_p(\bdsf^\sharp_Q,\Ad)$. 
\end{proof}
\subsubsection*{The definite case}
Now we consider the Gross periods of definite quaternionic Hida families. Assume that $\bdsf$ satisfies Hypothesis (CR,$\,\Sigma^-$). Let $\bdsf^D\in \eord\bfS^D(N,\psi,\bfI)$ be the primitive Jacquet-Langlands lift of $\bdsf$. Let $\pmq^c$ be the conductor of $\chi$. Let $\cP_\chi$ be the element in the group ring $\cO[\GL_2(\Q_\pmq)]$ defined as follows: $\cP_\chi=1$ if $\chi=1$, $\cP_\chi=\pMX{0}{1}{-N}{0}$ if $\chi=\psi_{(\pmq)}^{-1}$, and \
\[\cP_\chi:=\frakg(\chi^{-1})^{-1}\sum_{b\in (\Z_\pmq/\pmq^c\Z_\pmq)^\x}\chi(b)\cdot \pMX{1}{b\pmq^{-c}}{0}{1}\]
if $\chi\not =1,\psi_{(\pmq)}^{-1}$, where $\frakg(\chi^{-1})$ is the Gauss sum of $\chi^{-1}$. Put  \[\bdsf^D|[\chi](x):=\cP_\chi(\bdsf^D)(x)\chi(\nu(x))\in\eord\bfS^D(N\pmq^{2c},\psi\chi^2,\bfI)\]
for $x\in \wh D^\x$ and $\nu(x)$ the reduced norm of $x$.
% \ie \[\bdsf^D|[\chi](x)=\frakg(\chi^{-1})^{-1}\sum_{b\in(\Z_\pmq/\pmq^c\Z_\pmq)^\x}\chi(b)\bdsf^D(x\pMX{1}{b\pmq^{-c}}{0}{1})\chi(\nu(x)).\]

\begin{lm}\label{L:Dprimitive} The quaternionic form $\bdsf^D|[\chi]$ is a primitive Jacquet-Langlands lift of $\bdsf^\sharp$. In other words, $\bdsf^D|[\chi] \in \eord\bfS^D(N^\sharp,\psi\chi^2,\bfI)[\lam_{\bdsf^{\sharp D}}]$ is a generator over $\bfI$.
\end{lm}
\begin{proof}
First we claim that $\bdsf^D|[\chi]\in \eord\bfS^D(N^\sharp,\psi\chi^2,\bfI)[\lam_{\bdsf^{\sharp D}}]$. This is clear if $\chi=1$ or $\psi_{(\pmq)}^{-1}$. If $\chi\not =1,\psi_{(\pmq)}^{-1}$, then $\lam_{\bdsf^\sharp}(\bfU_\pmq)=0$, and it is not difficult to show that $\bfU_\pmq (\bdsf^D|[\chi])=0$ by a direct computation. This shows the claim. %and the operator $\bfU_\pmq$ is nilpotent on $\cV_{\pi_\pmq\ot\chi_\pmq}^{N(\Z_\pmq)}$. This implies that $\dim_\C \cV_{\pi_\pmq\ot\chi_\pmq}^{N(\Z_\pmq)}[\bfU_\pmq=0]=1$, and hence the claim follows.

To see that $\bdsf^D|[\chi]$ is primitive, it suffices to show that $\bdsf^D|[\chi]$ is non-vanishing modulo the maximal ideal $\frakm_\bfI$ of $\bfI$. Let $\bar{f}:=\bdsf^D\ot\chi\circ\nu\pmod{\frakm_\bfI}\in\sS^D_2(N^\sharp p^t,\psi\chi^2,\Fpbar)$ for some positive integer $t$. Define two operators on $\sS^D_2(N^\sharp p^t,\psi\chi^2,\Fpbar)$ by 
\[L_1=\sum_{a\in \Z_\pmq/\pmq^c\Z_\pmq}\addchar_{\Q_\pmq}(a\pmq^{-c})\rho(\pMX{1}{a\pmq^{-c}}{0}{1};\quad L_2=\sum_{b\in (\Z_\pmq/\pmq^c\Z_\pmq)^\x}\chi^{-1}(b)\rho(\pDII{b}{1}).\]
Then 
\begin{align*}L_2L_1 (\bdsf^D|[\chi]\pmod{\frakm_\bfI})=&\sum_{b}\sum_{a}\addchar_{\Q_\pmq}(ab\pmq^{-c})\rho(\pMX{1}{a\pmq^{-c}}{0}{1})\bar{f}\\
=&\pmq^c f'-\pmq^{c-1}\sum_{b\in\Z_\pmq/\pmq\Z_\pmq}\rho(\pMX{1}{b\pmq^{-1}}{0}{1})\bar{f}\\
=&\pmq^{c-1}(\pmq-\pDII{\pmq^{-1}}{1}\bfU_\pmq)\bar{f}.
\end{align*}
Suppose that $\bdsf^D|[\chi]\pmod{\frakm_\bfI}=0$. Then we deduce from the above equation that either 
\begin{align*}(\pmq-\pDII{\pmq^{-1}}{1}\bfa(f,\pmq)\chi(\pmq))\bar{f}=&0\text{ if }\pmq\divides N,
\intertext{ or }
(\pmq-\bfa(f,\pmq)\pDII{\pmq^{-1}}{1}-\psi(\pmq)\pDII{\pmq^{-2}}{1})\bar{f}=&0\text{ if }\pmq\ndivides N.\end{align*}
In any case, this implies that $\bar{f}=0$ by \emph{Ihara's lemma} for definite quaternion algebras \cite[Lemma 5.5]{CH17Crelle} and hence $\bdsf^D\pmod{\frakm_\bfI}=0$, which is a contradiction.
\end{proof}

\begin{prop}\label{P:period2}Let $\bdsf^{\sharp D}$ be a primitive Jacquet-Langlands lift of $\bdsf^\sharp$. There exists $u\in\bfI^\x$ such that for every arithmetic point $Q\in\frakX_\bfI^+$, we have \[\Omega_{\bdsf^D_Q}=u^2(Q)\cdot \Omega_{\bdsf^{\sharp D}_Q}.\]
\end{prop}
\begin{proof} Let $\bdsf':=\bdsf^D|[\chi]$. Then $\bdsf^{\sharp D}=v\cdot \bdsf'$ for some $v\in\bfI^\x$ by \lmref{L:Dprimitive}. Let $f=\bfU_p^{-n}\bdsf^D_Q$ be the $L_{k_Q-2}(\Cp)$-valued \padic modular form obtained by \thmref{T:HidaQ} (2). Taking a nonzero vector $\bfu\in L_{k_Q-2}(\Cp)$, we let $\varphi=\varPhi(f)_\bfu=\pair{\varPhi(f)}{\bfu}_{k_Q-2}$ be the matrix coefficient of the vector-valued automorphic forms associated with $f$ and $\bfu$ as in \eqref{E:MA3} and let $\varphi_\chi:=\cP_\chi\varphi\ot\chi\circ\nu$. Choosing $\bfv$ with $\pair{\bfu}{\bfv}_{k_Q-2}=1$, define $\varphi'=\varPhi(f)_{\bfv}$ and $\varphi'_\chi$ likewise. By \lmref{L:pairing.bal} and \eqref{E:Schur}, we have
\begin{align*}
\frac{\eta_{\bdsf'}(Q)}{\eta_{\bdsf^D}(Q)}=&
\frac{\pair{\bfU_p^{-n}\bdsf'_Q}{\bdsf'_Q}_{N^\sharp p^n}}{\pair{\bfU_p^{-n}\bdsf_Q^D}{\bdsf^D_Q}_{Np^n}}
=S_1\cdot S_2\end{align*}
where 
\[S_1:=\left(\frac{N^\sharp}{N}\right)^{\frac{k_Q-2}{2}}\cdot\frac{\vol(\wh R_N^\x)}{\vol(\wh R_{N^\sharp}^\x)},\quad S_2:=\frac{\chi_p(p^{-n})\pair{\rho(\Tau^D_{N^\sharp p^n})\varphi_\chi}{\varphi_\chi'}}{\pair{\rho(\Tau^D_{Np^n})\varphi}{\varphi'}}.\]
It is easy to see that 
\[S_1=\frac{[\SL_2(\Z):\Gamma_0(N^\sharp)](N^\sharp)^\frac{k_Q-2}{2}}{[\SL_2(\Z):\Gamma_0(N)]N^\frac{k_Q-2}{2}}.\]
%For choosing a non-trivial $D_v^\x$-invariant bilinear pairing $\pairing_v:\cV_{\pi^D_v}\ot\cV_{\pi^D_v}\to\C(\om_v^{-1})$ and taking an isomorphism $\cA(\pi^D)\iso \ot_v\cV_{\pi^D_v}$, we have $\pair{\varphi_1}{\varphi_2}=C\cdot\prod_v\pair{\varphi_{1,v}}{\varphi_{2,v}}$.
On the other hand,  \begin{align*}S_2=&
\frac{\chi(\wh N^\sharp)\pair{\rho(\Tau^D_{N^\sharp p^n})\cP_\chi(\varphi)}{\cP_\chi(\varphi')}}{\pair{\rho(\Tau^D_{Np^n})\varphi}{\varphi'}}\\
%=&\chi(\wh N^\sharp)\prod_{v\not =\pmq\infty}\frac{\pair{\rho(\Tau^D_{N^\sharp p^n,v})\varphi_v}{\varphi_v}_v}{\pair{\rho(\Tau^D_{Np^n,v})\varphi_v}{\varphi_v}_v}\cdot\frac{\pair{\rho(\Tau^D_{N^\sharp p^n,\pmq})\cP_\chi(\varphi_\pmq)}{\cP_\chi(\varphi_\pmq)}_\pmq}{\pair{\rho(\Tau^D_{Np^n,\pmq})\varphi_\pmq}{\varphi_\pmq}_\pmq}\\
%=&\chi(\wh N^\sharp)\frac{\pair{\rho(\Tau^D_{N^\sharp p^n,\pmq})\cP_\chi(\varphi_\pmq)}{\cP_\chi(\varphi_\pmq)}_\pmq}{\pair{\rho(\Tau^D_{Np^n,\pmq})\varphi_\pmq}{\varphi_\pmq}_\pmq}\\
=&\chi(\wh N^\sharp)\cdot\frac{\pair{\rho(\Tau^D_{N^\sharp p^n,\pmq})\cP_\chi(W_{\pi_\pmq})}{\cP_\chi(W_{\pi_\pmq})\ot\om_\pmq^{-1}}}{\pair{\rho(\Tau^D_{Np^n,\pmq})W_{\pi_\pmq}}{W_{\pi_\pmq}\ot\om_\pmq^{-1}}}.
\end{align*}
If $\chi=1$ or $\psi_{(\pmq)}^{-1}$, $S_1=S_2=1$. Suppose that $\chi\not =1,\psi_{(\pmq)}^{-1}$. Then $\cP_\chi W_{\pi_\pmq}(\pDII{a}{1})=\bbI_{\Z_\pmq^\x}(a)\chi_\pmq^{-1}(a)$, so we have $\cP_\chi W_{\pi_\pmq}\ot\chi_\pmq=W_{\pi_\pmq\ot\chi_\pmq}$, and hence
\[S_2=\chi(\wh N)\cdot\frac{\pair{\rho(\Tau^D_{N^\sharp p^n,\pmq})W_{\pi_\pmq\ot\chi}}{W_{\pi_\pmq\ot\chi}\ot\chi_\pmq^{-2}\om_\pmq^{-1}}}{\pair{\rho(\Tau_{N,\pmq})W_{\pi_\pmq}}{W_{\pi_\pmq}\ot\om_\pmq^{-1}}}\\
=\chi(\wh N)\cdot\frac{B_{\pi_\pmq\ot\chi_\pmq}}{B_{\pi_\pmq}}.\]
From the above computations of $S_1$ and $S_2$, we see that
\begin{align*}\frac{\eta_{\bdsf'}(Q)}{\eta_{\bdsf^D}(Q)}=&
\frac{\chi(\wh N)\varepsilon(1/2,\pi_\pmq\ot\chi_\pmq)(N^\sharp)^\frac{k_Q-2}{2}}{\varepsilon(1/2,\pi_\pmq)N^\frac{k_Q-2}{2}}\cdot 
\frac{[\SL_2(\Z):\Gamma_0(N^\sharp)]}{[\SL_2(\Z):\Gamma_0(N)]}\frac{\varepsilon(1/2,\pi_\pmq)B_{\pi_\pmq\ot\chi_\pmq}}{\varepsilon(1/2,\pi_\pmq\ot\chi_\pmq)B_{\pi_\pmq}}\\=&
\frac{\varepsilon^{\Sigma^-}(\bdsf_Q^{\sharp})}{\varepsilon^{\Sigma^-}(\bdsf_Q)}
\cdot\frac{\norm{\bdsf_Q^{\sharp\circ}}^2_{\Gamma_0(N^\sharp)}}{\norm{\bdsf_Q^\circ}^2_{\Gamma_0(N)}} \quad(\text{by }\eqref{E:normformula}),\end{align*}
and the lemma follows.
\end{proof}
\begin{Remark}\label{R:periodc}If $\bdsf$ satisfies (CR,$\Sigma^-$), then $\eta_{\bdsf^D}$ indeed generates the congruence ideal associated with the homomorphism $\lam_\bdsf:\bfT^D(N,\psi,\bfI)\to\bfI$. This strengthens \cite[Prop. 6.1]{CH17Crelle} by replacing (CR${}^+$) there with a weaker hypothesis (CR, $\Sigma^-$) here. Note that $\bfT^D(N,\psi,\bfI)$ is isomorphic to the $N^-$-new quotient of $\bfT(N,\psi,\bfI)$. In particular, this implies that the congruence ideal $(\eta_{\bdsf^D})$ contains $(\eta_\bdsf)$ and $(\eta_{\bdsf^D})=(\eta_{\bdsf})$ if the residual Galois representation $\rho_{\bdsf}\pmod{\frakm_\bfI}$ is ramified at all $\ell\in\Sigma^-$. This implies Hida's canonical period of $\bdsf$ is an integral multiple of the Gross period of $\bdsf$.
\end{Remark}

%!TEX root = TRIPLE3.tex

\def\cK{K}
\def\Art{{\rm Art}}
\def\imp{*}
\def\bdsG{\boldsymbol G}
\def\frakpbar{\ol{\frakp}}
\section{Applications to anticyclotomic $p$-adic $L$-functions}\label{S:application}
\subsection{Primitive Hida families of CM forms}
In this section, we show that when $\bdsg$ and $\bdsh$ are primitive Hida families of CM forms, then the unbalanced \padic triple product $L$-function specializes to a product of theta elements \'a la Bertolini and Darmon in \cite{BD96}.  As a consequence, the anticyclotomic exceptional zero conjecture can be deduced from the theorem of Greenberg and Stevens. Let $\cK$ be an imaginary quadratic field over $\Q$ of the absolute discriminant $D_K$. Suppose that $p\cO_\cK=\frakp\ol{\frakp}$, where $\frakp$ is the prime induced by the fixed embedding $\Qbar\hookto\C\iso\Qbarp$. Let $K_\infty$ be the $\Zp^2$-extension of $K$ and let $\Gamma_\infty=\Gal(K_\infty/K)$ be the Galois group. Let $K_{\frakp^\infty}$ be the $\frakp$-ramified $\Zp$-extension in $K_\infty$ and $\Gamma_{\frakp^\infty}=\Gal(\cK_{\frakp^\infty}/\cK)$ be the Galois group. Let $\frakc$ be an ideal of $\cO_K$ coprime to $p$. For each ideal $\fraka$ prime to $\frakp\frakc$, define $\sigma_\fraka\in \Gal(\cK(\frakc\frakp^\infty)/\cK)$ be the image of $\fraka$ under the geometrically normalized Artin map sending a prime ideal $\frakq$ to the geometric Frobenius $\Frob_\frakq$. For each place $w$ of $\cK$, we let $\Art_{w}:\cK_w^\x\to G_\cK^{ab}$ denote the restriction of the Artin map to $\cK_w^\x$. Then $\Art_\frakp$ induces an embedding $\Lam\to\cO\powerseries{\Gamma_{\frakp^\infty}}$ given by $[z]\mapsto \Art_\frakp(z)|_{\cK_{\frakp^\infty}}$. Let $I_\frakp^{\rm w}:=\Art_\frakp(1+p\Zp)|_{K_{\frakp^\infty}}\subset \Gamma_{\frakp^\infty}$. Let $p^b:=[\Gamma_{\frakp^\infty}:I_{\frakp}^{\rm w}]$. Note that $b=0$ if the class number $h_K$ of $K$ is prime to $p$. Fixing a topological generator $\gamma_{\frakp}$ of $\Gamma_{\frakp^\infty}$ such that $\gamma_\frakp^{p^b}=\Art_\frakp(1+p)|_{K_{\frakp^\infty}}$, let $l:\Gal(K_\infty/K)\to\Zp$ be the \emph{logarithm} defined by the equation
\[\sg|_{\cK_{\frakp^\infty}}=\gamma_\frakp^{l(\sg)}.\]
For each variable $S$, let $\Psi_S:\Gamma_\infty\to \cO\powerseries{S}^\x$ be the universal character defined by 
\[\Psi_S(\sg)=(1+S)^{l(\sg)},\quad\sg|_{\cK_{\frakp^\infty}}=\gamma_\frakp^{l(\sg)}.\]
Enlarge the coefficient ring $\cO$ so that $\cO$ contains an algebraic integer $\bfv\in\Zbar^\x$ such that $\bfv^{p^b}=1+p$. For any finite order character $\psi:G_\cK\to\cO^\x$ of tame conductor $\frakc$, we define 
\[\bftheta_\psi(S)(q)=\sum_{(\fraka,\frakp\frakc)=1}\psi(\sigma_\fraka)\cdot\Psi_{\bfv^{-1}(1+S)-1}^{-1}(\sigma_\fraka)q^{\norm{\fraka}}\in \cO\powerseries{S}\powerseries{q}.\]
Let $\sV:G_\Q\to G_K^{ab}$ be the transfer map and put $\psi_+=\psi\circ\sV$. %If $\psi$ is geometric of weight $(1,0)$, namely 
%\[\psi(\Art_\frakp(z_1)\Art_{\ol{\frakp}}(z_2))=z_1\text{ for }z_1,z_2\in \Zp^\x\text{ sufficiently close to }1,\]
%then $\psi_+^{\rm fin}:=\cyc^{-1}\cdot \psi\circ\sV$ descends to a Dirichlet character, and 
Then $\theta_\psi(S)$ is a primitive Hida families in $\eord\bfS(C,\psi_+\tau_{K/\Q}\Om^{-1},\cO\powerseries{S})$, where $C=\#(\cO_K/\frakc)D_K$ and $\tau_{K/\Q}$ is the quadratic character associated with $K/\Q$.
\subsection{Anticyclotomic \padic $L$-functions for modular forms}\label{SS:2.8}
Let $N$ be a positive integer relatively prime to $p$. Let $f\in \cS_{2r}(Np,\bfone)$ be a $p$-stabilized newform of weight $2r\geq 2$, tame conductor $N$ and trivial nebentypus and let $\chi$ be a ring class character of $K$ with the conductor $c\cO_K$. We recall the anticyclotomic \padic $L$-functions associated with $(f,\chi)$ in the definite setting. Decompose $N=N^+N^-$, where $N^+$ (resp. $N^-$ ) is a product of primes split (resp. non-split) in $K$. Suppose that 
\begin{itemize}\item $(Np,cD_K)=1$,
\item $N^-$ is a square-free product of an \emph{odd} number of primes,
\item the residual Galois representation $\bar\rho_{f,p}$ satisfies (CR, $\supp N^-$).
\end{itemize}
Let $f^\circ$ be the normalized newform of conductor $N^\circ=Np^{n_p}$ corresponding to $f$. Enlarging $\cO$ so that it contains all Fourier coefficients of $f$, let $\bbT:=\bbT_{2r}(N^\circ,\bfone)$ be the Hecke algebra of level $\Gamma_0(N^\circ)$ and let $\lam_{f^\circ}:\bbT\to\cO$ be the homomorphism induced by $f^\circ$. Denote by $\bbT_{N^-}$ be the $N^-$-new quotient of the $\bbT$. Then $\lam_{f^\circ}$ factors through $\bbT_{N^-}$, and we denote by $\lam_{f^\circ,N^-}$ the resulting morphism. Let $\eta_{f^\circ}\in\cO$ (resp. $\eta_{f^\circ,N^-}$) be the congruence number corresponding to $\lam_{f^\circ}$ (resp. $\lam_{f^\circ,N^-}$). It is clear that $\eta_{f^\circ,N^-}$ is a divisor of the congruence number $\eta_{f^\circ}$ of $f^\circ$. 

Let $K_\infty^-$ be the anticyclotomic $\Zp$-extension of $\cK$. Let $\bfc$ be the complex conjugation. We define the logarithm $\wtd l:\Gamma_\infty\to\Zp$ by $\wtd l(\sg):=l(\sg^{1-\bfc}|_{K_{\frakp^\infty}})$. Then the map $\wtd l$ factors through the Galois group $\Gamma_\infty^-:=\Gal(K_\infty^-/K)$ and induces an isomorphism $\wtd l:\Gamma_\infty^-\iso\Zp$ as $K_{\frakp^\infty}$ and the cyclotomic $\Zp$-extension $K_\infty^+$ are linearly disjoint. Let $\gamma_-$ be the generator of $\Gamma^-$ such that $\wtd l(\gamma^-)=1$. If $\zeta\in\mu_{p^\infty}$ is a $p$-power root of unity, denote by $\ep_\zeta:\Gamma_\infty^-\to\mu_{p^\infty}$ the character defined by $\ep_\zeta(\gamma^-)=\zeta$. Fixing a factorization $N^+\cO_K=\frakN\ol{\frakN}$, by \cite{BD96}, \cite[Thm. A]{CH17Crelle} and \cite[Thm. A]{PC17}, there exists a unique Iwasawa function $\Theta_{f/K,\chi}(W)\in\cO\powerseries{W}$ such that for each primitive $p^n$-th root of unity $\zeta$, 
\beq\label{E:int.8}\begin{aligned}\left(\Theta_{f/K,\chi}(\zeta-1)\right)^2=&(2\pi)^{-2r}\Gamma(r)^2\cdot \frac{L(f^\circ/K\ot\chi\ep_\zeta,r)}{\Omega_{f^\circ,N^-}}\cdot \al_p(f)^{-2n}p^{(2r-1)n}\cdot \cE_p(f,\zeta)^{2-n_p}\\
&\times u_K^2\sqrt{D_K}cD_K^{k-2}\chi\ep_\zeta(\sg_{\frakN})\cdot \varepsilon_p(f^\circ), 
\end{aligned}
\eeq
where \begin{itemize}
\item[--]  $\al_p(f)\in\cO^\x$ is the $p$-th Fourier coefficient of $f$,
\item[--] $L(f^\circ/K\ot\chi\ep_\zeta,s)$ is the Rankin-Selberg $L$-function of $f^\circ$ and the CM form $\theta_{\psi\ep_\zeta}$ attached to $\chi\ep_\zeta$, 
\item[--] \[\cE_p(f,\zeta):=\begin{cases}(1-\al_p(f)^{-1}p^{r-1}\chi(\frakp))(1-\al_p(f)^{-1}p^{r-1}\chi(\frakpbar))&\text{ if }\zeta=1,\\
1&\text{ if }\zeta\not =1.\end{cases}\]
\item[--] $\Omega_{f^\circ,N^-}$ is the Gross period of $f^\circ$ defined by \[\Omega_{f^\circ,N^-}:=2^{2r}\cdot \norm{f^\circ}^2_{\Gamma_0(N_{f^\circ})}\cdot \eta_{f^\circ,N^-}^{-1}.\]
\item[--] $u_K=\#(\cO_K^\x)/2$ and $\varep_p(f^\circ)\in\stt{\pm 1}$ is the local root number of $f^\circ$ at $p$.
\end{itemize}
When $\chi=\bfone$ is the trivial character, we write $\cL_f$ for $\cL_{f,\one}$.

\subsection{Factorization of \padic triple product $L$-functions}\label{SS:3.8}
Let $\bdsf\in \eord\bfS(N,\Om^{k-2},\bfI)$ be the primitve Hida family passing through $f$ at some arithmetic point $\Qx$ of weight $k_{\Qx}=2r$ and trivial finite part $\ep_{\Qx}=1$. Let $\ell\ndivides Np$ be a rational prime split in $K$ and let $\chi$ be a ring class character of conductor $\ell^m\cO_K$ for some $m>0$. Suppose that $\chi=\xi^{1-\bfc}$ for some ray class character $\xi$ modulo $\ell^m\cO_K$. Consider the primitive Hida families of CM forms
\begin{align*}\bdsg&=\bftheta_{\xi}(S_2)\in\eord\bfS(C,\xi_+\tau_{K/Q}\Om^{-1},\cO\powerseries{S_2});\\
 \bdsh&=\bftheta_{\xi^{-1}}(S_3)\in \eord\bfS(C,\xi^{-1}_+\tau_{K/Q}\Om^{-1},\cO\powerseries{S_3})\end{align*}
 with $C=D_K\ell^{2m}$. Let $\bdsF=(\bdsf,\bdsg,\bdsh)$ be the triple of primitive Hida families and let $\cL^{\bdsf}_{\bdsF}\in \cR=\bfI\powerseries{S_1,S_2}$ be the associated unbalanced $p$-adic $L$-function in \thmref{T:main.7} with $a=-r$ in \eqref{ev}. 

\begin{prop}\label{P:factor}Set \[W_2=\bfv^{-1}(1+S_2)^{1/2}(1+S_3)^{1/2}-1;\quad  W_3=(1+S_2)^{1/2}(1+S_3)^{-1/2}-1.\]
Then we have 
\begin{align*}\cL_{\bdsF}^\bdsf(\Qx,S_1,S_2)%=&\cL_p(\bdsf^\dagger\ot\chi_1\chi_2^{-c}(\Psi_{S_2}\Psi_{S_3})^\frac{1-c}{2})\cdot \cL_p(\bdsf^\dagger\ot\chi_1\chi_2^{-1}(\Psi_{S_2}\Psi_{S_3}^{-1})^\frac{1-c}{2})\\
=&\pm \bfw^{-1} \cdot \Theta_{f/K}(W_2)\cdot \Theta_{f/K,\xi^{1-\bfc}}(W_3)\cdot \frac{\eta_{f^\circ}}{\eta_{f^\circ,N^-}}\in\cO\powerseries{S_1,S_2},
\end{align*}
where $\bfw=\bfw(W_2,W_3)$ is a unit in $\cO\powerseries{S_1,S_2}$ given by 
\[\bfw:= u_K^2D_K^{2r-3/2}\ell^{m/2}\xi\Psi_{W_1}\Psi_{W_2}(\sg_\frakN^{1-\bfc}).\]
\end{prop}
\begin{proof}
For $i=2,3$, taking $\zeta_i$ primitive $p^{n_i}$-th roots of unity with $n_i>0$, we let
 $Q_2=\zeta_2\zeta_3\bfv-1$ and $Q_3=\zeta_2\zeta_3^{-1}\bfv-1$, so $\bdsg_\Qy$ and $\bdsh_\Qz$ are CM forms of weight one. Let $T_i=\bfv^{-1}(1+S_i)-1$,  $i=2,3$ and let \[\cX_1:=\Psi_{T_2}^{-1/2}\Psi_{T_3}^{-1/2}\circ \sV\colon G_\Q\to\cO\powerseries{S_1,S_2}^\x\] be a square root of $\det V_{\bdsg}\det V_{\bdsh}$. There is a decomposition of Galois representations
\begin{align*}\Ind_K^\Q\xi\Psi_{T_2}^{-1}\ot\Ind_K^\Q\xi^{-1}\Psi_{T_3}^{-1}\ot\cX_1^{-1}
%=&\Ind\chi^{1-c}\Psi_{s_2}^\frac{c-1}{2}\Psi_{S_2}^\frac{c-1}{2}\oplus \Ind\chi^{1-c}\lam^{c-1}\Psi_{S_2}^\frac{c-1}{2}\Psi_{S_3}^\frac{1-c}{2}\\
=&\Ind_K^\Q\Psi_{W_2}^{\bfc-1}\oplus \Ind_K^\Q\chi\Psi_{W_3}^{\bfc-1}.
\end{align*}
Following the notation in the introduction with $\ulQ=(\Qx,Q_2,Q_3)$, we thus have
\begin{align*}\bfV_\ulQ^\dagger=&V_f(r)\ot\Ind_K^\Q\ep_2\oplus V_f(r)\ot \Ind_K^\Q\chi\ep_3;\\
\Fil^+_\bdsf\bfV_\ulQ^\dagger=&\al_{f,p}\cyc^{r}\ot (\ep_{2,\frakp}\oplus\ep_{2,\frakp}^{-1}\oplus\chi_\frakp\ep_{3,\frakp}\oplus\chi_{\frakp}^{-1}\ep_{3,\frakp}^{-1}),
\end{align*}
where $\ep_i=\ep_{\zeta_i}\colon\Gamma_\infty^-\to\mu_{p^\infty}$ is the finite order character with $\ep_i(\gamma^-)=\zeta_i$, $i=2,3$. Now we explicate the items that appear in the formula of $\cL_{\bdsF}^\bdsf(\ulQ)$ in \thmref{T:main.7}:
\begin{itemize}
\item The $L$-values \[\Gamma_{\bfV_\ulQ^\dagger}(0)\cdot L(\bfV^\dagger_\ulQ,s)=4(2\pi)^{-4r}\Gamma(r)^4\cdot L(f^\circ/K\ot\ep_2,r)\cdot L(f^\circ/K\ot\chi\ep_3,r),\]
\item By definition, $\ep_2$ and $\ep_3$ are of conductors $p^{n_2}\cO_K$ and $p^{n_3}\cO_K$, so the modified Euler factor at $p$ is given by 
\begin{align*}
\cE_p(\Fil^+_\bdsf\bfV_\ulQ^\dagger)=&\frac{1}{\varepsilon(r,\al_{f,p}\chi_\frakp\ep_{3,\frakp})\varepsilon(r,\al_{f,p}\chi_\frakp^{-1}\ep_{3,\frakp}^{-1})\varepsilon(r,\al_{f,p}\ep_{2,\frakp})\varepsilon(r,\al_{f,p}\ep_{2,\frakp}^{-1})}\\
=&\al_p(f)^{-2(n_2+n_3)}\cdot \abs{p}^{(1-2r)(n_2+n_3)}\cdot \ep_{2,\frakp}(-1)\ep_{3,\frakp}(-1)\\
=&\al_p(f)^{-2(n_2+n_3)}\cdot \abs{p}^{(1-2r)(n_2+n_3)}.
\end{align*}
\item $\Omega_f=(-2\sqrt{-1})^{2r+1} \norm{f^\circ}^2_{\Gamma_0(N^\circ)}\cdot\eta_{f^\circ}^{-1}$ and $\Sigma_{\rm exc}=\emptyset$.
\end{itemize}
Comparing with the interpolation formula of $\Theta_f$ in \eqref{E:int.8}, we find that 
\[\left(\cL_{\bdsF}^\bdsf(\Qx,\bfv\zeta_2\zeta_3-1,\bfv\zeta_2\zeta_3^{-1}-1)\right)^2=\bfw(\zeta_2-1,\zeta_3-1)^{-2}\cdot \Theta_{f/K}(\zeta_2-1)^2\Theta_{f/K,\chi}(\zeta_3-1)^2\]
for all non-trivial $p$-power roots of unity $\zeta_2,\zeta_3$, and hence the proposition follows.
\end{proof}
\begin{Remark}[An Euler system construction for $\Theta_{f/K}$]\label{R:Kato}This square root $\Theta_{f/K}$ of the anticyclotomic \padic $L$-function  in the definite setting is constructed by using Gross points in definite quaternion algebras, and a priori there is no obvious Euler system construction. Below we explain how $\Theta_{f/K}$ can be actually recovered by the Euler system of generalized Kato classes \`a la Darmon and Rotger. Suppose that  the weight $k_\Qx=2$. In \cite{DR17JAMS}, Darmon and Rotger introduce a one-variable generalized Kato classes $\kappa(f,\bdsg\bdsh)\in {\rm H}^1(\Q,V_f\ot V_{\bdsg}\ot_{\cO\powerseries{S}}V_\bdsh)$ and prove that the image of $\kappa(f,\bdsg\bdsh)$ under the Coleman map over the anticyclotomic $\Zp$-extension, which we denote by ${\rm Col}$, is given by the one-variable unbalanced \padic $L$-function $\cL_\bdsF^\bdsf(Q_1,\bfv S-1,\bfv S-1)$ (\cite[Theorem 5.3]{DR17JAMS}). On the other hand, in virtue of \propref{P:factor} combined with a result of Vatsal on the non-vanishing of central $L$-values with anticyclotoic twist, we conclude that when $\chi$ is sufficiently ramified,  
\[{\rm Col}(\kappa(f,\bdsg\bdsh))=\cL_\bdsF^\bdsf(Q_1,\bfv S-1,\bfv S-1)=\Theta_{f/K}(S)\cdot (\text{non-zero constant}). \]
In a work joint with F. Castella \cite{CsH18}, we will make use of the explicit version of the above equation to prove first cases of a conjecture of Darmon-Rotger on the non-vanishing of generalized Kato classes.
\end{Remark}
\subsection{An improved $p$-adic $L$-function}
Let \[Z=(1+T_1)^{-1}(1+T_2)(1+T_3)\in\cR_0.\] 
In this subsection, we introduce a \emph{two-variable} improved $p$-adic $L$-function $\sL^\imp_\bdsF\in\cR/(Z)$ attached to $\bdsF=(\bdsf,\bdsg,\bdsh)$ a triple of primitive Hida families as in \subsecref{S:levelraising}. %Set  \[\wtd\Fil_\bdsf\bfV^\dagger:=\Fil_\bdsf^+\bfV^\dagger+V_\bdsf\ot\Fil^0 V_\bdsg\ot \Fil^0 V_\bdsh.\]
%Define \[\cE_p(\wtd\Fil_\bdsf\bfV^\dagger_\ulQ):=\frac{1}{\varepsilon(\WD_p(\Fil^+_\bdsf\bfV^\dagger_\ulQ))}\cdot \frac{1}{L(\bfV^\dagger_\ulQ/\wtd \Fil_\bdsf\bfV^\dagger_\ulQ,0)^2}.\]
To lighten the notation, we let $\al_p(?):=\bfa(p,?)$ be the $\bfU_p$-eigenvalues of Hida families $?\in\stt{\bdsf,\bdsg,\bdsh}$. Then we have the following
\begin{prop}\label{P:improved}  Suppose that $\brchf^{-1}\Om^{1+a}$ is unramified at $p$. Then
there exists an improved $p$-adic $L$-function $\cL_{\bdsF}^\imp\in \cR/(Z)$ such that
\[\cL_{\bdsF}^\bdsf\pmod{Z}=(1-\frac{\brchf\Om^{-a-1}(p)\al_p(\bdsg)\al_p(\bdsh)}{\al_p(\bdsf)})\cdot \cL^\imp_{\bdsF}.\]
Moreover, for $\ulQ=(\Qx,\Qy,\Qz)\in\frakX^\bdsf_\cR$ with $Z(\ulQ)=0$, we have
\[(\cL_{\bdsF}^\imp(\ulQ))^2=\frac{L(1/2,\itPi_\ulQ)}{(\sqrt{-1})^{2k_\Qx}\Omega_{\bdsf_\Qx}^2}\cdot\cE^\imp(\itPi_{\ulQ,p}),\]
where
\[\cE^\imp(\itPi_{\ulQ,p}):=\frac{1}{\varepsilon(\WD_p(\Fil^+_\bdsf\bfV_\ulQ^\dagger))}\cdot\frac{L_p(\Fil^+_\bdsf\bfV_\ulQ^\dagger,s) L_p(U'_\ulQ,s)^2}{L_p(\bfV_\ulQ^\dagger/\Fil^+_\bdsf\bfV_\ulQ^\dagger,s)L_p(\bfV_\ulQ^\dagger,s)}|_{s=0},\]
where $U'_\ulQ=(\Fil^0V_{\bdsf_\Qx})^\vee\ot\Fil^0V_{\bdsg_\Qy}\ot\Fil^0V_{\bdsh_\Qz}\ot\brchf^{-1}\Om^{a+1}$.
\end{prop}
\begin{proof} Let $\bdsG:=\bdsg^\star\cdot \bdsh^\star\pmod{Z}$. Then the argument in \lmref{L:bdsH1} shows that 
\[\bdsG\in \bfS(N,\psi_{1,(p)}\ol{\psi_1}^{(p)},\bfI_1)\wh\ot_{\bfI_1}\cR/(Z),\]
so we can define $\bdsG^{\rm aux}$ as $\bdsH^{\rm aux}$ in \eqref{E:bdsH}, replacing $\bdsH$ by  $\bdsG$ and define $\sL_\bdsF^\imp$ by \[\sL_\bdsF^\imp:=\bfa(1,1^*_{\breve\bdsf}(\Tr_{N/\condf}(\bdsG^{\rm aux})))\in\cR/(Z).\]
In what follows, we shall keep the notation in \subsecref{SS:Ichino}. For each $\ulQ=(\Qx,\Qy,\Qz)\in\frakX^\bdsf_\cR$ with  $\cR(\ulQ)=0$, \ie $k_\Qx=k_\Qy+k_\Qz$ and $\ep_\Qx=\ep_\Qy\ep_\Qz$, let $F=(f,g,h)=(\bdsf_\Qx,\bdsg_\Qy,\bdsh_\Qz)$. Applying the proof of \propref{P:inter1} to the improved $p$-adic $L$-function $\sL_{\bdsF}^\imp$, we obtain
\beq\label{E:4.8}
\frac{\sL_{\bdsF}^\bdsf(\ulQ)}{I(\rho(\bft_n)\phi_F^\star)}=\frac{\sL_{\bdsF}^\imp(\ulQ)}{I(\rho(\bft_n)\phi_F^{\star,\imp})},
\eeq
where $\phi_F^{\star,\imp}:=\rho(\cJ_\infty)\varphi^\star_f\boxtimes\varphi^\star_g\boxtimes \varphi_h^\star$, and
$I(\rho(\bft_n)\phi_F^{\star,\imp})$ is the global trilinear period integral 
\[I(\rho(\bft_n)\phi_F^{\star,\imp})=\int_{\A^\x\GL_2(\Q)\bksl \GL_2(\A)}\phi_F^{\star,\imp}(xt_n,x,x)\rmd^\tau x.\] 
Letting $\al_1=\om_{F,p}^{-1/2}(p)\bfa(p,f)p^{1-\frac{k_\Qx}{2}}$, $\al_2=\bfa(p,g)p^{1-\frac{k_\Qy}{2}}$ and $\al_3=\bfa(p,h)p^{1-\frac{k_\Qz}{2}}$, one verifies that
\[\phi_F^\star=1\ot 1\ot (1-\abs{p}\al_3\cdot \pi_h(\pDII{p^{-1}}{1}))\cdot\phi_F^{\star,\imp}\]
and that
\[I(\rho(\bft_n)\phi^\star)= I(\rho(\bft_n)\phi^{\star,\imp}_F)-\abs{p}^\frac{3}{2}\al_1\al_2\al_3\cdot I(\rho(\bft_{n-1})\phi^{\star,\imp}_F)\]
for $n$ sufficiently large. From the above equation, \eqref{E:4.8} and \propref{P:inter1}, we can deduce that  \[\sL_{\bdsF}^\bdsf(\ulQ)=(1-\abs{p}^\onehalf\om_{g,p}\om_{h,p}(p)\al^{-1}_1\al_2\al_3)\cdot \sL_\bdsF^\imp(\ulQ).\]
Now as in \thmref{T:main.7}, we apply the above construction to a suitable Dirichlet twist $\bdsF'$ of $\bdsF$ so that $\bdsF'$ satisfies the minimal hypothesis and define $\cL_{\bdsF}^\imp:=\sL_{\bdsF'}^\imp\cdot \sqrt{\psi_{1,(p)}(-1)(-1)I_{\bdsF'}}^{-1}$. Then $\cL_{\bdsF}^\imp$ clearly does the job.
%\[\frac{I(\rho(\bft_n)\phi^{\star,\imp})}{\al_1\beta_1^{-1}\Abs(p^n)},\quad \al_1=\bfa(p,f)\abs{p}^\frac{k_\Qx-1}{2}\om_F^{-1/2}, \beta_1=\al_1^{-1}\eta_1(p)\]
%is independent of $n$, where
%$\eta_1:=\om_f\om_F^{-/12}=\psi_{1\A}^{-1}\Om_\A^{a+1}$ is the adelization of $\brchf^{-1}\Om^{a+1}$, and hence
%\[\sL_\bdsF^\imp(\ulQ)=C_f\cdot I(\rho(\bft_n)\phi^{\star,\imp}_F). \]

To see the interpolation formula, applying the proof of \corref{C:Ichino.imb} and \thmref{T:main.7} to $\sL_\bdsf^\imp$, we can show that
\[(\sL_{\bdsF}^\imp(\ulQ))^2=\frac{L(1/2,\itPi_\ulQ)}{\Omega_{\bdsf_\Qx}^2}\cdot \sI_{\itPi_\ulQ,p}^{\imp}\cdot \prod_{\pmq\divides N}I_{\bdsF,\pmq}(\ulQ)\cdot\prod_{\ell\in\Sigma_{\rm exc}}(1+\ell^{-1}),\]
where $\sI_{\itPi_\ulQ,p}^{\imp}$ is the improved \padic zeta integral defined in \remref{R:improved}. Then the interpolation formula follows from the expression of $\sI_{\itPi_\ulQ,p}^{\imp}$ given in \remref{R:improved}.
\end{proof}
%\begin{Remark}The method used here can be also applied to the construction of improved \padic $L$-functions in the balanced case. \end{Remark}

\subsection{An alternative proof of anticyclotomic exceptional zero conjecture}We return to the setting in \subsecref{SS:2.8} and \subsecref{SS:3.8}. Suppose that $f=f^\circ$ is the newform attached to an elliptic curve $E_{/\Q}$ of conductor $Np$ with split multiplicative reduction at $p$. For a ring class character $\chi$, put
\[\cL_p(f/K\ot\chi,s):=\Theta_{f/K,\chi}(\bfv^s-1)\text{ for }s\in\Zp.\]
Then we know $\cL_p(f/K,0)=0$. Write $\frakp^{h_K}=\uf\cO_K$ with $\uf\in K^\x$ and let $\log_{\uf/\ol{\uf}}\colon\Cp^\x\to\Cp$ be the \padic logarithm such that $\log_{\uf/\ol{\uf}}(\uf/\ol{\uf})=0$. We provide a Greenberg-Stevens style proof of the anityclotomic exceptional zero conjecture for elliptic curves that was proved in \cite{BD99Duke}.
\begin{thm}[Bertolini and Darmon] Let $q_E$ be the Tate period of $E$. Then we have
\[\frac{d\cL_p(f/K,s)}{ds}|_{s=0}=\pm\frac{\log_{\uf/\ol{\uf}}(q_E)}{\Ord_p(q_E)}\cdot \sqrt{\frac{L(E/K,1)u_K^2D_K^{1/2}}{4\pi^2\Omega_{f,N^-}}}.\]
\end{thm}
\begin{proof}By \cite[Theorem D]{CH17Crelle}, we can choose a ring class character $\chi$ of $\ell$-power conductor with $\ell\ndivides Np$ split in $K$ such that  $\cL_p(f/K\ot\chi^2,0)\not =0$. Let $\bdsf=\bdsf(T)\in \Zp\powerseries{T}\powerseries{q}$ be the primitive Hida family passing through $f$ at the weight two specialization $T=\bfu^2-1$ with $\bfu:=1+p$. 
Let $\bdsF=(\bdsf(T),\bftheta_{\chi}(S_2),\bftheta_{\chi^{-1}}(S_3))$ be the triple of Hida families and let $\cL_{\bdsF}^\bdsf=\cL_{\bdsF}^{\bdsf}(T,S_2,S_3)$ be the unbalanced $p$-adic $L$-function attached to $\bdsF$ in \thmref{T:main.7}. Fixing a lift $\wtd\cL_{\bdsF}^\imp\in\cR$ of 
$\cL_{\bdsF}^\imp\pmod{Z}$, we define analytic functions on $\Zp^3$:
\begin{align*}\cL_p(k_1,k_2,k_3):=&\cL_{\bdsF}^\bdsf(\bfu^{k_1}-1,\bfv^{k_2}-1,\bfv^{k_3}-1);\\
\cL_p^\imp(k_1,k_2,k_3):=&\wtd\cL_\bdsF^\imp(\bfu^{k_1}-1,\bfv^{k_2}-1,\bfv^{k_3}-1)\end{align*}
for $(k_1,k_2,k_3)\in \Zp^3$. Let $a_\bdsf(k_1)=\al_p(\bdsf)(\bfu^{k_1}-1)$, 
\[a_\bdsg(k_2)=\al_p(\bdsg)(\bfv^{k_2}-1)=\chi(\Frob_{\ol{\frakp}})\bfv^{l(\Frob_{\ol{\frakp}})(1-k_2)};\quad a_\bdsh(k_3)=\chi^{-1}(\Frob_{\ol{\frakp}})\bfv^{l(\Frob_{\ol{\frakp}}))(1-k_3)}.\]
It is clear that
\[a_\bdsf(2)=1;\quad a_\bdsg(1)a_\bdsg(1)=1.\]
By \propref{P:improved}, there exists $\cH(T_1,S_1,S_2)\in \cR$ and $H(k_1,k_2,k_3)=\cH(\bfu^{k_1}-1,\bfv^{k_2}-1,\bfv^{k_3}-1)$ such that 
\beq\label{E:3.app}
\cL_p(k_1,k_2,k_3)=(1-\frac{a_\bdsg(k_2)a_\bdsh(k_3)}{a_\bdsf(k_1)})\cdot\cL_p^{\imp}(k_1,k_2,k_3)+H(k_1,k_2,k_3)\cdot (\bfu^{-k_1+k_2+k_3}-1)
\eeq
(the nebentypus $\brchf=1$, $\brchg=\brchh=\Om^{-1}$ and $a=-1$). We may assume $L(f/K,1)\not =0$, so the root numbers of $f$ and its quadratic twist $f\ot\tau_{K/\Q}$ are $+1$. This in turns implies that the root numbers of $\bdsf$ and $\bdsf\ot\tau_{K/\Q}$ are $-1$, and hence the one-variable Iwasawa function $\cL_p(k_1,1,1)$ vanishes identically. Taking the derivative with respect to $k_1$ on the both sides of \eqref{E:3.app}, we find that 
\[0=\frac{\pd\cL_p}{\pd k_1}(2,1,1)= a_\bdsf'(2)\cdot \cL_p^{\imp}(2,1,1)-H(2,1,1)\cdot\log_p\bfu.\]
This implies that 
\[H(2,1,1)\cdot \log_p\bfu=a_\bdsf'(2)\cdot\cL_p^{\imp}(2,1,1);\]
By an elementary calculation and a theorem of Greenberg-Stevens \cite[Theorem 3.18]{GS93}, \[a_\bdsg'(1)=\frac{\log_p\ol{\uf}}{h_K};\quad a_\bdsf'(2)=-\onehalf\cdot\frac{\log_p(q_E)}{\Ord_p(q_E)}.\]
It follows that
\begin{align*}\frac{\pd\cL_p}{\pd k_2}(2,1,1)=&\frac{\pd\cL_p}{\pd k_3}(2,1,1)
=(-\frac{\log_p\ol{\uf}}{h_K})\cdot\sL^{\imp}(2,1,1)+H(2,1,1)\cdot \log_p\bfu\\
=&(-\frac{\log_p\ol{\uf}}{h_K}-\onehalf \frac{\log_p(q_E)}{\Ord_p(q_E)})\cdot\cL_p^{\imp}(2,1,1).\end{align*}
By \propref{P:factor}, we have
\begin{align*}\cL_p(2,k_2,k_3)=&v(k_2,k_3)\cdot \cL_p(f/K,\frac{k_2+k_3-2}{2})\cL_p(f\ot\chi^2,\frac{k_2-k_3}{2})
\end{align*}
for some nowhere vanishing analytic function $v(k_2,k_3)$. Letting $v=v(1,1)\not =0$, we find that
\begin{align*}v\cdot\cL_p'(f/K,0)\cL_p(f/K\ot\chi^2,0)=&\frac{\pd\cL_p}{\pd k_2}(2,1,1)+\frac{\pd\cL_p}{\pd k_3}(2,1,1)\\
=&(-1)(\frac{\log_p (q_E)}{\Ord_p(q_E)}+\frac{2\log_p\ol{\uf}}{h_K})\cdot\cL_p^\imp(2,1,1)\\
=&(-1)\frac{\log_{\uf/\ol{\uf}}(q_E)}{\Ord_p(q_E)}\cdot\cL_p^\imp(2,1,1).
\end{align*}
On the other hand, the interpolation formula in \propref{P:improved} shows that\[\cL_p^\imp(2,1,1)^2=v^2\cdot (2\pi)^{-2}\frac{L(f/K,1)}{\Omega_{f,N^-}}\cdot u_K^2\sqrt{D_K}\cdot\cL_p(f/K\ot\chi^2,0)^2.\]
Combining the above two equations, we obtain
\[(\cL_p'(f/K,0))^2=\left(\frac{\log_{\uf/\ol{\uf}}(q_E)}{\Ord_p(q_E)}\right)^2\cdot\frac{L(f/K,1)}{4\pi^2\Omega_{f,N^-}}\cdot u_K^2\sqrt{D_K},\]
and the theorem follows.
\end{proof}

\bibliographystyle{amsalpha}
\bibliography{mybib_triple}

\providecommand{\bysame}{\leavevmode\hbox to3em{\hrulefill}\thinspace}
\providecommand{\MR}{\relax\ifhmode\unskip\space\fi MR }
\providecommand{\MRhref}[2]{%
  \href{http://www.ams.org/mathscinet-getitem?mr=#1}{#2}
}
\providecommand{\href}[2]{#2}
\begin{thebibliography}{Coa89b}

\bibitem[BD96]{BD96}
M.~Bertolini and H.~Darmon, \emph{Heegner points on {M}umford-{T}ate curves},
  Invent. Math. \textbf{126} (1996), no.~3, 413--456.

\bibitem[BD99]{BD99Duke}
Massimo Bertolini and Henri Darmon, \emph{{$p$}-adic periods, {$p$}-adic
  {$L$}-functions, and the {$p$}-adic uniformization of {S}himura curves}, Duke
  Math. J. \textbf{98} (1999), no.~2, 305--334.

\bibitem[BD07]{BD07}
\bysame, \emph{Hida families and rational points on elliptic curves}, Invent.
  Math. \textbf{168} (2007), no.~2, 371--431. \MR{2289868}

\bibitem[BDP13]{BDP13}
M.~Bertolini, H.~Darmon, and K.~Prasanna, \emph{Generalized {H}eegner cycles
  and {$p$}-adic {R}ankin {$L$}-series}, Duke Math. J. \textbf{162} (2013),
  no.~6, 1033--1148.

\bibitem[Bum97]{Bump97Grey}
D.~Bump, \emph{Automorphic forms and representations}, Cambridge Studies in
  Advanced Mathematics, vol.~55, Cambridge University Press, Cambridge, 1997.

\bibitem[Cas73]{Casselman73MA}
W.~Casselman, \emph{On some results of {A}tkin and {L}ehner}, Math. Ann.
  \textbf{201} (1973), 301--314.

\bibitem[CC18]{ChenYao16}
S.-Y. Chen and Yao Cheng, \emph{On {D}eligne's conjecture for central values of
  certain automorphic ${L}$-functions on {${\rm GL}(3)\times {\rm GL}(2)$}},
  arXiv:1806.09767.

\bibitem[CH18a]{CsH18}
Francesc Castella and Ming-Lun Hsieh, \emph{On the non-vanishing of generalized
  {K}ato classes for elliptic curves of rank two}, arXiv:1809.09066.

\bibitem[CH18b]{CH17Crelle}
Masataka Chida and Ming-Lun Hsieh, \emph{Special values of anticyclotomic
  {$L$}-functions for modular forms}, J. Reine Angew. Math. \textbf{741}
  (2018), 87--131.

\bibitem[Coa89a]{Coates89Bourbaki}
John Coates, \emph{On {$p$}-adic {$L$}-functions}, Ast\'erisque (1989),
  no.~177-178, Exp.\ No.\ 701, 33--59, S\'eminaire Bourbaki, Vol. 1988/89.
  \MR{1040567}

\bibitem[Coa89b]{Coates89II}
\bysame, \emph{On $p$-adic {$L$}-functions attached to motives over {$\Q$}
  {II}}, Boletim da Sociedade Brasileira de Matem{\'a}tica - Bulletin/Brazilian
  Mathematical Society \textbf{20} (1989), no.~1, 101--112.

\bibitem[Col16]{Collins16}
Dan Collins, \emph{Anticyclotomic $p$-adic {$L$}-functions and {I}chino's
  formula}, preprint. arXiv:1612.06948.

\bibitem[CPR89]{CP89}
John Coates and Bernadette Perrin-Riou, \emph{On {$p$}-adic {$L$}-functions
  attached to motives over {${\bf Q}$}}, Algebraic number theory, Adv. Stud.
  Pure Math., vol.~17, Academic Press, Boston, MA, 1989, pp.~23--54.

\bibitem[Del79]{Deligne79}
P.~Deligne, \emph{Valeurs de fonctions {$L$} et p\'eriodes d'int\'egrales},
  Automorphic forms, representations and {$L$}-functions ({P}roc. {S}ympos.
  {P}ure {M}ath., {O}regon {S}tate {U}niv., {C}orvallis, {O}re., 1977), {P}art
  2, Proc. Sympos. Pure Math., XXXIII, Amer. Math. Soc., Providence, R.I.,
  1979, With an appendix by N. Koblitz and A. Ogus, pp.~313--346.

\bibitem[Dim14]{Dim14}
Mladen Dimitrov, \emph{On the local structure of ordinary {H}ecke algebras at
  classical weight one points}, Automorphic forms and {G}alois representations.
  {V}ol. 2, London Math. Soc. Lecture Note Ser., vol. 415, Cambridge Univ.
  Press, Cambridge, 2014, pp.~1--16.

\bibitem[DLR15]{DLR15}
Henri Darmon, Alan Lauder, and Victor Rotger, \emph{Stark points and {$p$}-adic
  iterated integrals attached to modular forms of weight one}, Forum Math. Pi
  \textbf{3} (2015), e8, 95. \MR{3456180}

\bibitem[DR14]{DR14ASEN}
Henri Darmon and Victor Rotger, \emph{Diagonal cycles and {E}uler systems {I}:
  {A} {$p$}-adic {G}ross-{Z}agier formula}, Ann. Sci. \'Ec. Norm. Sup\'er. (4)
  \textbf{47} (2014), no.~4, 779--832.

\bibitem[DR17]{DR17JAMS}
\bysame, \emph{Diagonal cycles and {E}uler systems {II}: {T}he {B}irch and
  {S}winnerton-{D}yer conjecture for {H}asse-{W}eil-{A}rtin {$L$}-functions},
  J. Amer. Math. Soc. \textbf{30} (2017), no.~3, 601--672.

\bibitem[DT94]{DiaomondTaylor94}
F.~Diamond and R.~Taylor, \emph{Nonoptimal levels of mod {$l$} modular
  representations}, Invent. Math. \textbf{115} (1994), no.~3, 435--462.

\bibitem[FO12]{OF12Crelle}
Olivier Fouquet and Tadashi Ochiai, \emph{Control theorems for {S}elmer groups
  of nearly ordinary deformations}, J. Reine Angew. Math. \textbf{666} (2012),
  163--187. \MR{2920885}

\bibitem[Fon94]{F94}
Jean-Marc Fontaine, \emph{Repr\'esentations {$l$}-adiques potentiellement
  semi-stables}, Ast\'erisque (1994), no.~223, 321--347, P\'eriodes $p$-adiques
  (Bures-sur-Yvette, 1988). \MR{1293977}

\bibitem[GJ78]{GJ78}
Stephen Gelbart and Herv{\'e} Jacquet, \emph{A relation between automorphic
  representations of {${\rm GL}(2)$} and {${\rm GL}(3)$}}, Ann. Sci. \'Ecole
  Norm. Sup. (4) \textbf{11} (1978), no.~4, 471--542. \MR{533066}

\bibitem[Gre94]{Greenberg94}
Ralph Greenberg, \emph{Iwasawa theory and {$p$}-adic deformations of motives},
  Motives ({S}eattle, {WA}, 1991), Proc. Sympos. Pure Math., vol.~55, Amer.
  Math. Soc., Providence, RI, 1994, pp.~193--223. \MR{1265554}

\bibitem[GS93]{GS93}
Ralph Greenberg and Glenn Stevens, \emph{{$p$}-adic {$L$}-functions and
  {$p$}-adic periods of modular forms}, Invent. Math. \textbf{111} (1993),
  no.~2, 407--447.

\bibitem[GS16]{GS16}
Matthew Greenberg and Marco~Adamo Seveso, \emph{Triple product $p$-adic
  {$L$}-functions for balanced weights}, arXiv:1506.05681.

\bibitem[Hel07]{Helm07}
David Helm, \emph{On maps between modular {J}acobians and {J}acobians of
  {S}himura curves}, Israel J. Math. \textbf{160} (2007), 61--117.

\bibitem[Hid85]{Hida85Inv}
Haruzo Hida, \emph{A {$p$}-adic measure attached to the zeta functions
  associated with two elliptic modular forms. {I}}, Invent. Math. \textbf{79}
  (1985), no.~1, 159--195. \MR{774534}

\bibitem[Hid88a]{Hida88AJM}
H.~Hida, \emph{Modules of congruence of {H}ecke algebras and {$L$}-functions
  associated with cusp forms}, Amer. J. Math. \textbf{110} (1988), no.~2,
  323--382. \MR{935010 (89i:11058)}

\bibitem[Hid88b]{Hida88Annals}
\bysame, \emph{On p-adic hecke algebras for {GL(2)} over totally real fields},
  Annals of Mathematics \textbf{128} (1988), 295--384.

\bibitem[Hid88c]{Hida88Fourier}
Haruzo Hida, \emph{A {$p$}-adic measure attached to the zeta functions
  associated with two elliptic modular forms. {II}}, Ann. Inst. Fourier
  (Grenoble) \textbf{38} (1988), no.~3, 1--83.

\bibitem[Hid90]{Hida90Michigan}
\bysame, \emph{{$p$}-adic {$L$}-functions for base change lifts of {${\rm
  GL}_2$} to {${\rm GL}_3$}}, Automorphic forms, {S}himura varieties, and
  {$L$}-functions, {V}ol.\ {II} ({A}nn {A}rbor, {MI}, 1988), Perspect. Math.,
  vol.~11, Academic Press, Boston, MA, 1990, pp.~93--142.

\bibitem[Hid93]{Hida93Blue}
\bysame, \emph{Elementary theory of {$L$}-functions and {E}isenstein series},
  London Mathematical Society Student Texts, vol.~26, Cambridge University
  Press, Cambridge, 1993.

\bibitem[Hid94]{Hida94Duke}
\bysame, \emph{On the critical values of {$L$}-functions of {${\rm GL}(2)$} and
  {${\rm GL}(2)\times{\rm GL}(2)$}}, Duke Math. J. \textbf{74} (1994), no.~2,
  431--529.

\bibitem[Hid06]{Hida06blue}
H.~Hida, \emph{Hilbert modular forms and {I}wasawa theory}, Oxford Mathematical
  Monographs, The Clarendon Press Oxford University Press, Oxford, 2006.

\bibitem[Hid16]{Hida16Pune}
Haruzo Hida, \emph{Arithmetic of adjoint {L}-values}, {$p$}-adic aspects of
  modular forms, World Sci. Publ., Hackensack, NJ, 2016, pp.~185--236.

\bibitem[HK91]{HK91Triple}
Michael Harris and Stephen~S. Kudla, \emph{The central critical value of a
  triple product {$L$}-function}, Ann. of Math. (2) \textbf{133} (1991), no.~3,
  605--672. \MR{1109355}

\bibitem[HT01a]{HarrisTaylor01}
M.~Harris and R.~Taylor, \emph{The geometry and cohomology of some simple
  {S}himura varieties}, Annals of Mathematics Studies, vol. 151, Princeton
  University Press, Princeton, NJ, 2001, With an appendix by Vladimir G.
  Berkovich.

\bibitem[HT01b]{HT01Triple}
Michael Harris and Jacques Tilouine, \emph{{$p$}-adic measures and square roots
  of special values of triple product {$L$}-functions}, Math. Ann. \textbf{320}
  (2001), no.~1, 127--147. \MR{1835065}

\bibitem[Hu17]{HuYueKe17}
YueKe Hu, \emph{The subconvexity bound for triple product {$L$}-functions in
  level aspect}, Amer. J. Math. \textbf{139} (2017), no.~1, 215--259.

\bibitem[Hun17]{PC17}
Pin-Chi Hung, \emph{On the non-vanishing {${\rm mod}\, \ell$} of central
  {$L$}-values with anticyclotomic twists for {H}ilbert modular forms}, J.
  Number Theory \textbf{173} (2017), 170--209. \MR{3581914}

\bibitem[Ich08]{Ichino08Duke}
Atsushi Ichino, \emph{Trilinear forms and the central values of triple product
  {$L$}-functions}, Duke Math. J. \textbf{145} (2008), no.~2, 281--307.

\bibitem[II10]{II10GAFA}
Atsushi Ichino and Tamutsu Ikeda, \emph{On the periods of automorphic forms on
  special orthogonal groups and the {G}ross-{P}rasad conjecture}, Geom. Funct.
  Anal. \textbf{19} (2010), no.~5, 1378--1425.

\bibitem[Ish17]{Ishikawa}
Isao Ishikawa, \emph{On the construction of twisted triple product $p$-adic
  {$L$}-functions}, 2017, Thesis (Ph.D.)--Kyoto University.

\bibitem[Jac72]{Jacquet72Part2}
H.~Jacquet, \emph{Automorphic forms on {${\rm GL}(2)$}. {P}art {II}}, Lecture
  Notes in Mathematics, Vol. 278, Springer-Verlag, Berlin, 1972.

\bibitem[JL70]{JacquetLanglands70}
H.~Jacquet and R.~P. Langlands, \emph{Automorphic forms on {${\rm GL}(2)$}},
  Lecture Notes in Mathematics, Vol. 114, Springer-Verlag, Berlin, 1970.

\bibitem[LV11]{LongoVigni11}
Matteo Longo and Stefano Vigni, \emph{Quaternion algebras, {H}eegner points and
  the arithmetic of {H}ida families}, Manuscripta Math. \textbf{135} (2011),
  no.~3-4, 273--328.

\bibitem[Miy06]{Miyake06book}
Toshitsune Miyake, \emph{Modular forms}, english ed., Springer Monographs in
  Mathematics, Springer-Verlag, Berlin, 2006, Translated from the 1976 Japanese
  original by Yoshitaka Maeda. \MR{2194815}

\bibitem[MV10]{MV10}
Philippe Michel and Akshay Venkatesh, \emph{The subconvexity problem for {${\rm
  GL}_2$}}, Publ. Math. Inst. Hautes \'Etudes Sci. (2010), no.~111, 171--271.

\bibitem[MW86]{MW86}
B.~Mazur and A.~Wiles, \emph{On {$p$}-adic analytic families of {G}alois
  representations}, Compositio Math. \textbf{59} (1986), no.~2, 231--264.

\bibitem[NPS14]{NPS14}
Paul~D. Nelson, Ameya Pitale, and Abhishek Saha, \emph{Bounds for
  {R}ankin-{S}elberg integrals and quantum unique ergodicity for powerful
  levels}, J. Amer. Math. Soc. \textbf{27} (2014), no.~1, 147--191.

\bibitem[Oht95]{Ohta95}
Masami Ohta, \emph{On the {$p$}-adic {E}ichler-{S}himura isomorphism for
  {$\Lambda$}-adic cusp forms}, J. Reine Angew. Math. \textbf{463} (1995),
  49--98. \MR{1332907}

\bibitem[Oht00]{Ohta00}
\bysame, \emph{Ordinary {$p$}-adic \'etale cohomology groups attached to towers
  of elliptic modular curves. {II}}, Math. Ann. \textbf{318} (2000), no.~3,
  557--583. \MR{1800769}

\bibitem[Orl87]{Orl87}
Tobias Orloff, \emph{Special values and mixed weight triple products (with an
  appendix by {D}on {B}lasius)}, Invent. Math. \textbf{90} (1987), no.~1,
  169--188.

\bibitem[Pra90]{Prasad90}
Dipendra Prasad, \emph{Trilinear forms for representations of {${\rm GL}(2)$}
  and local {$\epsilon$}-factors}, Compositio Math. \textbf{75} (1990), no.~1,
  1--46.

\bibitem[PSR87]{PSR87}
I.~Piatetski-Shapiro and Stephen Rallis, \emph{Rankin triple {$L$} functions},
  Compositio Math. \textbf{64} (1987), no.~1, 31--115. \MR{911357}

\bibitem[PW11]{PW11CoM}
R.~Pollack and T.~Weston, \emph{On anticyclotomic {$\mu$}-invariants of modular
  forms}, Compos. Math. \textbf{147} (2011), no.~5, 1353--1381.

\bibitem[Ram00]{Dinakar00}
Dinakar Ramakrishnan, \emph{Modularity of the {R}ankin-{S}elberg {$L$}-series,
  and multiplicity one for {${\rm SL}(2)$}}, Ann. of Math. (2) \textbf{152}
  (2000), no.~1, 45--111. \MR{1792292}

\bibitem[Rib84]{Ribet84Ihara}
Kenneth~A. Ribet, \emph{Congruence relations between modular forms},
  Proceedings of the {I}nternational {C}ongress of {M}athematicians, {V}ol.\ 1,
  2 ({W}arsaw, 1983), PWN, Warsaw, 1984, pp.~503--514. \MR{804706}

\bibitem[Sch02]{Schmidt02RJ}
Ralf Schmidt, \emph{Some remarks on local newforms for {$\rm GL(2)$}}, J.
  Ramanujan Math. Soc. \textbf{17} (2002), no.~2, 115--147.

\bibitem[SU06]{SU06}
Christopher Skinner and Eric Urban, \emph{Sur les d\'eformations {$p$}-adiques
  de certaines repr\'esentations automorphes}, J. Inst. Math. Jussieu
  \textbf{5} (2006), no.~4, 629--698.

\bibitem[Tat79]{Tate79Corvallis}
J.~Tate, \emph{Number theoretic background}, Automorphic forms, representations
  and {$L$}-functions ({P}roc. {S}ympos. {P}ure {M}ath., {O}regon {S}tate
  {U}niv., {C}orvallis, {O}re., 1977), {P}art 2, Proc. Sympos. Pure Math.,
  XXXIII, Amer. Math. Soc., Providence, R.I., 1979, pp.~3--26. \MR{546607}

\bibitem[Tay04]{Taylor04ICM}
Richard Taylor, \emph{Galois representations}, Ann. Fac. Sci. Toulouse Math.
  (6) \textbf{13} (2004), no.~1, 73--119. \MR{2060030}

\bibitem[Wal85]{Wald85}
J.-L. Waldspurger, \emph{Sur les valeurs de certaines fonctions {$L$}
  automorphes en leur centre de sym\'etrie}, Compositio Math. \textbf{54}
  (1985), no.~2, 173--242.

\bibitem[Wil88]{Wiles88}
A.~Wiles, \emph{On ordinary {$\lambda$}-adic representations associated to
  modular forms}, Invent. Math. \textbf{94} (1988), no.~3, 529--573.

\bibitem[Wil95]{Wiles95}
\bysame, \emph{Modular elliptic curves and {F}ermat's last theorem}, Ann. of
  Math. (2) \textbf{141} (1995), no.~3, 443--551.

\end{thebibliography}
\end{document}